\documentclass[11pt]{article}

\usepackage{latexsym}
\usepackage{epsfig}
\usepackage{amsfonts,amssymb,amsmath,amsthm}
\usepackage{graphicx,multicol,color}
\usepackage{stmaryrd}
\usepackage{float}
\usepackage[utf8]{inputenc}
\usepackage[T1]{fontenc}
\usepackage{caption}
\usepackage{natbib}
\usepackage{hyperref}
\usepackage{enumitem}

\usepackage{bold-extra}
\usepackage{color}                                            
\usepackage{array}                                            
\usepackage{longtable}                                        
\usepackage{calc}                                             
\usepackage{multirow}                                         
\usepackage{hhline}                                           
\usepackage{ifthen}

\usepackage[margin=2cm,footskip=1cm]{geometry}

\usepackage{authblk}
\usepackage{lmodern}
\usepackage{xspace}

\let\oldabstract\abstract
\makeatletter
\renewcommand\abstract{%
  \providecommand\keywords{\par\medskip\noindent\textit{Keywords:}\xspace}
  \oldabstract\noindent\ignorespaces}
\makeatother

\usepackage{algorithmic, algorithm}

\newtheorem{theo}{Theorem}

\newtheorem{prop}[theo]{Proposition}

\newtheorem{lem}[theo]{Lemma}
\newtheorem{cor}[theo]{Corollary}

\newcommand\R{\mathbb{R}}
\newcommand\N{\mathbb{N}}

\newcommand\E{\mathbb{E}}
\renewcommand\P{\mathbb{P}}

\newcommand\F{\mathcal{F}}

\newcommand\varep{\varepsilon}

\newcommand\var{\mathbb{V}}

\newcommand\li{\lim_{n \rightarrow +\infty}\limits}

\newcommand{\1}{{\mathbf 1}}

\newcommand{\one}{\ifmmode {\sf 1}\hspace{-.26em}{\sf
l}\hspace{-.35em}{\sf \_} \else ${\sf 1}\hspace{-.26em}{\sf
l}\hspace{-.35em}{\sf \_}$ \fi}

\newcommand\Et{E}

\renewcommand{\theequation}{\arabic{equation}}
\renewcommand{\thetheo}{\arabic{theo}}
\renewcommand{\Box}{\mbox{\rule{1ex}{1ex}}}

\newcounter{exno}
\newcommand{\ex}[1]{\refstepcounter{exno}\label{#1}}

\title{Spatial birth-death-move processes : basic properties and estimation of their intensity functions}

\author[1]{Fr\'ed\'eric Lavancier}
\author[2]{Ronan Le Gu\'evel}
\affil[1]{LMJL, BP 92208,  2 Rue de la Houssini\`ere, F-44322 Nantes Cedex 03, France, \texttt{frederic.lavancier@univ-nantes.fr}}
\affil[2]{Univ Rennes, CNRS, IRMAR - UMR 6625, F-35000 Rennes, France,

\texttt{ronan.leguevel@univ-rennes2.fr}}

\begin{document}

\date{}

\maketitle

\begin{abstract} 

Many spatio-temporal data record the time of birth and death of individuals, along with their spatial trajectories during their lifetime, whether through continuous-time observations or discrete-time observations. Natural applications include epidemiology, individual-based modelling in ecology, spatio-temporal dynamics observed in bio-imaging, and computer vision.  The aim of this article is to estimate in this context the birth and death intensity functions, that depend in full generality on the current spatial configuration of all alive individuals.  While the temporal evolution of the  population size is a simple birth-death process, observing the lifetime and trajectories of all individuals calls for a new paradigm. To formalise this framework, we introduce spatial birth-death-move processes, where the birth and death dynamics depends on the current spatial configuration of the population and where individuals can move during their lifetime according to a continuous Markov process with possible interactions.
We consider non-parametric kernel estimators of their birth and death intensity functions. The setting is original because each observation in time belongs to a non-vectorial, infinite dimensional space and the dependence between observations is barely tractable. We prove the consistency of the estimators in presence of  continuous-time and discrete-time observations, under fairly simple conditions. We moreover discuss how we can take advantage in practice of structural assumptions made on the intensity functions and we explain how data-driven bandwidth selection can be conducted, despite the unknown (and sometimes undefined) second order moments of the estimators. We  finally apply our statistical method to the analysis of the spatio-temporal dynamics of proteins involved in exocytosis in cells, providing new insights on this complex mechanism.

 \end{abstract}

\section{Introduction}

Simple birth-death processes have a long history, ever since at least \cite{Feller1939} and \cite{Kendall49}. They constitute the basic model to explain the temporal evolution of the size of a population. However in many applications, we not only observe the size of population over time, but we have access to the spatio-temporal information of all individuals of the population, whether through continuous-time observations or discrete-time observations. 
This happens in epidemiology \citep{masuda2017} and ecology \citep{pommerening2019} where ``individuals'' may represent animals, plants or infected subjects, and a birth can represent  the biological birth, the time of appearance in some region or the time of infection, for instance. In these applications, the birth time and death time of all individuals are observed, along with their trajectories (e.g. their spatial displacement or their growth) during their lifetime.  This context also occurs  in bio-imaging: the dataset that we will study later for illustration concerns the dynamics of proteins involved in the exocytosis mechanisms in cells. For biological reasons, these proteins are visible at some time point (their birth) and disappear some time later (their death), their ``lifetime'' being closely related  to their activity during the exocytosis process. Meanwhile, they follow some motion in the cell. As a last example, this kind of dynamics is also encountered in computer vision  \citep{wang2002}, where new objects (like flying birds or snowflakes) appear at some time point in a video, then follow some motions before disappearing from the video.

Our aim in this article is to provide a non-parametric method to estimate the intensity function of births and the intensity function of deaths for this kind of data. These functions rule the waiting time before a new birth and a new death, respectively. 
For instance, if there are no interactions between the individuals and if each individual has the same constant birth rate $\beta$  over time (i.e. the same probability $\beta dt$ to give birth in the infinitesimal small interval $[t,t+dt]$ for any $t$), then the global birth intensity function of the population is just a linear function of the size of the population.
This is the standard setting of simple birth-death processes and the parametric estimation of the birth and death rates in this case has been widely studied \citep{darwin1956, wolff1965, keiding1975}. A more general approach for simple birth-death processes is to ignore the linearity assumption for the intensity functions, and only  assume that they depend on the size of the population. A natural non-parametric estimator in this case is studied in \cite{wolff1965} and \cite{reynolds1973}: for each size $n$ of the population, this estimator just counts the number of  observed births (respectively deaths) when the population has this cardinality $n$, divided by the time spent by the population at this size. However, in full generality, the intensity functions may depend on the spatial locations of the individuals, and not only on the size of the population, due to spatial competitions or/and spatial dispersions, for example. Observing the trajectories of all individuals, as in the applications we have in mind, opens the possibility to investigate this general dependence, which is the objective of this contribution. A major difficulty comes from the fact that the intensity functions now depend on the current spatial configuration of all individuals of the population, the number of which varies over time. This means that these functions are defined on the space of finite point configurations, a non-vectorial and infinite dimensional space. 
Furthermore, the underlying process becomes much more complex than a simple birth-death process and it needs to be formalised precisely.

There is a rich amount of articles dealing with the non-parametric estimation of intensity functions in other contexts. 
For temporal and spatial point processes, kernel estimators or penalised projection estimators of the intensity of events are for instance considered in \cite{diggle1985,guan2008,reynaud2003,reynaud2010}. For L\'evy processes observed at discrete times, the density of jumps is addressed by similar methods and also specific Fourier approaches in \cite{van2007,figueroa2009,comte2011} to cite a few.  These works can be viewed as generalisations of the standard problem of non-parametric estimation of a density \citep{silverman1986}. 
A common feature is also that the space in which the considered processes evolve (that is the state space) is $\mathbb R$ or $\mathbb R^d$.
In our case, the state space is more general and the intensity depends on the current value of the process, which makes the estimation challenge closer to a non-parametric regression problem than to  a density estimation problem.
Even if non-parametric regression estimators are nowadays routinely used  \citep{hardle1990}, even for functional data (see \cite{ferraty2006} for an introductory review), few applications concern the estimation of intensity of events in stochastic processes.  
A non-parametric estimator of the death intensity of a branching diffusion process is  considered in \cite{hopfner2002}, but in the very particular case where the diffusions evolve in $\R$ and independently of each other.
Another range of works concerns the estimation of the intensity function in survival analysis, commonly called hazard rate function. Assuming a multiplicative intensity model, a standard non-parametric estimator consists in smoothing the famous Nelson–Aalen estimator by kernel methods \citep{ramlau1983}. This approach can be extended when the intensity not only depends on time but also on  covariates, see  \cite{martinussen2007} and the references therein. In the latter case, the problem resembles a non-parametric regression problem. Generalisations to the estimation of the jump intensity of piecewise-constant processes and piecewise deterministic Markov processes are considered in \cite{azais2013,azais2014}, where the value of the process plays the role of covariate in the jump intensity. The state space of the process in these works can be quite general,  
but the jump intensity and the estimator have a very specific form which implies a decoupling between time and space, in line with the multiplicative model of Nelson–Aalen. 
The main differences of our setting with the existing literature therefore lie in two aspects: first our intensity functions depend on the actual value $X_t$ of the process, without any decoupling between time and space; and second, our state space is infinite dimensional and non-vectorial, a noticeable difference with most setting of non-parametric regression, including functional spaces (which are vectorial).

To achieve our objective, the first task is to formalise properly the stochastic process $(X_t)_{t\geq 0}$ corresponding to the dynamics described earlier. Here $X_t$ denotes the spatial configuration of all alive individuals at time $t$. In other words $X_t$ is a set of points, the number of which corresponds to the number of alive individuals at time $t$, each point representing the location of one individual. 
When the individuals do not move during their lifetime, then $(X_t)_{t\geq 0}$ is a spatial birth-death process, introduced by \cite{preston}. These processes have been used in \cite{moller1994} to describe the evolution of dunes and in \cite{Sadahiro19} to analyse the opening and closure of shops and restaurants in a city. 
The idea of allowing the individuals to move during their lifetime in this process is not new and has been considered in several ways.  
An instance is the growth-interaction process  introduced in forestry  to model the appearance of new plants, their growth and their death \citep{renshaw2001,sarkka2006,pommerening2019}. In these applications, the growth process plays the role of the motion and is deterministic, making this process part of the class of piecewise-deterministic Markov processes \citep{davis1984}. 
When each individual  evolves independently according to a stochastic diffusion, the dynamics is a particular instance of general branching diffusion processes, widely studied in probability theory \citep{skorokhod1964,athreya2012}. 
On the opposite way, when all individuals move on the same stochastic flow, we obtain a birth death process on a flow, introduced  by \cite{cinlar1991}. 
Intermediate specific interactions between the individual diffusions are considered in \cite{eisele1981} and \cite{locherbach2002likelihood}.
Finally, a completely different stochastic motion has been considered in \cite{huber2012}, where the individuals can only jump during their lifetime, which defines the so-called birth-death swap (or shift) process. 
Spatial birth-death-move processes, as we introduce them in Section~\ref{basics}, allow any  continuous Markov motion of the individuals during their lifetime, with possible interactions between them. As such, they contain all previous examples, to the exception of birth-death swap processes (the motion being not continuous). 
Apart from the motion between two jumps, whether they are a birth or a death, the dynamics of a birth-death-move process depends on the birth intensity  $\beta$ and the death intensity $\delta$, that are in full generality two functions of the current spatial configuration of $(X_t)_{t\geq 0}$, and on probability transition kernels for the births and the deaths. 
We present in Section~\ref{basics} all basic ingredients to clearly understand the full dynamics of this mechanistic model,  
without technical details. 
We then describe some examples, including the aforementioned references as particular cases, and we provide an algorithm of simulation on a finite interval.  
 A more formal presentation is given in the supplementary material along with some theoretical properties needed for the statistical study conducted in the rest of the article.

We address in Section~\ref{sec:estimation}  the non-parametric estimation of the birth and death intensity functions $\beta$ and $\delta$, and of the total intensity function $\alpha=\beta+\delta$.  We consider two settings, whether the process $(X_t)_{t\geq 0}$ is observed continuously in the time interval $[0,T]$ or at discrete time points $t_0,\dots,t_m$. In both settings we introduce a kernel estimator and prove that it is consistent under natural conditions, when $T$ tends to infinity and the discretisation step tends to 0. This consistency holds true whatever the birth and death transitions kernels are, and whatever the inter-jumps motions are.  
Our estimator typically involves a bandwidth $h_T$, which for consistency must tend to 0 but not too fast, as usual in non-parametric inference. The second order moments of our estimator are not available in closed form, making least squares cross-validation or standard plug-in methods impossible to set up to choose the bandwidth.  We however explain how $h_T$ can be selected in practice by  a partial likelihood cross-validation procedure based on the counting process defined by the cumulative number of jumps. 
We moreover discuss several strategies of estimation, depending on structural assumptions made on the intensity functions. Their performances are assessed in two simulation studies carried out in Section~\ref{simulations}.
Consider for instance the estimation of $\alpha(x)$ in the continuous case, where $x$ is a configuration of points. In a pure non-parametric approach, this estimation relies through the kernel on the distance between $x$ and the observed configurations of $(X_t)$ for $t\in[0,T]$. This  approach is ambitious given the infinite dimension of the state space of the process.  It is however consistent if $\alpha$ is regular enough, as a consequence of our theoretical results and confirmed in the simulation study. This first approach may constitute in practice a first step towards more structural hypotheses on $\alpha$. 
In a second approach, we exploit this kind of hypotheses by assuming that $\alpha(x)$ only depends on some specific $p$ features of $x$. For instance, we may assume that it depends only on the cardinality of $x$,  the common setting in standard birth-death processes, or we may assume that it depends on geometric characteristics of the configuration $x$. Under this assumption, we make our estimator depend only on the distance between the relevant characteristics of $x$ and the ones of the observations $(X_t)$ for $t\in[0,T]$. 
In this case the estimation problem reduces to a non-parametric problem in dimension $p$, improving the quality of estimation. This effect is clearly demonstrated in our simulation study and also reflected in our theoretical results where the rate of convergence becomes in such framework the usual rate of kernel estimators in dimension $p$.
A particular case is the estimator of  \cite{wolff1965} and \cite{reynolds1973} discussed before, where only observations with exactly the same cardinality as $x$ are used for the estimation of $\alpha(x)$. We instead allow to use all observations, which makes sense if $\alpha$ has some regularity properties, a situation where our estimator outperforms the previous one in our simulation study. 
This second approach is a way to question the classical assumption in birth-death processes, namely  that each individual has constant birth rate and death rate, implying that $\alpha(x)$ is a linear function in the cardinality of $x$. Testing formally this hypothesis based on our non-parametric estimator is an interesting perspective for future investigation.

We finally apply our methodology in Section~\ref{sec:data} to the analysis of the spatio-temporal dynamics of proteins involved in exocytosis.   
We observe a sequence of 1199 frames  showing the dynamics of proteins in a living cell. Each frame contains several tens   of proteins, that can be of two types (Langerin or Rab11). During the sequence, proteins appear, disappear and move, in keeping with a birth-death-move process. Classical approaches either study the trajectories of each protein independently without considering spatial interactions \citep{Briane19, pecot2018}, or study the spatial configurations of proteins at some time points without temporal insight \citep{Costes2004,Bolte2006,Lagache2015,Lavancier18}. 
In contrast, our method allows to investigate the joint spatio-temporal dynamics of all proteins, a new approach. 
We question several biological hypotheses, which imply that the birth intensity function should be constant whatever the current configuration of proteins is, and that the death intensity function could instead depend on the current activity. Our study confirms these hypotheses and further reveals the temporal interaction between Langerin and Rab11 during the exocytosis mechanisms.

Some complements on the data analysis and technical results, including the proofs, are available in the supplementary material. The codes and data are accessible in the GitHub repository at \url{https://github.com/lavancier-f/Birth-Death-Move-process}.

\section{Spatial birth-death-move processes}\label{basics}

We introduce in this section spatial birth-death-move processes. A more formal presentation and some theoretical properties are given in the supplementary material. 
After their general definition presented in the first part,
we discuss several examples including the well-known simple birth-death processes, spatial birth-death processes, growth-interaction processes,  
and other models and applications considered in the literature so far. 
We conclude by the presentation of a simulation algorithm on a finite time interval.

\subsection{Definition and notation}\label{BDM def}
The birth-death-move dynamics takes place in a space $E = \bigcup_{n=0}^{+\infty} E_n$ where the spaces $E_n$ are disjoint and $E_0$ consists of a single element, written $\varnothing$ for short, i.e. $E_0=\{\varnothing\}$. The main example we have in mind for $E_n$ is the space of point configurations in $\R^d$ with cardinality $n$, meaning that  $x\in E_n$ if and only if $x=\{x_1,\dots,x_n\}$ where $x_i\in\R^d$, $i=1,\dots,n$. The presentation of this section refers for simplicity to this situation, but more general spaces $E_n$ can be considered, as exemplified in Section~\ref{sec-examples} and detailed in the supplementary material. Accordingly, since a birth-death-move process $(X_t)_{t\geq 0}$ takes its values in $E_n$, we will often call the value of $X_t$ a configuration.  We denote by $n(x)$ the cardinality of $x$ defined by $n(x)=n$ if and only if $x\in E_n$.  Three ingredients rule the birth-death-move process  $(X_t)_{t\geq 0}$:
\begin{enumerate}
\item  The birth intensity function $\beta : E \rightarrow \R^+$ and the death intensity function $\delta : E \rightarrow \R^+$, both assumed to be continuous on $E$. The value $\beta(x)$ (resp. $\delta(x))$ governs the probability to have a birth (resp. a death) at time $t$ given that the current configuration of the process is $X_t=x$. We denote by $T_1, T_2, \dots$ the sequence of jump times (whether a birth or a death) of the process $X_t$ and we let $T_0=0$. The specific distribution of the waiting time before the next jump  is given below in \eqref{wait}. Given that a jump occurs at time $T_j$, it is a birth with probability $\beta(X_{T_j^-})/(\beta(X_{T_j^-})+\delta(X_{T_j^-}))$ and it is a death otherwise.  Here $X_{T_j^-}$ denotes the configuration of the process just before the jump. 
\item Given that a birth (resp. a death) occurs and the process is in configuration $x$ just before the jump, we need the probability transition kernel for a birth $K_{\beta}(x,.)$ (resp. for a death $K_{\delta}(x,.)$). Specifically, if $x=\{x_1,\dots,x_n\}\in E_n$,  $K_{\beta}(x,.)$ is the probability distribution that indicates where the new born point is more likely to appear in $\R^d$. Similarly if there is a death, $K_{\delta}(x,.)$ indicates which point of $x$ is more likely to disappear. 
\item Between two jumps (whether a birth or a death), we assume that $X_t$ evolves according to a continuous Markov process, which is the ``move step''. After each jump $T_j$, we denote the process that drives this motion  as $(Y_t)_{t\geq 0}$ with $Y_0=X_{T_j}$.  This means that between two jumps $T_{j}$ and $T_{j+1}$, $X_t=Y_{t-T_j}$ for $T_j\leq t<T_{j+1}$. 
Strictly speaking, the process $Y_t$ depends on $j$ since a different process is generated after each jump $T_j$ conditionally on $Y_0=X_{T_j}$.  
In the formal presentation in the supplementary material, we use the notation $Y_t^{(j)}$ to stress this dependence. However we assume that the distribution of $Y_t^{(j)}$ does not depend on $j$ and is similar after all jumps. In the following, we will  simply write $Y_t$ when there is no confusion.
The process $Y_t$ also depends on $n=n(X_{T_j})$ between $T_{j}$ and $T_{j+1}$, meaning that it is a multivariate process  with $n$ components, each being a stochastic process evolving in $\R^d$. This dependence in $n$ being implicit, we also omit it in the notation. 
\end{enumerate}
 
With this construction, the birth-death-move process is continuous except at the jump times where it is right continuous.
As introduced above, we denote by $X_{t^-}$ the limit from the left of $X_t$, i.e. $X_{t-}:=\lim_{s\to t, s<t} X_s$. At each jump time $T_j$, $X_{T_j^-}$ is therefore the configuration of the process before the jump while $X_{T_j}$ corresponds to the configuration after the jump. If $t$ is not a jump time, then $X_{t^-}=X_t$.

To fully define the dynamics, it remains to specify the distribution of the waiting times before each jump. We denote by $\alpha=\beta+\delta$ the intensity of jumps, whether it is a birth or a death. We will assume that $\alpha$ is bounded from below and above, i.e. there exist $\alpha_*>0$ and $\alpha^*<\infty$ such that for every $x \in E$,  $\alpha_* \leq \alpha(x) \leq \alpha^*$. We also prevent a death in $E_0$ by assuming that $\delta(\varnothing)=0$. Given the jump time $T_j$, given $X_{T_j}$ and given a realisation of the process $(Y_t)_{t\geq 0}$ conditionally on $Y_0=X_{T_j}$, the next jump $T_{j+1}$ is distributed as 
\begin{equation}\label{wait}
\P(T_{j+1}-T_j>t | T_j, X_{T_j}, (Y_t)_{t\geq 0})=\exp\left( - \int_0^t \alpha( Y_u)du \right).
\end{equation}
This specific form is necessary to ensure the Markov property of the process, as verified in the supplementary material.  Note that to be able to compute this formula for any $t$,  the process $(Y_t)_{t\geq 0}$ needs to be known up to $t=\infty$.  We show however in Section~\ref{sec:simu} how we can simulate $X_t$ on a finite time interval  without this obligation. The proposed algorithm may also constitute an illuminating alternative point of view to understand the dynamics.

\bigskip

The thorough probabilistic study of the birth-death-move process is beyond the scope of this article, but we establish in the supplementary material some important properties that are needed for its statistical study. 
Among them, we first verify that the birth-death-move process is a proper time-homogeneous Markov process. We then specialise our results to the following hypothesis, that is assumed henceforth. 
Let $\beta_{n}=\sup_{x\in E_{n}} \beta(x)$ and $\delta_{n}=\inf_{x\in E_{n}} \delta(x)$. 
 \begin{enumerate}[label={\bf(H\arabic*)},ref=H\arabic*]
\item  There exists $n^*$ such that for all $n\geq n^*$, $\beta_n = 0$. Furthermore $\delta_{n}>0$ for all $n\geq 1$.\label{existence}
\end{enumerate}

In light of this hypothesis, we can redefine $E=\bigcup_{n=0}^{n^*} E_n$, meaning that there is a maximal number  $n^*$ of possible individuals in the population. Under  \eqref{existence},  the process $(X_t)_{t\geq 0}$ converges to a stationary distribution,  denoted by $\mu_\infty$, and we get  a uniform rate of convergence. Recall that the existence of a stationary distribution is a prerequisite for the consistency of most statistical procedures of a stochastic process.
It is worth noticing that $\mu_\infty$ may exist even if  \eqref{existence} is not satisfied, provided the sequence $(\delta_n)$ compensates in a proper way the sequence $(\beta_n)$ to avoid explosion of the process. This is for instance proved in \cite{preston} for a pure spatial birth-death process, the special case of a birth-death-move process where there is no move between the jumps. The same kind of statement in the general case  is not straightforward to establish and would imply technicalities that we prefer to avoid in this article.  
We think however that assuming the existence of $n^*$ is reasonable for most statistical applications, especially since the value of $n^*$ has not to be known for our statistical procedure.

To help the reader, we finally summarise the main notation: $X_t$ and $X_{t^-}$ denote the birth-death-move process at time $t$ and its limit from the left; $E=\bigcup_{n=0}^{n^*} E_n$ is the state space of $(X_t)$ and $E_0=\{\varnothing\}$; $\beta(.), \delta(.), \alpha(.)$ are the birth (resp. death and total) intensity functions ($\alpha=\beta+\delta$); $T_1,T_2,\dots$ denote the jump times of the process and $T_0=0$; 
$N_t$ will denote the number of total jumps (births and deaths) up to time $t$; $K_\beta(x,.)$ and $K_\delta(x,.)$ are the probability transition kernels for a birth (resp. a death) given that the configuration of the process before jump is $x$; $(Y_t)_{t\geq 0}$ is the continuous Markov process that drives the motion of $X_t$ between $T_j$ and $T_{j+1}$; $\mu_\infty$ denotes the stationary distribution of $(X_t)_{t\geq 0}$.

\subsection{Examples}\label{sec-examples}

The general definition above includes spatial birth-death processes, as introduced by \cite{preston}.
 They correspond to $(Y_t)_{t\geq 0}$ being the constant random variable $Y_t=Y_0$ for any $t\geq 0$, so that $(X_t)_{t\geq 0}$  does not move between two jumps. 
When $E_n$ is the set of point configurations in $\R^d$ with cardinality $n$ (the main example described in the previous section), we obtain a spatial birth-death process in $\R^d$. This important special case is treated in details in \cite{preston} and further studied in \cite{moller1989}. Special instances are discussed in \cite{comas2008} and some applications to real data sets are considered in \cite{moller1994} and  \cite{Sadahiro19} for example. Perfect simulation of spatial point processes moreover relies on these processes \citep[Chapter~11]{moeller:waagepetersen:03}. 
But other spaces $E_n$ can also be considered. The simplest choice 
$E_n=\{n\}$, implying $E=\N$, leads to the simple birth-death process where $(X_t)_{t\geq 0}$ is interpreted as the evolution of a population size, each jump corresponding to the addition  of a new individual (a birth) or to a diminution (a death). In this case the intensity functions $\beta$ and $\delta$ are just sequences, i.e. functions of $n$. The standard historical simple birth-death process, as introduced in  \cite{Feller1939} and \cite{Kendall49}, corresponds to a constant birth rate and a constant death rate for each individual of the population, leading to linear functions of $n$ for $\beta$ and $\delta$. The case of general sequences $\beta$ and $\delta$ is considered in \cite{wolff1965} and \cite{reynolds1973}, who studied their estimation by maximum likelihood. This approach will be a particular case of our procedure, see Example \ref{ex simple} (i) in Section~\ref{sec:estimation}.

Allowing each point of a spatial birth-death process in $\R^d$ to move according to a continuous Markov process leads to a spatial birth-death-move process  in $\R^d$. A simple example is to assume that each point independently follows a  Brownian motion in $\R^d$, which means that $(Y_t)_{t\geq 0}$ is  in $E_n$ a vector of $n$ independent Brownian motions. 
In fact, when the motions  are independent of each other, the process can be viewed as a branching process with immigration \cite{athreya2012}, where the branching mechanism just consists in a death.  
In \cite{wang2002}, a spatial birth-death-move dynamics in the plane has been adopted to model and track the joint trajectories of elements in a video, like snowflakes or flying birds. The motion in this application is composed of independent autoregressive discrete-time processes. 
Some interacting motions are considered in \cite{comas2008b}, where the authors study by simulation the effect of some deterministic and stochastic
Gibbsian motions in a spatial birth-death  process. Some specific interacting diffusion processes are also considered from a more theoretical point of view in  \cite{eisele1981,locherbach2002likelihood,cinlar1991,hopfner1999}, the two last references concerning the special case where all individuals follow the same stochastic flow.  
 In our real-data application in Section~\ref{sec:data}, we observe the location of proteins in a planar projection of a cell during some time interval. The motion can be different from a protein to another and some interactions may occur. Previous studies \citep{Briane19} have identified three main regimes: Brownian motion, superdiffusive motion (like a Brownian process with drift) and subdiffusive motion (like the Ornstein-Uhlenbeck process). The process $(Y_t)_{t\geq 0}$ in this application could then be a vector of $n$ such processes, and some interactions between these $n$ processes  may further been introduced. In Section~\ref{sec:data}, we do not actually address the choice of a model for   $(Y_t)_{t\geq 0}$, but rather focus on the estimation of the intensity functions $\beta$ and $\delta$ by the procedures developed in Section~\ref{sec:estimation}. Fortunately, no knowledge about the process $(Y_t)_{t\geq 0}$  is needed for these estimations, except that it is a continuous Markov process.

Spatial birth-death-move processes also include spatio-temporal growth interaction models used in individual-based modelling in ecology. In this case $E_n$ is the set of point configurations in $\R^2\times \R_+$ with cardinality $n$, where $\R^2$ represents the space of  location of the points (the plants in ecology) and  $\R_+$ the space of their associated mark (the height of plants, say). Each birth in the process corresponds to the emergence of a new plant in $\R^2$ associated with some positive mark. Through the birth kernel $K_\beta$, we may favor a new plant to appear  nearby existing ones (or not in case of competitions) and  the new mark  may be set to zero or be generated according to some specific distribution (\cite{renshaw2001} chose for instance a uniform distribution on $[0,\epsilon]$ for some small $\epsilon>0$). The growth process only concerns the marks. Let us denote by $(U_i(t),m_i(t))_{t\geq 0}$, for $i=1,\dots,n$, each component of the process  $(Y_t)_{t\geq 0}$, to distinguish the location $U_i(t)\in\R^2$ to the mark $m_i(t)\in\R_+$ of a plant $i$. We thus have $U_i(t)=U_i(0)$ for all $i$, and some continuous Markov dynamics can be chosen for $(m_1(t),\dots,m_n(t))$. In \cite{renshaw2001,renshaw2009,comas2009,habel2019} several choices for this so-called growth interaction process are considered. Furthermore, while in the previous references the birth rate of each plant is constant, the death rate may depend on the location and size of the other plants, leading to a non trivial death intensity function $\delta$.

\subsection{Simulation on a finite time interval}\label{sec:simu}

Algorithm~\ref{algo_simu} shows how we can simulate a birth-death-move process on the time interval  $[0,T]$ for some $T<\infty$, starting  from an initial configuration $X_0$. It requires as an input to be able to simulate  $Y_t$ on a finite interval and to be able to simulate a birth and a death with respect to the kernels $K_\beta$ and $K_\delta$, respectively. The algorithm is a straightforward implementation of the construction explained in Section~\ref{BDM def}, where after each jump we first simulate whether the next jump will occur before $T$ (this happens with probability $1-p$ using the notation of Algorithm~\ref{algo_simu}) or not.

\begin{algorithm}
\caption{Simulation of a birth-death-move process on the finite interval $[0,T]$}
\label{algo_simu}
\begin{algorithmic}
 \STATE {\bf set}  $t=0$ and $j=0$.
 \WHILE{$t<T$}
 \STATE {\bf generate}  $Y_s$ for $s\in [0, T-T_j]$ conditionally on $Y_0=X_{T_j}$
  \STATE {\bf set}  $p=\exp\left( - \int_0^{T-T_j} \alpha( Y_u)du \right)$ 
    \STATE {\bf generate} $U_1\sim U([0,1])$
  \IF{$U_1\leq p$} 
    \STATE {\bf set} $X_s=Y_{s-T_j}$ for $s\in [T_j,T]$ and $t\leftarrow T$
  \ELSE  
   \STATE {\bf generate} the waiting time $\tau_j$ before the next jump according to the distribution
   $$\forall s\in [0, T-T_j],\quad P(\tau_{j}< s| \tau_j\leq T-T_j)=\frac 1 {1-p} \left(1-\exp\left( - \int_0^{s} \alpha( Y_u)du \right)\right)$$
   \STATE {\bf set} $T_{j+1}=T_j+\tau_j$ and $X_s=Y_{s-T_j}$ for $s\in [T_j,T_{j+1})$
     \STATE {\bf generate} $U_2\sim U([0,1])$
     \IF{$U_2\leq \beta(Y_{\tau_j})/\alpha(Y_{\tau_j})$}  
     	\STATE {\bf generate}  $X_{T_{j+1}}$ according to the transition kernel $K_\beta(Y_{\tau_j},.)$
     	\ELSE 
	\STATE {\bf generate} $X_{T_{j+1}}$ according to the transition kernel $K_\delta(Y_{\tau_j},.)$
    \ENDIF
 \STATE  $t\leftarrow t+\tau_j$ and $j\leftarrow j+1$
\ENDIF
 \ENDWHILE
 \end{algorithmic}
\end{algorithm}

The simulation can be speeded up by noticing that most inter-jump times $\tau_j=T_{j+1}-T_j$ are  likely to happen much before  $T-T_j$,  so that instead of generating at each step $(Y_t)$ on $[0,T-T_j]$, it is often sufficient to simulate it on a smaller interval. Choose for instance $\tau_{\max}=\alpha_*\log (1/\epsilon)$ for some $\epsilon>0$ where we recall that $\alpha_*$ is a lower bound of the total intensity function $\alpha$. In view of \eqref{wait}, this choice implies that $P(\tau_j>\tau_{\max})<\epsilon$ for any $j$. Then it is sufficient in most cases to generate $(Y_t)$ on $[0,\min(\tau_{\max},T-T_j)]$ only, since $\tau_j<\min(\tau_{\max},T-T_j)$ with high probability. The rare situations when 
$\tau_j>\min(\tau_{\max},T-T_j)$ can be handled as in Algorithm~\ref{algo_simu}. The detailed procedure is provided in the supplementary material.

\newpage	

\section{Estimation of the intensity functions}\label{sec:estimation}

\subsection{Continuous time observations}
Assume that we observe continuously the birth-death-move process $(X_t)$ in the time interval $[0,T]$ for some $T>0$. Let $(k_{T})_{T \geq 0}$ be a family of  non-negative functions on $\Et\times\Et$ such that $k^* := \sup_{x,y \in \Et} \sup_{T \geq 0} |k_{T}(x,y)| < \infty$. For $x\in E$ and $y\in E$, the role of $k_T(x, y)$ is to quantify the proximity of $x$ and $y$. It can be an indicator function or a kernel function that involves a bandwidth $h_T$. Some typical choices for $k_T$ are discussed in the examples below.
Using the convention $0/0=0$, a natural estimator of $\alpha(x)$ for a given configuration $x\in\Et$ is
\begin{equation}\label{est_cont}
\hat{\alpha}(x) = \frac{1}{\hat T(x)}\int_0^T k_T( x,X_{s^-})dN_s  = \frac{1}{\hat T(x)} \sum_{j=1}^{N_T}  k_T( x,X_{T_j^-}), \end{equation}
where  $N_T$ is the number of jumps before $T$, i.e. $N_T=Card\{j\geq 1 : T_j \leq T\}$ and 
\begin{equation}\label{defT}\hat T(x)=\int_0^T  k_T( x,X_{s}) ds\end{equation}
is an estimation of the time spent by $(X_s)_{0 \leq s \leq T}$ in configurations similar to $x$. In words, $\hat\alpha(x)$ counts the number of times  $(X_s)_{0 \leq s \leq T}$ has jumped when it was in configurations similar to $x$ over the time spent in these configurations. 
Similarly, we consider the following estimators of $\beta(x)$ and $\delta(x)$:
 \[\hat{\beta}(x) =  \frac{1}{\hat T(x)} \sum_{j=1}^{N_T}  k_T( x,X_{T_j^-}) \1_{\{\text{a birth occurs at $T_{j}$}\}},\]
 \[\hat{\delta}(x) = \frac{1}{\hat T(x)} \sum_{j=1}^{N_T}  k_T( x,X_{T_j^-}) \1_{\{\text{a death occurs at $T_{j}$}\}}.\]

\ex{ex BD}
\noindent{\it Example \ref{ex BD}}:  
If $(X_t)_{t\geq 0}$ is a pure spatial birth-death process, corresponding to the case where there is no motion between its jumps, then $X_{T_j^-}=X_{T_{j-1}}$ and $\hat T(x)$ becomes a discrete sum, so that 
\[\hat\alpha(x)= \frac{\sum_{j=0}^{N_T-1} k_T(x,X_{T_j})}{\sum_{j=0}^{N_T-1}(T_{j+1}-T_j)  k_T(x,X_{T_j}) + (T-T_{N_T}) k_T(x,X_{T_{N_T}})}.\]
Similar simplifications occur in this case for $\hat\beta(x)$ and $\hat\delta(x)$.

\bigskip

The next proposition establishes the consistency of our estimators under the following assumptions.
 \begin{enumerate}[label={\bf(H\arabic*)},ref=H\arabic*]
\setcounter{enumi}{1} 
\item Setting $v_T(x) =  \int_{\Et} k_T(x,z)\mu_{\infty}(dz)$,
  \[\lim_{T \to \infty}  T v_T(x) = \infty.\] \label{H3}
 \item Let  $\gamma$ be either $\gamma=\beta$ or $\gamma=\delta$ or $\gamma=\alpha$. Setting $w_T(x) =  \frac{1}{v_T(x)} \int_{\Et} (\gamma(z) - \gamma(x)) k_T(x,z)\mu_{\infty}(dz)$,
    \[\lim_{T \to \infty} w_T(x) = 0.\]\label{H4}
 \end{enumerate}

The meaning of Assumptions \eqref{H3} and \eqref{H4} will appear more clearly in the following examples. In fact, $1/(Tv_T(x))$ may be understood as a variance term, while $w_T(x)$ can be seen as a bias term. Strictly speaking, this interpretation is wrong as discussed in the supplementary material.  
However, in line with this point of view, the following rate of convergence is a standard bias-variance tradeoff in non-parametric kernel-based estimation. Its proof is given in the supplementary material.
\begin{prop}\label{Maintheo}
Let  $\gamma$ be either $\gamma=\beta$ or $\gamma=\delta$ or $\gamma=\alpha$. Assume \eqref{existence}, \eqref{H3} and \eqref{H4}, then for any $x\in E$ and any $\epsilon>0$
\begin{equation*}
\P(|\hat{\gamma}(x) - \gamma(x)|>\epsilon)\leq  c\left( \frac{1}{T v_T(x)} + w_T^2(x)\right)
\end{equation*}
where $c$ is a positive constant depending on $\epsilon$. Consequently $\hat{\gamma}(x)$ is a consistent estimator as $T\to\infty$.  
\end{prop}

 \bigskip

\ex{ex kernel}
\noindent{\it Example \ref{ex kernel}}:  
A standard choice for $k_T$ is 
\begin{equation}\label{generalkT} k_T(x,y)=k\left(\frac{d(x,y)}{h_T}\right)\end{equation}
where $k$ is a bounded kernel function on $\R$, $d$ is a pseudo-distance on $E$ and $h_T>0$ is a bandwidth parameter.  
In this case, \eqref{H3} and \eqref{H4} can be understood as hypotheses on the bandwidth $h_T$, demanding that $h_T$ tends to 0 as $T\to\infty$ but not too fast, as usual in non-parametric estimation. To make this interpretation clear, assume that $k(u)=\1_{|u|<1}$ and 
that $\gamma$ is Lipschitz with constant $\ell$. 
Then, denoting $B(x,h_T):=\{y\in E, d(x,y)<h_T\}$, we have that $v_T(x)=\mu_\infty(B(x,h_T))$ and
\begin{align*}
|w_T(x)| &\leq \frac{1}{v_T(x)}  \int_{B(x,h_T)} \left|\gamma(z) - \gamma(x)\right| \mu_{\infty}(dz)
\leq  \frac{\ell}{v_T(x)} \int_{B(x,h_T)} d(x,z) \mu_{\infty}(dz)\leq \frac{\ell}{v_T(x)} h_T v_T(x) = \ell h_T. 
\end{align*}
In this setting, \eqref{H3} and \eqref{H4} are satisfied whenever  $h_T\to 0$ and $T\mu_\infty(B(x,h_T))\to\infty$.
Note moreover that the rate $w_T(x)=O(h_T)$ in this example is in agreement with the standard rate for the bias of the kernel estimator of a Lipschitz function, justifying the interpretation of $w_T(x)$ as a bias term.  On the other hand $1/Tv_T(x)=O(1/(T\mu_\infty(B(x,h_T))))$ depends on the underlying dimension through the pseudo-distance $d$ defining the ball $B(x,h_T)$, as usual for the variance term of a kernel estimator. These rates are made explicit in the particular setting of Example~\ref{ex geom}.

\bigskip

\ex{ex Hausdorff}
\noindent{\it Example \ref{ex Hausdorff}}: 
When $E_n$ is the space of point configurations in $\R^d$ with cardinality $n$, we can take $k_T$ as in \eqref{generalkT} where $d(x,y)$ is a pseudo-distance on the space $E$ of finite point configurations in $\R^d$. Several choices for this  pseudo-distance are possible. A first standard option is the Hausdorff distance
\[d_H(x,y)= \max\{ \max_{u\in x} \min_{v\in y} \|u-v\| ,  \max_{v\in y} \min_{u\in x} \|v-u\| \}\]
if $x\neq \varnothing$ and $y\neq\varnothing$, while $k_T(x,y)=\1_{x=y}$ if $x=\varnothing$ or $y=\varnothing$.
Some alternatives are discussed in \cite{mateu2015}. 
Another option, that will prove to be more appropriate for our applications in Section~\ref{simulations}, is the optimal matching distance  
introduced by \cite{schuhmacher2008}. 
Letting $x=\{x_1,\dots,x_{n(x)}\}$, $y=\{y_1,\dots,y_{n(y)}\}$ and assuming that $n(x)\leq n(y)$, this distance is defined for some $\kappa>0$ by 
\[d_{\kappa}(x,y)=\frac{1}{n(y)} \left(\min_{\pi\in\mathfrak{S}_{n(y)}} \sum_{i=1}^{n(x)}(\|x_i-y_{\pi(i)}\|\wedge\kappa) +  \kappa (n(y)-n(x))  \right),\]
where $\mathfrak{S}_{n(y)}$ denotes the set of permutations of $\{1,\dots,n(y)\}$. 
In words, $d_{\kappa}$ calculates the total (truncated) distance between $x$ and its best match with a sub-pattern of $y$ with cardinality $n(x)$, and then it adds a penalty $\kappa$ for the difference of cardinalities between $x$ and $y$. 
If all point configurations belong to a bounded subset $W$ of $\R^d$, a natural choice for $\kappa$ is to take the diameter of $W$, in which case the distance between two point patterns with the same cardinalities corresponds to the average distance between their optimal matching. 
The definition of $d_{\kappa}$ when $n(x)\geq n(y)$ is similar by inverting the role played by $x$ and $y$. 
The choice of $d_H$ or $d_{\kappa}$ does not exploit any particular structural form of $\gamma(x)$, allowing for a pure non-parametric estimation. Note however that some regularities are implicitly demanded on $\gamma(x)$ because of \eqref{H4}, as illustrated in Example~\ref{ex kernel} where $\gamma$ is assumed to be Lipschitz.

\bigskip

\ex{ex simple}
\noindent{\it Example \ref{ex simple}}:  If we assume that $\gamma(x)=\gamma_0(n(x))$ only depends on the cardinality of $x$ through some function $\gamma_0$ defined on $\N$, the setting becomes similar to simple birth-death processes, except that we allow continuous motions between jumps. We may consider two strategies in this case:
\begin{itemize}
\item[(i)] We recover the standard non-parametric likelihood estimator of the intensity  studied in \cite{wolff1965} and \cite{reynolds1973} by choosing $k_T(x,y)=1$ if $n(x)=n(y)$ and $k_T(x,y)=0$ otherwise. 
Then $\hat\alpha(x)$ (resp. $\hat\beta(x)$, $\hat\delta(x)$) just counts the number of jumps (resp. of births, of deaths) of the process $(X_s)_{0 \leq s \leq T}$ when it is in $E_{n(x)}$ divided by the time spent by the process in $E_{n(x)}$. We get in this case that $v_T(x)=\mu_\infty(E_{n(x)})$ does not depend on $T$ and $v_T(x)w_T(x)=0$, so \eqref{H3} and \eqref{H4} are satisfied whenever $\mu_\infty(E_{n(x)})\neq 0$. 

\item[(ii)] Alternatively, we may choose $k_T$ as in \eqref{generalkT} with $d(x,y)=|n(x)-n(y)|$, in which case $\hat\gamma(x)$ differs from the previous estimator in that not only configurations in $E_{n(x)}$ are taken into account in $\hat\gamma(x)$ but all configurations in $E_{n}$ provided $n$ is close to $n(x)$.  
For this reason this new estimator can be seen as a smoothing version of the previous one, the latter being in fact  the limit when $h_T\to 0$ of the former. This smoothing makes sense if we assume some regularity properties on $n\mapsto\gamma_0(n)$, as demanded by \eqref{H4}, and results in a less variable estimation (see the simulation study of Section~\ref{simulations}).
\end{itemize}

 \bigskip

\ex{ex geom}
\noindent{\it Example \ref{ex geom}}: In the spirit of the previous example, we may assume that  $\gamma(x)=\gamma_0(f(x))$ only depends on some specific characteristics of $x$ encoded in a function $f: E\to \R^p$ for some $p\geq 1$ and $\gamma_0$ is defined on $\R^p$.  
For instance in our simulation study of Section~\ref{simu:geom}, $E_n$ is the space of point configurations in $[0,1]^2$ with cardinality $n$ and we consider for $f(x)$ the maximal area of the Delaunay cells associated to the configuration $x$, while $\gamma_0$ is an increasing function. For this example, jumps are more likely to happen when there is a large ``available'' empty region in between the elements of $x$. The estimation problem then reduces to a non-parametric estimation in dimension $p$ and we may choose $k_T$ as in \eqref{generalkT} with $d(x,y)=\|f(x)-f(y)\|$. 
Under this setting, it is not hard to prove the following corollary, showing that we recover the standard rate of convergence for kernel estimators of Lipschitz functions in dimension $p$.

\begin{cor}
Assume $\gamma(x)=\gamma_0(f(x))$ where $f: E\to \R^p$ for some $p\geq 1$ and $\gamma_0$ is a Lipschitz function defined on $\R^p$.  If $k_T$ is as in \eqref{generalkT} with $d(x,y)=\|f(x)-f(y)\|$ and $\int_{\R} |t|^p k(|t|) dt<\infty$,  and if $f(Z)$ admits a bounded continuous density with respect to the Lebesgue measure in $\R^p$ when $Z\sim\mu_\infty$, then under   \eqref{existence},
\begin{equation*}
\P(|\hat{\gamma}(x) - \gamma(x)|>\epsilon)\leq  c\left( \frac{1}{T h_T^p} + h_T^2\right).
\end{equation*}
\end{cor}

  \subsection{Discrete time observations}\label{sec:discrete}

Assume now that we observe the process at $m+1$ time points $t_0,\dots, t_m$ where $t_0=0$ and $t_m=T$. We denote  $\Delta t_j = t_{j}-t_{j-1}$, for $j=1,\dots,m$, and $\Delta_m = \max_{j=1...m} \Delta t_j$ the maximal discretization step. We thus have $T=\sum_{j=1}^{m} \Delta t_j$ and $T\leq m\Delta_m$. We assume further that $m\Delta_m/T$ is uniformly bounded. The asymptotic properties of this section will hold when both $T\to\infty$ and $\Delta_m\to 0$, implying $m\to\infty$.

To consider a discrete version of the estimator \eqref{est_cont} of $\alpha(x)$, a natural idea is to replace the involved integrals by Riemann sums, leading to the estimator
\begin{equation}
\frac{\sum_{j=0}^{m-1} \Delta N_{t_{j+1}} k_T( x, X_{t_j})}{\sum_{j=0}^{m-1} \Delta t_{j+1} k_T( x, X_{t_j})}\end{equation}
where $\Delta N_{t_j}=N_{t_{j}} - N_{t_{j-1}}$, for $j=1,\dots,m$.
However $\Delta N_{t_j}$ is not necessarily observed because we can miss some jumps in the interval $(t_{j-1}, t_{j}]$ (for instance the birth of an individual can be immediately followed by its death). For this reason we introduce an approximation $D_j$ of  $\Delta N_{t_j}$ which is observable. 
For the asymptotic validity of our estimator, we specifically require that this approximation satisfies:
\begin{enumerate}[label={\bf(H\arabic*)},ref=H\arabic*]
\setcounter{enumi}{3} 
\item For all $j\geq 1$, $D_j= \Delta N_{t_j}$ if $\Delta N_{t_j} \leq 1$ and $D_j \leq \Delta N_{t_j}$ if $\Delta N_{t_j} \geq 2$. \label{HDj}
\end{enumerate}
When $\Delta_m\to 0$, the case $\Delta N_{t_j} \geq 2$ becomes unlikely, so that asymptotically the only important case is $\Delta N_{t_j} \leq 1$, a situation where $D_j = \Delta N_{t_j}$ is a perfect approximation under \eqref{HDj}.
An elementary example is to take $D_j = \1_{n(X_{t_{j-1}}) \neq n(X_{t_{j}})}$. This choice satisfies \eqref{HDj} but some better approximations may be available. For instance, when $E_n$ is the space of point configurations in $\R^d$ with cardinality $n$, assuming that we can track the points between times $t_{j-1}$ and $t_{j}$, 
we can choose for $D_j$ the number of new points in $X_{t_{j}}$  (that is a lower estimation of the number of births), plus the number of points that have disappeared between  $t_{j-1}$ and $t_{j}$ (that is a lower estimation of the number of deaths),  in short $D_j = n(X_{t_{j}}\setminus X_{t_{j-1}})+n(X_{t_{j-1}}\setminus X_{t_{j}}).$

Our estimator of $\alpha(x)$ in the discrete case is then defined as follows: 
\begin{equation}\label{est discrete}\hat{\alpha}_{(d)}(x) = \frac{\sum_{j=0}^{m-1} D_{j+1} k_T( x, X_{t_j})}{\sum_{j=0}^{m-1} \Delta t_{j+1} k_T( x, X_{t_j})}.\end{equation}
Obvious adaptations lead to the definition of $\hat\beta_{(d)}(x)$ and $\hat\delta_{(d)}(x)$, the discrete versions of $\hat\beta(x)$ and $\hat\delta(x)$. The details are provided in the supplementary material.

To get the consistency of these estimators, we assume that the discretization step $\Delta_m$ asymptotically vanishes at the following rate.
\begin{enumerate}[label={\bf(H\arabic*)},ref=H\arabic*]
\setcounter{enumi}{4} 
\item Let $v_T(x)$ be as in \eqref{H3}, \[ \lim_{T\to\infty} \frac{\Delta_m}{v_T^2(x)}\to 0.\] \label{stepto0}
\end{enumerate}
We also need some regularity assumptions on the process $(k_T(x,Y_t))_{t\geq 0}$.
 This is necessary to control the difference between $\hat{T}(x)$ defined in \eqref{defT} and its discretized version $\hat{T}_{(d)}(x) =\sum_{j=0}^{m-1} \Delta t_{j+1} k_T( x, X_{t_j})$ appearing in the denominator of \eqref{est discrete}.

\begin{enumerate}[label={\bf(H\arabic*)},ref=H\arabic*]
\setcounter{enumi}{5} 
\item There exist $\ell_T(x)\geq 0$ and $a>0$ such that for any $s,t$ satisfying $|s-t|<\Delta_m$ and for all $y\in E$, \begin{equation*}
 \E\left[|k_T(x,Y_s)-k_T(x,Y_t)|\big | Y_0=y \right]\leq \ell_T(x) |s-t|^a\quad\text{with}\quad\lim_{T\to\infty} \frac{\Delta_m^{a} \ell_T(x) }{v_T^2(x)}\to 0.\end{equation*}\label{discrete condition}
  \end{enumerate}
  The inequality in \eqref{discrete condition} with $a>1$ implies the continuity of the process $(k_T(x,Y_t))_{t\geq 0}$, in virtue of   the Kolmogorov continuity theorem.  In fact continuous processes that do not satisfy this inequality are the exception, and  a  standard way to prove the continuity of a stochastic process consists precisely in verifying this condition.   The last condition in  \eqref{discrete condition} demands that $\Delta_m$ converges sufficiently fast to  zero to capture this regularity. 
We come back to Examples~\ref{ex BD}-\ref{ex geom} in the supplementary material to show that these assumptions are mild and reduce, when the kernel takes the form \eqref{generalkT}, to assuming that the bandwidth $h_T$ tends to 0 but not too fast and that $\Delta_m$ tends to zero sufficiently fast, the latter rate depending on the regularity exponent $a$. We finally get the following consistency result, the proof of which is given  in the supplementary material.

  \begin{prop}\label{th discret}
  Let  $\gamma$ be either $\gamma=\beta$ or $\gamma=\delta$ or $\gamma=\alpha$. Assume \eqref{existence}-\eqref{discrete condition}, then
\[\hat{\gamma}_{(d)}(x) - \gamma(x)=O_p\left(\frac{1}{T v_T(x)} + w_T^2(x) +  \frac{\Delta_m}{v_T^2(x)} + \frac{\Delta_m^{a} \ell_T(x) }{v_T^2(x)}\right)\]
as $T\to\infty,$ whereby $\hat{\gamma}_{(d)}(x)$ is a consistent estimator of $\gamma(x).$  
  \end{prop}

  \subsection{Bandwidth selection by partial likelihood cross-validation}\label{CV}
  
  We assume in this section that $k_T$ takes the form \eqref{generalkT}, in which case we have to choose in practice a value of the bandwidth $h_T$ to implement our estimators, whether for $\hat\gamma(x)$ in the continuous case or for $\hat\gamma_{(d)}(x)$ in the discrete case. We explain in the following how to select $h_T$ by (partial) likelihood cross-validation, a widely used procedure in kernel density estimation of a probability distribution and intensity kernel estimation of a point process, see for instance \cite[Chapter 5]{loader2006}. Remark that the alternative popular plug-in method and least-squares cross-validation method, that are both based on second order properties of the estimator, do not seem adapted to our setting because, as discussed in the supplementary material, $\hat\gamma(x)$ is not necessarily integrable. Furthermore we only have an upper bound for the rate of convergence and this one depends on unknown quantities that appear difficult to estimate.

 We focus for simplicity on the estimation of $\alpha(x)$ but the procedure adapts straightforwardly to the estimation of $\beta(x)$ or $\delta(x)$.
 As verified in the supplementary material, the intensity of $N_t$  is $\alpha(X_{t^-}).$
  By the Girsanov theorem \cite[Chapter 6.2]{bremaud}, the log-likelihood of $(N_t)_{0\leq t\leq T}$ with respect to the unit rate Poisson counting process on $[0,T]$ is therefore
  \[\int_0^T (1-\alpha(X_{s^-})) ds +  \int_0^T \log\alpha(X_{s^-}) dN_s = T- \int_0^T \alpha(X_{s})ds + \sum_{j=1}^{N_T}  \log \alpha(X_{T_j^-}) .\]

Note that this is not the log-likelihood of the process $(X_t)_{0\leq t\leq T}$ but only of $(N_t)_{0\leq t\leq T}$. 
This is the reason why we call our approach a partial likelihood cross-validation.   
For continuous time observations, bandwidth selection then  amounts to choose $h_T$ as
\[\hat h= \underset{h}{\rm argmax} \sum_{j=1}^{N_T} \log \hat\alpha^{(-)}_h (X_{T_j^-}) - \int_0^T   \hat\alpha^{(-)}_h (X_{s})ds\]
where $\hat\alpha^{(-)}_h (X_s)$ is the estimator \eqref{est_cont} of $\alpha(x)$ for $x=X_s$, associated to the choice of bandwidth $h_T=h$, but without using the observation $X_s$. To carry out this removal, we suggest to discard all observations in the time interval $[T_{N_s},T_{N_s+1}]$, which gives
\[\hat\alpha^{(-)}_h (X_s)= \frac{\sum_{i=1,i\neq N_s+1}^{N_T}  k(d( X_s,X_{T_i^-})/h)} {\int_{[0,T]\setminus [T_{N_s},T_{N_s+1}] } k(d( X_s,X_{u})/h) du}.\]
In particular 
\[\hat\alpha^{(-)}_h (X_{T_j^-}) =  \frac{\sum_{i=1,i\neq j}^{N_T}  k(d( X_{T_j^-},X_{T_i^-})/h) }{ \int_{[0,T]\setminus [T_{j-1},T_{j}] } k(d(X_{T_j^-} ,X_{u})/h) du}.\]

For discrete time observations, this cross-validation procedure becomes
\[\hat h_{(d)}= \underset{h}{\rm argmax} \sum_{j=0}^{m-1} D_{j+1}  \log \hat\alpha^{(-)}_{(d),h} (X_{t_j}) - \sum_{j=0}^{m-1} \Delta t_{j+1}  \hat\alpha^{(-)}_{(d),h} (X_{t_j}),\]
using the same notation as in Section~\ref{sec:discrete} and where 
\[ \hat\alpha^{(-)}_{(d),h} (X_{t_j}) = \frac{\sum_{i=0,i\neq j}^{m-1} D_{i+1} k(d( X_{t_j},X_{t_i})/h)}{\sum_{i=0,i\neq j}^{m-1} \Delta t_{i+1} k(d( X_{t_j},X_{t_i})/h)}.\]

 \section{Simulation study}\label{simulations}
 
 \subsection{First situation: a dependence on the cardinality}\label{simu:card}

In order to assess the performances of our intensity estimator, we simulate a birth-death-move process in the square window $W=[0,1]^2$ during the time interval $[0,T]$ with $T=1000$ (the time unit does not matter).  The initial configuration at $t=0$ consists of the realization of a Poisson point process with intensity $100$ in $W$.  
For the jump intensity function $\alpha$, we choose $\alpha(x)=\exp(5(n(x)/100-1))$, i.e. $\alpha(x)$ only depends on the cardinality of $x$, as in the setting of Example~\ref{ex simple}, and this dependence is exponential. We fix a truncation value $n^*=1000$. If $0<n(x)<n^*$, each jump is a birth or a death with equal probability. If $n(x)=0$ the jump is a birth, and it is a death if  $n(x)=n^*.$
Each birth consists of the addition of one point which is drawn uniformly on $[0,1]^2$, and  each death consists of the removal of an existing point uniformly over the points of $x$. Finally, between each jump, each point of $x$ independently evolves according to a planar Brownian motion with standard deviation $2.10^{-3}$. 
Many other dynamics could have been simulated. Our choice for this first example is motivated by the real-data dynamics treated in Section~\ref{sec:data}, that we try to roughly mimic. The next example considered in Section~\ref{simu:geom} is quite different as it involves geometric characteristics of $x.$

Figure~\ref{sim-samples} shows, for one simulated trajectory, the configuration of the process at the initial time $T_0=0$ and after the first jump time (at $t=T_1$). This first jump was a birth and the location of the new point is indicated by a red dot, pointed out by a  red arrow. The locations of the other points at time $T_1$ are indicated by black dots, while their initial locations are recalled in gray, illustrating the motions between $T_0$ and $T_1$.  For this simulation, $N_T=1530$ jumps have been observed in the time interval $[0,T]$. The two last plots of  Figure~\ref{sim-samples} show the evolution of the number of points before each jump time and the time elapsed between each jump time.

 \begin{figure}[ht]
\begin{center}
{\small
\begin{tabular}{cccc}
 \includegraphics[scale=.31]{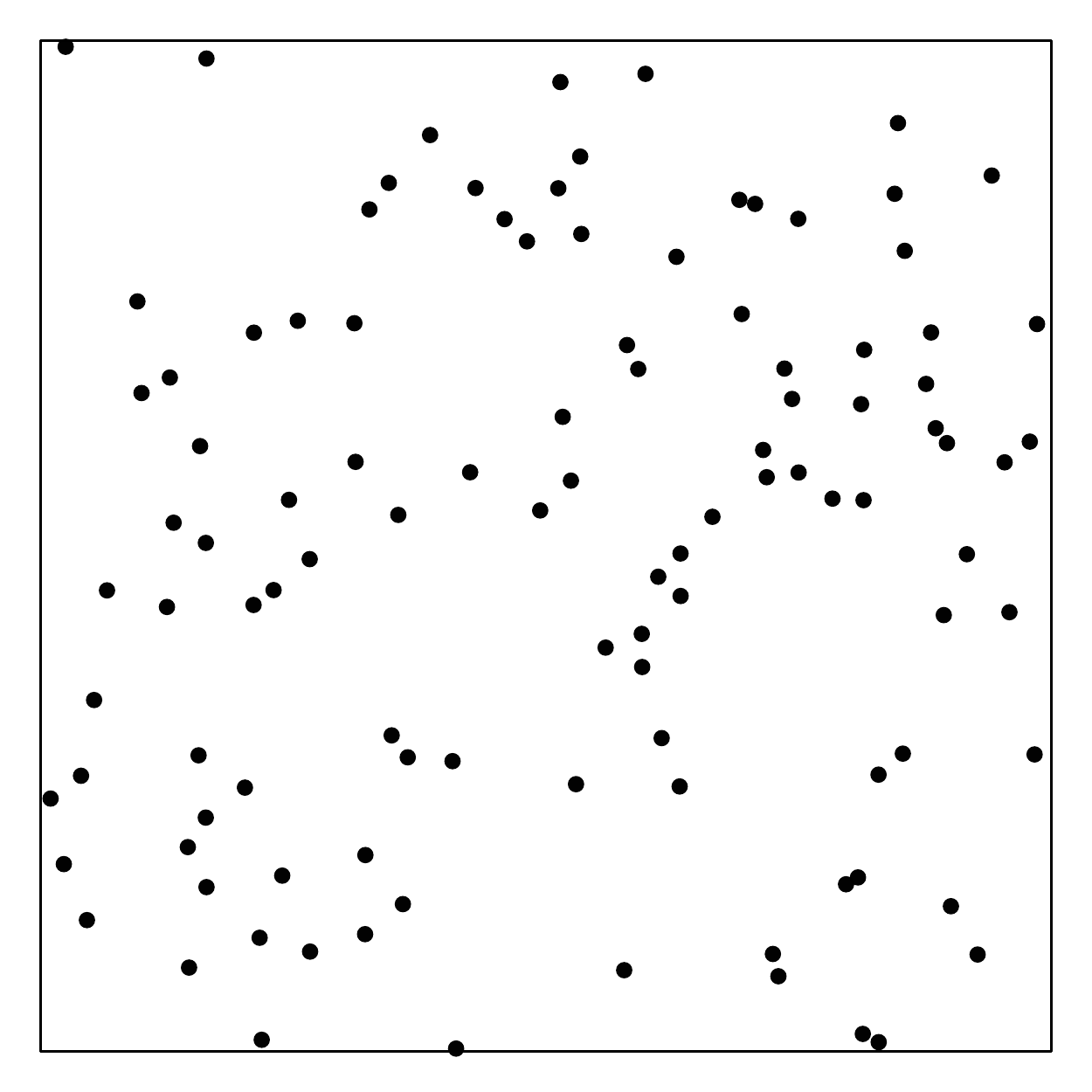} & \includegraphics[scale=.31]{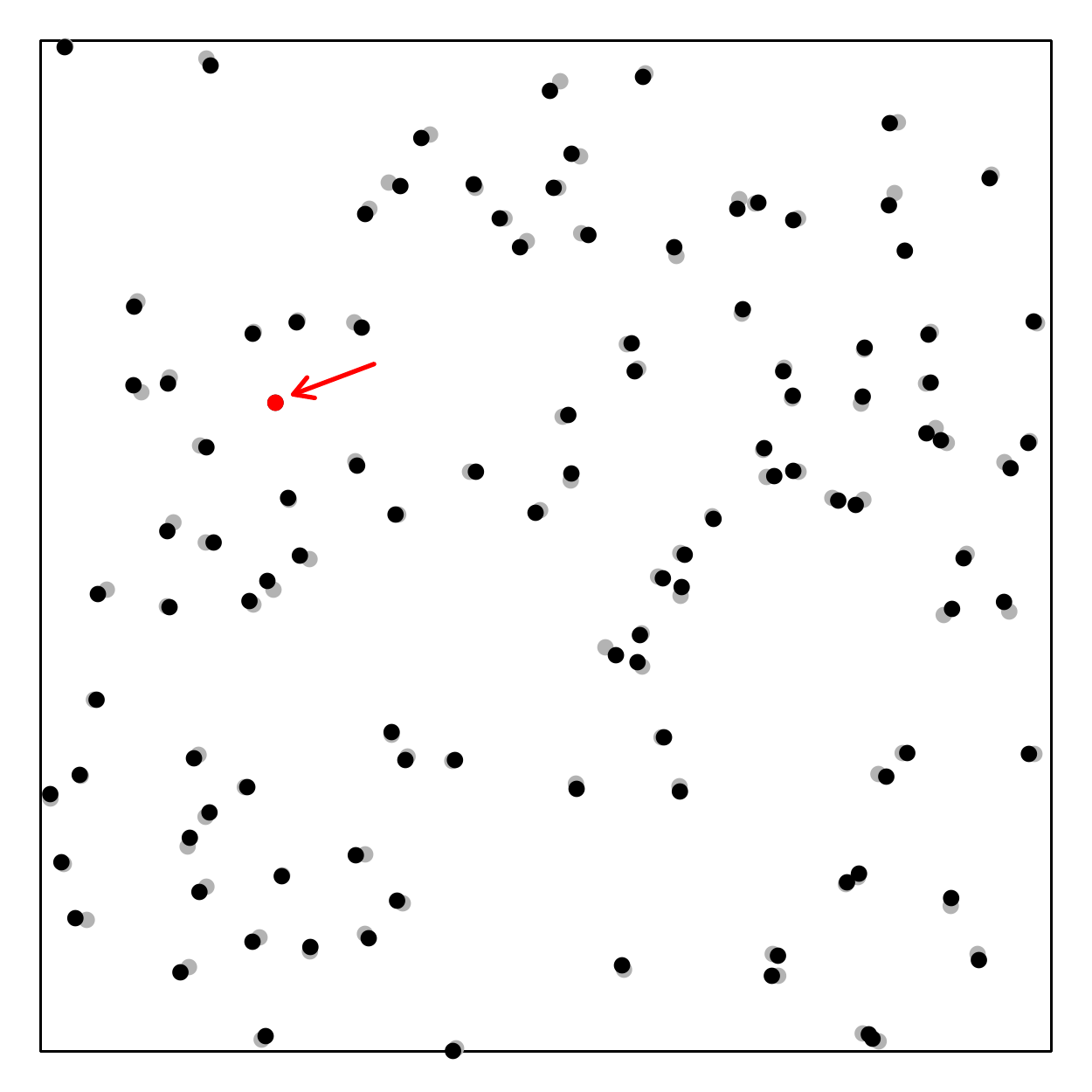} &\includegraphics[scale=.31]{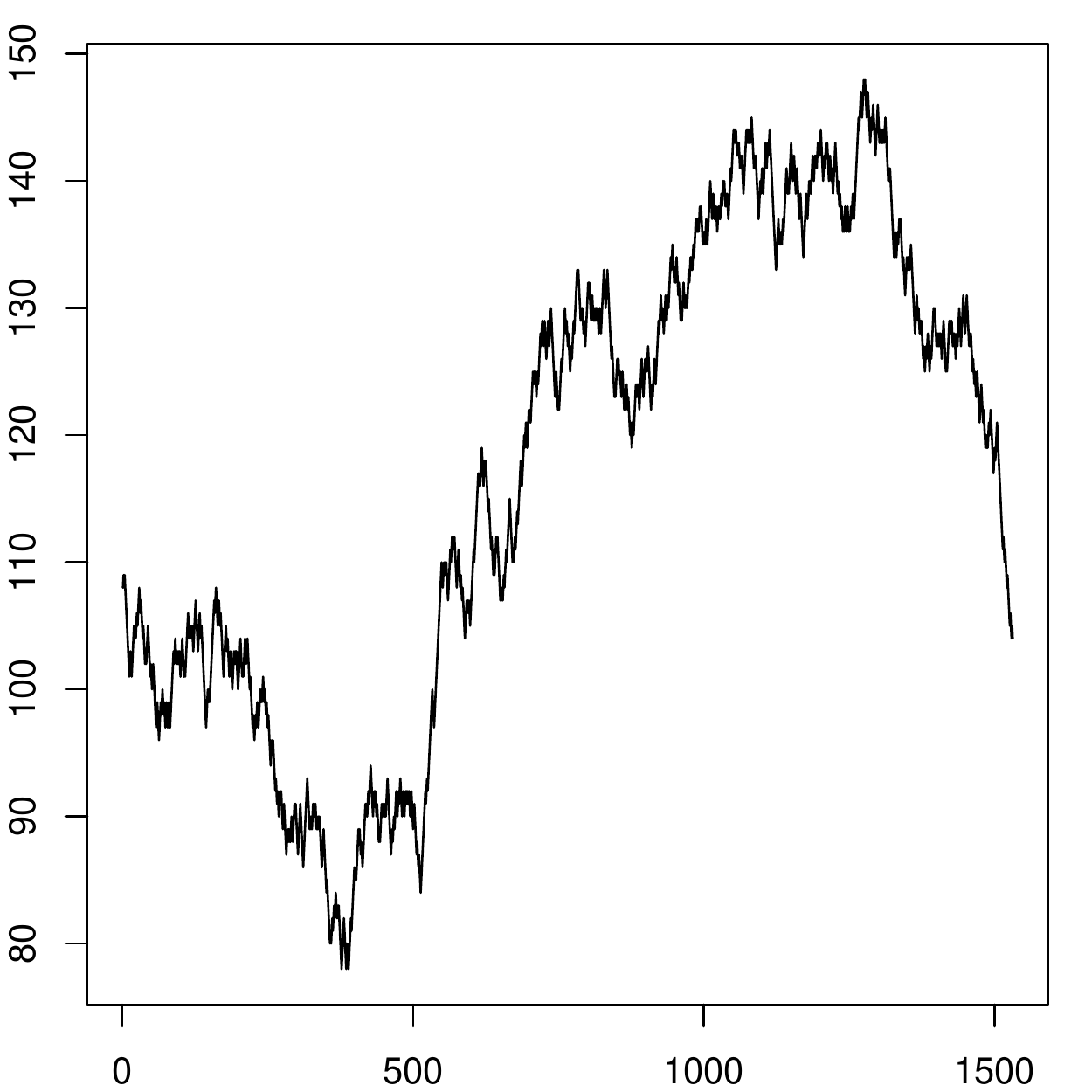} & \includegraphics[scale=.31]{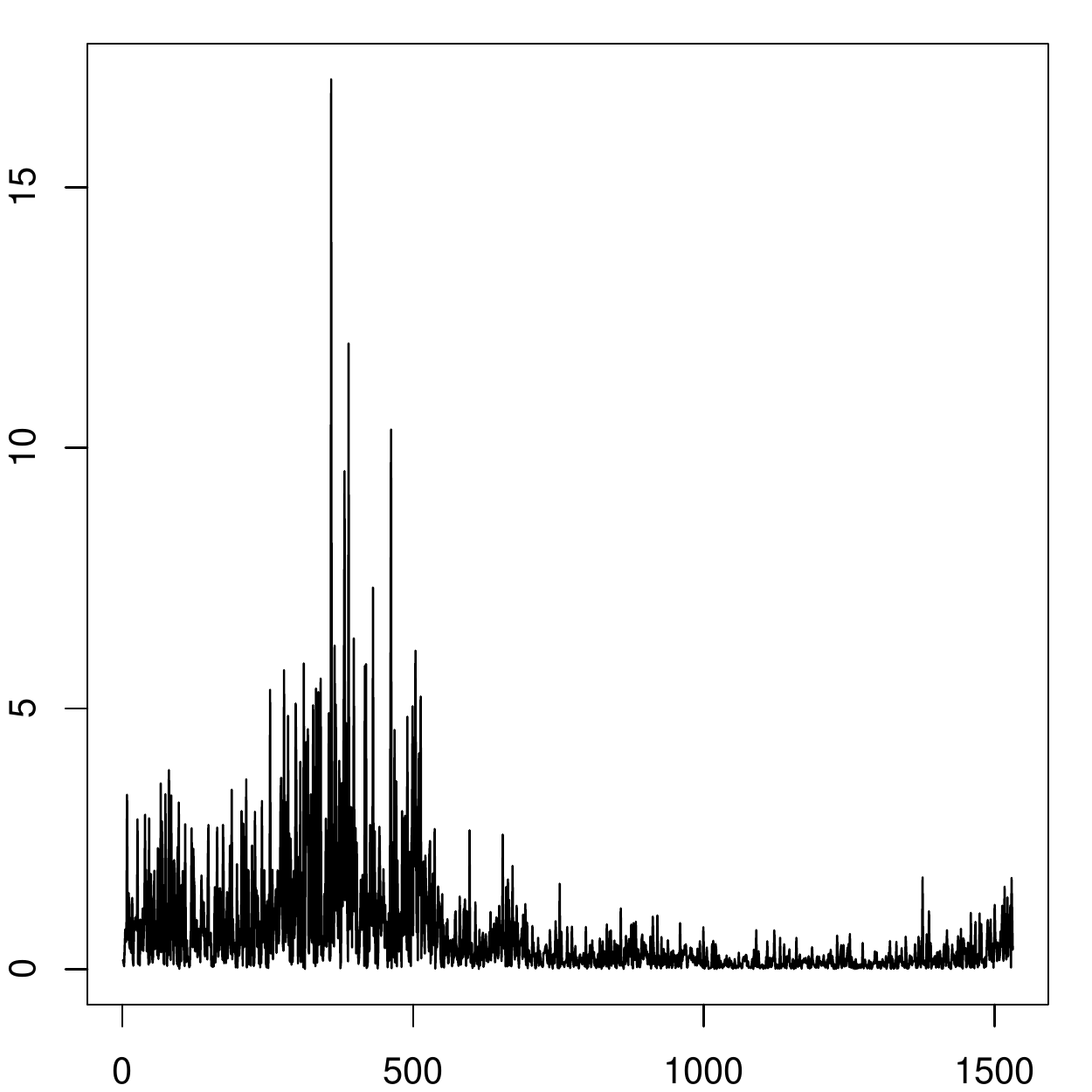}\\
$(a)$ $T_0=0$, $n=108$  & $(b)$ $T_1=0.17$, $n=109$ & $(c)$ & $(d)$
\end{tabular}}
 \caption{{\small  $(a)$ Initial point configuration of the simulation considered in Section~\ref{simu:card}; $(b)$ Configuration after the first jump (i.e. at $t=T_1$) that was a birth. This second plot shows the new born point in red (pointed out by a  red arrow); the other points show as black dots while their initial locations at time $T_0$ are recalled in gray. 
$(c)$ Number of points before each jump time; $(d)$ Time elapsed between each jump time.}}
\label{sim-samples} 
\end{center}
\end{figure}

For this simulation, we consider the estimation of $\alpha(x)$ at each $x=X_{T_i}$, $i=0,\dots,N_T$. This choice for the target point configurations $x$ is a way to analyse the temporal evolution of the intensity, as we intend to do in the real-data application of Section~\ref{sec:data}.
Assuming for the moment that we observe the process continuously in $[0,T]$, we consider the estimator \eqref{est_cont} and the following four different choices for $k_T$.  The first two strategies are fully non-parametric and follow Example~\ref{ex Hausdorff} : the first one depends on the Hausdorff distance $d_H$, while the second one depends on the distance $d_{\kappa}$ where we have chosen $\kappa=\sqrt 2$ to be the diameter of $W$. The two other strategies assume (rightly) that $\alpha(x)$ only depends on the cardinality of $x$. They respectively correspond to the cases (i) and (ii)  of Example~\ref{ex simple}.  Each time needed, we select the bandwidth by likelihood cross-validation as explained in Section~\ref{CV} and the kernel function $k$ in \eqref{generalkT} is just the standard Gaussian density.  The result of these four estimations are showed in Figure~\ref{fig:alphaest}. The first row depicts the evolution of the true value of $\alpha(X_{T_i})$ in black, for $i=0,\dots,N_T$, along with its estimation in red. The second row shows the scatterplot $(n(X_{T_i}),\hat\alpha(X_{T_i}))$ along with the ground-truth curve $\exp(5(n(X_{T_i})/100-1))$.

\begin{figure}[ht]
\begin{center}
{\small
\begin{tabular}{cccc}
 \includegraphics[scale=.3]{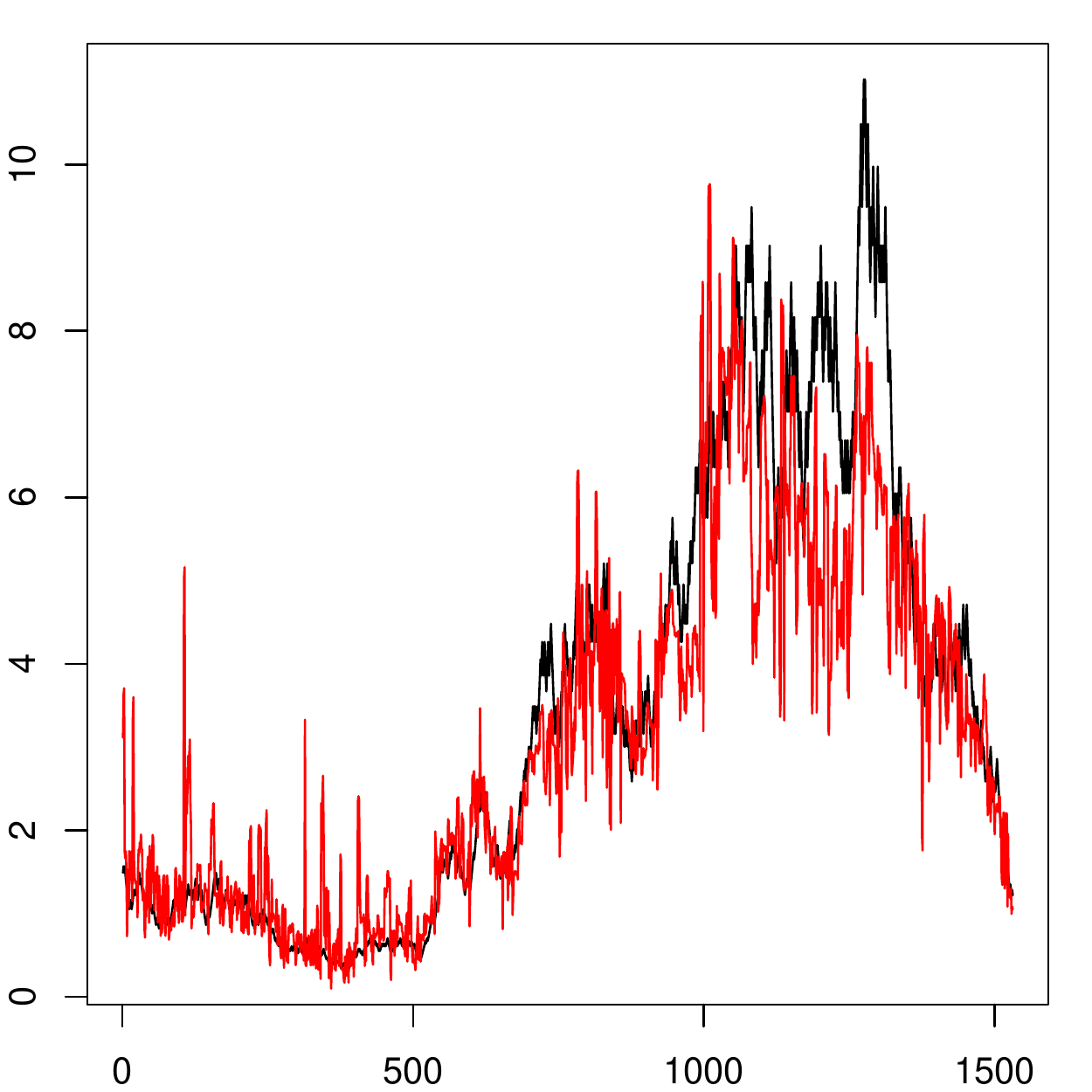} & \includegraphics[scale=.3]{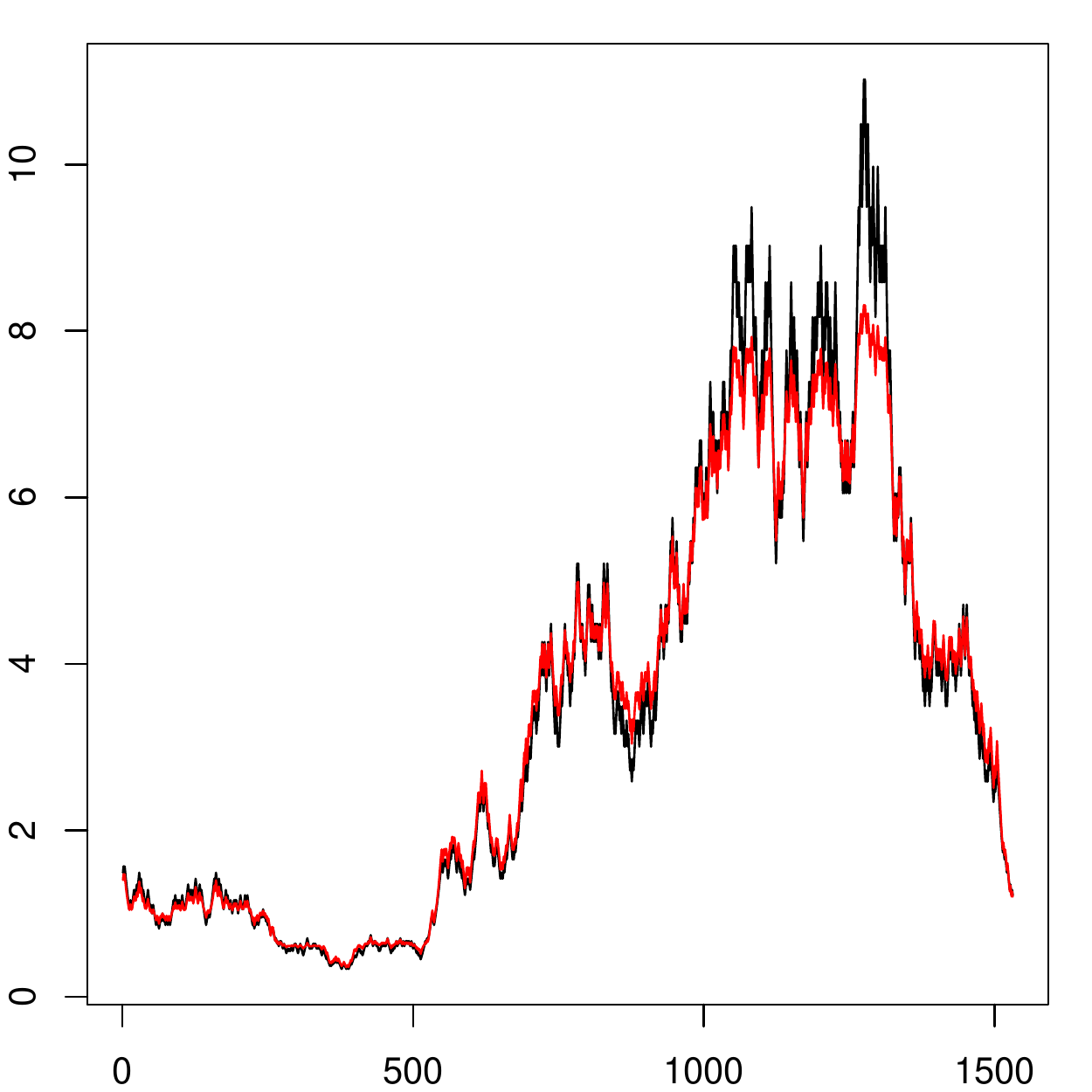} & \includegraphics[scale=.3]{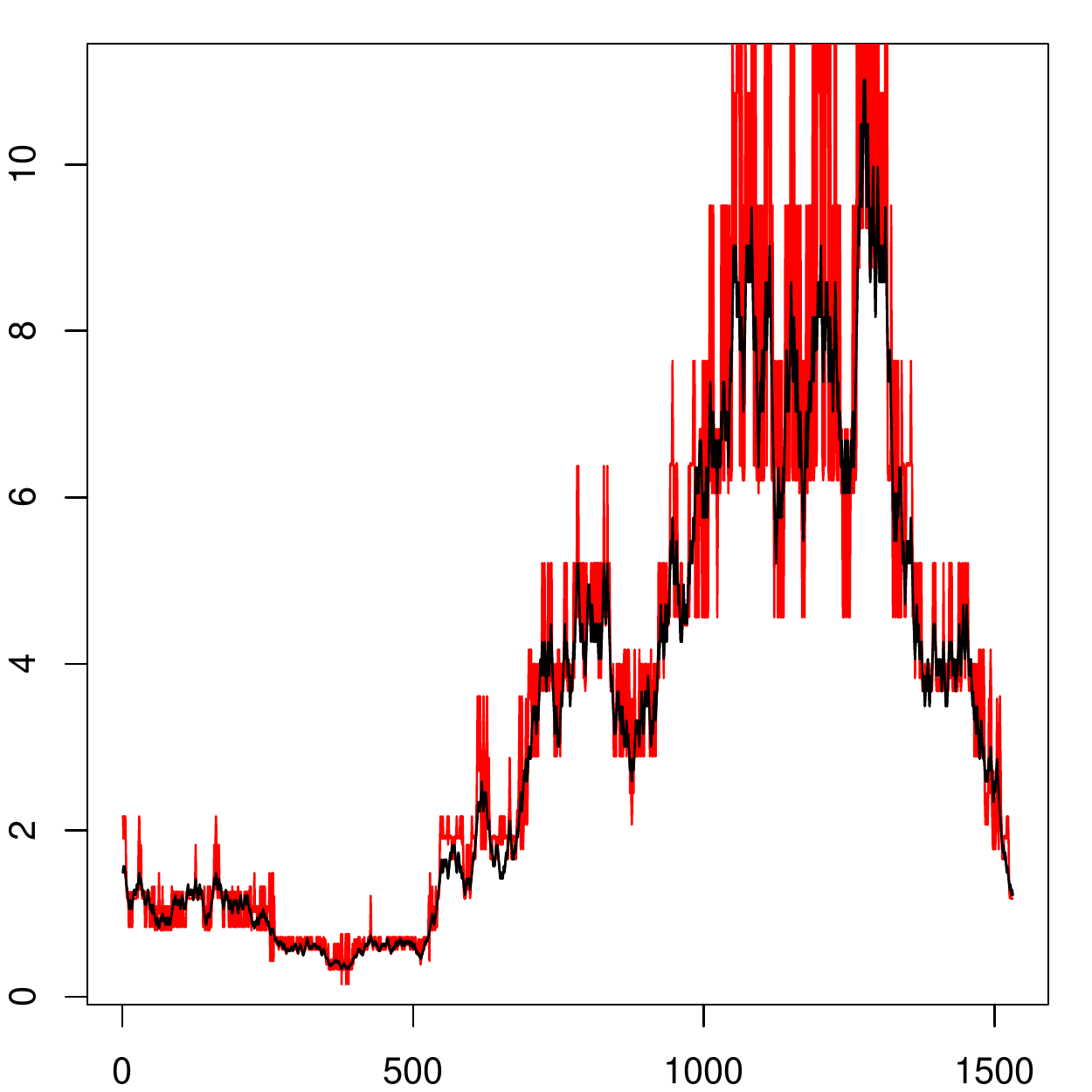} & \includegraphics[scale=.3]{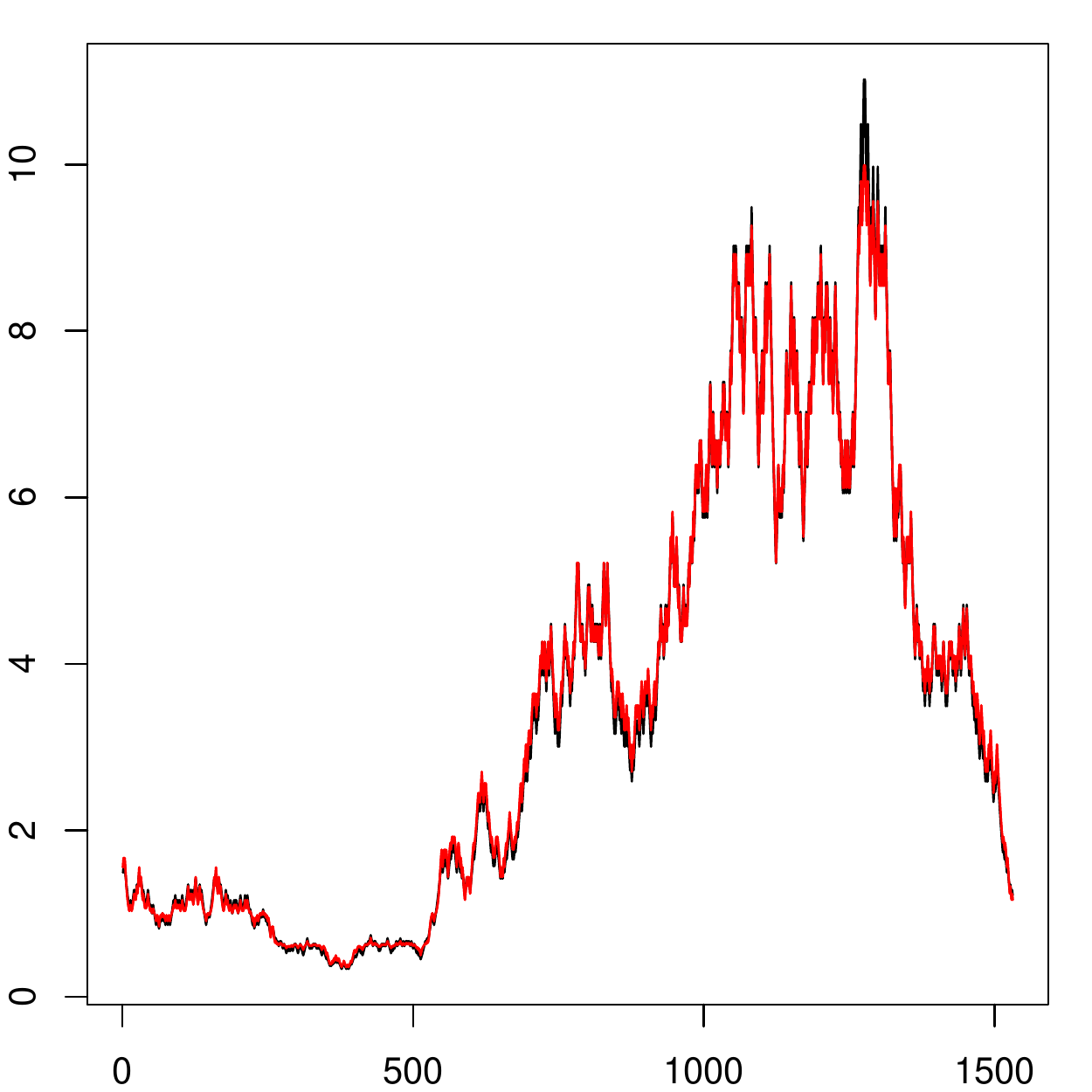} \\ 
 (a) & (b) & (c) & (d) \\
 \includegraphics[scale=.3]{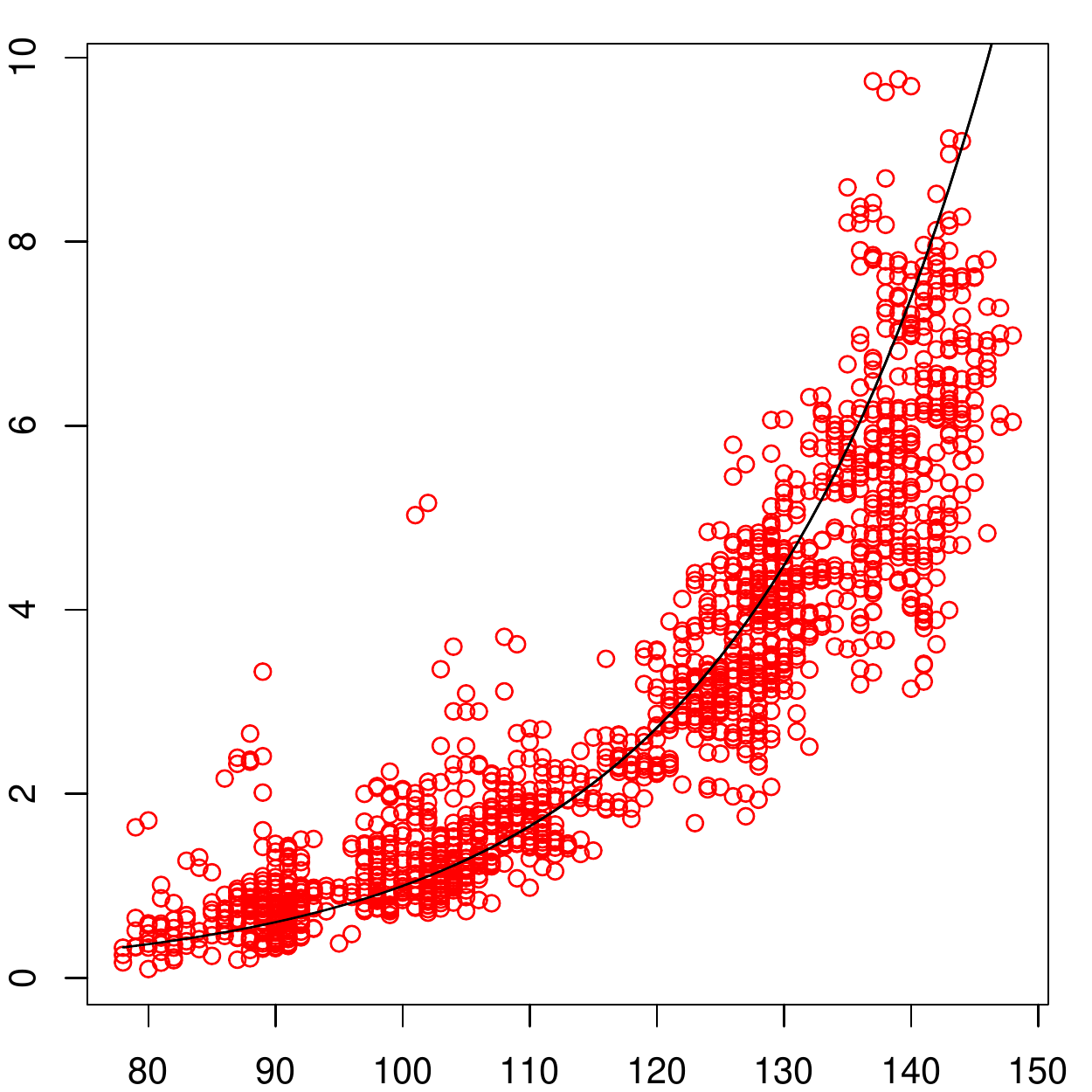} & \includegraphics[scale=.3]{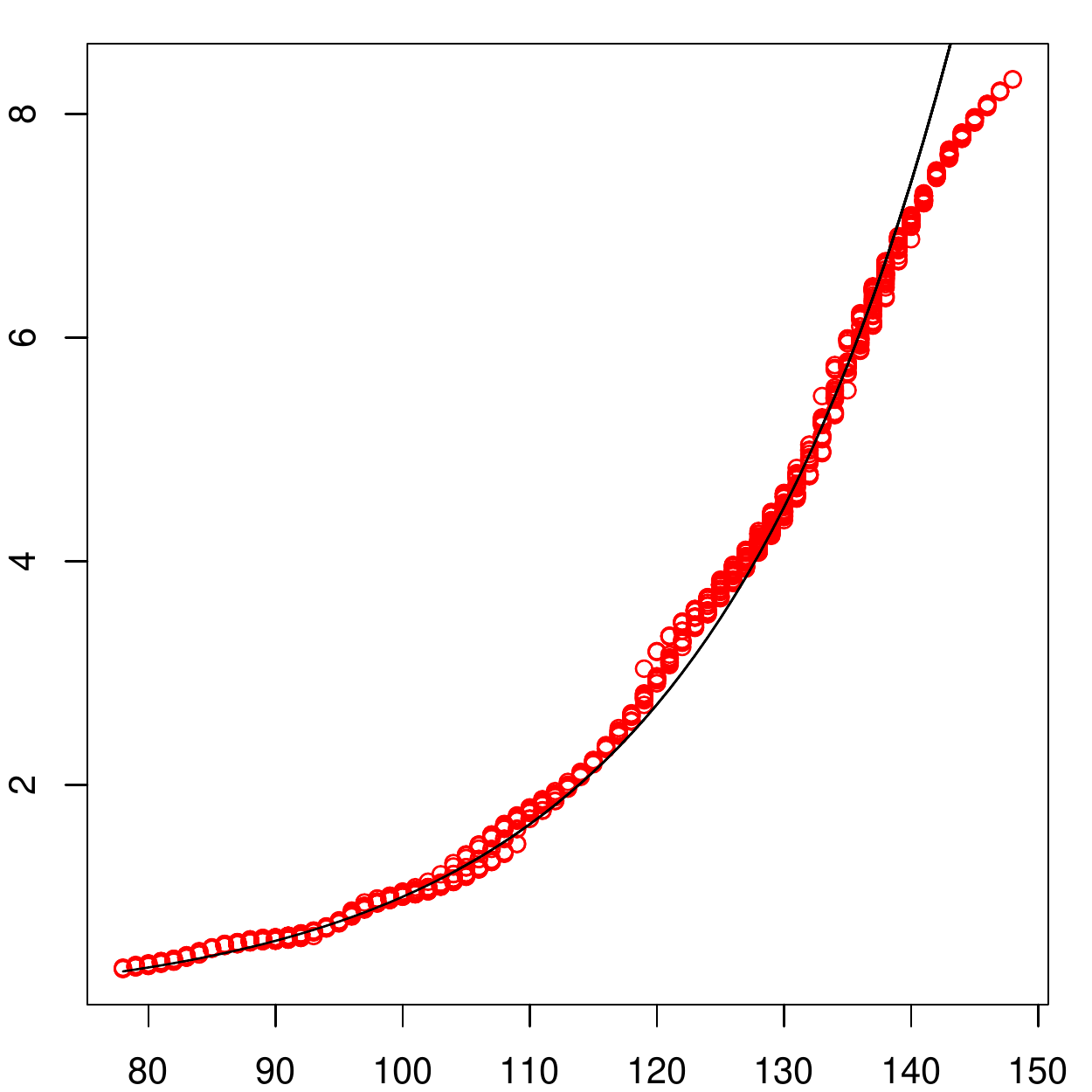} &  \includegraphics[scale=.3]{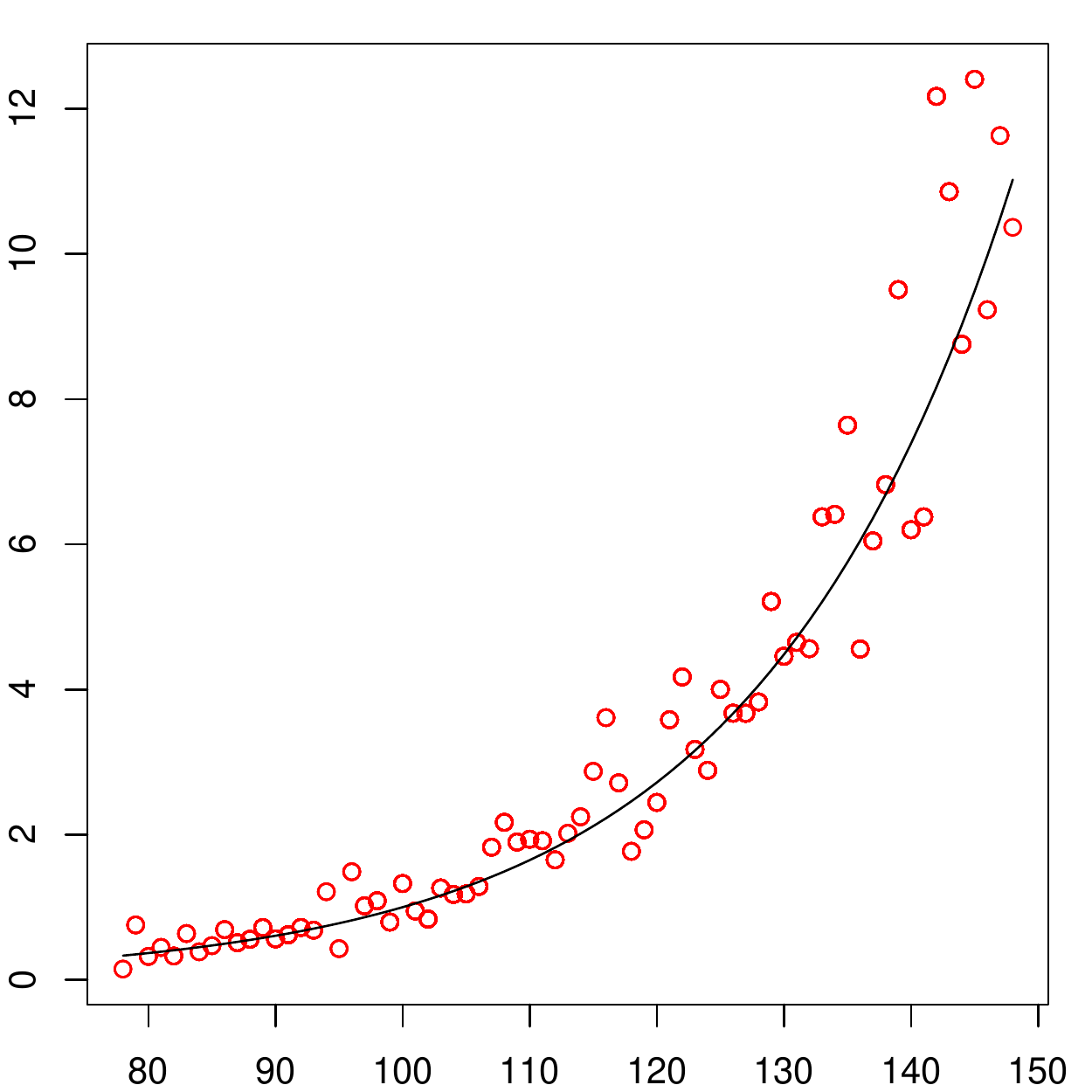} & \includegraphics[scale=.3]{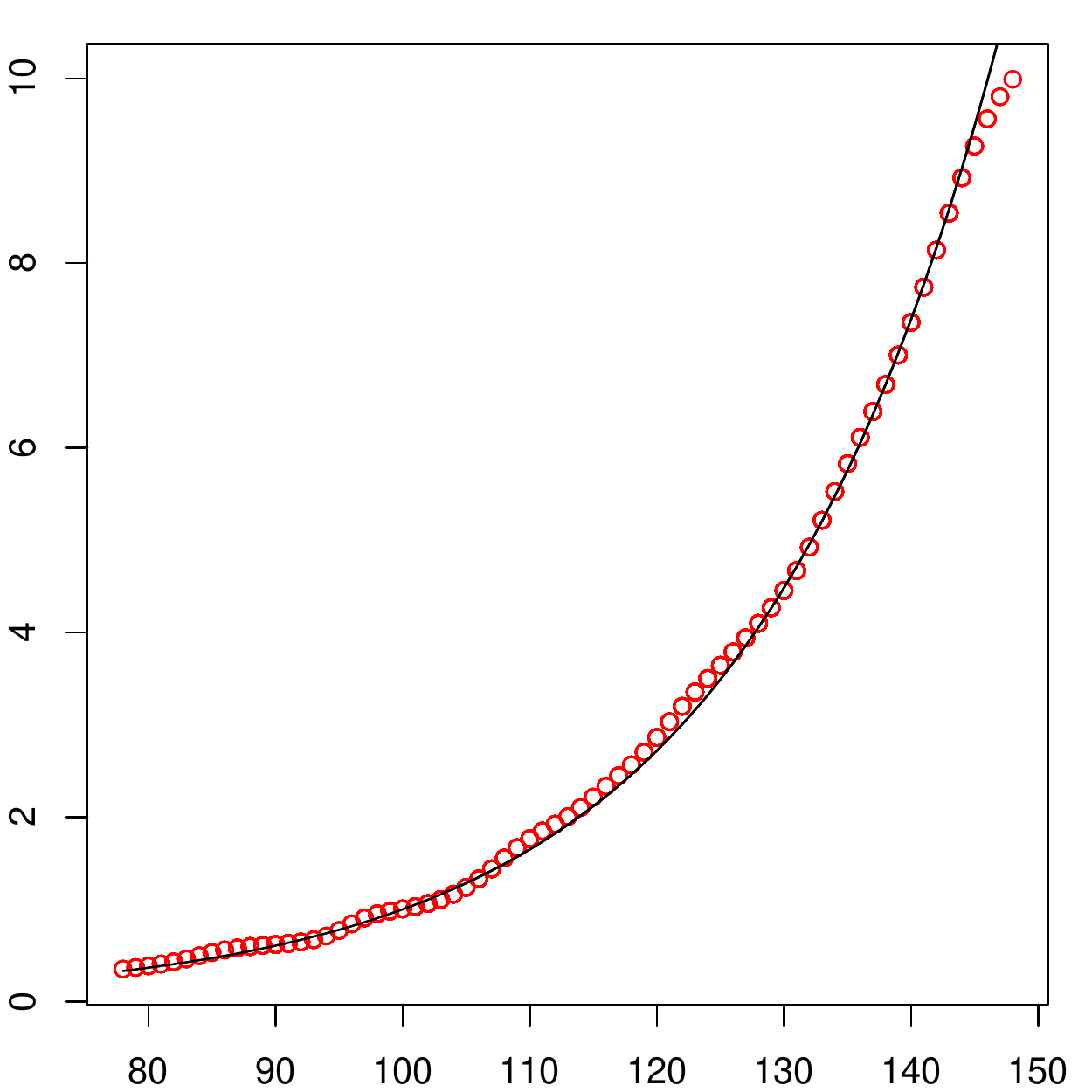} \\ 
 (e) & (f) & (g) & (h) 
\end{tabular}}
 \caption{{\small (a) True value of $\alpha(X_{T_i})$ in black, for $i=0,\dots,N_T$, along with its estimation in red, for the simulation of Section~\ref{simu:card}, using the estimator \eqref{est_cont} with the kernel $k_T$ specified in Example~\ref{ex Hausdorff}, i.e. based on the Hausdorff distance $d_H$; (b) Same as in (a) when the distance in $k_T$ is $d_\kappa$; (c)  Same as in (a) when $k_T$ is as in Example~\ref{ex simple} (i); (d) Same as in (a) when $k_T$ is as in Example~\ref{ex simple} (ii); (e)-(h) Ground-truth curve $\exp(5(n(X_{T_i})/100-1))$, in black, along with the scatterplots $(n(X_{T_i}),\hat\alpha(X_{T_i}))$, in red, for the same estimators as in (a)-(d) respectively.}}
\label{fig:alphaest} 
\end{center}
\end{figure}

From the two left scatterplots of Figure~\ref{fig:alphaest}, we see that the  two non-parametric estimators based on $d_H$ and $d_\kappa$ are able to detect a dependence between $\alpha(x)$ and $n(x)$. Between them, the estimator based on $d_\kappa$ seems by far more accurate. Concerning the two other estimators that assume a dependence in $n(x)$, the second one is much less variable. As indicated in Example~\ref{ex simple}, this is because this estimator exploits the underlying regularity properties of $\alpha(x)$ in $n(x)$, unlike the other one. It is remarkable that the non-parametric estimator based on $d_\kappa$, because it takes advantage of the smoothness of $\alpha(x)$, achieves better performances than the third estimator, although this one assumes the dependence in $n(x)$. All in all, the best results come from the last estimator, defined in  Example~\ref{ex simple}-(ii), that takes advantage of both the dependence in $n(x)$ and the smoothness of $\alpha(x)$.  These conclusions are confirmed by the mean squared errors reported in the first row of Table~\ref{table:mse}, computed from the estimation 
of $\alpha(x_j)$ at 100 point configurations $x_j=X_{T_j}$ for $j$ regularly sequenced from 1 to $N_T=1530$.

\begin{table}
\centering
\begin{tabular}{l|cccc}
  & $d_H$ & $d_{\kappa}$ & Ex. \ref{ex simple} (i) & Ex. \ref{ex simple} (ii)  \\ \hline
Continuous time obs. & 151 (34) &  18  (8) &  93 (21) & 1.8 (1.1) \\
Discrete time,  $m+1=5000$  &   266 (62) &  18  (8) & 141 (44) &  3.0  (1.9) \\
Discrete time, $m+1=1000$ & 226 (56) &  20  (9) &  NA &  4.1  (1.8)\\
Discrete  time, $m+1=100$ & 376 (94) &  36  (16) & NA & 36  (19) \\
Discrete time, $m+1=30$ & 767 (131) & 182  (48) & NA & 128  (37)
\end{tabular}
\caption{{\small Mean square errors of the estimation of $\alpha(x_j)$ at 100 different point configurations $x_j$, along with their standard deviations in parenthesis. The same four estimators as in Figure~\ref{fig:alphaest} are considered. 
First row:  estimation based on continuous time observations in $[0,T]$;  other rows: estimation based on $m+1$ observations regularly spaced in $[0,T]$  implying $N_T/m=1530/m$ jumps in average between each observation. 
 }}
\label{table:mse}
\end{table}

In order to assess the effect of discretisation, we consider $m+1$ observations $X_{t_j}$ taken from the above simulated continuous trajectory, regularly spaced from $t_0=0$ to $t_m=T=1000$. This implies an average of $N_T/m=1530/m$ jumps between two observations. We then apply the estimator \eqref{est discrete} of $\alpha(x)$ at the same 100 point configurations $x_j$ as above and the same four choices of $k_T$, based on the $m+1$ observations. For the approximation $D_j$ of the number of jumps between two observations $X_{t_{j-1}}$ and $X_{t_{j}}$,  we take the number of new points observed in $X_{t_{j}}$ plus the number of points having disappeared from $X_{t_{j-1}}$. The mean square errors of the results, for different values of $m$, are reported in Table~\ref{table:mse}, along with their standard deviations. Note that for some configurations $x_j$, there was no observation with cardinality $n(x_j)$ making impossible the computation of the third estimator, which explains the presence of some NA's in the table. As seen from Table~\ref{table:mse}, the comparison between the four estimators are in line with the continuous case and their performances increase with $m$, as could be expected. For illustration, Figure~\ref{fig:discrete} compares the true values of $\alpha(x_j)$ for $j=1,\dots,100$ with their estimation by the second estimator based on $d_\kappa$ (in blue) and the last estimator (in red), for the same different values of $m$ as in Table~\ref{table:mse}. While the quality of estimation degrades when $m$ decreases, it remains quite decent even for small values of $m$, in particular when $m+1=100$ implying more than 15 jumps in average between each observation.

\begin{figure}[ht]
\begin{center}
{\small
\begin{tabular}{cccc}
 \includegraphics[scale=.3]{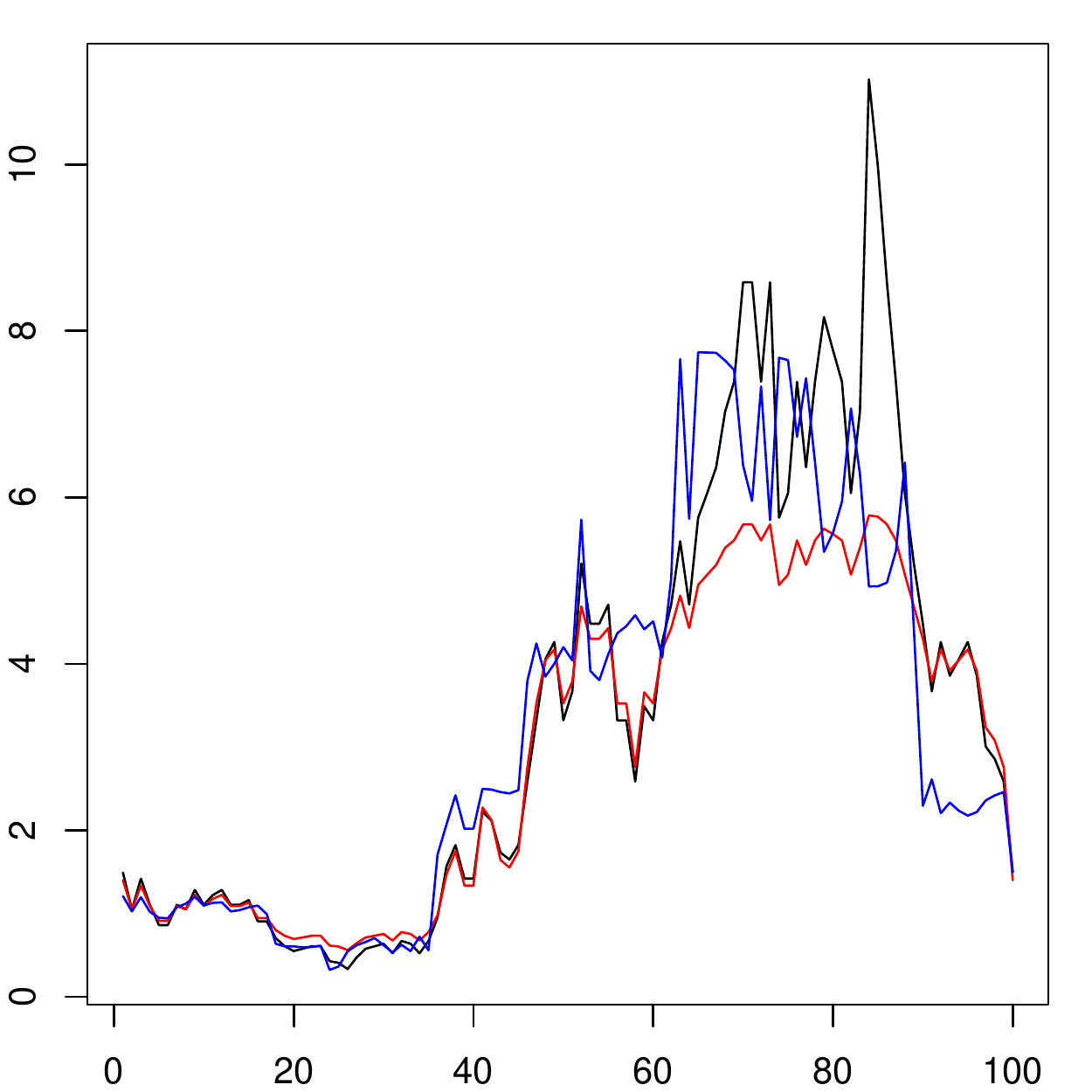} & \includegraphics[scale=.3]{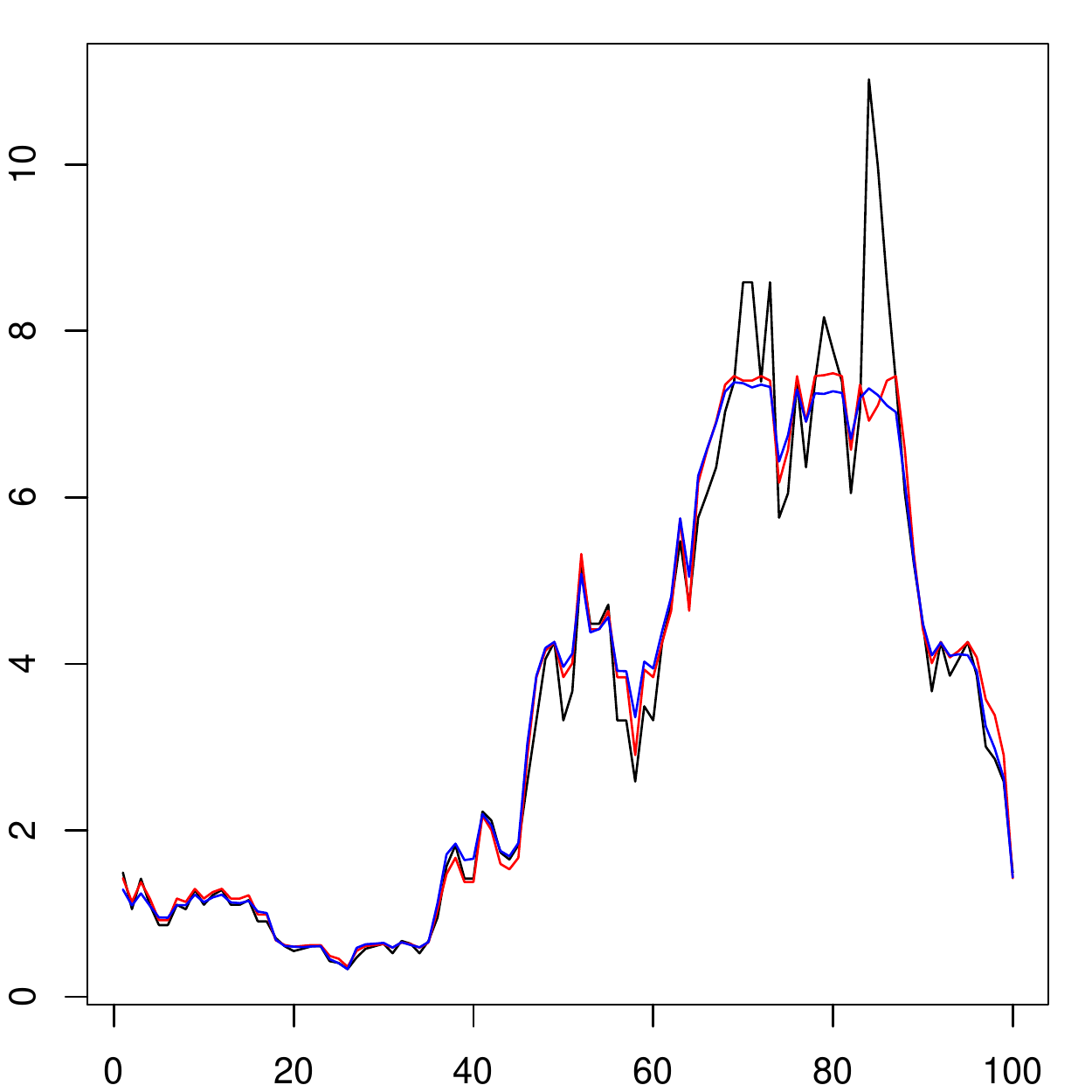} & \includegraphics[scale=.3]{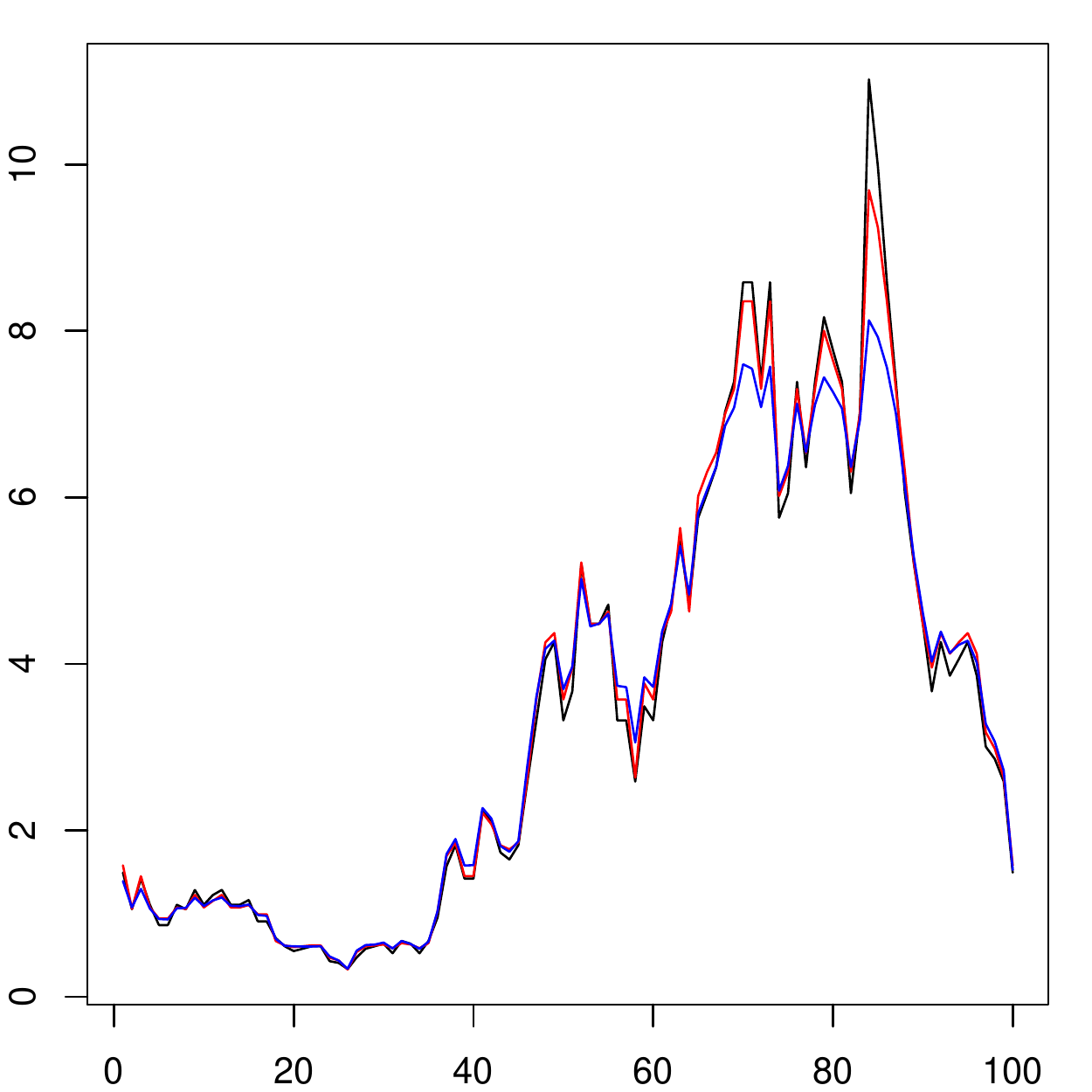} & \includegraphics[scale=.3]{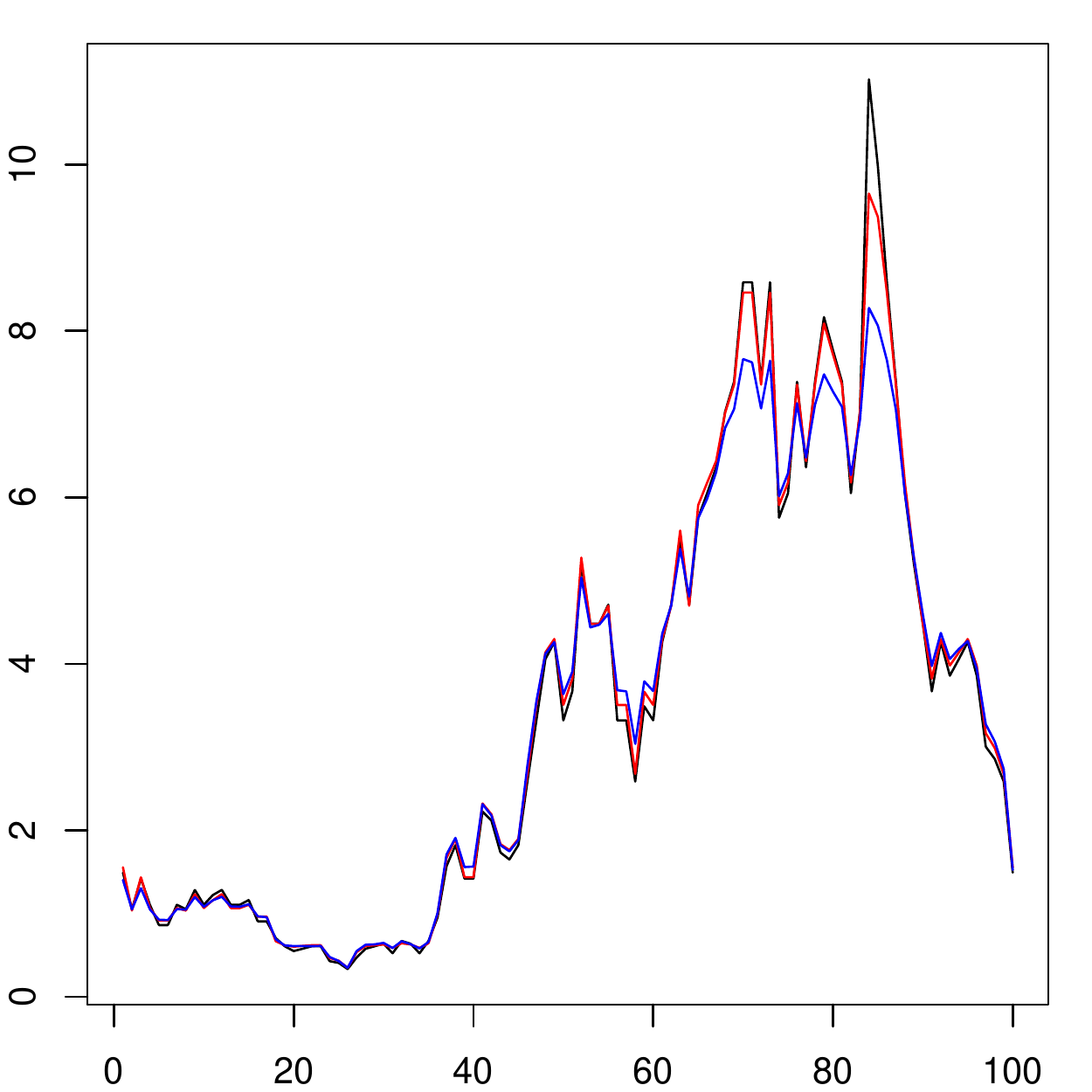} \\ 
(a) $m+1=30$ & (b) $m+1=100$ & (c) $m+1=1000$ & (d) $m+1=5000$\\
 \includegraphics[scale=.3]{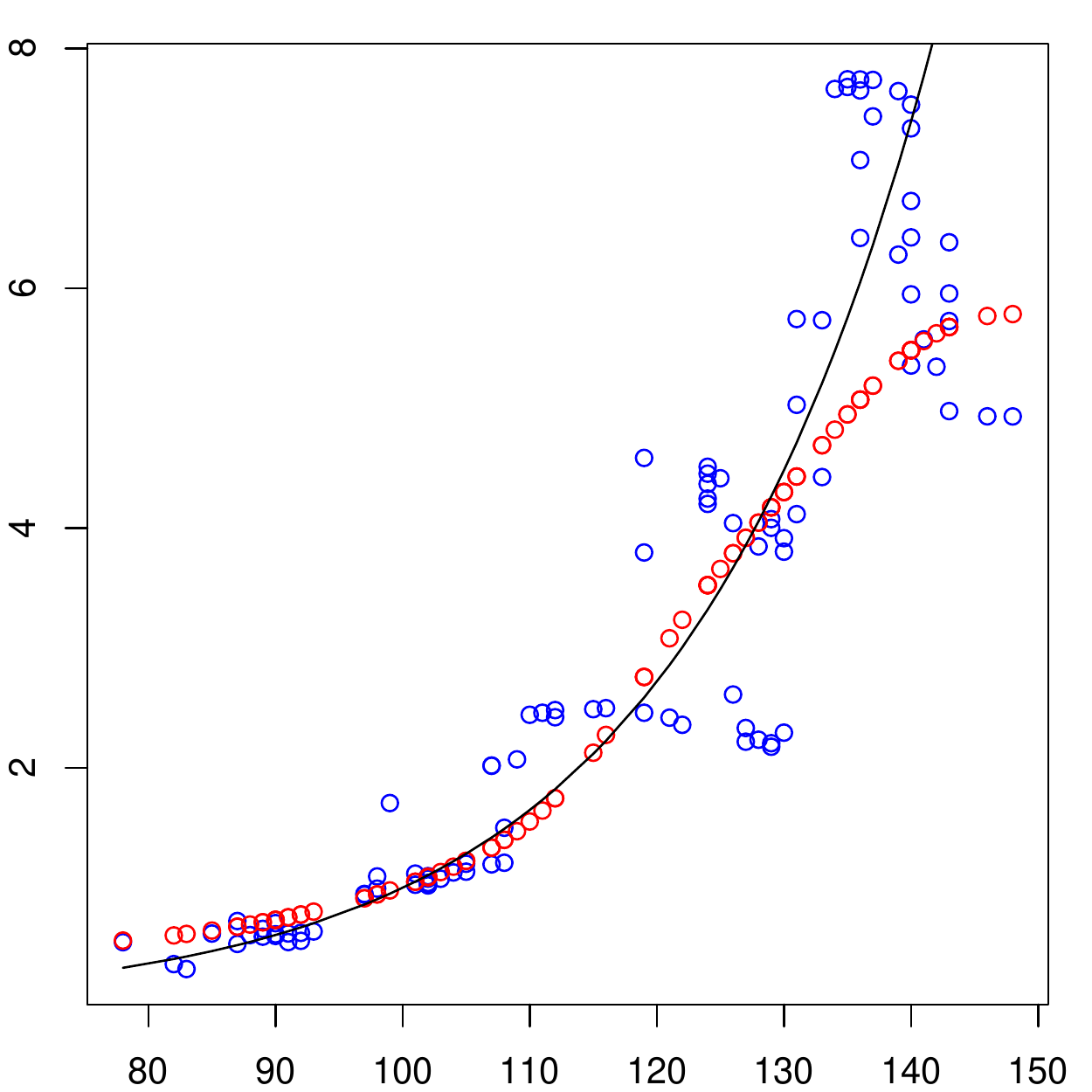} & \includegraphics[scale=.3]{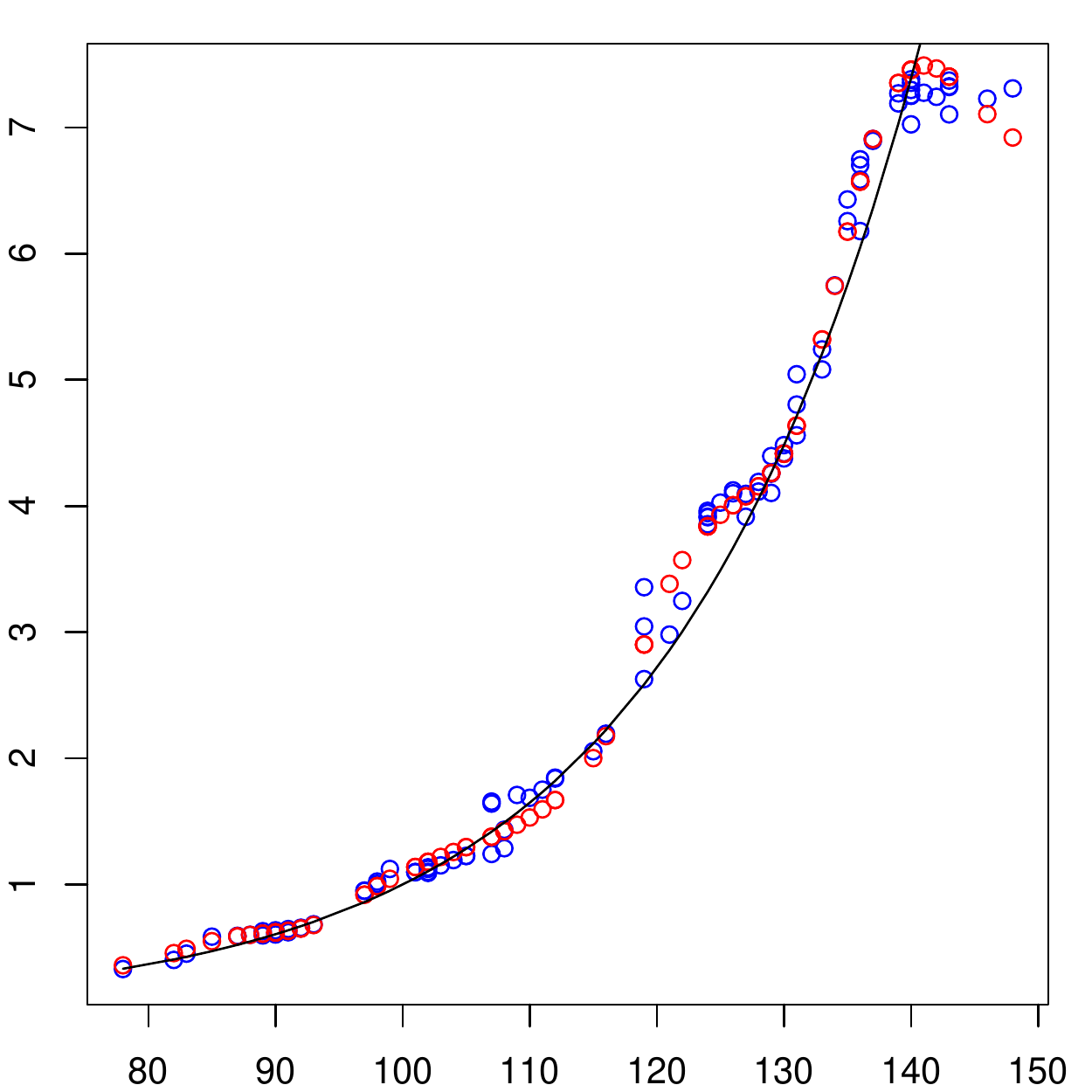} &  \includegraphics[scale=.3]{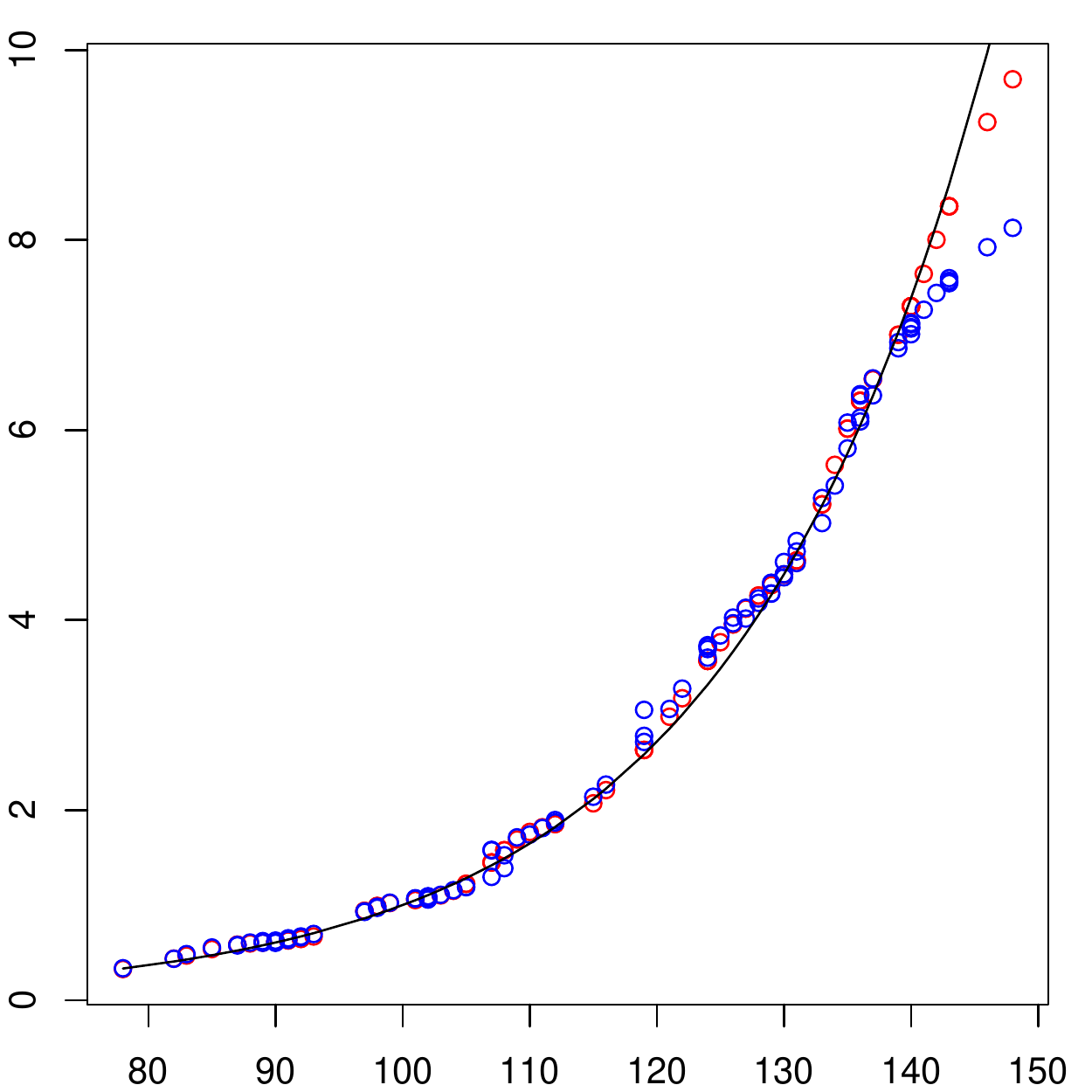} & \includegraphics[scale=.3]{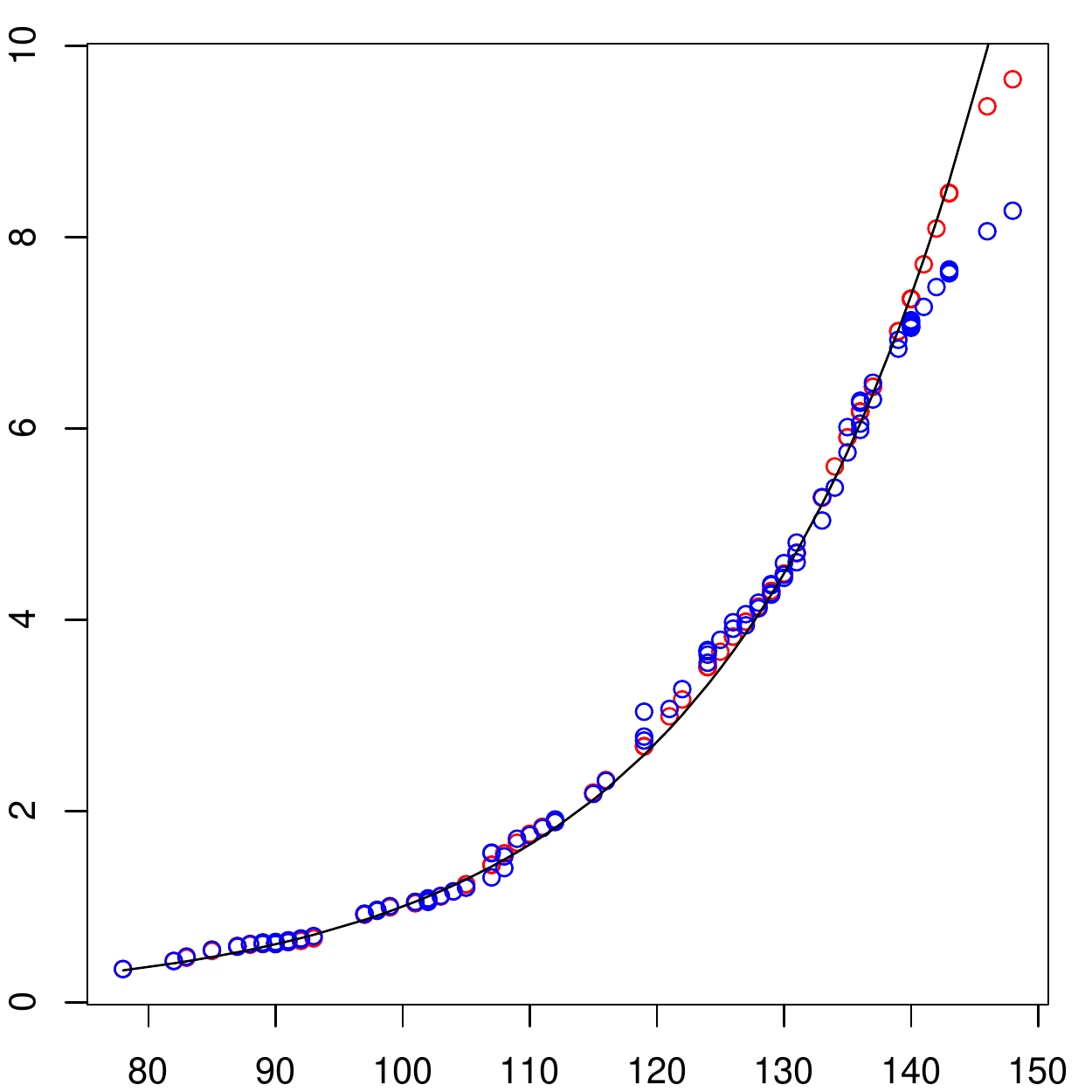} \\ 
\end{tabular}}
\caption{{\small (a) Top: true value of $\alpha(x_j)$ in black, for 100 different $x_j$ extracted from the simulated continuous trajectory of Section~\ref{simu:card}, along with
 the discrete time estimations  $\hat\alpha_{(d)}(x_j)$ based on $m+1=30$ observations,  with the choice of distance $d_\kappa$, in blue, and with the choice of $k_T$ as in Example~\ref{ex simple}-(ii), in red. Bottom:  ground-truth curve $\exp(5(n(x_j)/100-1))$, in black, along with the scatterplots $(n(x_j),\hat\alpha(x_j))$ for the same two estimators; (b)-(d) Same plots as in (a) but with $m+1=100$, $m+1=1000$ and $m+1=5000$ observations, respectively. 
}}\label{fig:discrete} 
\end{center}
\end{figure}

\subsection{Second situation: a dependence on geometric characteristics}\label{simu:geom}

For this example, we simulate a birth-death-move process on $W=[0,1]^2$ 
where the birth intensity function and  the birth transition kernel depend on the Delaunay tessellation associated to the current configuration. Specifically, given a point pattern $x$ in $W$,  the birth transition kernel $K_\beta(x,.)$ is the probability distribution on $W$ which is constant in each cell of the Delaunay tessellation of $x$ and equal to the area of the cell (up to the normalisation). An example of such distribution is given in the leftmost plot of Figure~\ref{fig:areadel}. Note that in order to allow a birth outside the convex hull of $x$, we add the corners of $W$ in the construction of the tessellation. This distribution makes a birth more likely to occur in the large empty spaces in between the points of $x$. Concerning the birth intensity function, we let $\beta(x)=\exp(50maxarea(x))/25)$ where $maxarea(x)$ denotes the maximal area of the cells. This choice  implies that births appear more quickly if there is a large ``available'' empty space. An alternative natural choice would have been to take the mean area of the cells instead of the maximal area.  But  choosing the mean area would imply a strong dependence between $\beta(x)$ and $n(x)$, making this example a bit too close to the previous example treated in Section~\ref{simu:card}. This is because for stationary spatial point processes, the expected area of a typical Delaunay cell is inversely proportional to the intensity of the process (\cite{chiu2013}), making the empirical mean area very close to the inverse number of points. Choosing the maximal area instead allows us to investigate a  different situation, where $\beta(x)$ does not depend on $n(x)$ too much. 
Finally, we let the death intensity function depend on the number of points: $\delta(x)=(n(x)/100)\1_{n(x)\leq 1000}$, while the death transition kernel $K_\delta(x,.)$ is just the uniform distribution over the points of $x$. 
The move process between each jump is a Brownian motion with standard deviation $2.10^{-3}$, independently applied to each point of the configuration.

\begin{figure}[ht]
\begin{center}
{\small
\begin{tabular}{ccc}
 \includegraphics[scale=.37]{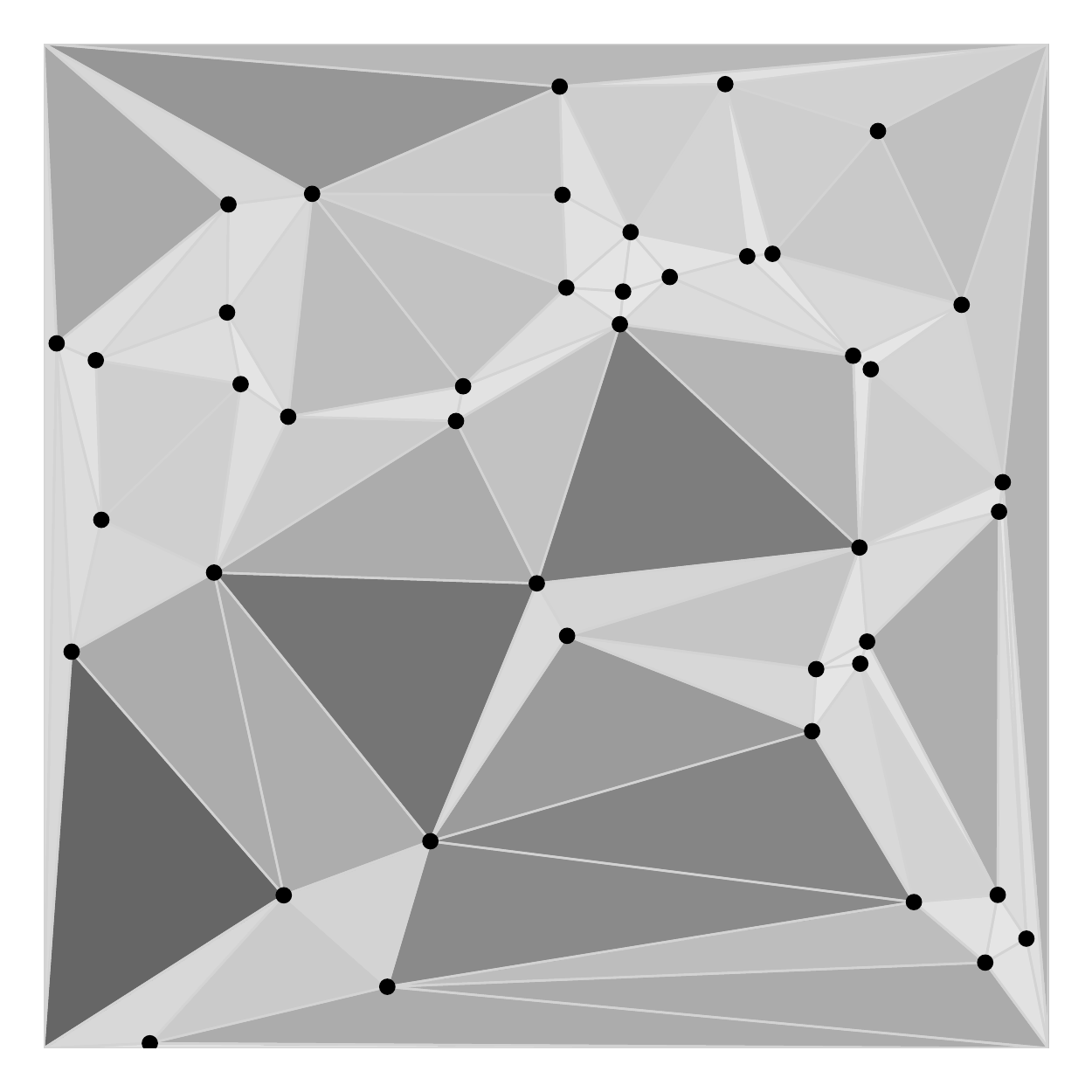} & \includegraphics[scale=.37]{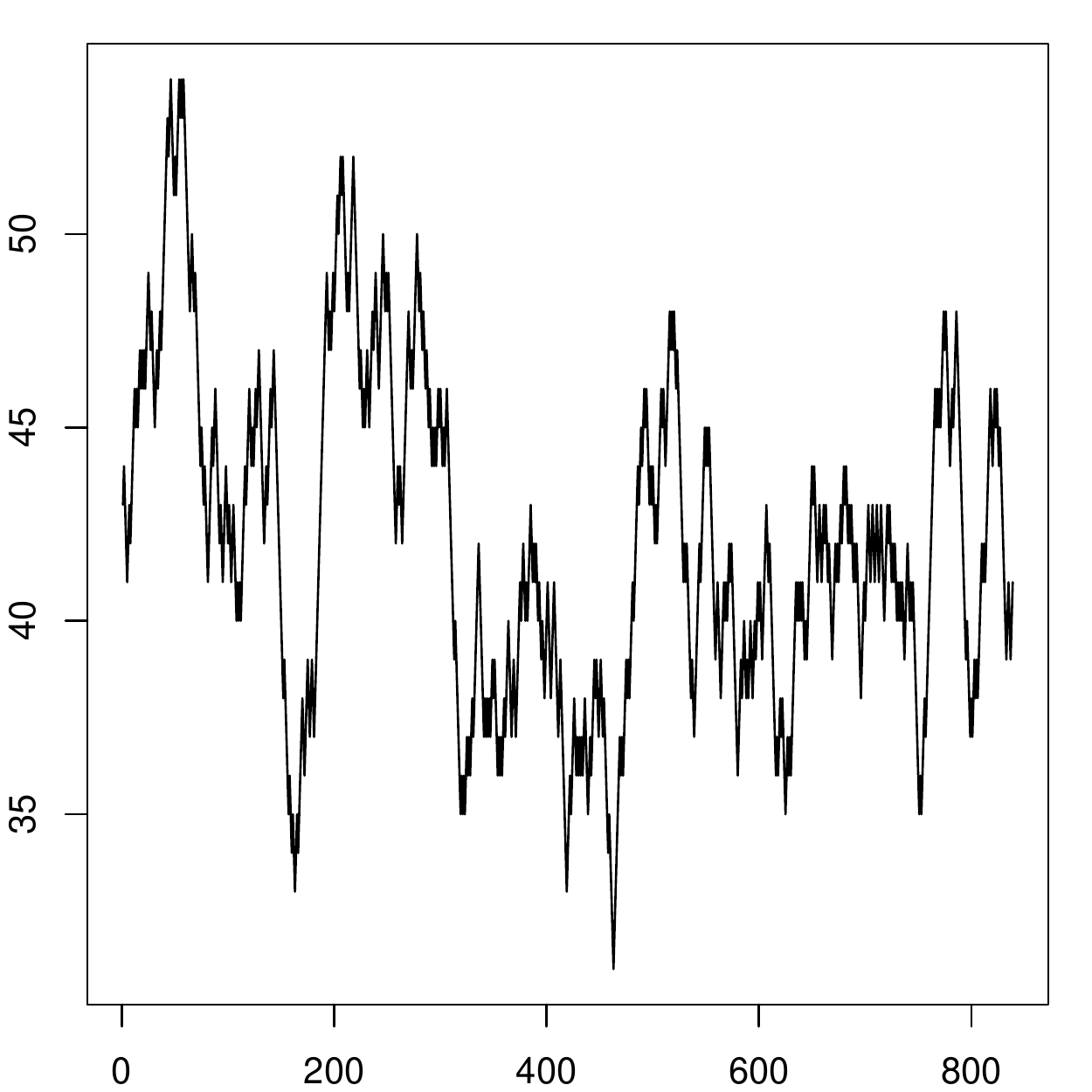} & \includegraphics[scale=.37]{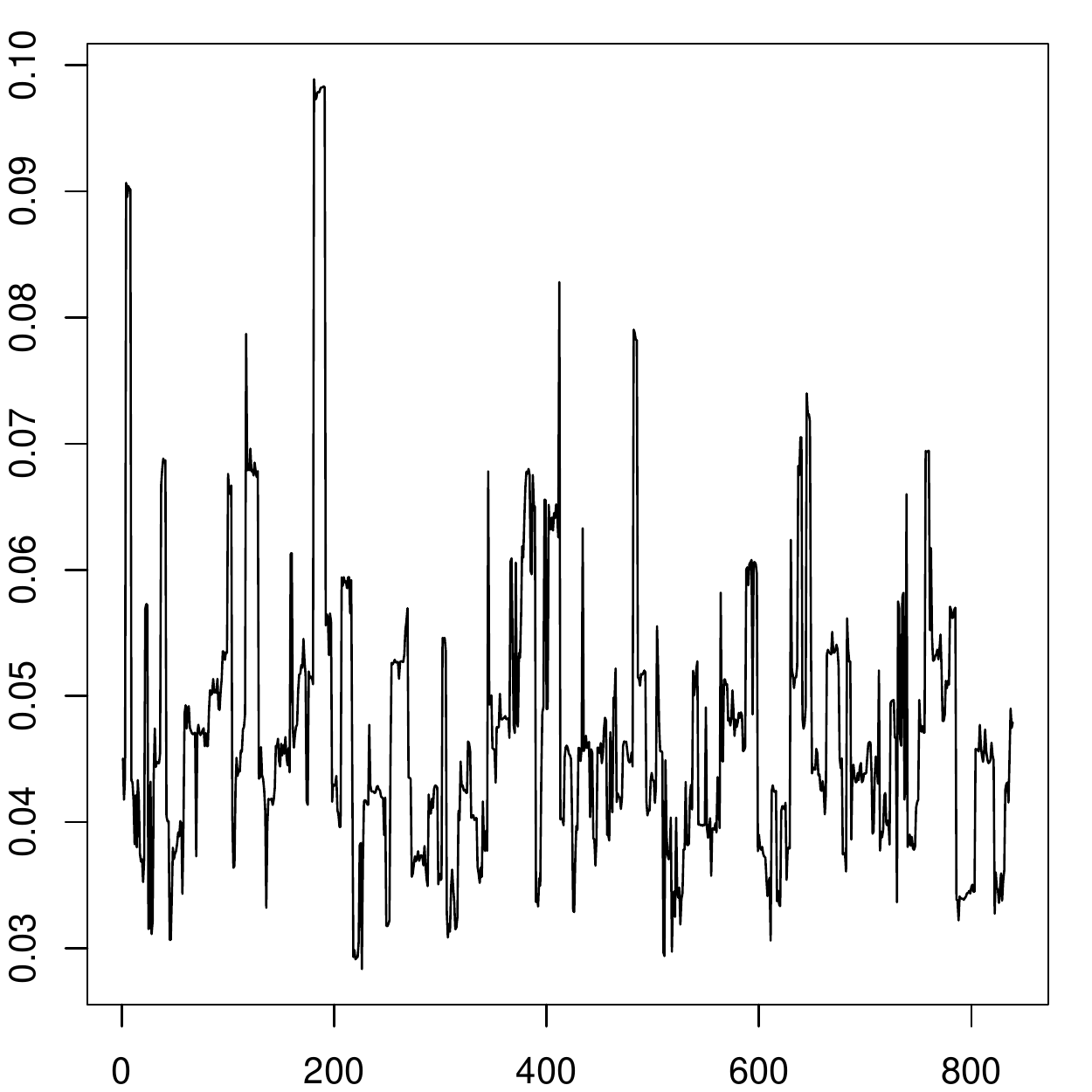}\\
 (a) & (b) & (c) \end{tabular}}
\caption{{\small (a) Initial point configuration at $t=0$ for the simulated trajectory of Section~\ref{simu:geom}, along with its Delaunay tessellation where each cell is coloured proportionally to its area. This coloured map represents the probability distribution  defining the birth transition kernel, while the birth intensity is an increasing function of the maximal area;
(b) Evolution of the number of points before each jump; (c) Evolution of the maximal area of the Delaunay cells before each jump.}}\label{fig:areadel} 
\end{center}
\end{figure}

\begin{figure}[ht]
\begin{center}
{\small
\begin{tabular}{cccc}
 \includegraphics[scale=.4]{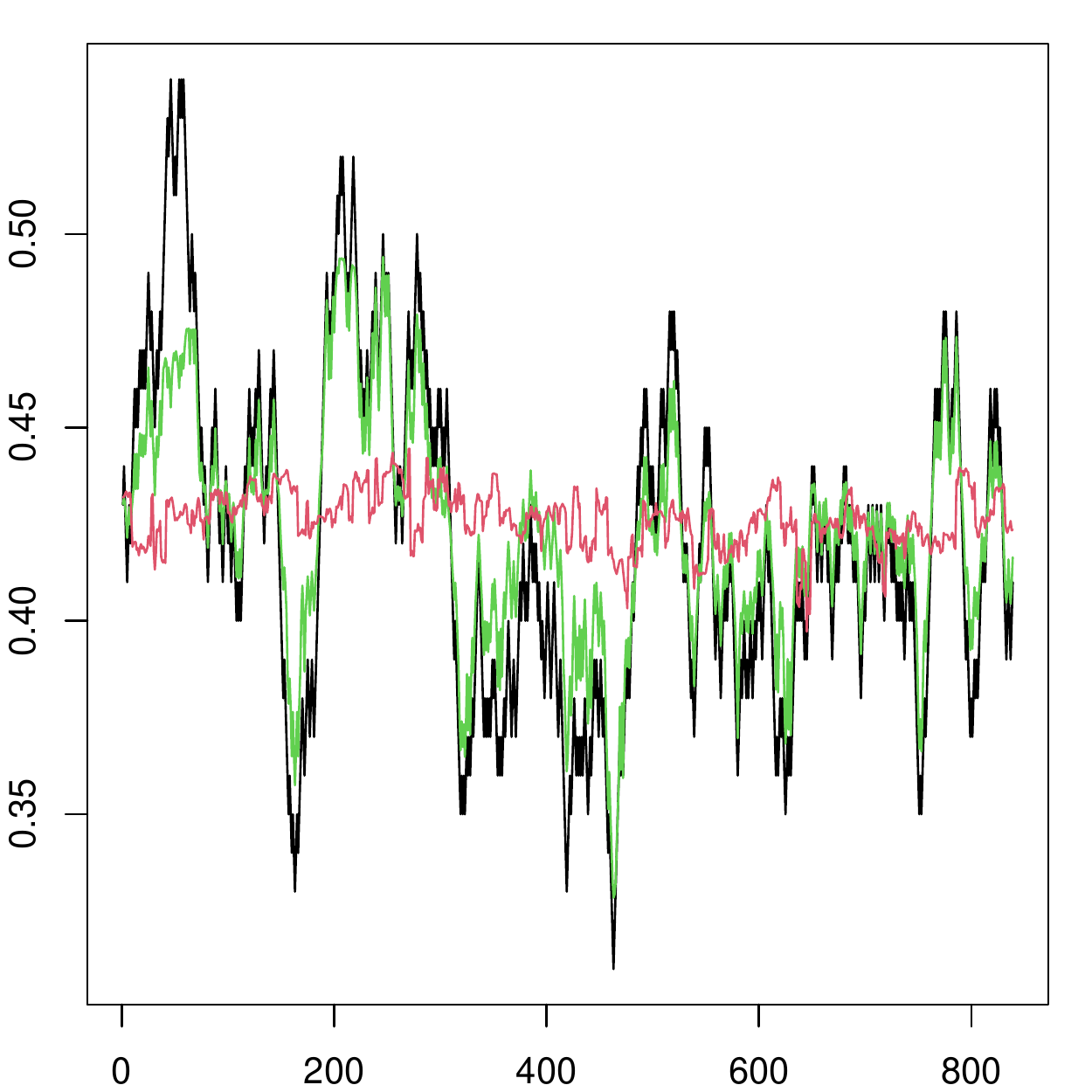} & \includegraphics[scale=.4]{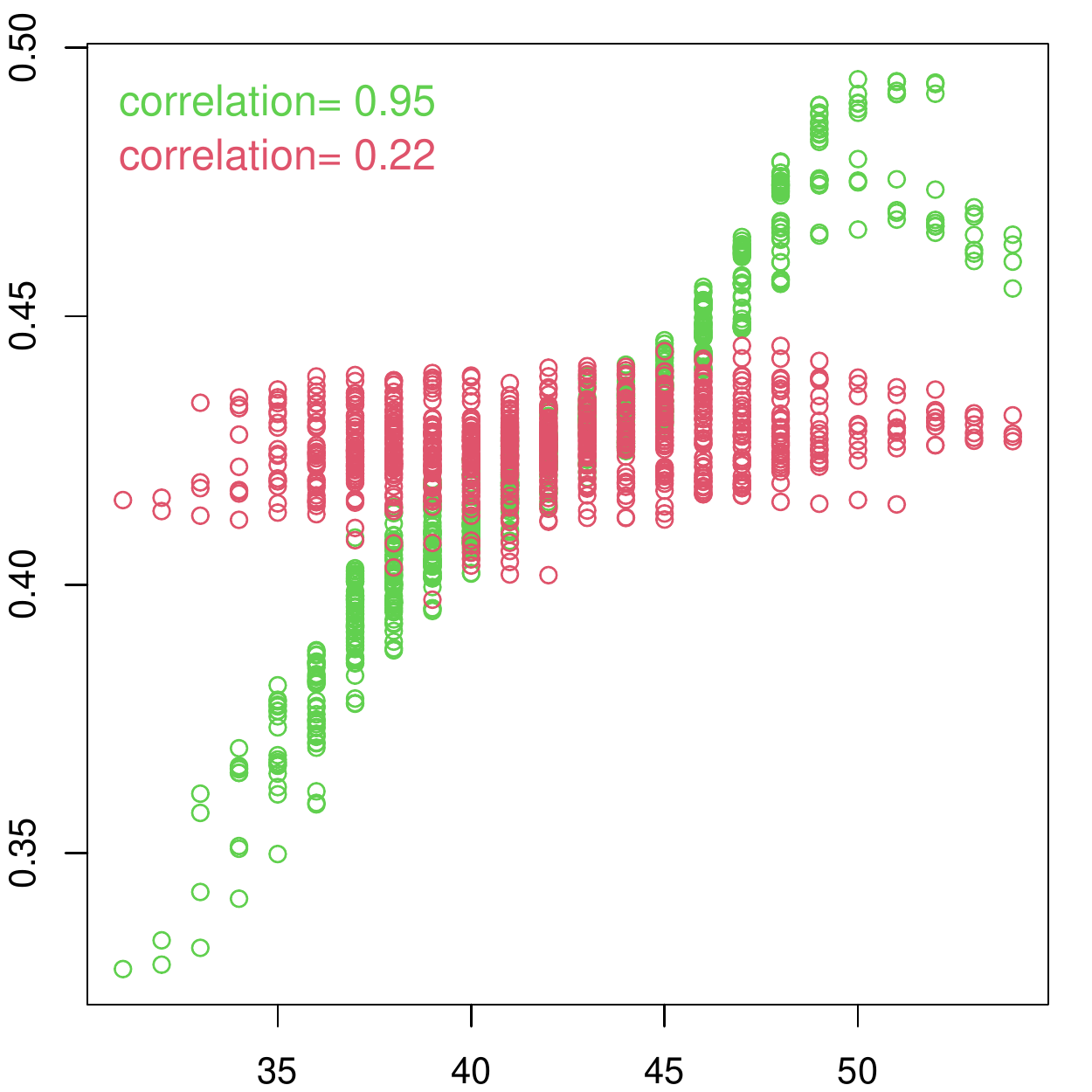} \\ (a) & (b) \\
 \includegraphics[scale=.4]{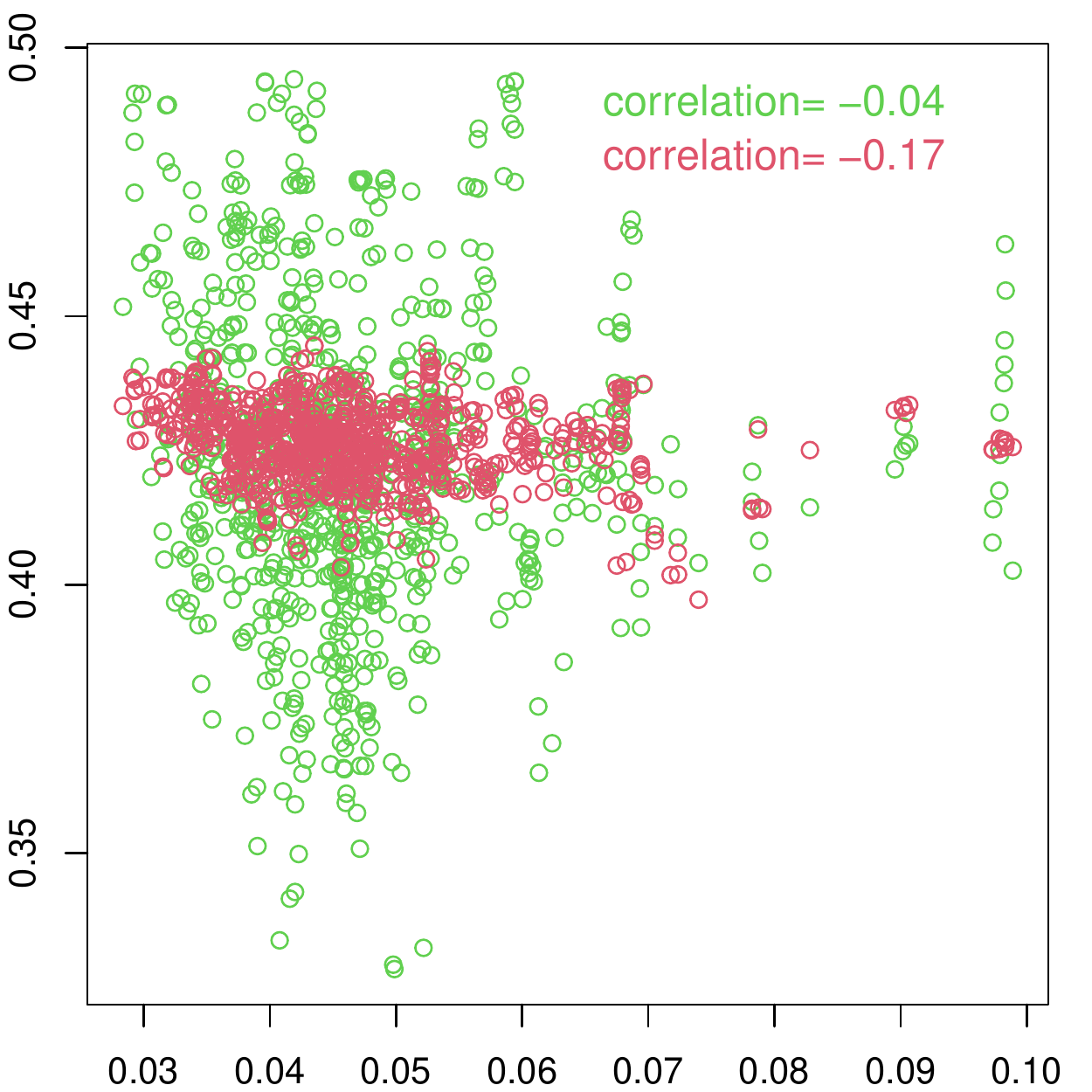} &  \includegraphics[scale=.4]{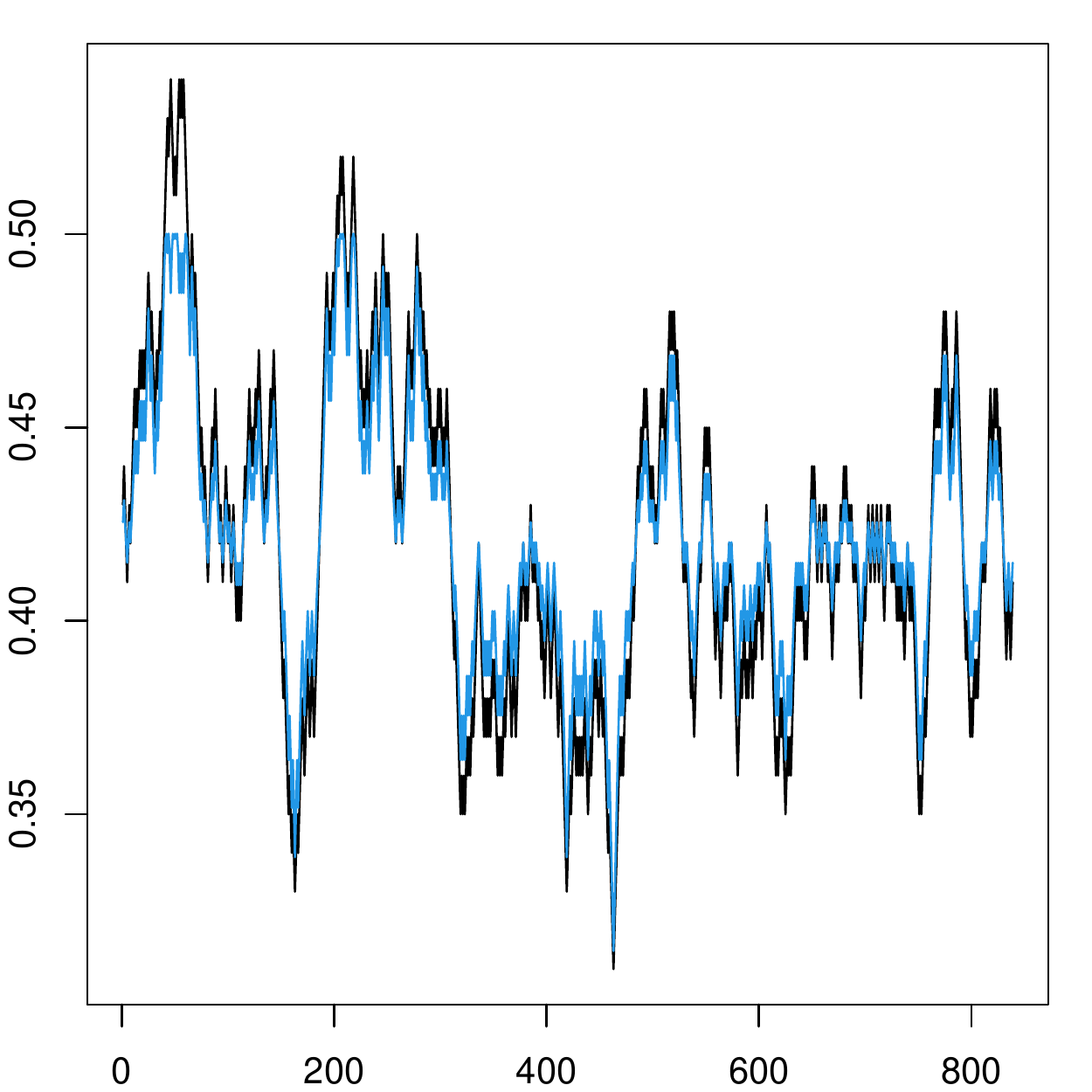}\\
(c) & (d) 
\end{tabular}}
\caption{{\small (a) True value of $\delta(X_{T_i})$ in black, for $i=0,\dots,N_T$ for the simulation of Section~\ref{simu:geom}, along with its non-parametric estimation based on the Hausdorff distance (in red) and based on $d_\kappa$ (in green); (b) Scatterplots $(n(X_{T_i}),\hat\delta(X_{T_i}))$ for these two estimators; (c) Scatterplots $(maxarea(X_{T_i}),\hat\delta(X_{T_i}))$ for these two estimators; (d) True value of  $\delta(X_{T_i})$ in black
along with its estimation based on  $d(x,y)=|n(x)-n(y)|$ in blue.}}\label{fig:deldeath} 
\end{center}

\end{figure}

\begin{figure}[ht]
\begin{center}
{\small
\begin{tabular}{cccc}
 \includegraphics[scale=.4]{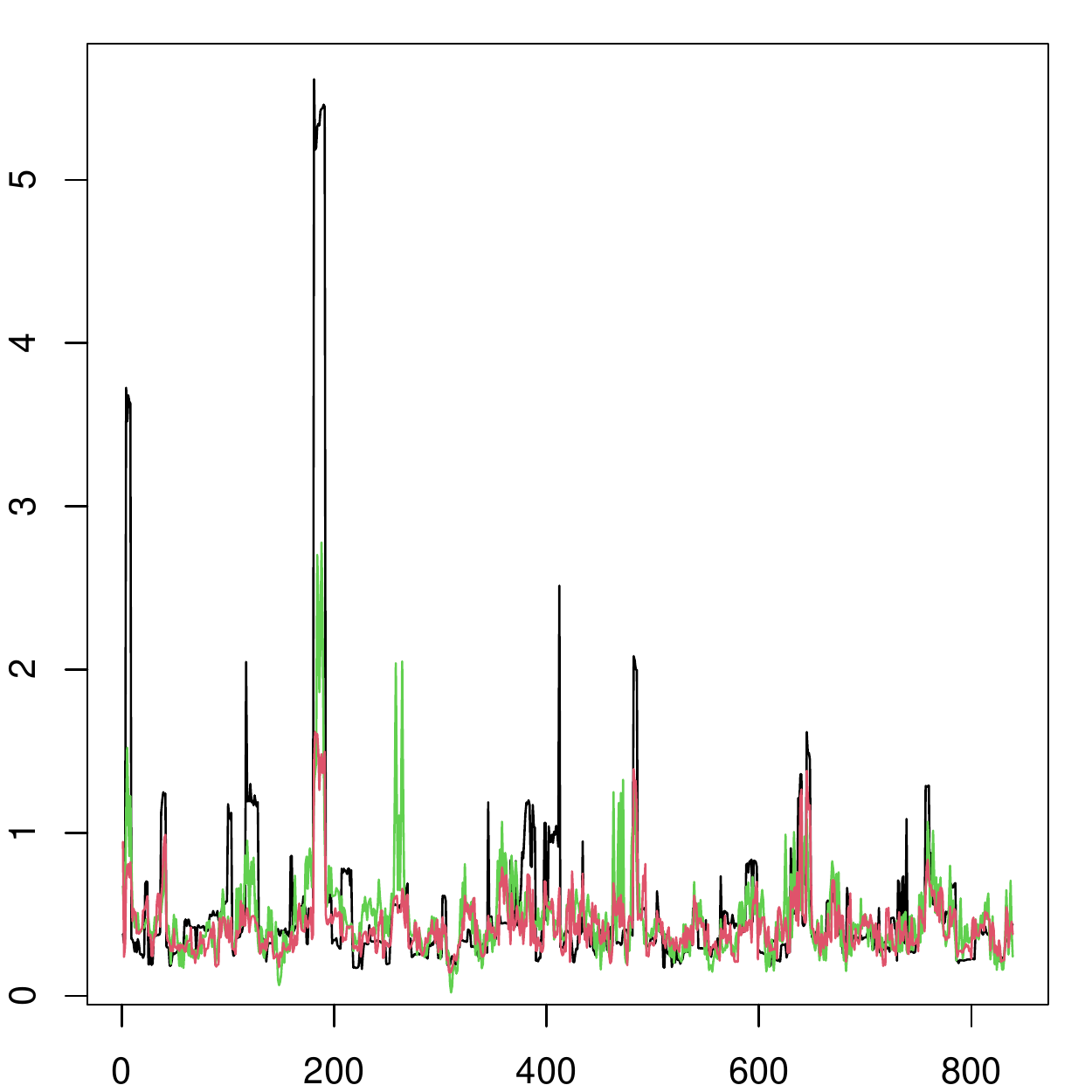} & \includegraphics[scale=.4]{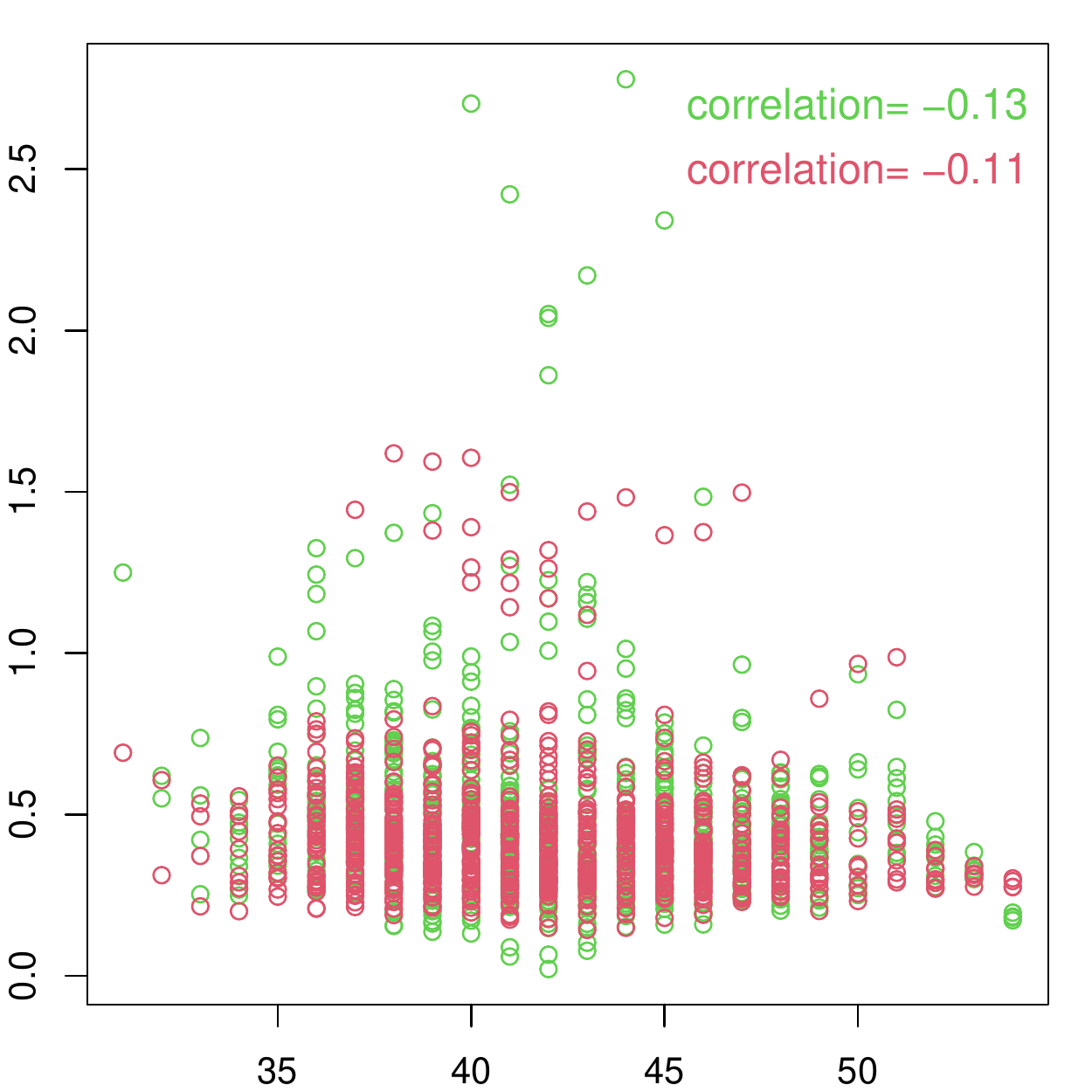} \\ (a) & (b) \\
\includegraphics[scale=.4]{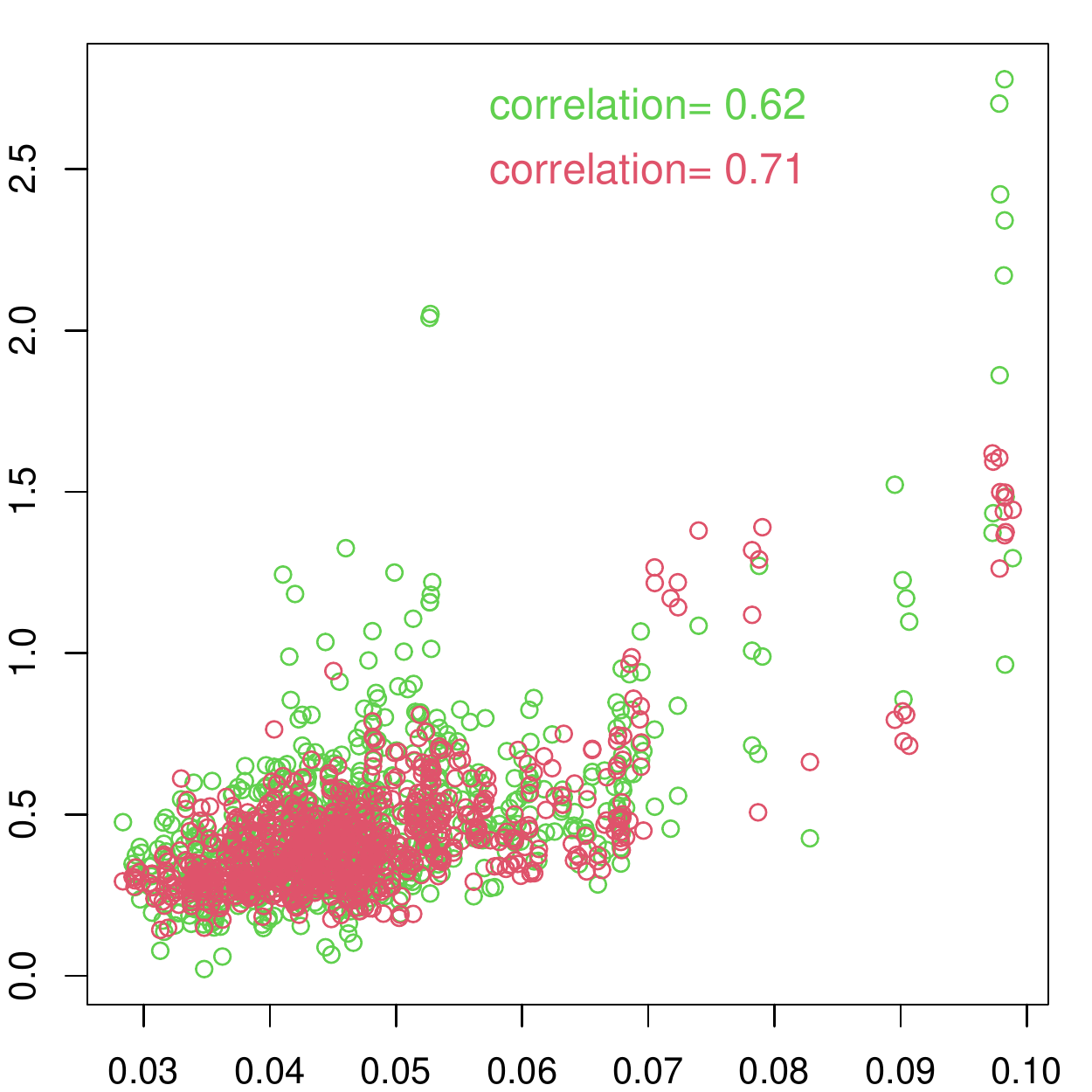} &  \includegraphics[scale=.4]{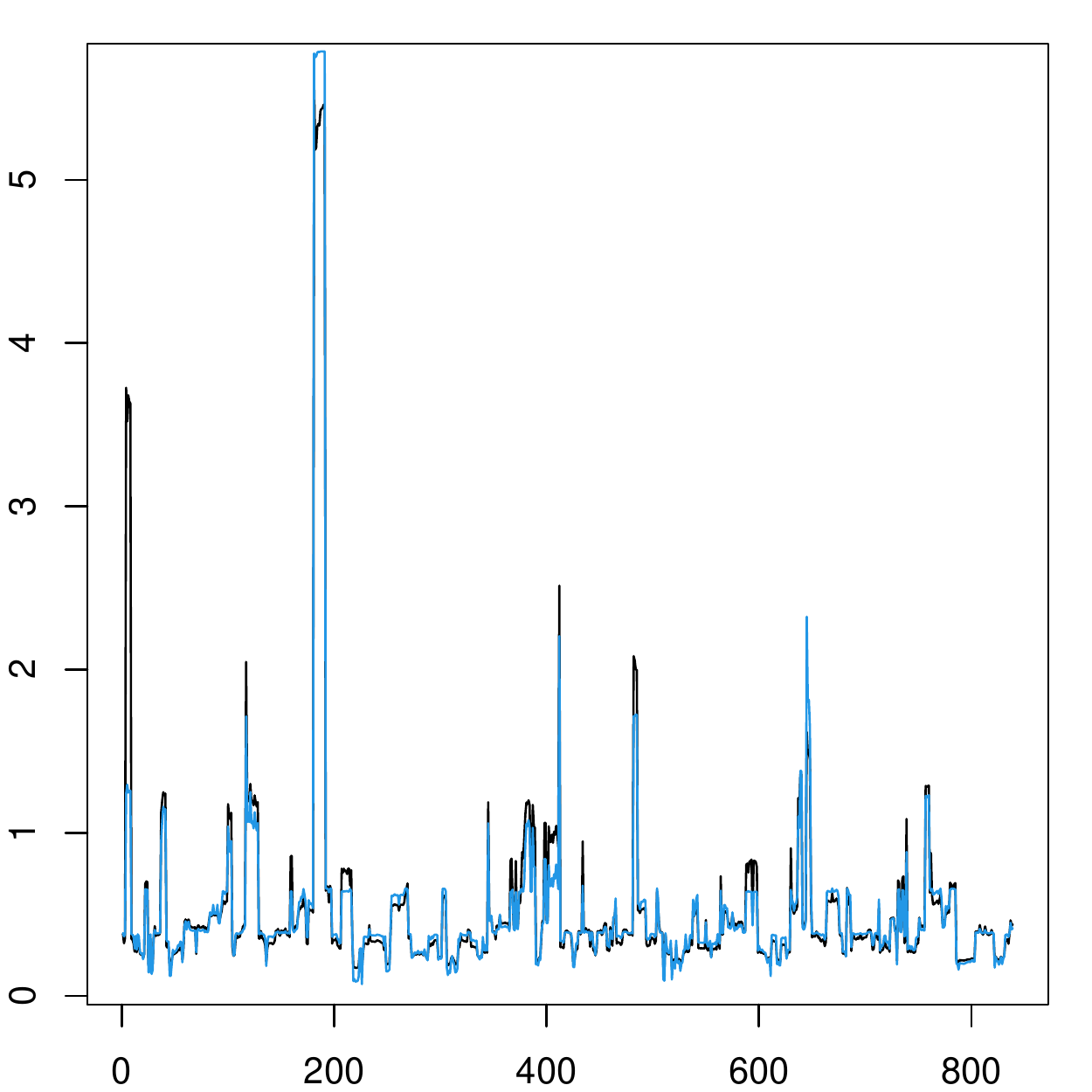}\\
(c) & (d) 
 \end{tabular}}
\caption{{\small Same plots as in Figure~\ref{fig:deldeath} but concerning the estimation of the birth intensity function $\beta(.)$ and where the estimator in the last plot (d) is based on $d(x,y)=|maxarea(x)-maxarea(y)|$.}}\label{fig:delbirth} 
\end{center}
\end{figure}

The initial point pattern at $t=0$ is shown in the first plot of Figure~\ref{fig:areadel} along with its associated Delaunay tessellation. In our simulation 839 jumps have been observed. The evolution of the number of points over the jump times is represented in the middle plot of this figure, while the rightmost plot shows the evolution of the maximal area of the Delaunay cells before each jump. 

The estimation of the death and the birth intensity functions are considered in Figures~\ref{fig:deldeath} and~\ref{fig:delbirth}, respectively. 
We first consider the two purely non-parametric estimators associated to the Hausdorff distance (in red) and to the distance $d_\kappa$ (in green). The results are shown in Figures~\ref{fig:deldeath}-(a) and~\ref{fig:delbirth}-(a), where the ground-truth curves are in black. As a first remark, we can see that for the death intensity function, the estimator based on the Hausdorff distance completely fails  in  this simulation, while the estimator based on $d_\kappa$ works pretty well. For the estimation of the birth intensity function, both estimators behave similarly but none is very accurate. These observations tend to confirm our preference, in a non-parametric approach, for the choice $d_\kappa$ instead of the Hausdorff distance, and they also show the limits of a pure non-parametric approach in such a high-dimensional estimation problem. However, based on these estimations, we can question the dependence between the intensities and some features of the configurations through scatterplots. Figures~\ref{fig:deldeath}-(b) and~\ref{fig:delbirth}-(b) show the scatterplots between the estimations at each jump time and the number of points at these times, while Figures~\ref{fig:deldeath}-(c) and~\ref{fig:delbirth}-(c) correspond to the scatterplots between the estimations and the maximal area of the Delaunay cells. 
These plots seem to indicate a dependence between the death intensity function and the number of points (but not the maximal area of cells), and a dependence between the birth intensity function and the maximal area of cells (but not the number of points).
Finally, if we apply accordingly our estimator with the choice $d(x,y)=|n(x)-n(y)|$ for the death intensity and the choice $d(x,y)=|maxarea(x)-maxarea(y)|$ for the birth intensity, we obtain the estimations on  Figures~\ref{fig:deldeath}-(d)  and~\ref{fig:delbirth}-(d). 
The results are then very good because we have identified the relevant feature defining $\delta(.)$ and $\beta(.)$, reducing the estimation problem to a non-parametric estimation in dimension 1.

\section{Data analysis}\label{sec:data}

Our data consist of a sequence of $m=1199$ frames showing the locations of two types of proteins inside a living cell, namely Langerin and Rab11 proteins, both being involved in exocytosis mechanisms in cells. The total length of the sequence is 171 seconds for a 140 ms time interval between each frame. These images have been acquired by 3D multi-angle TIRF (total internal reflection fluorescence) microscopy technique \citep{Boulanger2014}, and we observe projections along the $z$-axis on a 2D plane close to the plasma membrane of the cell. The raw sequence can be seen online in our  GitHub repository.   
As a result of  this acquisition, we observe  several tens of proteins of each type on each frame following some random motions, while some new proteins appear at some time point and others disappear.  The reason why a protein becomes visible can be simply due to its appearance into the axial resolution of the microscope, or because it becomes fluorescent only at the last step of the exocytosis process due to the pH change close to the plasma membrane. Similarly, a protein disappears from the image when it exits the axial resolution or when it disaggregates after the exocytosis process. Between its appearance and disappearance, the dynamics of a protein depends on its function and its environment. The whole spatio-temporal process at hand is thus composed of multiple fluorescent spots appearing, moving and disappearing over time, all of them possibly in interaction with each other. 
The underlying biological challenge is to be able to decipher this complex spatio-temporal dynamics, and in particular to understand the interaction between the different types of involved proteins, in the present case Langerin and Rab11 proteins \citep{Gidon, Boulanger2014}. 
Existing works either study the trajectories of each individual protein, independently to the other proteins  \citep{Briane19, pecot2018}, or investigate the interaction between different types of proteins frame by frame (which is the co-localization problem), without temporal insight \citep{Costes2004,Bolte2006,Lagache2015,Lavancier18}. As far as we know, the present approach is the first attempt to tackle the joint spatio-temporal dynamics of two types of proteins involved in exocytosis mechanisms.

To analyse the data, we do not consider the raw sequence but the post-processed sequence introduced in \cite{pecot2008patch,pecot2014background},  leading to more valuable data, where the most relevant regions of the cell, corresponding to the locations of the most dynamical proteins, have been enhanced thanks to a specific filtering procedure. The post-processed sequence for each type of protein is available online in our  GitHub repository.  We then apply the U-track algorithm developed in  \cite{jaqaman2008} in order to track over time the locations of proteins. For both types of them, the result is a sequence of $1199$ point patterns that follow a certain birth-death-move dynamics for the reasons explained earlier. Figure~\ref{fig:data} shows the repartition of Langerin (resp. Rab11) proteins in the first (resp. second) row, for a few frames extracted from these two sequences. The two leftmost plots correspond to the two first frames: we observe that a new Langerin protein and two new Rab11 proteins, visible in red, appeared between times $t_0=0$ and $t_1=0.14s$. In the second plot, we also recalled the initial positions of proteins as gray dots. Close inspection reveals that the proteins have slightly moved between the two frames. This motion is more apparent on the full sequence available online.
For the Langerin sequence, we observe   21 to 76 proteins per frame ($36.4$ in average) and $1.26$ jumps in average between each frame, $50.7\%$ of which being deaths and $49.3\%$ being births. For the Rab11 sequence, 
there are 10 to 52 proteins per frame ($22.3$ in average) and $0.85$ jumps in average between each frame, $50.6\%$ being deaths and $49.4\%$ being births. 
Based on these observations, we would like to question several biological hypotheses:
\begin{itemize}
\item[i)] The first one assumes that each protein may appear at any time independently on the configuration and the number of proteins already involved in the exocytosis process. This would imply a constant birth intensity function  over the whole sequence, for both types of proteins. 
\item[ii)] The second hypothesis is that each protein disappears independently of the others after its exocytosis process. Accordingly, the death intensity function at a configuration $x$ should then depend linearly on the cardinality of $x$, that is on the number of currently active proteins.
\item[iii)] The third hypothesis is that Langerin and Rab11 proteins interact during the exocytosis process, which should imply a correlation between their respective intensities. 
\end{itemize}
  
  \begin{figure}[ht]
\begin{center}
{\small
\begin{tabular}{cccc}
 \includegraphics[scale=.3]{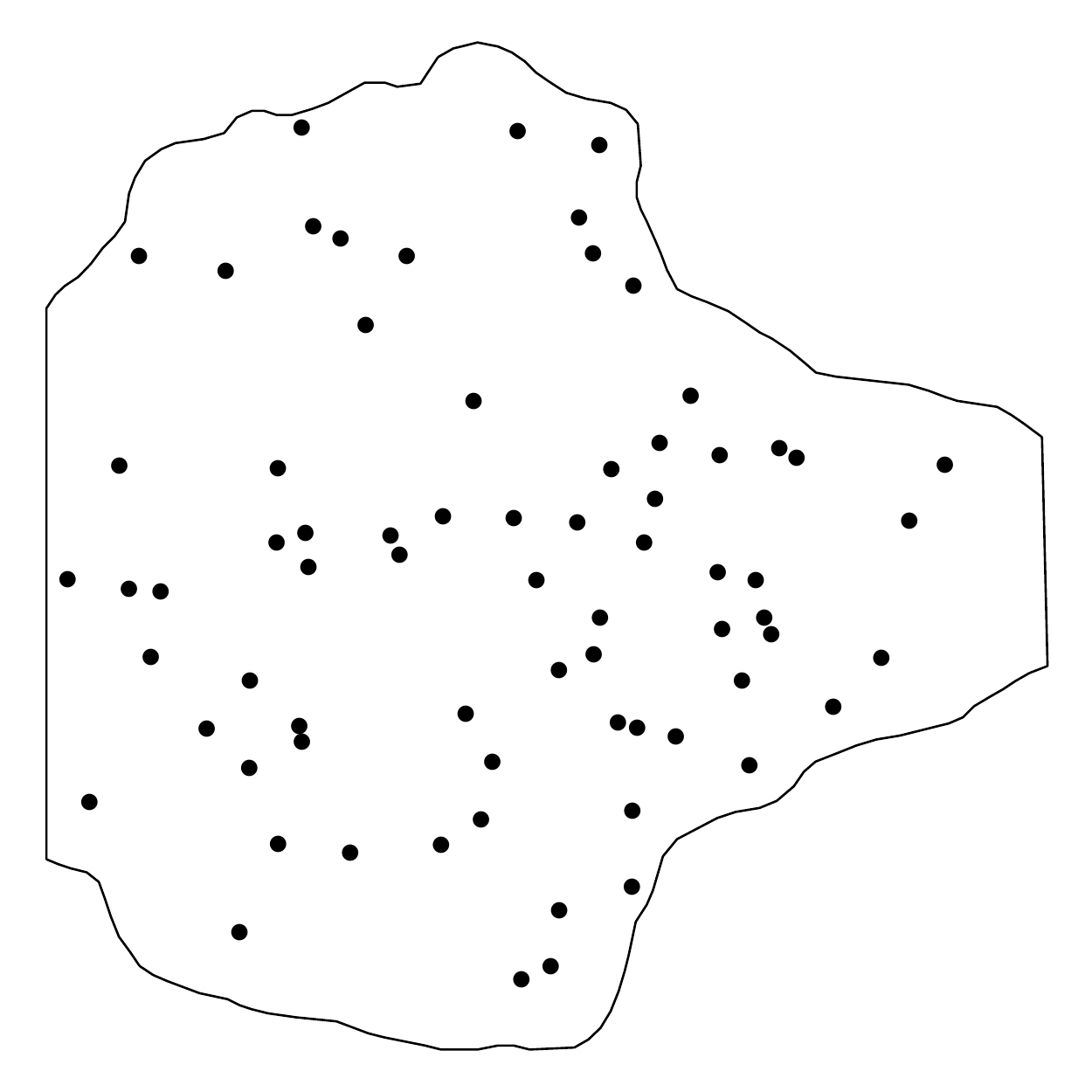} &  \includegraphics[scale=.3]{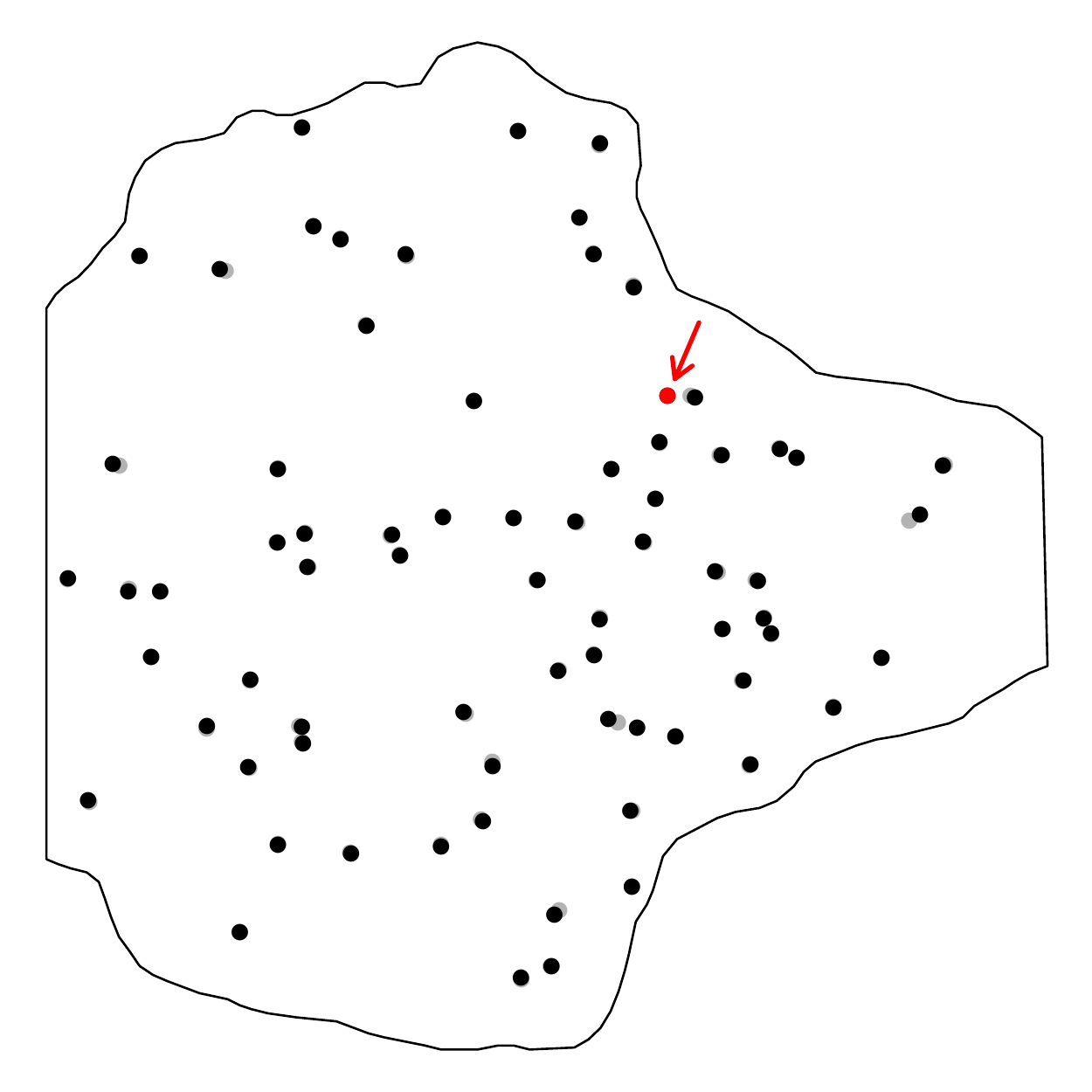}  & \includegraphics[scale=.3]{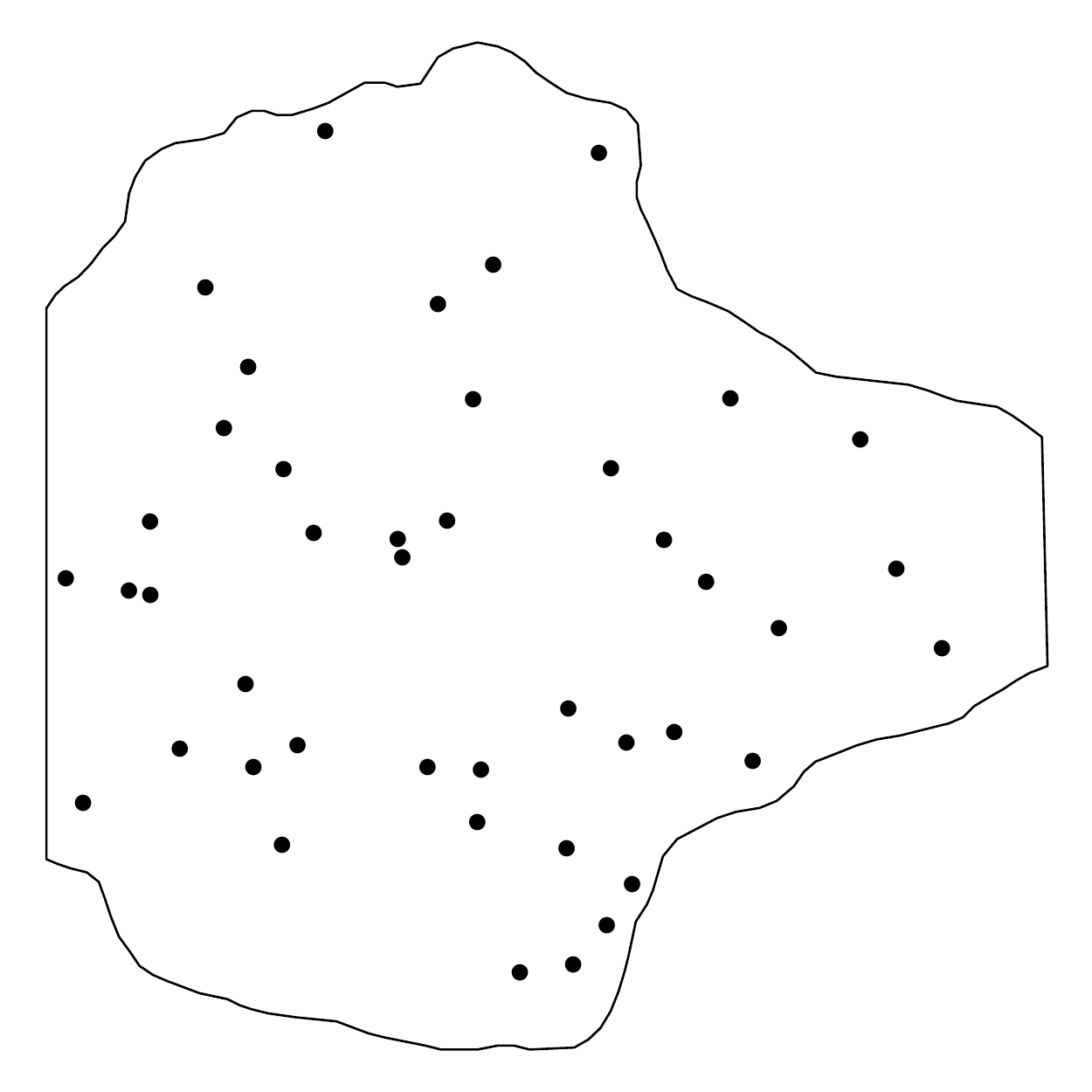} &\includegraphics[scale=.3]{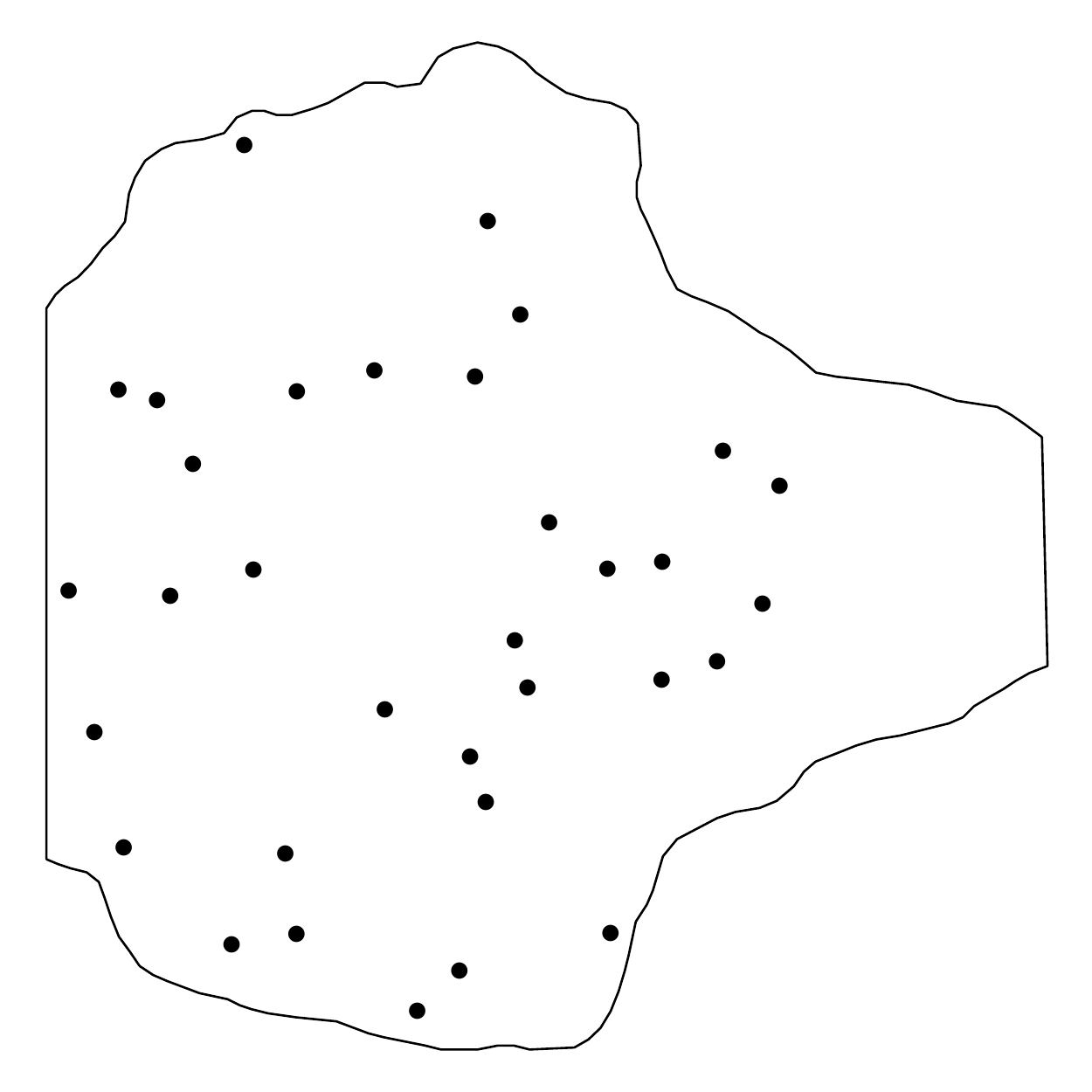}\\
 (a) $t_0=0$ & (b) $t_1=0.14s$ & (c) $t_{100}=14s$ & (d) $t_{1000}=140s$ \\
 \includegraphics[scale=.3]{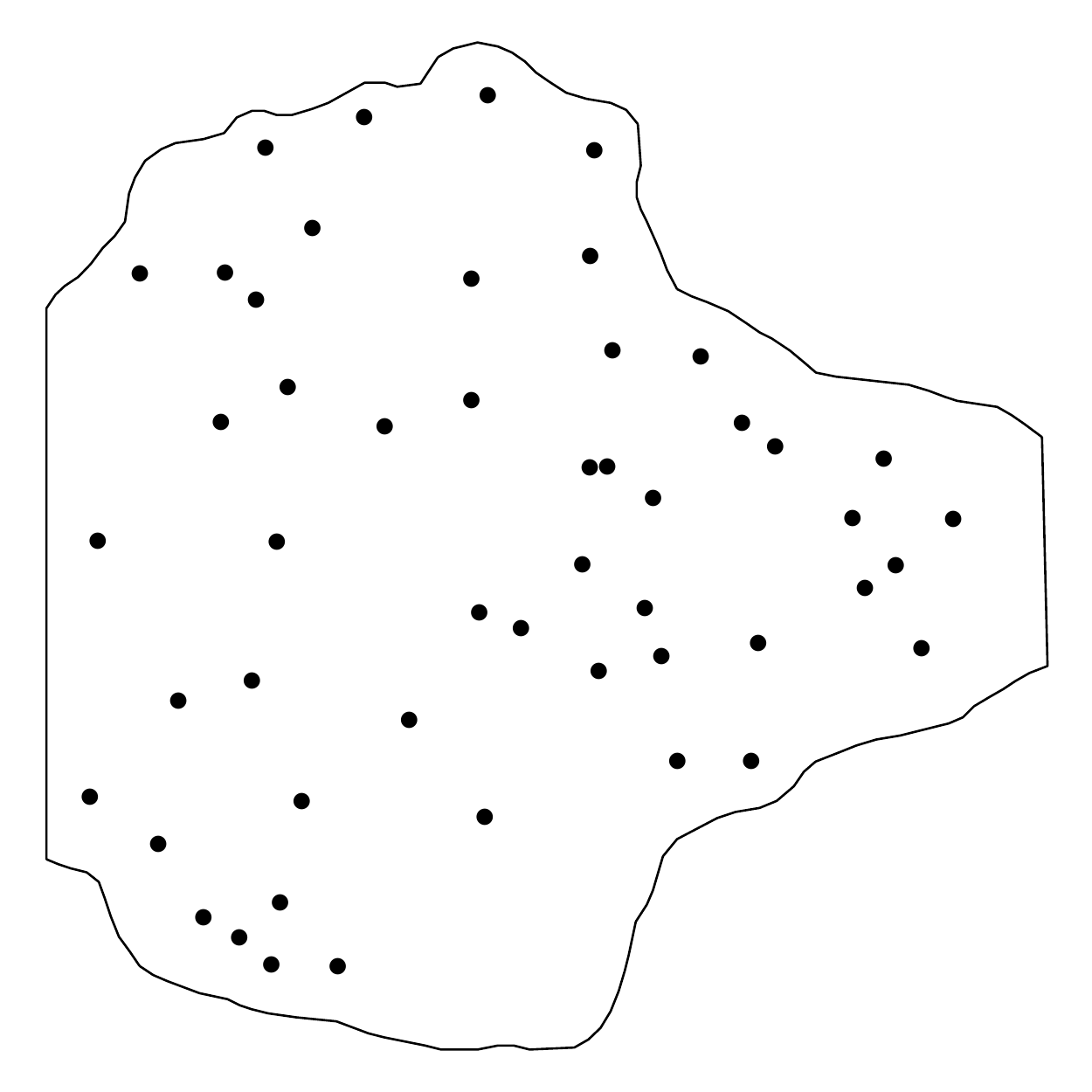} &  \includegraphics[scale=.3]{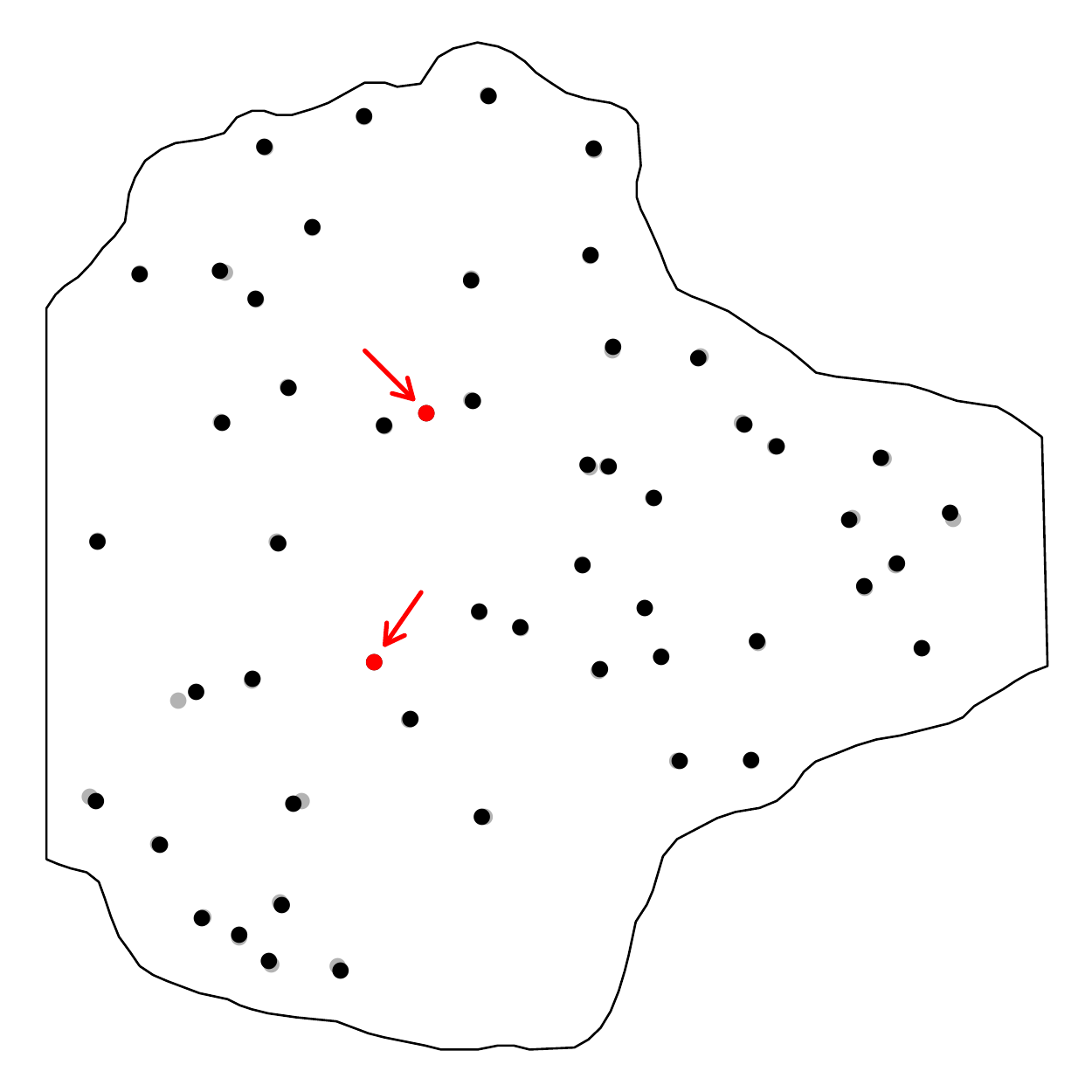} &  \includegraphics[scale=.3]{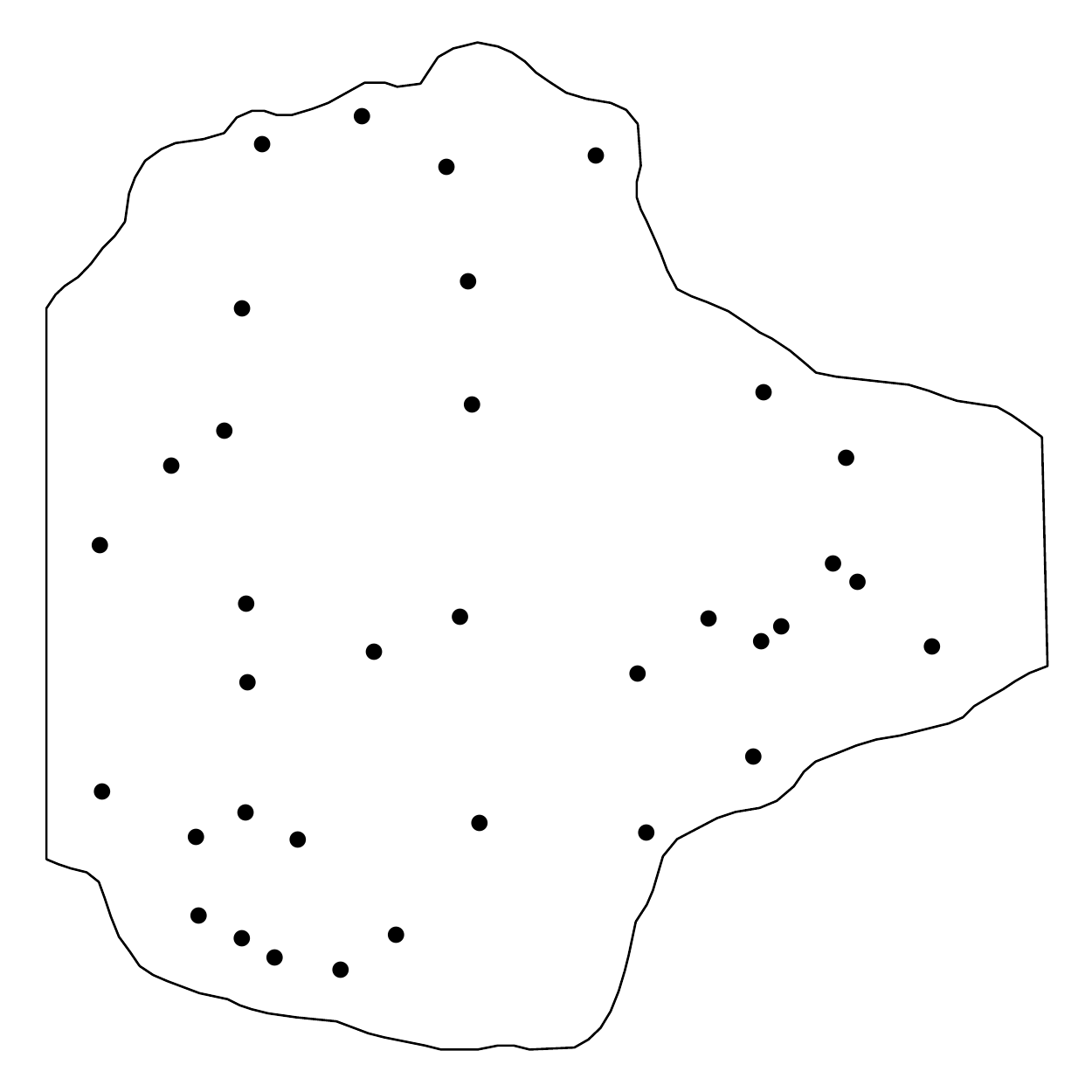} &\includegraphics[scale=.3]{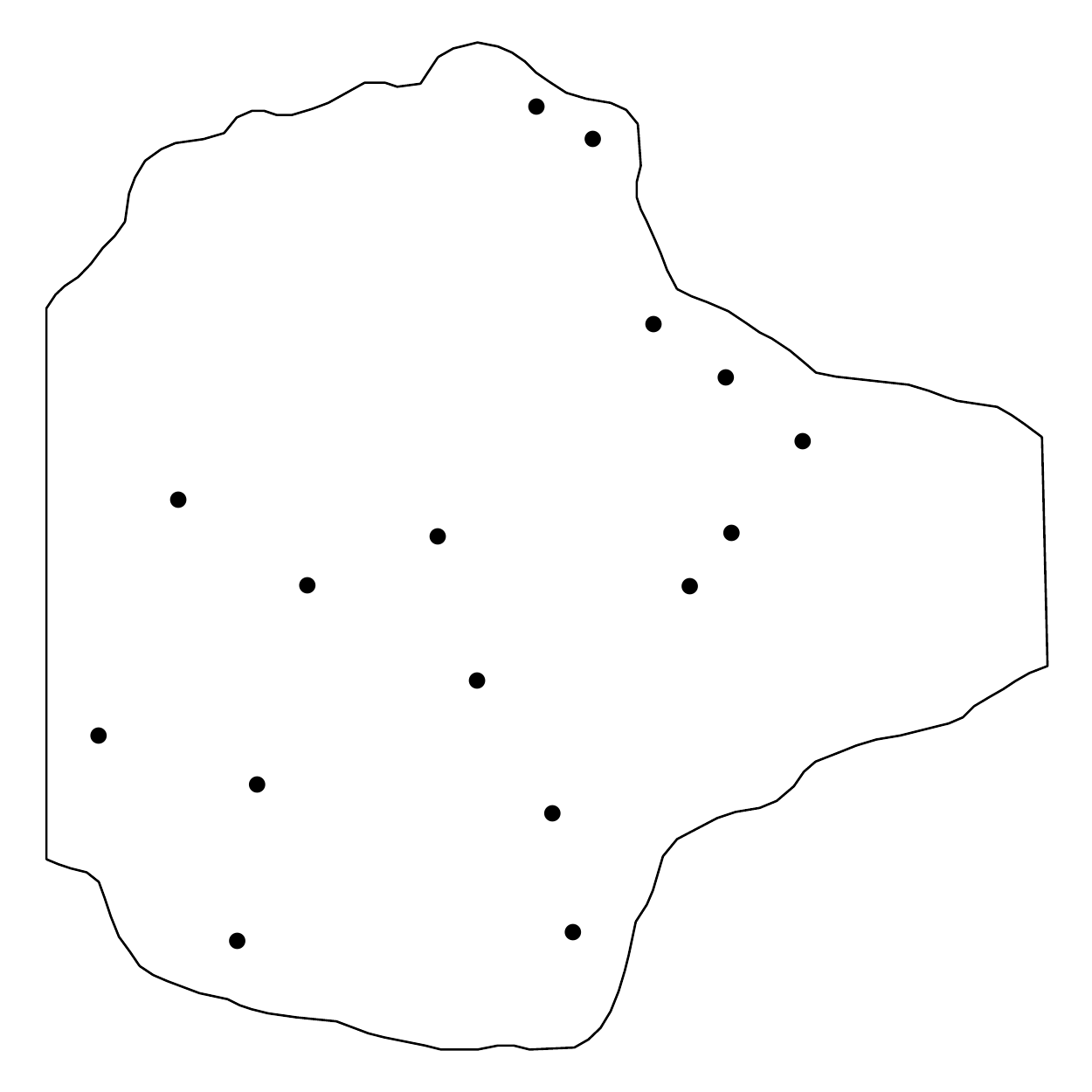}
 \end{tabular}}
\caption{{\small Locations of Langerin proteins (first row) and Rab11 proteins (second row) at different time points $t_j$, corresponding to the $j$-th frame extracted from two sequences of $1199$ images, observed simultaneously in the same living cell in TIRF microscopy. In the second plot (b) corresponding to time $t_1$, the new proteins that appeared during the first time interval are represented as red dots (pointed out by arrows), while the initial positions of the other proteins are recalled in gray.
}}\label{fig:data} 
\end{center}
\end{figure}

We estimate separately the birth intensity function and the death intensity function of both sequences thanks to our estimator \eqref{est discrete}, where $D_j$ is replaced by the observed number of new proteins having appeared (resp. the number of proteins having disappeared) between frames $j-1$ and $j.$
Motivated by the simulation study conducted in the previous section, we consider the non-parametric estimator based on the distance $d_\kappa$, where $\kappa$ is the diameter of the observed cell, and the estimator from Example~\ref{ex simple}-(ii) that assumes the intensities only depend on the cardinality.  

For the birth intensities, both estimators agree on a constant value of $4.45$ births per second for the Langerin sequence and $2.98$ births per second for the Rab11 sequence. A graphical representation of the estimated birth intensities is available in the online supplementary material. These constant values result from the choice of a large value of the bandwidths by cross-validation, and they are in agreement with the first biological hypothesis.

For the death intensities, we represent their estimations in Figure~\ref{fig:death}-(a) and Figure~\ref{fig:death}-(c)  for the Langerin and Rab11 sequences, respectively. For each sequence, both considered estimators provide similar results. Figure~\ref{fig:death}-(b) and Figure~\ref{fig:death}-(d) show the evolution of the estimated death intensities with respect to the number of observed proteins. From these plots we deduce first, that the death intensities seem to depend on the number of proteins and second, that this dependence seems to be linear, up to some value where the death intensities decrease.
To make sure that this decrease  is not due to an estimation artefact, we conducted a short simulation study, reported in the supplementary material, that showed that in presence of a true linear dependence, this kind of slump is unlikely to be observed in the estimation. The observed decrease thus seems to be significant and it means  that when many proteins are active, they tend to spend more time than usual in the exocytosis process. This may be explained by some saturation phenomena, inducing some interactions between the proteins. However, further experiments on new cells must be made to investigate this new hypothesis and be sure that this is not due to experimental conditions.

\begin{figure}[ht]
\begin{center}
{\small
\begin{tabular}{cccc}
 \includegraphics[scale=.3]{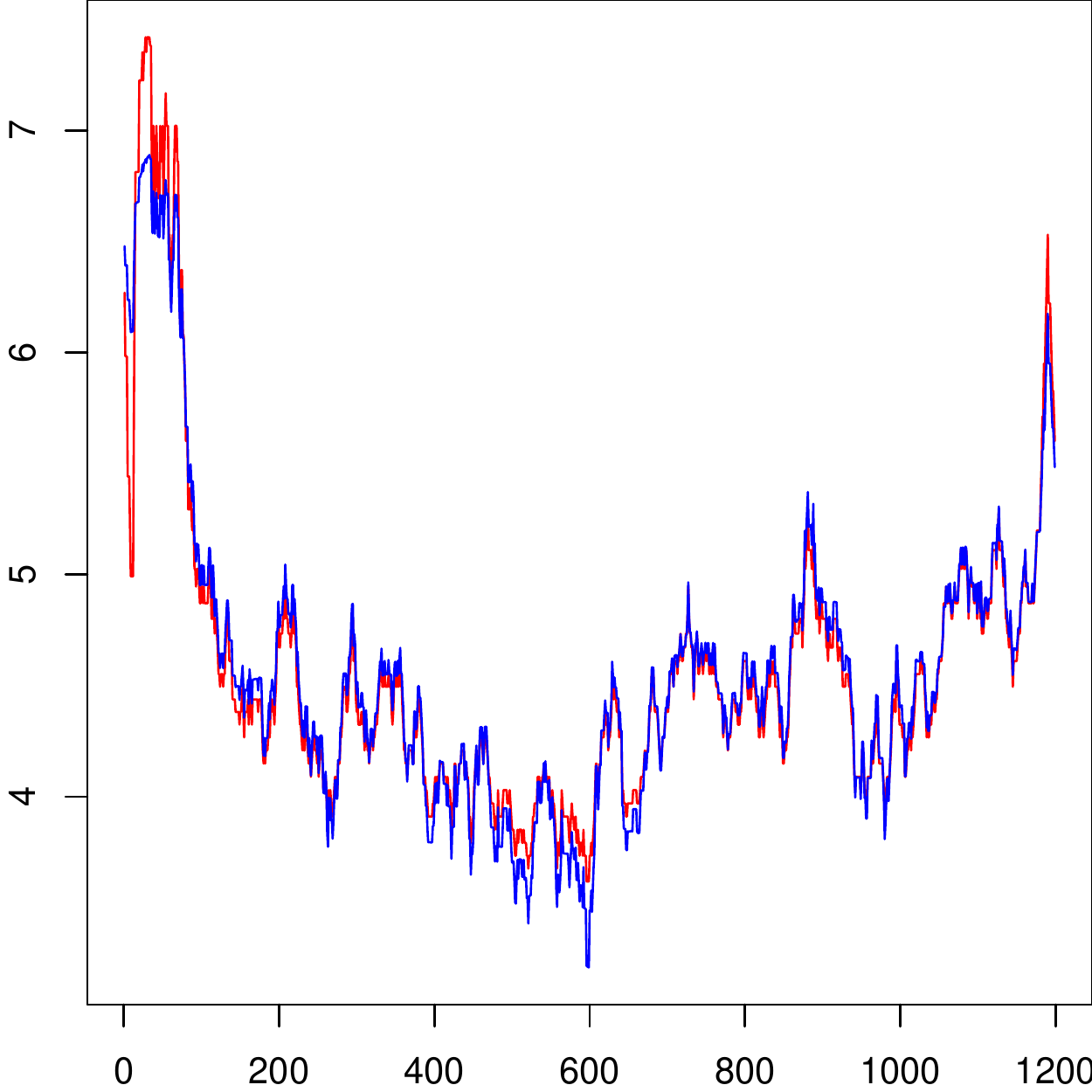} &  \includegraphics[scale=.3]{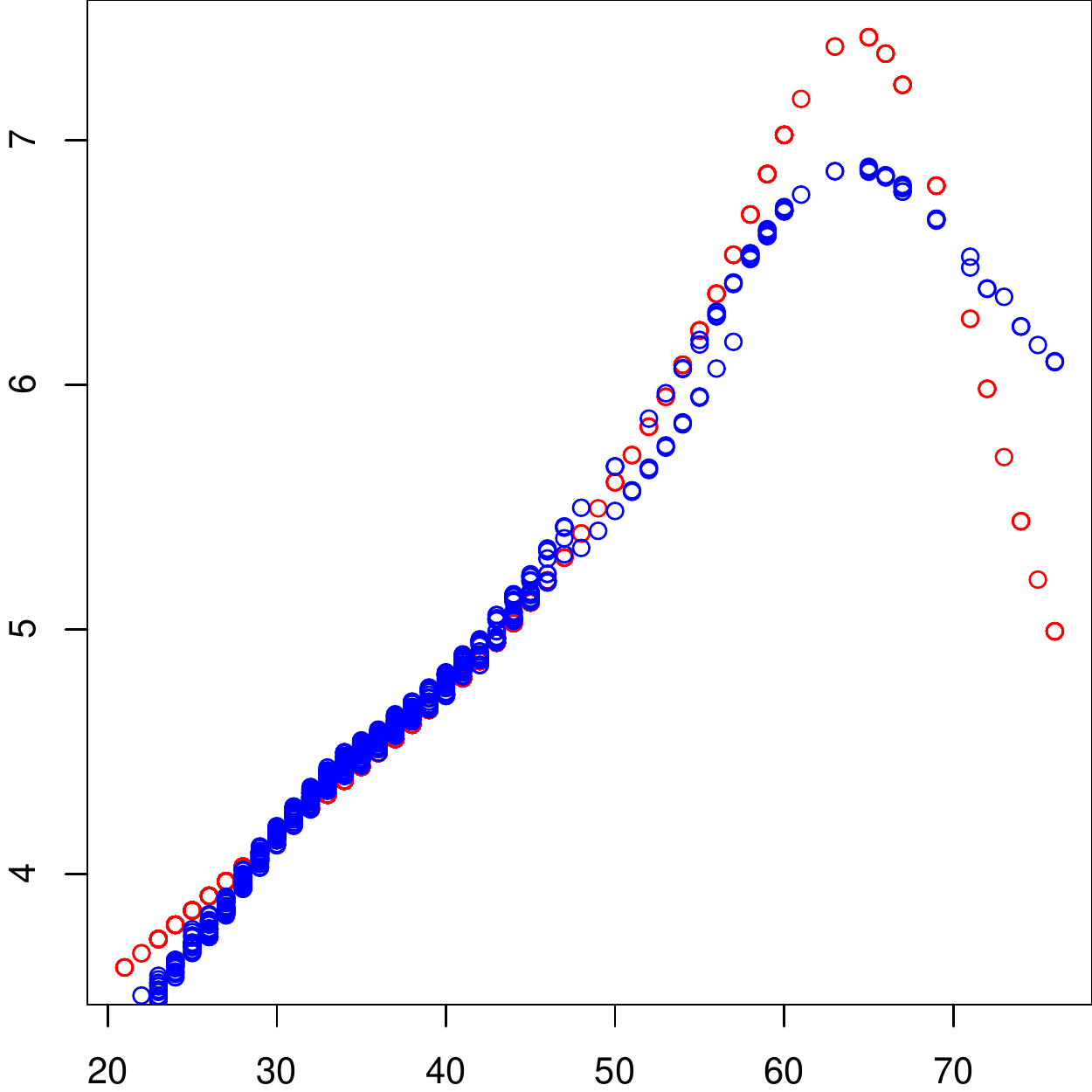} & \includegraphics[scale=.3]{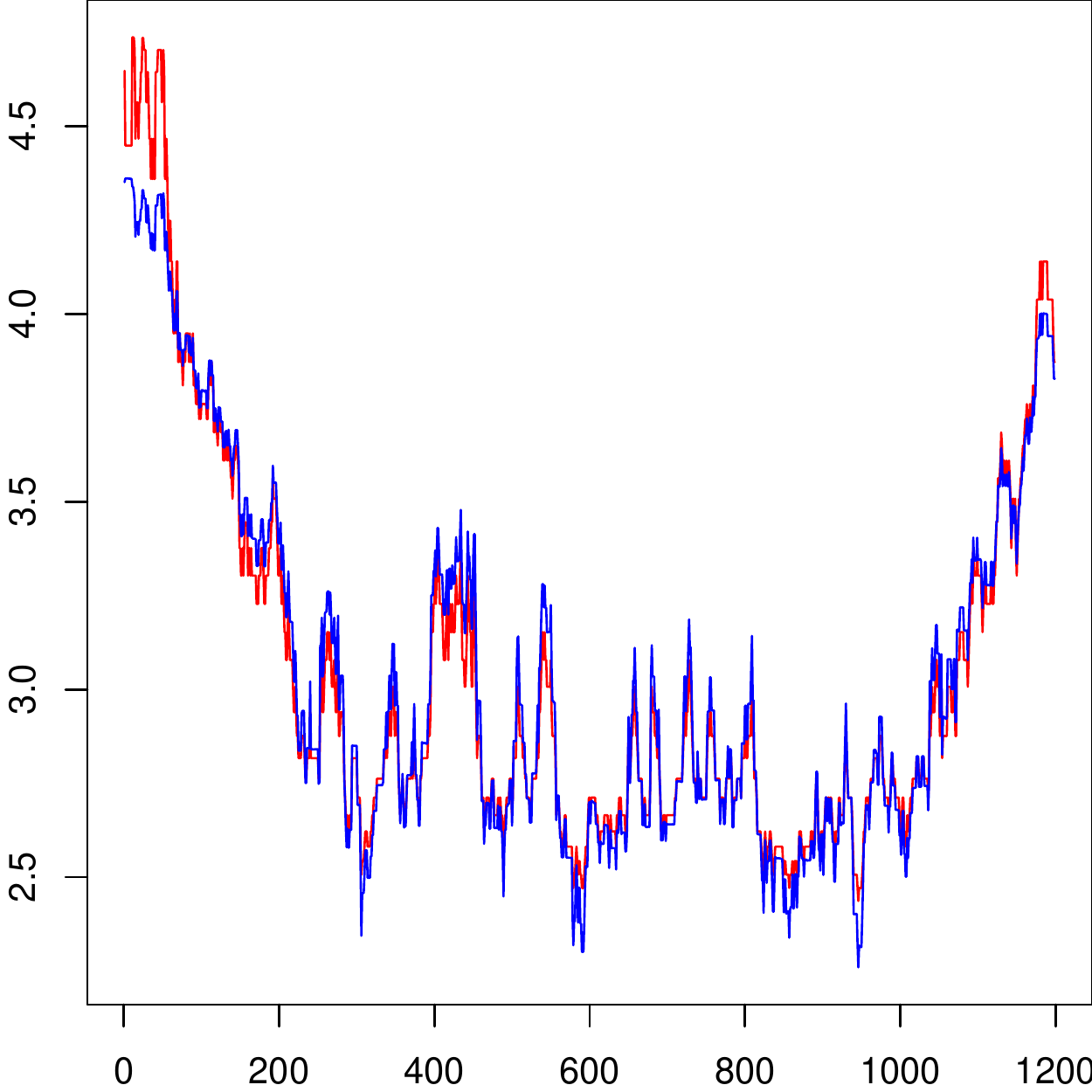} &\includegraphics[scale=.3]{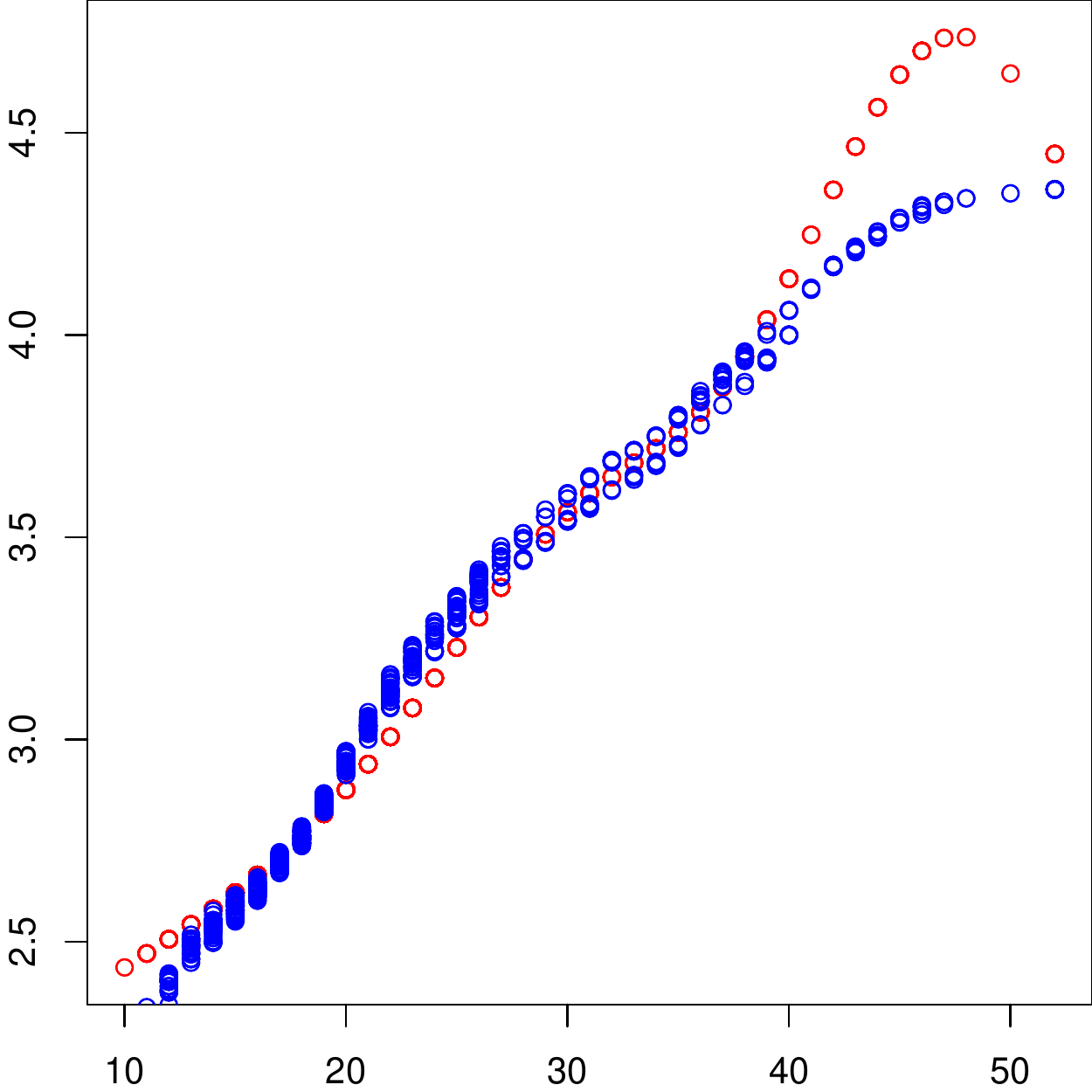}\\
 (a) & (b) & (c) & (d)
 \end{tabular}}
\caption{{\small (a) Estimation of the death intensity $\delta_L(x_j)$  at each configuration $x_j$ of the Langerin sequence, for all frames $j$ from 1 to 1199, by  the discrete time estimator given by  \eqref{est discrete} based on $d_\kappa$ (in blue) and the discrete time estimator given by  \eqref{est discrete}  with the choice of $k_T$ as in Example~\ref{ex simple}-(ii) (in red). (b) Scatterplot of  $(n(x_j),\hat\delta_L(x_{j}))$ for the same two estimators and $j=1,\dots,1199$. (c)-(d) Same plots as (a)-(b) but for the Rab11 sequence.}}\label{fig:death} 
\end{center}
\end{figure}

Finally, we reproduce  in Figure~\ref{fig:correlation} the estimated death intensities for both types of proteins, based on the estimators from Example~\ref{ex simple}-(ii) (that are the red curves in Figure~\ref{fig:death}), along with the estimated cross-correlation function (ccf) between these two estimated death intensities. This last plot represents for each lag $h=-20,\dots,20$ the empirical correlation between the death intensity of the Rab11 sequence at frame $j$ and the death intensity of the Langerin sequence at frame $j+h$, for $j=1,\dots,m-h$. The leftmost plot of  Figure~\ref{fig:correlation} provides evidence that the death intensities of the two types of proteins follow the same trend, which is confirmed by the global high values of the empirical ccf. This observation is consistent with the biological hypothesis that both proteins interact. Interestingly, the ccf is asymmetric, showing higher values for positive lags than for negative lags. This behaviour is confirmed on the ccf obtained from the detrended curves, as illustrated in the online supplementary material. 
 This asymmetry tends to corroborate previous studies  \citep{Gidon, Boulanger2014,Lavancier18}, where it has been concluded that Rab11 seems to be active before Langerin.

Further discussion about this data analysis is provided in the supplementary material.

\begin{figure}[ht]
\begin{center}
{\small
\begin{tabular}{cc}
 \includegraphics[scale=.42]{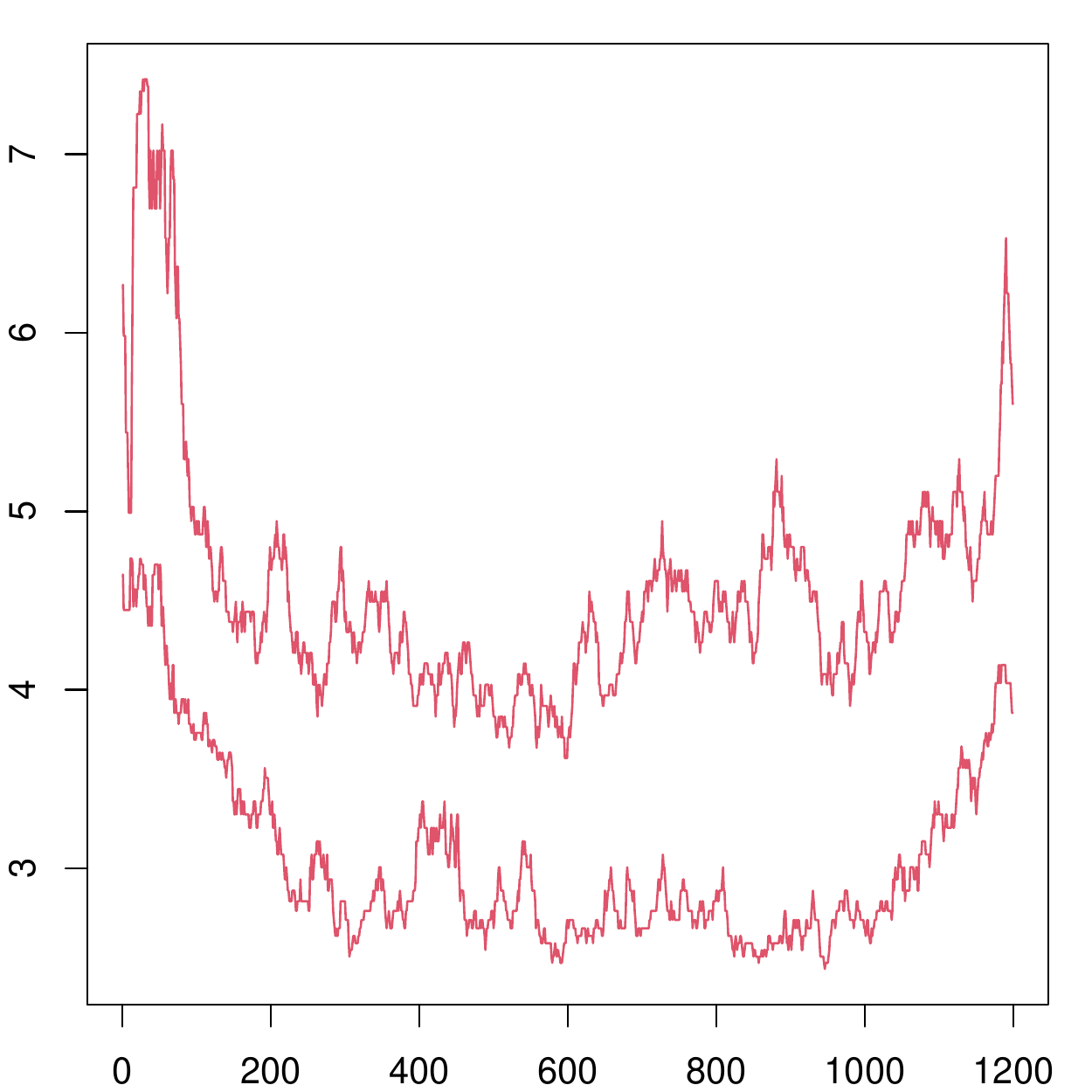} & \includegraphics[scale=.42]{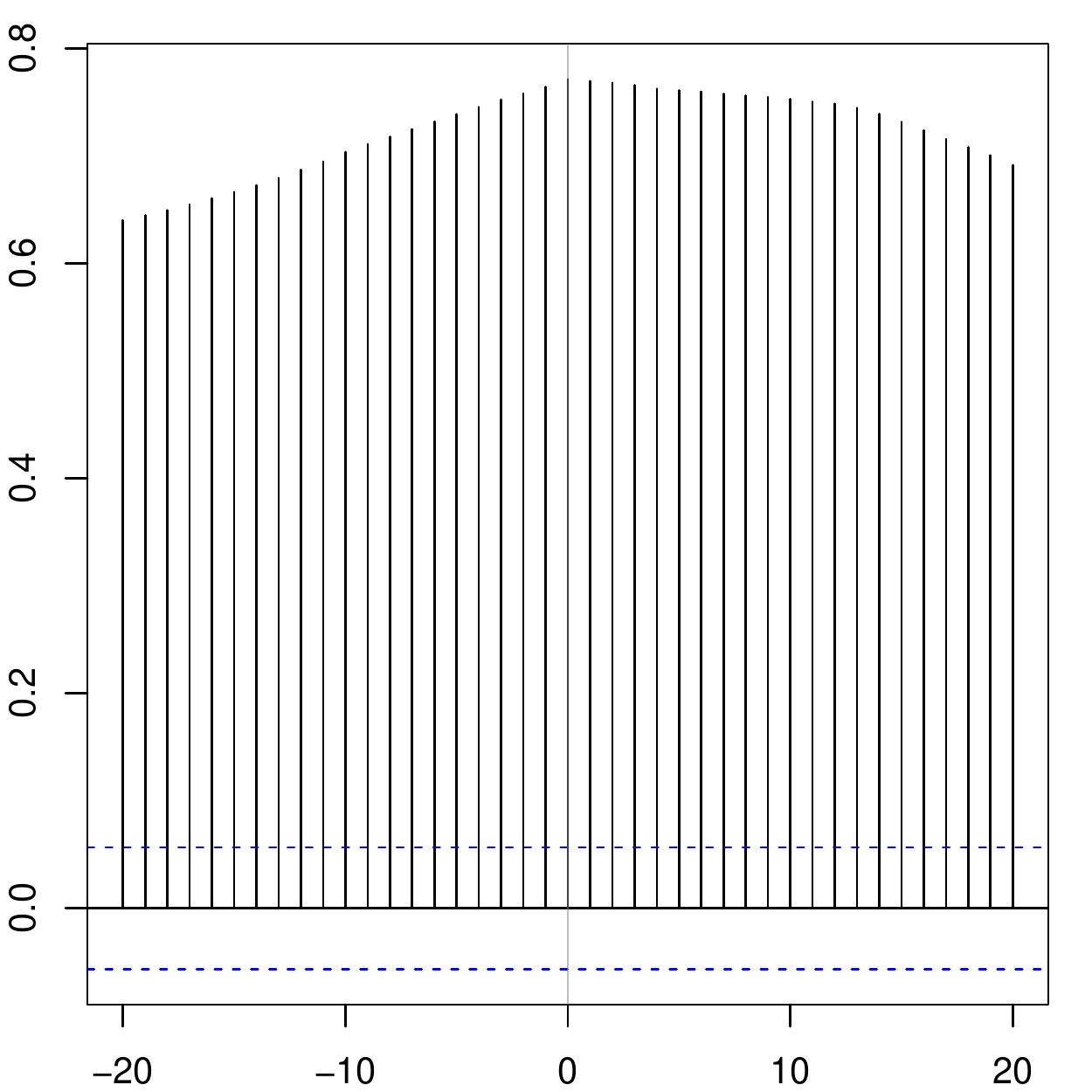}\\ (a) & (b) \end{tabular}}
\caption{{\small (a) Estimation of the death intensity for the Langerin sequence (top curve) and the Rab11 sequence (bottom curve) over the 1199 observed frames. These correspond to the red curves in Figure~\ref{fig:death}; (b) Empirical cross-correlation function between these two estimated death intensities, where the reference is Rab11 and the lag is applied to Langerin.}}\label{fig:correlation} 
\end{center}
\end{figure}

\vspace{-0.5cm}

 \section*{Acknowledgements}
We thank \cite{Boulanger2014} for making the Langerin and Rab11 sequences data analysed in Section~\ref{sec:data} available. We are  grateful to Cesar Augusto Valades-Cruz for assistance in tracking the proteins from these sequences. We also thank Charles Kervrann, Thierry Pécot and  Jean Salamero for helpful discussions and insights on the biological context.
We finally thank the editor and the reviewers for valuable comments that help improving the first version of this article.

 
 \newpage 
 
 \appendix
 
 \renewcommand{\thefigure}{S\arabic{figure}}
\renewcommand{\thealgorithm}{S\arabic{algorithm}}
\renewcommand{\theequation}{S\arabic{equation}}
\renewcommand{\thetheo}{S\arabic{theo}}
\renewcommand{\Box}{\qedsymbol}

\renewcommand{\thesection}{S-\arabic{section}}

{ \centering
 \section*{Supplementary material}
 }
 
This supplementary material describes in Section~\ref{supp-algo} an alternative simulation algorithm of birth-death-move processes, which is more effective that the one presented in the main manuscript. In Section~\ref{supp-estimator}, we study some moment properties of the main estimator and  we discuss the assumptions made in the discrete observations setting. Section~\ref{supp-data}  contains some supplementary information concerning the data analysis carried out in Section~6 of the main manuscript. Section~\ref{appendix:basics} includes  the formal definition of birth-move-processes and some of their theoretical properties. 
The other sections contain the proofs of the theoretical results and two appendices gather some technical lemmas.  All numbering and references in this supplementary material begin with the letter S, the other references referring to the main article.

 \section{Fast algorithm of simulation}\label{supp-algo}

In the main manuscript, Algorithm~\ref{algo_simu} shows how we can simulate a birth-death-move process on the time interval $[0,T]$. After each jump (whether a birth or a death) occurring at time $T_j$, this algorithm requires the simulation of the process $(Y_t)$ on $[0,T-T_j]$ to generate the next jump time $T_{j+1}$.  But as mentioned in the main manuscript, it is very likely that the inter-jump times $\tau_j=T_{j+1}-T_j$ are much smaller than $T-T_j$, making the simulation of $(Y_t)$ on $[0,T-T_j]$ inessential. We may exploit this remark by choosing $\tau_{\max}>0$, a value for which $P(\tau_j>\tau_{\max})$ is thought to be small for any $j$. For instance, if we know the lower bound $\alpha_*$ of $\alpha(x)$, a possible choice is $\tau_{\max}=\alpha_*\log (1/\epsilon)$ which implies $P(\tau_j>\tau_{\max})<\epsilon$ in view of  \eqref{wait}. However the exact value of $P(\tau_j>\tau_{\max})$ has not to be known  and the following Algorithm~\ref{algo_simu_fast} is exact whatever the choice of $\tau_{\max}$ is. 

The justification of Algorithm~\ref{algo_simu_fast} is the following. After each jump at $t=T_j$,  we let $\tau_{j,\max}=\min(\tau_{\max},T-T_j)$ and $p_1$ corresponds to $p_1=P(\tau_{j} >  \tau_{j,\max})$. With high probability, specifically with probability $1-p_1$, $U_1>p_1$ and the simulation of $\tau_j$ given that $\tau_{j} \leq  \tau_{j,\max}$ is the last part of  Algorithm~\ref{algo_simu_fast} which requires the simulation of $Y_t$ on $[0, \tau_{j,\max}]$ only. The first part of the algorithm, i.e. when $U_1\leq p_1$, corresponds to the rare cases when $\tau_{j} >  \tau_{j,\max}$. Conditionally on this event, we then simulate whether $\tau_j > T-T_j$ or not. This requires the simulation of $Y_t$ on $[0, T-T_j]$. Since  $\tau_{j,\max}\leq T-T_j$, 
$$p_2=P(\tau_j > T-T_j | \tau_{j} >  \tau_{j,\max}) = P(\tau_j > T-T_j)/ P(\tau_{j} >  \tau_{j,\max})=\exp\left( - \int_0^{T-T_j} \alpha( Y_u)du \right)/p_1.$$ 
Then, either $U_2\leq p_2$, which means that $\tau_j > T-T_j$, and there is no more jumps until $t=T$, or $U_2>p_2$ and we need to simulate $\tau_j$ given that $\tau_{j,\max}<\tau_j < T-T_j$. The latter distribution is 
\begin{align*}
P\left(\tau_{j}< s | \tau_{j,\max}<\tau_j < T-T_j\right)&=P\left( \tau_{j,\max}<\tau_{j}< s\right)/P(\tau_{j,\max}<\tau_j < T-T_j)\\
&=\frac{\exp\left( - \int_0^{\tau_{j,\max}} \alpha( Y_u)du \right) -\exp\left( - \int_0^{s} \alpha( Y_u)du \right)}{ P(\tau_j < T-T_j|\tau_{j} >  \tau_{j,\max})P(\tau_{j} >  \tau_{j,\max})}\\
&=\frac 1 {(1-p_2)p_1}\left(\exp\left( - \int_0^{\tau_{j,\max}} \alpha( Y_u)du \right) -\exp\left( - \int_0^{s} \alpha( Y_u)du \right)\right).
\end{align*}
Once $\tau_j$ has been simulated, we have $T_{j+1}=T_j+\tau_j$ and it remains to generate the post-jump location $X_{T_{j+1}}$ as detailed in Algorithm~\ref{algo_birth-death}.

\begin{algorithm}[H]
\caption{Fast simulation of a birth-death-move process on the finite interval $[0,T]$}
\label{algo_simu_fast}
\begin{algorithmic}
 \STATE {\bf set}  $t=0$ and $j=0$.
 \WHILE{$t<T$}
 \STATE {\bf set} $\tau_{j,\max}=\min(\tau_{\max},T-T_j)$
 \STATE {\bf generate}  $Y_s$ for $s\in [0, \tau_{j,\max}]$ conditionally on $Y_0=X_{T_j}$
  \STATE {\bf set}  $p_1=\exp\left( - \int_0^{\tau_{j,\max}} \alpha( Y_u)du \right)$ 
    \STATE {\bf generate} $U_1\sim U([0,1])$
  \IF{$U_1\leq p_1$} 
   \STATE {\bf generate}  $Y_s$ for $s\in [\tau_{j,\max},T-T_j]$
   \STATE {\bf set}  $p_2=\exp\left( - \int_0^{T-T_j} \alpha( Y_u)du \right)/p_1$
    \STATE {\bf generate} $U_2\sim U([0,1])$
    \IF{$U_2\leq p_2$} 
     \STATE {\bf set} $X_s=Y_{s-T_j}$ for $s\in [T_j,T]$ and $t\leftarrow T$
     \ELSE
      \STATE {\bf generate} the waiting time $\tau_j$ before the next jump according to the distribution
       \begin{multline*}\forall s\in [\tau_{j,\max}, T-T_j],\quad P\left(\tau_{j}< s | \tau_{j,\max}<\tau_j < T-T_j\right)\\
       =\frac 1 {p_1(1-p_2)} \left(\exp\left( - \int_0^{\tau_{j,\max}} \alpha( Y_u)du \right) -\exp\left( - \int_0^{s} \alpha( Y_u)du \right)\right)\end{multline*}
    \STATE {\bf set} $T_{j+1}=T_j+\tau_j$ and $X_s=Y_{s-T_j}$ for $s\in [T_j,T_{j+1})$
    \STATE {\bf generate} $X_{T_{j+1}}$ by Algorithm~\ref{algo_birth-death})
     \STATE  $t\leftarrow t+\tau_j$ and $j\leftarrow j+1$
\ENDIF

  \ELSE  
   \STATE {\bf generate} the waiting time $\tau_j$ before the next jump according to the distribution
   $$\forall s\in [0, \tau_{j,\max}],\quad P(\tau_{j}< s| \tau_j\leq \tau_{j,\max})=\frac 1 {1-p_1} \left(1-\exp\left( - \int_0^{s} \alpha( Y_u)du \right)\right)$$
   \STATE {\bf set} $T_{j+1}=T_j+\tau_j$ and $X_s=Y_{s-T_j}$ for $s\in [T_j,T_{j+1})$
    \STATE {\bf generate} $X_{T_{j+1}}$ by Algorithm~\ref{algo_birth-death})
    \STATE  $t\leftarrow t+\tau_j$ and $j\leftarrow j+1$
\ENDIF
 \ENDWHILE
 \end{algorithmic}
\end{algorithm}

\begin{algorithm}[h]
\caption{Simulation of the post-jump location at $t=T_{j+1}$ (where $\tau_j=T_{j+1}-T_j$)}
\label{algo_birth-death}
\begin{algorithmic}
     \STATE {\bf generate} $U\sim U([0,1])$
     \IF{$U\leq \beta(Y_{\tau_j})/\alpha(Y_{\tau_j})$}  
     	\STATE {\bf generate}  $X_{T_{j+1}}$ according to the transition kernel $K_\beta(Y_{\tau_j},.)$
     	\ELSE 
	\STATE {\bf generate} $X_{T_{j+1}}$ according to the transition kernel $K_\delta(Y_{\tau_j},.)$
    \ENDIF
  \end{algorithmic}
\end{algorithm}

\section{Supplementary information concerning the estimators}\label{supp-estimator}

\subsection{Moments of the estimator}

As explained in the main manuscript, we may interpret $w_T(x)$ in Proposition~\ref{Maintheo} as a bias term and $1/(Tv_T(x))$ as a variance term of $\hat\gamma(x)$.  Strictly speaking, this interpretation is wrong because the estimator $\hat\gamma(x)$ is not integrable in general. 
This is due to the presence of $\hat T(x)$ in its definition. 
Even if  we know  from Corollary~\ref{moments T} that $\E(\hat T(x))\to\infty$ as $T\to\infty$, $\hat T(x)$ may take infinitely small values for some $x$, in which case $\hat\gamma(x)$ may be arbitrarily large. The following lemma provides a clear example.

\begin{lem}\label{not L1}
As in Example~\ref{ex simple} (i), let $k_T(x,y) = \1_{n(x)=n(y)}$. If  $\P(X_0\in E_{n(x)})\neq 0$,  then $\E(\hat\alpha(x))=\infty$. 
\end{lem}

\begin{proof}
We have
\begin{align*}
        \E(\hat{\alpha}(x)) &\geq \E(\hat{\alpha}(x)\1_{N_T=1} \1_{X_0 \in E_{n(x)}})\\
        & = \E\left( \frac{k_T(x,X_{T_1^-})}{\int_0^{T_1} k_T(x,Y_s^{(0)})ds + \int_{T_1}^T k_T(x,Y_{s-T_1}^{(1)})ds} \1_{T_1 \leq T}\1_{T_2-T_1 > T - T_1} \1_{X_0 \in E_{n(x)}}\right).
        \end{align*}
 Since by assumption $k_T(x,y) = \1_{n(x)=n(y)}$, we have $k_T(x,X_{T_1^-})\1_{X_0 \in E_{n(x)}}=k_T(x,X_0)\1_{X_0 \in E_{n(x)}}=1$. Similarly for all $s>0$, $k_T(x,Y_s^{(0)})\1_{X_0 \in E_{n(x)}}=1$ and $k_T(x,Y_{s-T_1}^{(1)})\1_{X_0 \in E_{n(x)}}=0$. Therefore 
          \begin{align*}
      \E(\hat{\alpha}(x)) &\geq \E\left(\frac{1}{T_1} \1_{T_1 \leq T}\1_{T_2-T_1 > T - T_1} \1_{X_0 \in \Et_{n(x)}}\right)\\
        &= \E \left(\1_{X_0 \in \Et_{n(x)}} \int_0^T \frac{1}{s} \alpha(Y_s^{(0)}) e^{-\int_0^s \alpha(Y_u^{(0)})du} e^{-\int_0^{T-s}\alpha(Y_{u}^{(1)})du} ds\right)\\
        &\geq \alpha_* e^{-\alpha^*T} \mu_0(E_{n(x)}) \int_0^T \frac{1}{s}  e^{-\alpha^*s}ds= \infty.
     \end{align*}
     \end{proof}

Arguably, a statistician would not trust any estimation of $\gamma(x)$ if $\hat T(x)$ is very small and in our opinion the  result of this lemma does not rule out using $\hat\gamma(x)$ for reasonable  configurations $x$, that are configurations for which a minimum time has been spent by the process in configurations similar to $x$. 
To reflect this idea, let us consider the following modified estimator, for a given small $\xi>0$,  
\[\hat{\gamma}_\xi (x) =\hat{\gamma}(x)\1_{\hat T(x)>\xi}.\]
By Corollary~\ref{moments T}, $\1_{\hat T(x)>\xi}$ converges in probability to $1$, while $\hat{\gamma}(x)$ is consistent under the assumptions of  Proposition~\ref{Maintheo}. Therefore the alternative estimator $\hat{\gamma}_\xi (x)$ is also consistent under the same assumptions,  
but it has the pleasant additional property to be mean-square consistent, as stated in the following proposition whose proof is given in Section~\ref{sec:proofs}.
\begin{prop}\label{theo:MSE}
 Let  $\gamma$ be either $\gamma=\beta$ or $\gamma=\delta$ or $\gamma=\alpha$. Assume \eqref{existence}, \eqref{H3} and \eqref{H4}, then for all $\xi>0$ and all $0<\eta<1/3$,
\[\E\left[ ( \hat{\gamma}_\xi(x)- \gamma(x))^2\right] = O\left(  \frac{1}{(Tv_T(x))^{1/3 -\eta}} + w_T^2(x)\right) \]
as $T\to\infty$, whereby $\hat{\gamma}_\xi(x)$ is mean-square consistent. 
\end{prop}

\subsection{Discrete time observations}

In the main manuscript we provide the formula for the estimator of the total intensity function $\hat\alpha_{(d)}(x)$, in case of discrete time observations, under the assumption~\ref{HDj}. Regarding the birth intensity function, its estimator  is similarly defined in this case as 
\begin{equation*}
\hat{\beta}_{(d)}(x) = \frac{\sum_{j=0}^{m-1} D_{j+1}^\beta k_T( x, X_{t_j})}{\sum_{j=0}^{m-1} \Delta t_{j+1} k_T( x, X_{t_j})},\end{equation*}
where $D_{j+1}^\beta$ is an approximation of the number of births $\Delta N^\beta_{t_j}$ between $t_j$ and $t_{j+1}$. Similarly as for  $D_j$ in $\hat\alpha_{(d)}(x)$, this approximation must satisfy:
\begin{itemize}[leftmargin=12mm]
\item[$\bf(H4^\beta)$] For all $j\geq 1$, $D_j^\beta= \Delta N^\beta_{t_j}$ if $\Delta N^\beta_{t_j} \leq 1$ and $D^\beta_j \leq \Delta N^\beta_{t_j}$ if $\Delta N^\beta_{t_j} \geq 2$. \end{itemize}
Likewise, the estimator of the death intensity function for discrete time observations is 
\begin{equation*}
\hat{\delta}_{(d)}(x) = \frac{\sum_{j=0}^{m-1} D_{j+1}^\delta k_T( x, X_{t_j})}{\sum_{j=0}^{m-1} \Delta t_{j+1} k_T( x, X_{t_j})},\end{equation*}
where $D_{j+1}^\delta$ is an approximation of the number of deaths $\Delta N^\delta_{t_j}$ between $t_j$ and $t_{j+1}$ and satisfies:
\begin{enumerate}[leftmargin=12mm]
\item[$\bf(H4^\delta)$] For all $j\geq 1$, $D_j^\delta= \Delta N^\delta_{t_j}$ if $\Delta N^\delta_{t_j} \leq 1$ and $D^\delta_j \leq \Delta N^\delta_{t_j}$ if $\Delta N^\delta_{t_j} \geq 2$. \end{enumerate}

We now discuss the assumptions of Proposition~\ref{th discret} concerning the consistency of $\hat\alpha_{(d)}(x)$, $\hat\beta_{(d)}(x)$ and $\hat\delta_{(d)}(x)$, for the same examples as in the continuous time observations setting.

\bigskip

\noindent{\it Example \ref{ex BD} (continued)}: 
If $(X_t)_{t\geq 0}$ is a pure spatial birth-death process, then \eqref{discrete condition} is satisfied with $\ell_T(x)=0$.

\bigskip

\noindent{\it Example \ref{ex kernel} (continued)}: 
Assume  $k_T$ takes the general form \eqref{generalkT} and that $u\mapsto k(u)$ is Lipschitz with constant $c_k$. Then the inequality in  \eqref{discrete condition} holds true whenever $\E(d(Y_s,Y_t)|Y_0=y)\leq c_Y |s-t|^a$, in which case $\ell_T(x)=c_kc_Y/h_T$. If $k$ is moreover assumed to be supported on $[-1,1]$ and to satisfy $k(u)\geq k_* \1_{|u|<c}$ for some $k_*>0$ and $0<c<1$,  the same interpretation of \eqref{H3} and \eqref{H4} as for the choice $k(u)=\1_{|u|<1}$  remains valid for a Lipschitz function $\gamma$. We obtain in this case that \eqref{H3}-\eqref{H4} and \eqref{stepto0}-\eqref{discrete condition} hold true whenever $h_T\to 0$, $T\mu_\infty(B(x,ch_T))\to\infty$, $\Delta_m=o(\mu_\infty^2(B(x,ch_T))$ and $\Delta_m^a=o(h_T\mu_\infty^2(B(x,ch_T))$. In other words $h_T$ must tend to 0 but not too fast, and $\Delta_m$ must tend to zero sufficiently fast, the latter rate depending on the regularity exponent $a$.

\bigskip

\noindent{\it Example \ref{ex Hausdorff} (continued)}: 
If each $E_n$ is the space of point configurations in $\R^d$ with cardinality $n$, $(Y_t)_{t\geq 0}$ is a  multivariate process formed of $n(Y_0)$ components, each of them being a stochastic process on $\R^d$. 
 In this case, for usual pseudo-distances $d$, $d(Y_s,Y_t)$ can be controlled by the Euclidean distance $\|Y_s-Y_t\|$. This is true when $d$ is the Hausdorff distance or the optimal matching distance $d_\kappa$. The condition $\E(d(Y_s,Y_t)|Y_0=y)\leq c_Y |s-t|^a$ mentioned in the previous example when $k_T$ takes the form \eqref{generalkT} is then implied by $\E(\|Y_s-Y_t\||Y_0=y)\leq c_Y |s-t|^a$. The latter condition holds true for most continuous processes, including Brownian motions, fractional Brownian motions, the Ornstein–Uhlenbeck process, and more generally any solution to a stochastic differential equation with Lipschitz coefficients \cite[Chapter 2]{kunita}.

\bigskip

  \noindent{\it Example \ref{ex simple}  (continued)}:  Assume that  $\gamma(x)=\gamma_0(n(x))$, then
\begin{itemize}
\item[(i)]  if we choose $k_T(x,y)=\1_{n(x)=n(y)}$ to recover the standard non-parametric likelihood estimator of \cite{wolff1965} and \cite{reynolds1973}, then we can take $\ell_T(x)=0$ in \eqref{discrete condition}  so that \eqref{H3}-\eqref{H4} and \eqref{stepto0}-\eqref{discrete condition} are satisfied  whenever $\mu_\infty(E_{n(x)})\neq 0$ and $\Delta_m\to 0$.
   
\item[(ii)] if we choose $k_T$ as in \eqref{generalkT} with $d(x,y)=|n(x)-n(y)|$, then $k_T(x,Y_s)-k_T(x,Y_t)=0$ and   \eqref{discrete condition} is also satisfied with $\ell_T(x)=0$.
   \end{itemize}

 \noindent{\it Example \ref{ex geom}  (continued)}:  Assume more generally that  $\gamma(x)=\gamma_0(f(x))$ for some function $f: E\to \R^p$. Following the discussion above for Example \ref{ex kernel}, if we choose $k_T$ as in \eqref{generalkT} with $d(x,y)=\|f(x)-f(y)\|$, the key condition  is $\E(\|f(Y_s)-f(Y_t)\||Y_0=y)\leq c_Y |s-t|^a$. If $E_n$ is the space of point configurations in $\R^d$ with cardinality $n$, this condition holds true if for instance $f$ is Lipschitz and $(Y_t)_{t\geq 0}$ is any of the processes listed in Example~\ref{ex Hausdorff}.

\section{Supplementary information concerning the data analysis}\label{supp-data}

As a complement to the data analysis carried out in Section~\ref{sec:data} of the main manuscript, Figure~\ref{fig:birth} shows the estimation of the birth intensity for the Langerin sequence (on the left) and the Rab11 sequence (on the right). The are two estimations: the first one is the non-parametric method based on the distance $d_\kappa$ (in blue), where $\kappa$ is the diameter of the observed cell, and the second one in red uses the estimator from Example~\ref{ex simple}-(ii) that assumes the intensities only depend on the cardinality.  The scales of the plots have been chosen to be similar to the scales of the death intensities estimations displayed in Figure~\ref{fig:death} in the main manuscript. Based on these results, we conclude that the birth intensity is constant ($4.45$ births per second) for the Langerin sequence. It seems also constant for the Rab11 sequence if we refer to the second estimator ($2.98$ births per second). The first estimator (in blue) for the Rab11 sequence shows more variability, but we decide to consider these fluctuations as noise and to trust the second estimator to conclude to a constant birth intensity. Providing a formal statistical test to decide whether the intensity is constant or not is beyond the scope of this contribution, but it clearly constitutes an interesting perspective for future investigations.  

\begin{figure}[ht]
\begin{center}
{\small
\begin{tabular}{cc}
 \includegraphics[scale=.4]{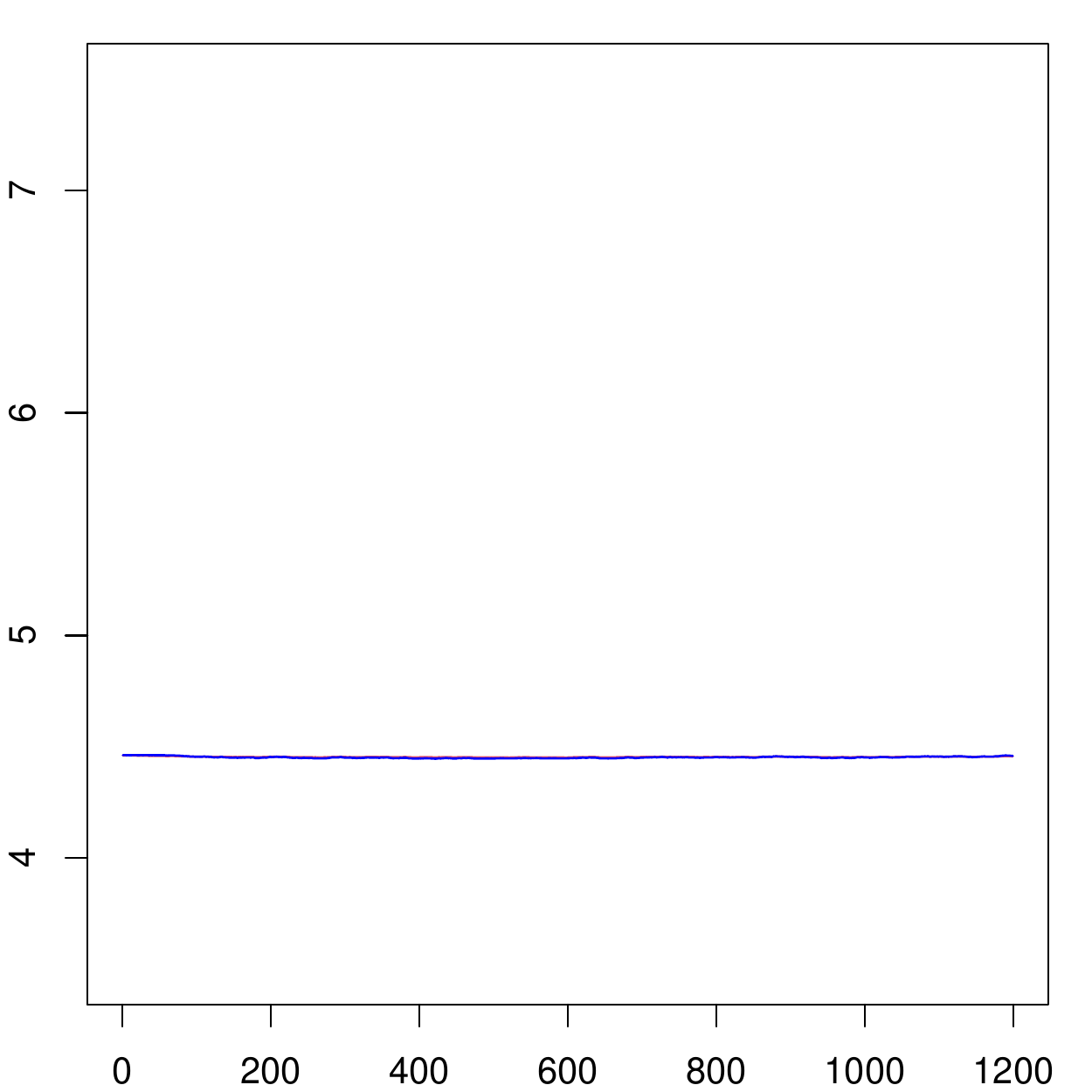} & \includegraphics[scale=.4]{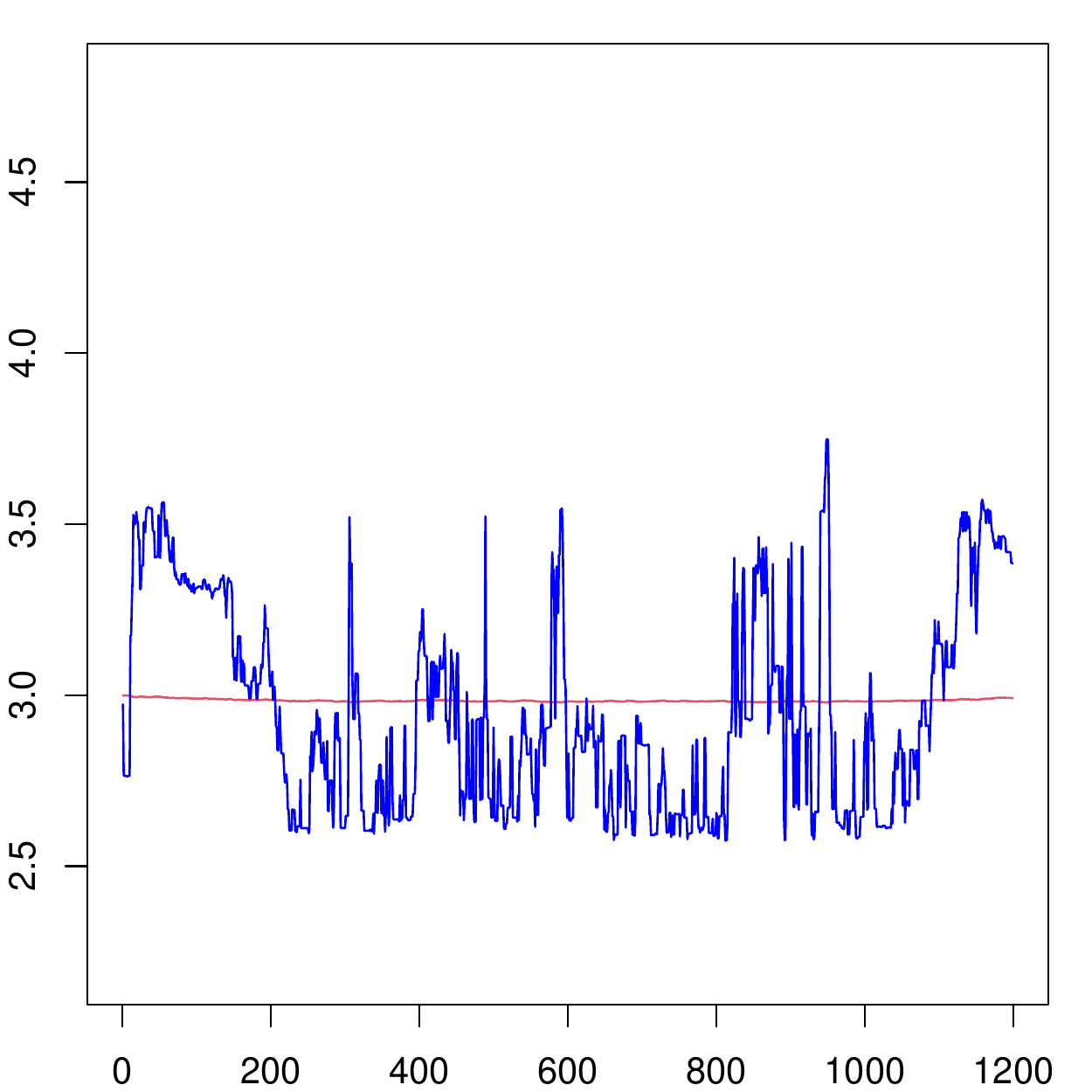} \\
 (a) & (b) 
 \end{tabular}}
\caption{{\small (a) Estimation of the birth intensity at each configuration $x_j$ of the Langerin sequence, for all frames $j$ from 1 to 1199, by  the discrete time estimator given by  \eqref{est discrete} based on $d_\kappa$ (in blue) and the discrete time estimator given by  \eqref{est discrete}  with the choice of $k_T$ as in Example~\ref{ex simple}-(ii) (in red). Both curves are almost constant and overlaid. (b) Same plot as (a) but for the Rab11 sequence. For both plots the $y$-scale is similar as the one used for the corresponding death intensity estimates, see Figure~\ref{fig:death} in the main manuscript.}}\label{fig:birth} 
\end{center}
\end{figure}

Concerning the death intensity functions, we have observed in Figure~\ref{fig:death} of the main manuscript that they depend on the cardinality of the proteins' configurations. Of course this does not rule out a dependence on other characteristics of the configurations. In an attempt to figure out whether geometrical characteristics  might influence the death intensities, we have represented in Figure~\ref{fig_geometric} the estimated death intensities of Langerin (top row) and Rab11 (bottom row) with respect to the maximal nearest neighbour distance of each configuration (left), to the maximal area of the Delaunay cells of each configuration (middle), and as a recall to the number of proteins (right). The considered estimator is the non-parametric one using the distance $d_\kappa$. The two first characteristics are geometric quantities indicating the maximal available empty space in each configuration. The plots in  Figure~\ref{fig_geometric} do not show a clear dependence with these quantities. The slight negative dependence that we might observe in the top middle plot is to our opinion spurious and due to the expected dependence between the maximal Delaunay area and the number of points (the less points, the more empty space).

\begin{figure}[H]
\begin{center}
\begin{tabular}{ccc}
 \includegraphics[scale=.35]{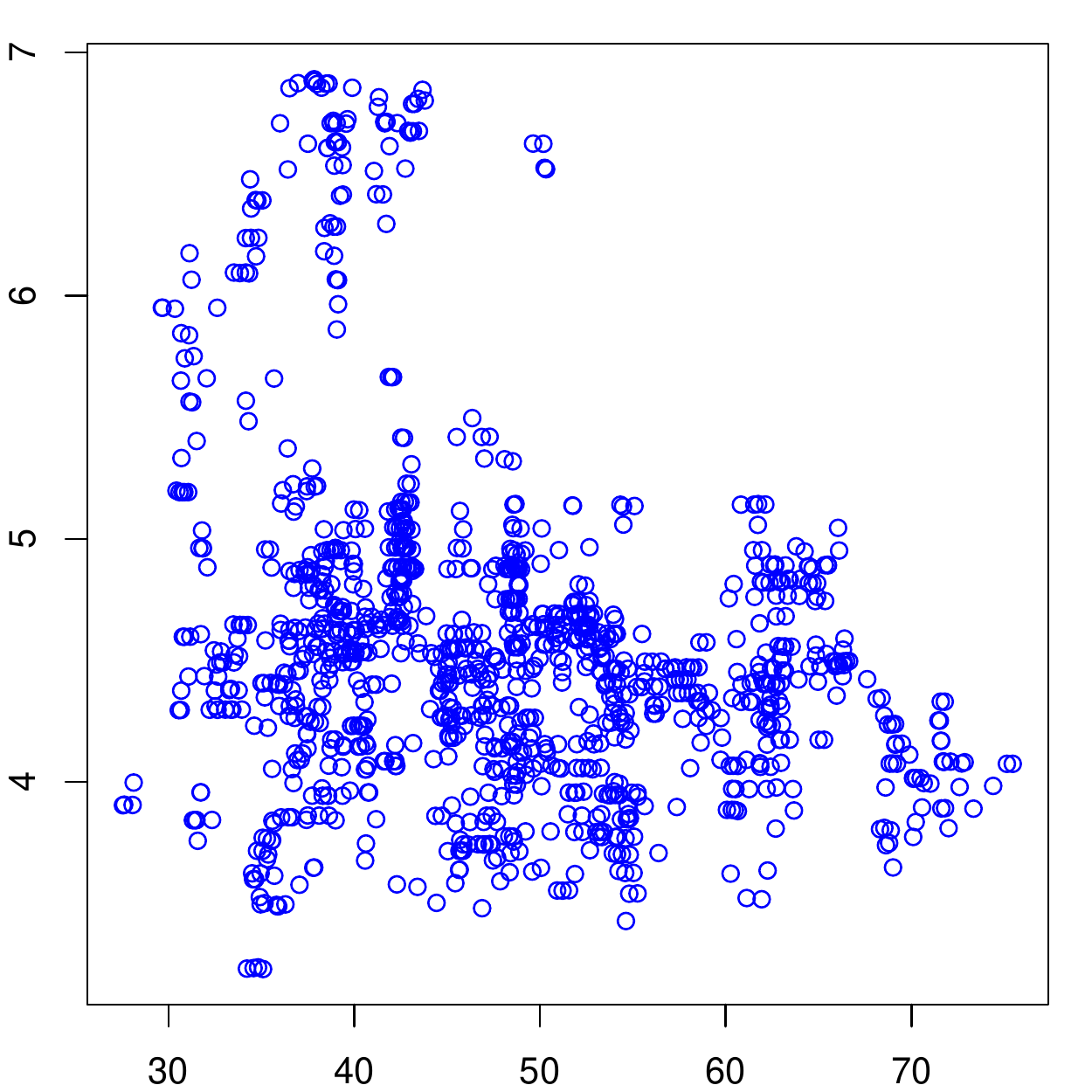} &
 \includegraphics[scale=.35]{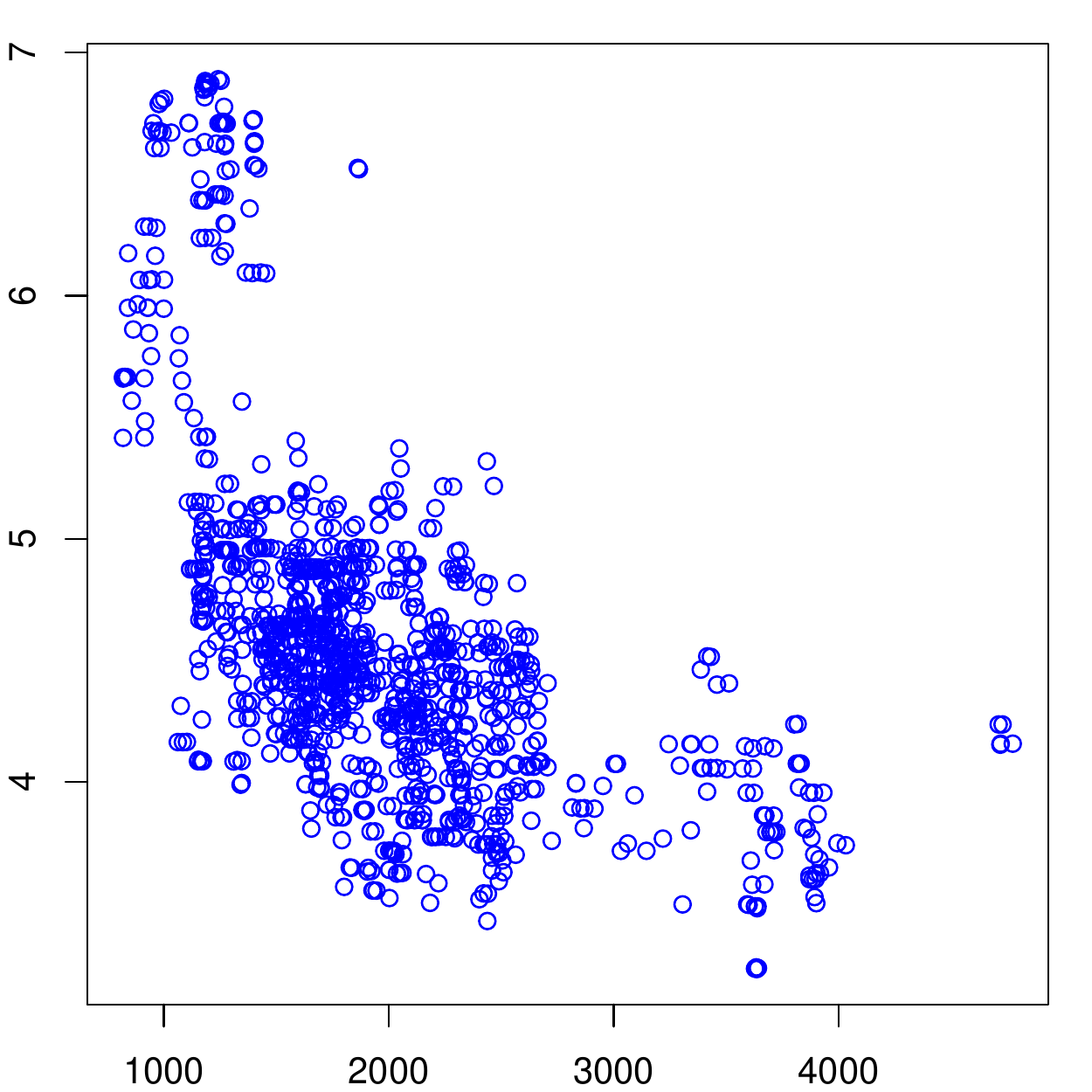} &
 \includegraphics[scale=.35]{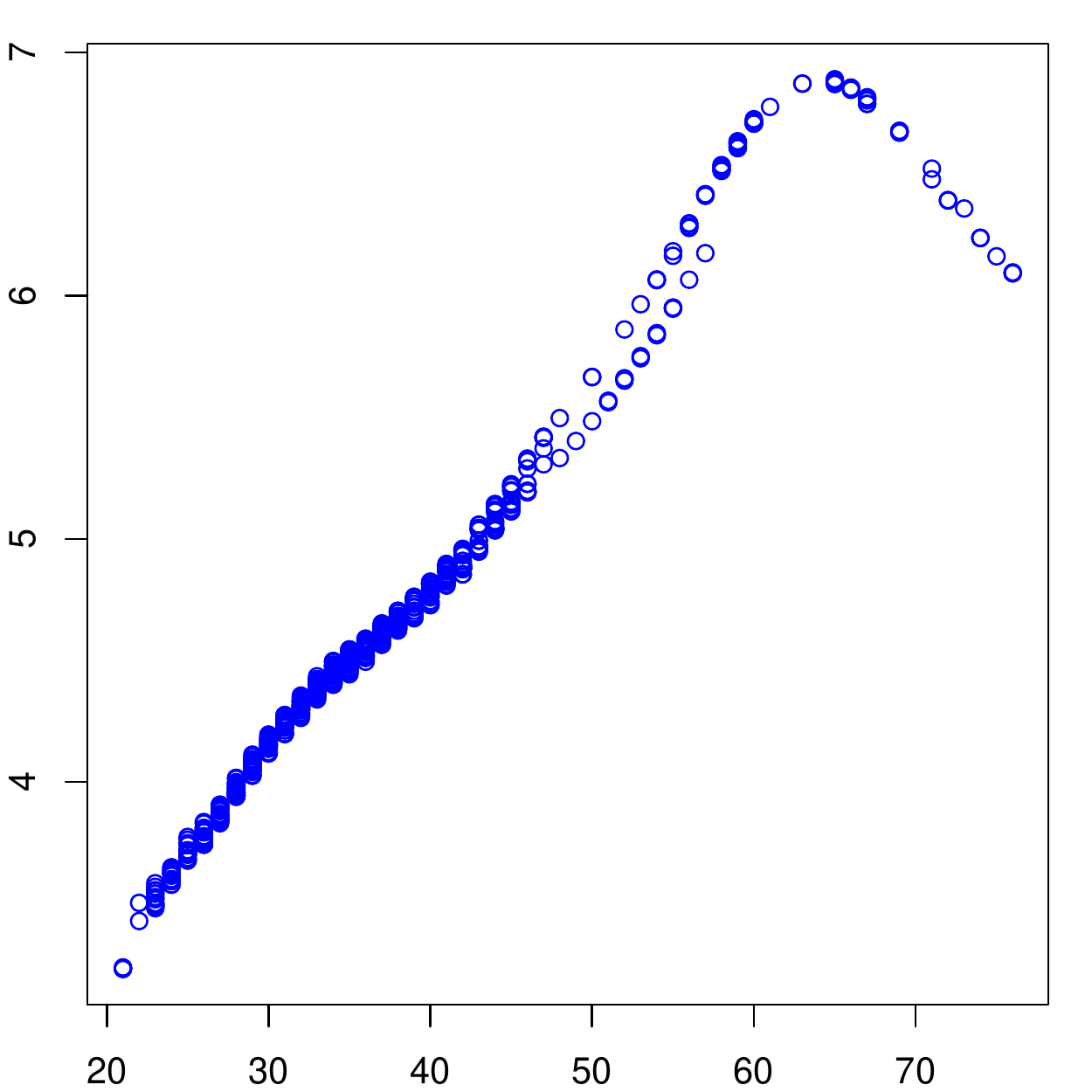}\\
\includegraphics[scale=.35]{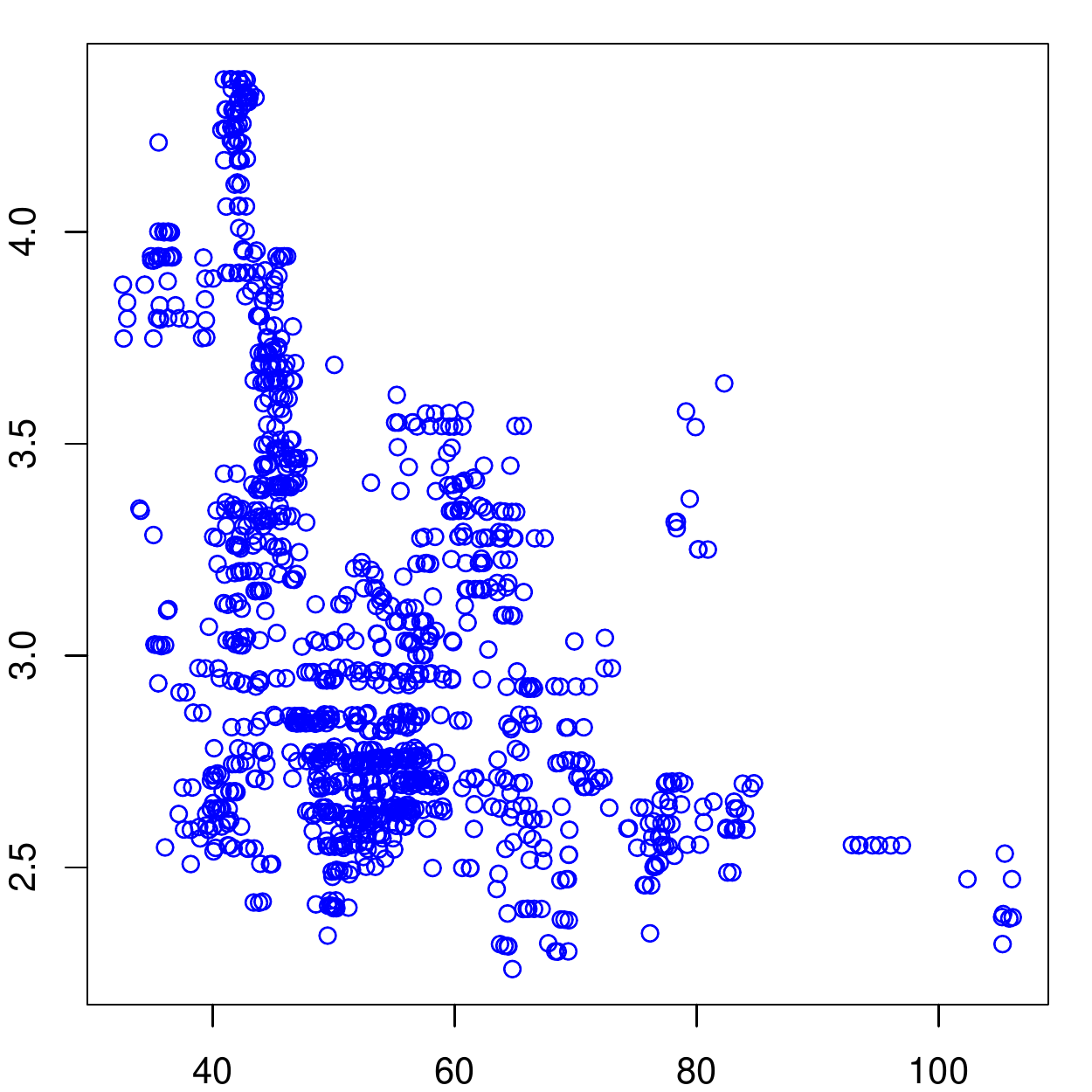} &
 \includegraphics[scale=.35]{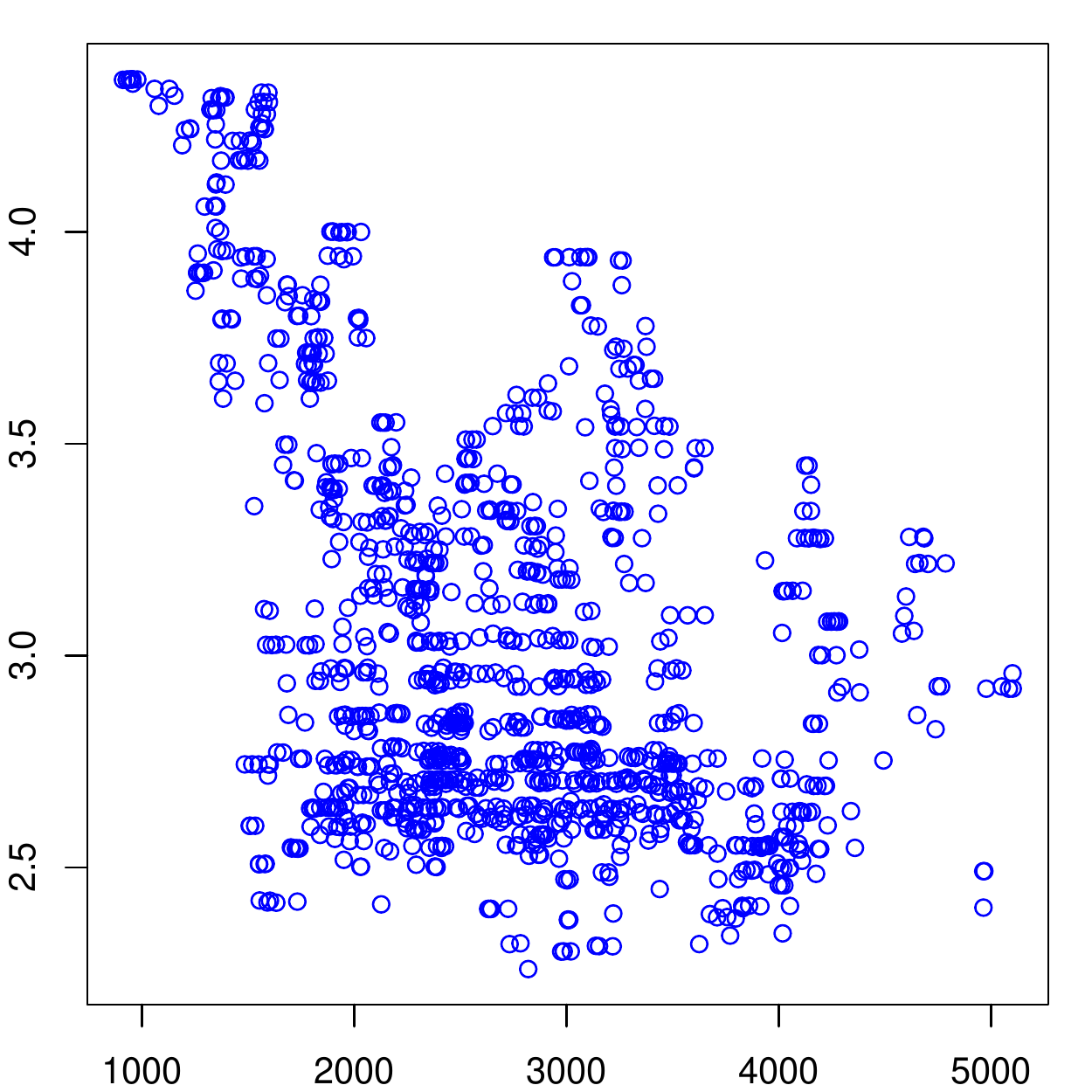} &
 \includegraphics[scale=.35]{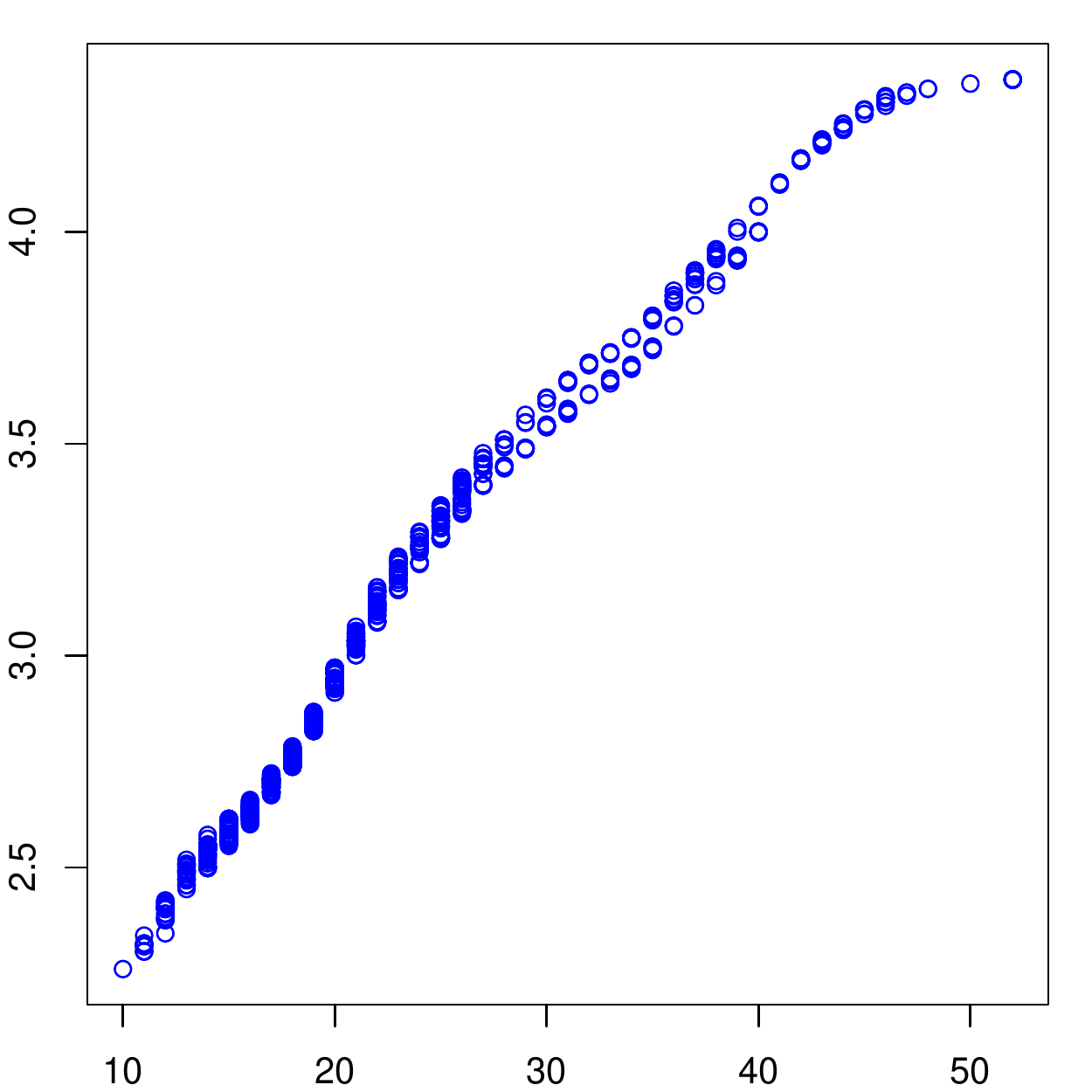}
 \end{tabular}
 \end{center}
\caption{Estimated death intensities of Langerin (top row) and Rab11 (bottom row) with respect to the maximal nearest neighbour distance of each configuration (left), to the maximal area of the Delaunay cells of each configuration (middle), and  to the number of proteins (right).}
\label{fig_geometric}
\end{figure}

We moreover verify in Figure~\ref{simu-slump} that the decrease of the death intensity for the highest  numbers of proteins in the Langerin sequence is not due to an estimation artefact. To this end, we simulate 100 processes with the same characteristics as the sequence of Langerin proteins. Specifically, the time of simulation is 171 seconds, the birth rate is constant and equals to 4.45 births per second and we assume that the death rate is linear and follows the same trend as for the Langerin sequence (represented as a black line in  Figure~\ref{simu-slump}). The starting  configuration of the process consists of approximately 60 points uniformly distributed. We considered for these 100 simulations the estimator based on the distance $d(x,y)=|n(x)-n(y)|$ for the same choice of bandwidth than that used for the Langerin sequence. The scatterplots of these 100 estimators with respect to the number of points are shown as gray lines in Figure~\ref{simu-slump}. The observed scatterplot from the Langerin sequence is shown in red, proving that the observed slump is unlikely to happen if the true death intensity was linear in the number of proteins. This leads us to conclude that the death intensity is subject to a significant decrease when the number of proteins is high.

\begin{figure}[ht]
\begin{center}
{\small
\begin{tabular}{c}
 \includegraphics[scale=.5]{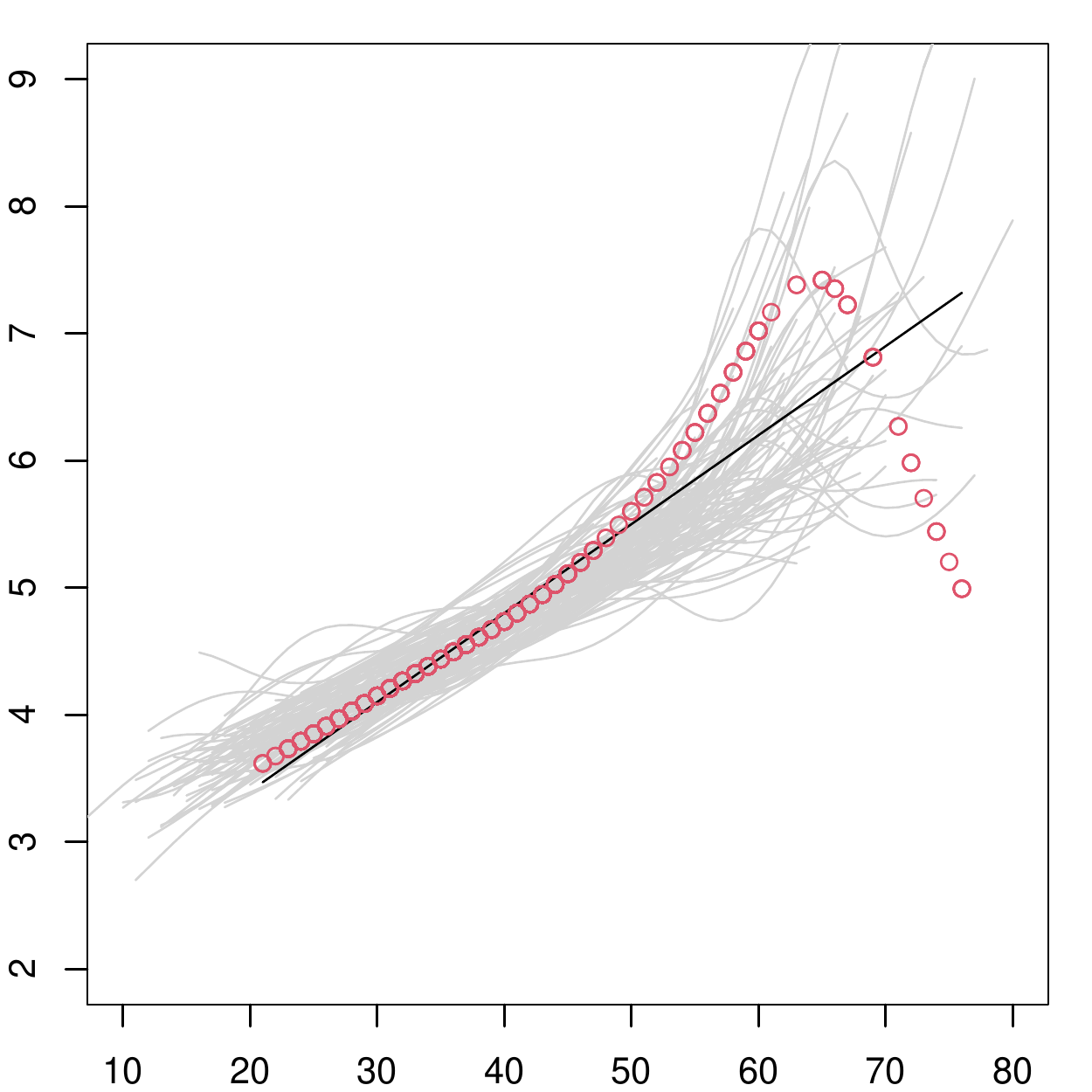}  
 \end{tabular}}
\caption{{\small In gray: scatterplots between the estimation of the death intensity and the number of points for 100 simulations where the true death intensity is the black line. In red: scatterplot between the estimation of the death intensity of the Langerin sequence and the number of proteins.}}\label{simu-slump} 
\end{center}
\end{figure}

Finally, Figure~\ref{fig:death trends}  completes Figure~\ref{fig:death} of the main manuscript. It shows the ccf of the detrented death intensities between Langerin and Rab11, where the lag is applied to Langerin. The estimated death intensities are represented as red curves in Figure~\ref{fig:death trends} while their trend is shown in black. These trends are the result of a loess smoothing with a smoothing parameter (span) equals to $0.1,0.2,0.4$ and $0.8$ from left to right, respectively. Whatever the degree of smoothing, the ccf is clearly asymmetric, with more significative correlations for the positive lags. This confirms the observation made in the main manuscript, namely that Rab11 seems to be active before Langerin.

\begin{figure}[ht]
\begin{center}
{\small
\begin{tabular}{cccc}
 \includegraphics[scale=.3]{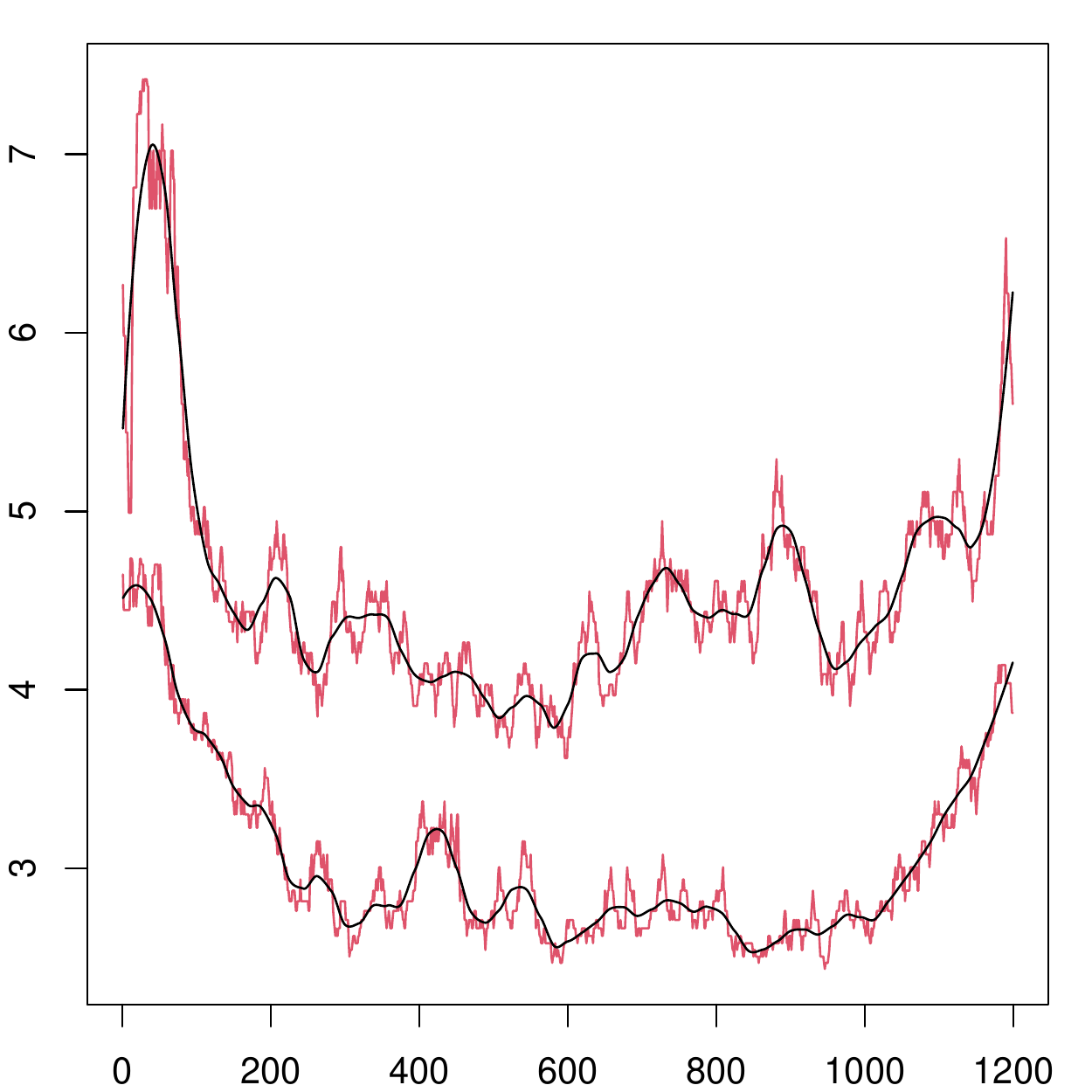} &  \includegraphics[scale=.3]{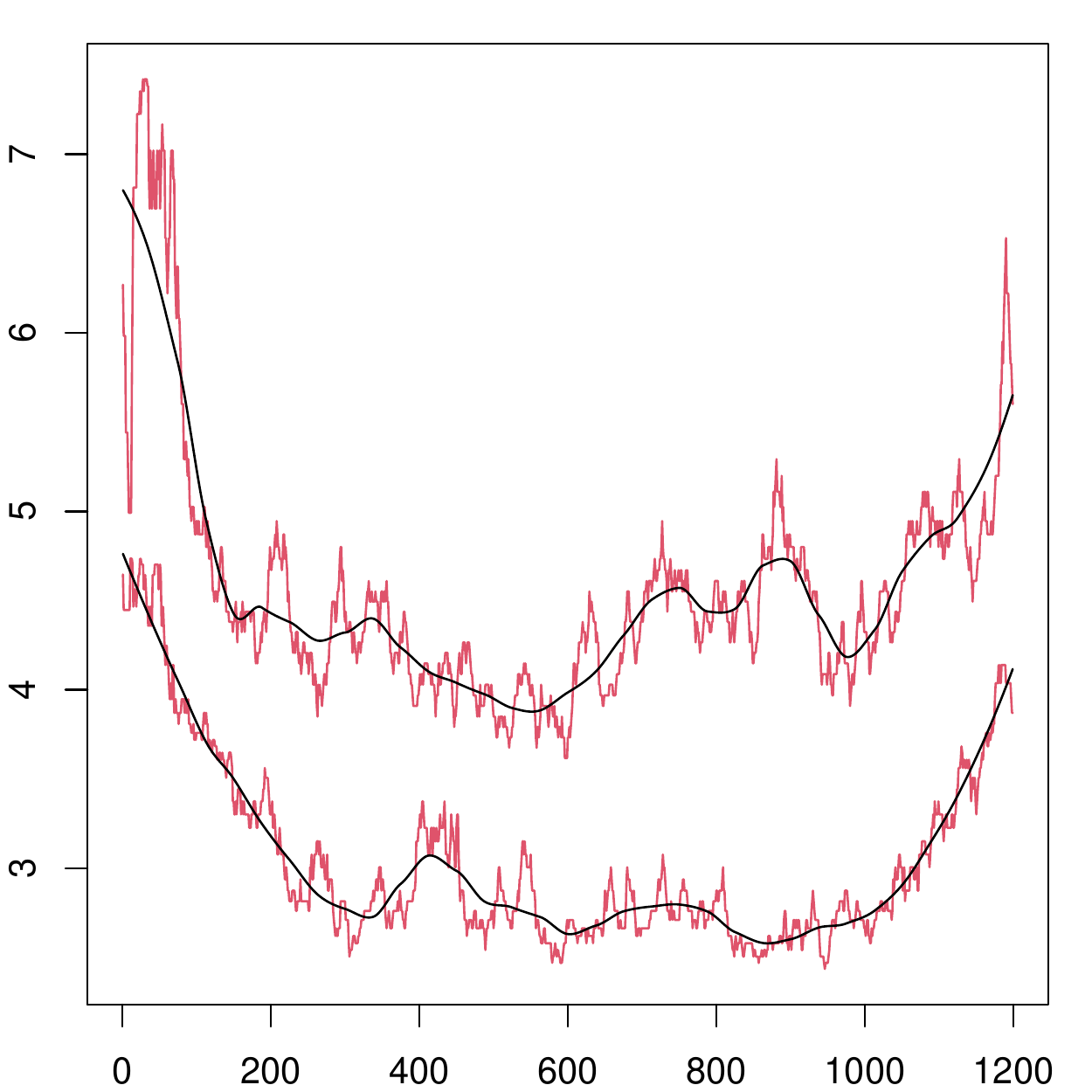}   &  \includegraphics[scale=.3]{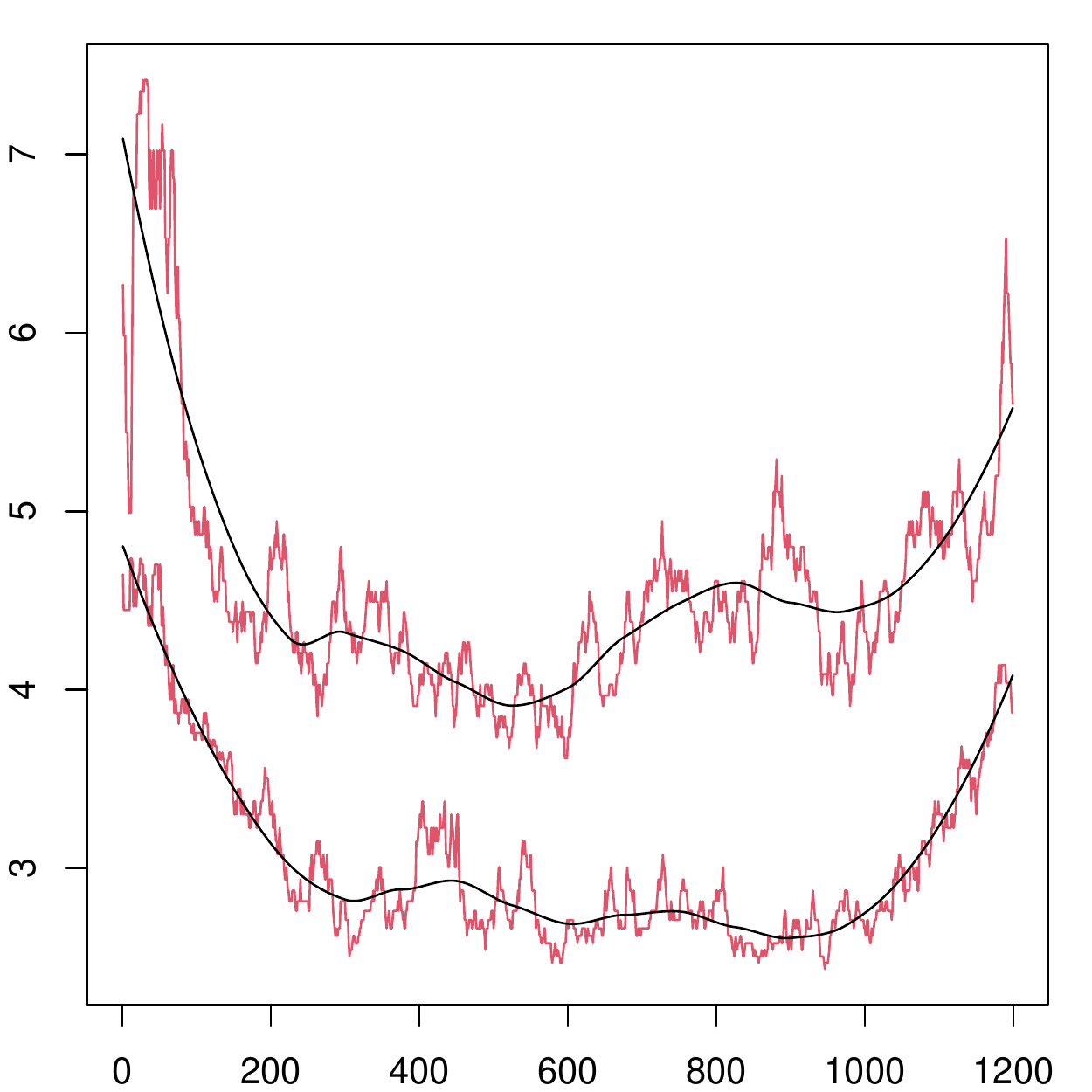} &  \includegraphics[scale=.3]{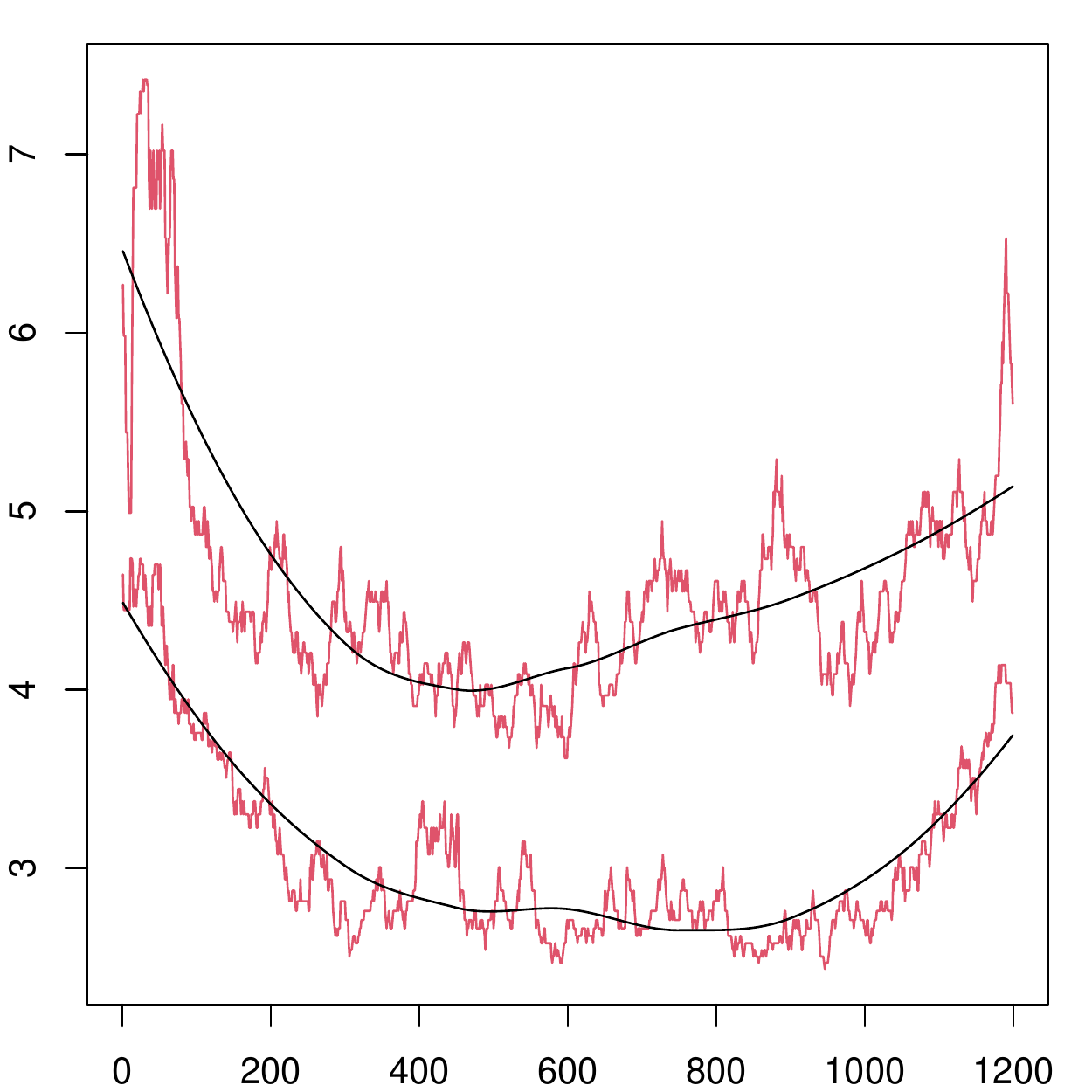}\\
 $span=0.1$ &  $span=0.2$ &  $span=0.4$ &  $span=0.8$ \\ 
  \includegraphics[scale=.3]{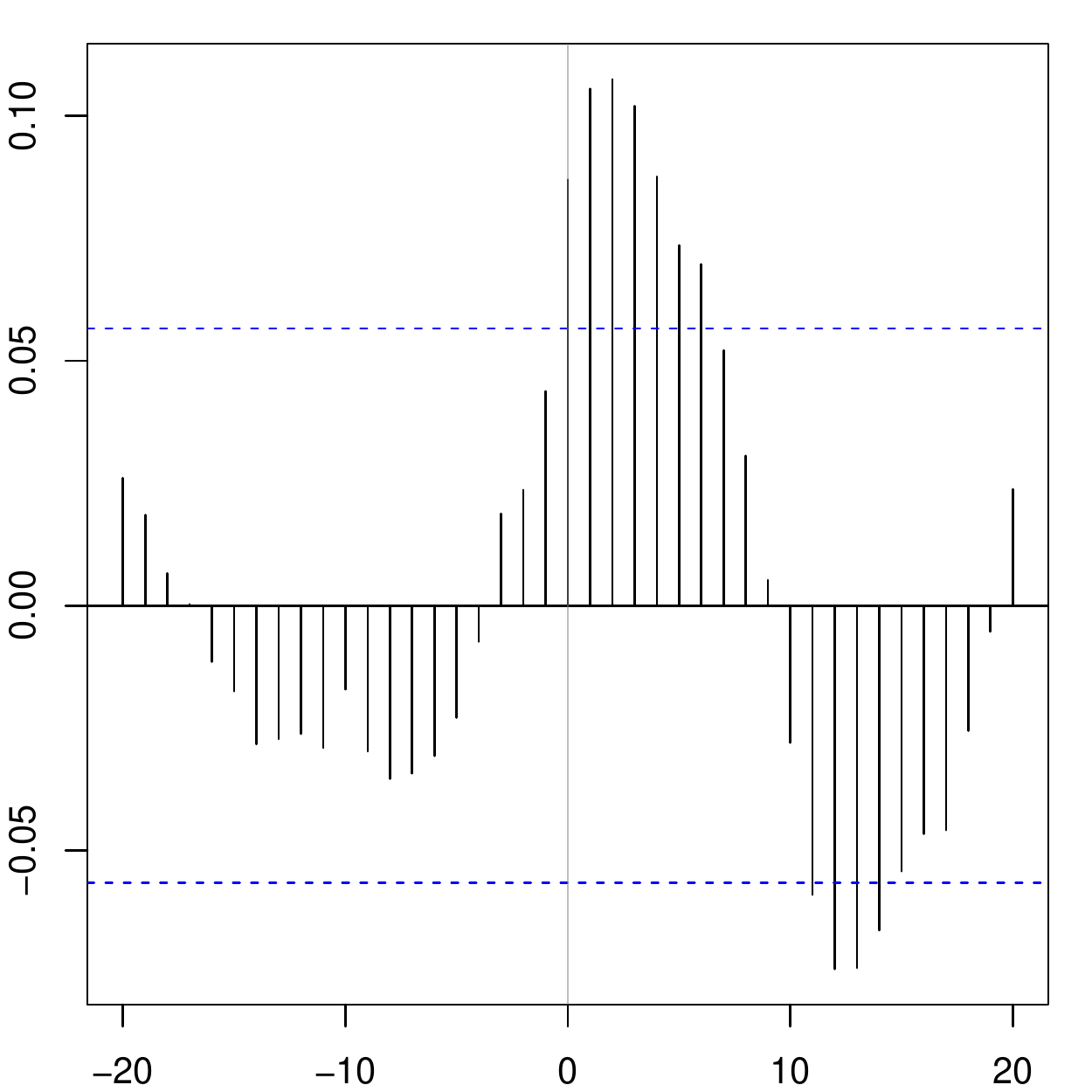}  & 
    \includegraphics[scale=.3]{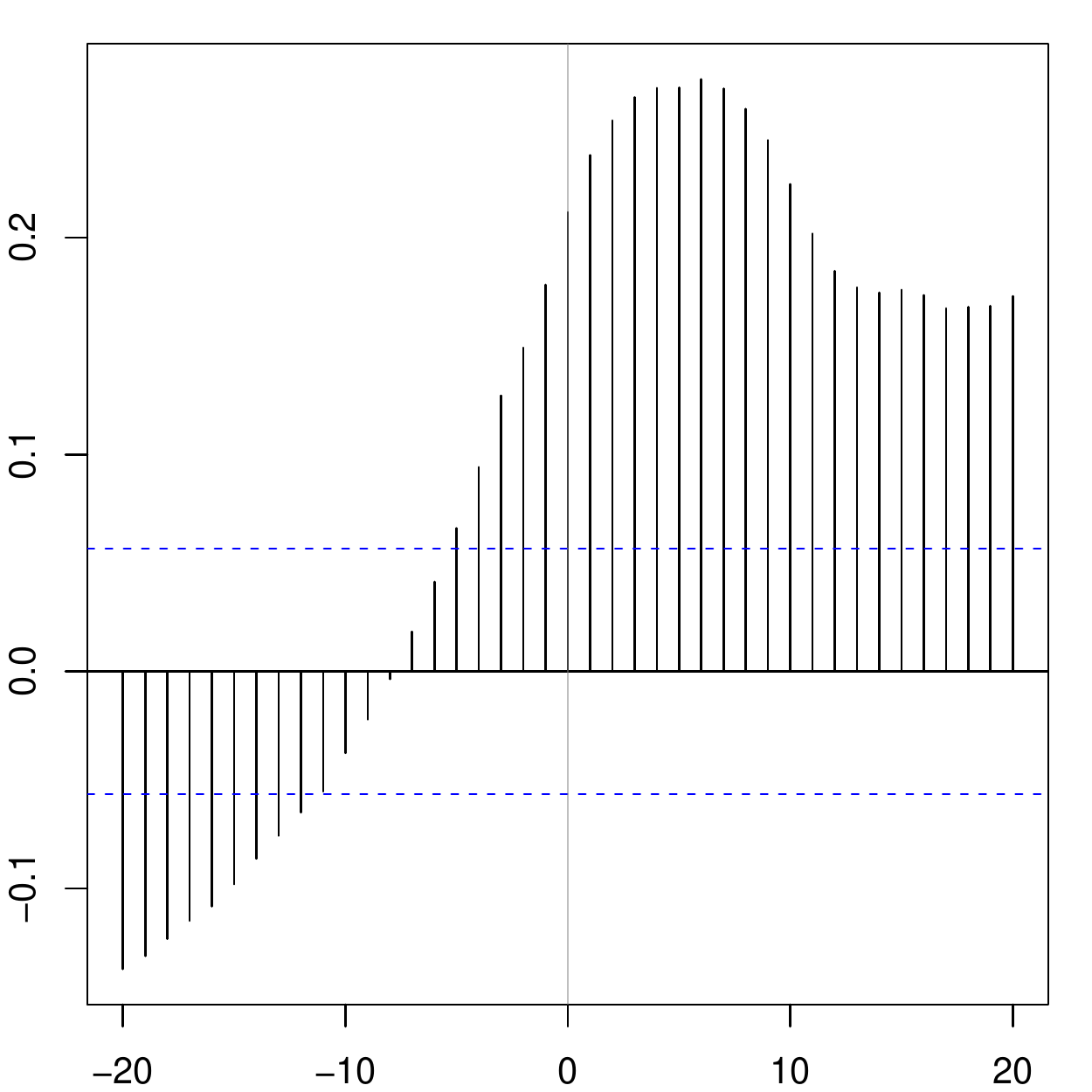} & 
      \includegraphics[scale=.3]{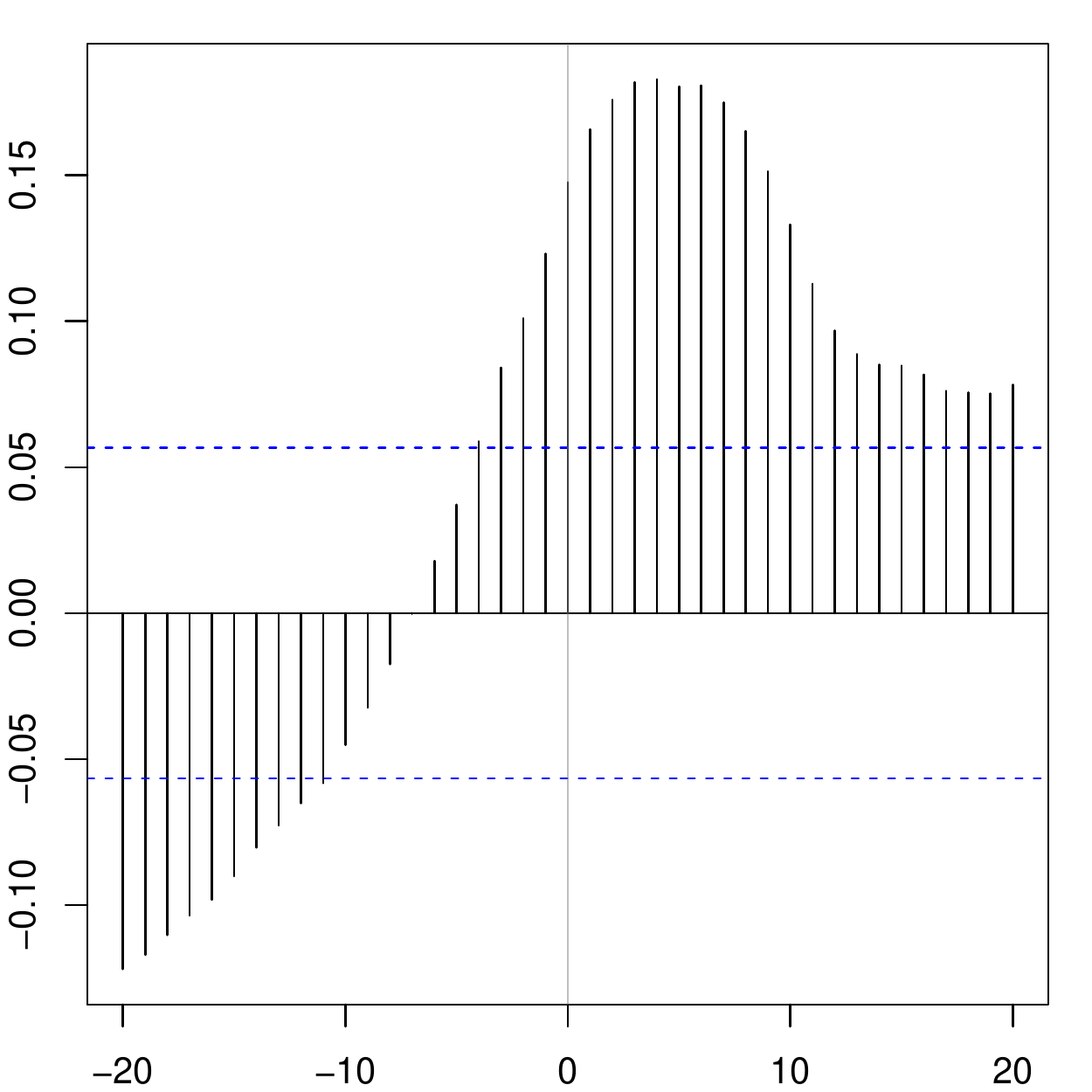} & 
        \includegraphics[scale=.3]{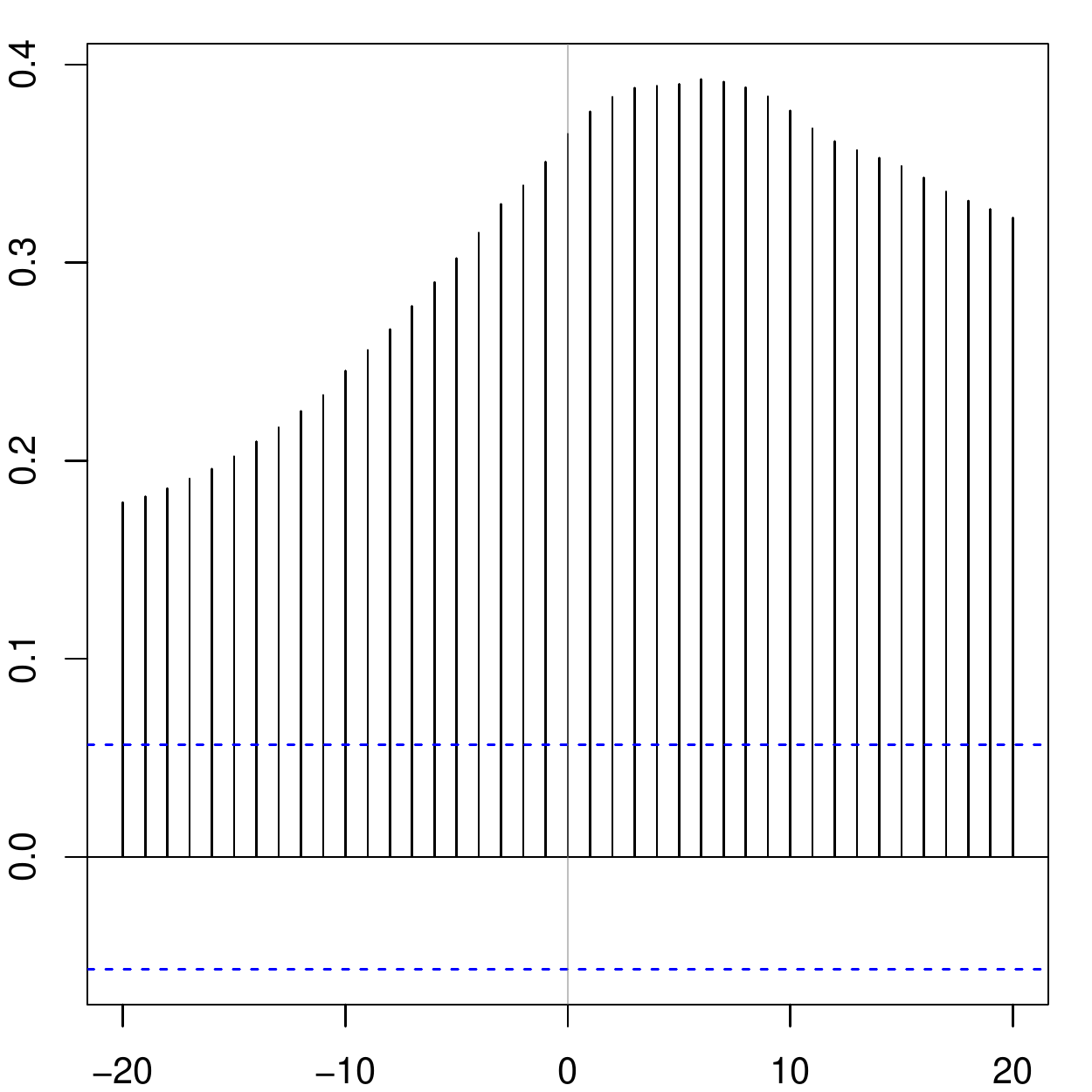} 
 \end{tabular}}
\caption{{\small First row: estimation of the death intensity for the Langerin sequence (top red curve) and for the Rab11 sequence (bottom red curve), along with their estimated trends (in black) obtained by a loess smoothing with parameter $span$ where, from left to right, $span=0.1,0.2,0.4,0.8$. Second row:  empirical ccf for the detrended curves (that are the difference between the red and the black curves of the first row).}}\label{fig:death trends} 
\end{center}
\end{figure}

The study of this dataset is far to be exhaustive. We have examined the intensities independently for each type of proteins. A natural extension would be to consider a joint intensity and/or marginal intensities that may depend on the current configurations of both types of proteins. This approach calls for the definition of multivariate birth-death-move processes, which certainly constitutes an exciting perspective of research. It would be also interesting to construct formal testing procedures in order to decide for instance whether the intensity function is constant (as we conclude it for the birth intensities in our applied study), or whether it depends on some features like the cardinality of the configurations (as for the death intensities in our application). 
Beyond the estimation of the intensities, the understanding of the whole spatio-temporal dynamics involves the estimation of the birth and death transition kernels, and the understanding of the inter-jumps dynamics.  These aspects will be the subject of future works.

 \section{Formal definition of  birth-death-move processes and some theoretical properties}\label{appendix:basics}

Birth-death-move processes are a special case of what we call a jump-move process. Informally, the dynamics of a jump-move process consists of a continuous stochastic motion in a space $E$ during a random time, followed by a jump in $E$, followed by a new continuous motion during a new random time, then a new jump, and so on. This dynamics shares close similarities with Markov processes obtained by ``piecing out'' strong Markov processes, as introduced in a very abstract way in \cite{ikeda1968}. We present in Section~\ref{appendix:BDM def} below a mechanistic construction which seems to us more accessible to understand and where the motions between jumps have not to be strong Markov but only Markov processes. 
We then introduce birth-death-move processes in Section~\ref{appendix:bdm} as a particular case of jump-move processes, and we derive in Section~\ref{appendix:bdm theory} important theoretical properties that are needed for the statistical study conducted in the main manuscript. The proofs  are provided in Section~\ref{sec:proofs}. 

\subsection{Spatial jump-move processes}\label{appendix:BDM def}

Let $E$ be a Polish space equipped with the Borel $\sigma$-algebra $\mathcal{E}$. Let further  $(\Omega,\F)$ be a measurable space and $(\P_x)_{x \in E}$ a family of probability measures on $(\Omega,\F)$. 

In order to define the jump-move process $(X_t)_{t\geq 0}$, we first introduce the process that will drive the inter-jumps motions. Let $( (Y_t)_{t\geq 0}, (\P_x)_{x \in E})$ be a continuous time-homogeneous Markov process on $E$ with respect to its natural filtration $(\F_t^{Y})_{t \geq 0}$. We denote its transition kernel on $E\times \mathcal E$ by $(Q_t^Y)_{t \geq 0}$, i.e. for any $x\in E$ and any $A\in \mathcal E,$
\[Q_t^{Y}(x,A)= \P_x(Y_t \in A).\]
We then let $(Y_t^{(j)})_{t\geq 0}$, $j\geq 0$, be a sequence of processes on $E$ identically distributed as $(Y_t)_{t\geq 0}$. In our construction below, $(Y_t^{(j)})_{t\geq 0}$ will support the motion  of $(X_t)_{t\geq 0}$ after the $j^{\textrm{th}}$ jump.

Moreover, to deal with the jumps of $(X_t)_{t\geq 0}$, we introduce the jump intensity function $\alpha : E \rightarrow \R^+,$ which is assumed to be continuous on $E$, as well as a probability transition kernel for the jumps denoted by $K : E \times \mathcal{E} \rightarrow [0,1].$
  We assume that $\alpha$ is bounded from below and above, i.e. there exist $\alpha_*>0$ and $\alpha^*<\infty$ such that for every $x \in E$,  $\alpha_* \leq \alpha(x) \leq \alpha^*$.

Given an initial distribution $\mu_0$ on $E$ and an initial configuration $X_0 \sim \mu_0$, the jump-move process  $(X_t)_{t\geq 0}$ is constructed as follows:
 \begin{itemize}

 \item[{1)}] Given $X_0 = z_0$, generate $(Y^{(0)}_t)_{t\geq 0}$ conditional on $Y^{(0)}_0=z_0$ according to the kernel $(Q_{t}^Y(z_0,.))_{t\geq 0}$. Then,
 
  \begin{itemize}
\item Given $X_0 = z_0$ and $(Y^{(0)}_t)_{t\geq 0}$, generate the first inter-jump time $\tau_1$ according to the cumulative distribution function
 \[F_1(t)  = 1-\exp\left( - \int_0^t \alpha( Y^{(0)}_u)du \right).\]
 
  \item Given $X_0 = z_0$, $(Y^{(0)}_t)_{t\geq 0}$ and $\tau_1$, generate the first post-jump configuration  $z_1$ according to the transition kernel $K \left( Y_{\tau_1}^{(0)}, \cdotp \right).$
 
 \item Set $T_1=\tau_1$, $X_t = Y_t^{(0)}$  for $t \in [0,T_1)$ and $X_{T_1}=z_1$.
 \end{itemize}
   \end{itemize}

And iteratively, for $j\geq 1$,
 \begin{itemize}
\item[{j{\small +1})}]  Given $X_{T_j} = z_j,$ generate $(Y^{(j)}_t)_{t\geq 0}$ conditional on $Y^{(j)}_0=z_j$ according to  $(Q_{t}^Y(z_j,.))_{t\geq 0}$. Then 
  \begin{itemize}
\item  Given $X_{T_j} = z_j$ and $(Y^{(j)}_t)_{t\geq 0}$, generate $\tau_{j+1}$ according to the cumulative distribution function 
   \begin{equation}\label{cdf tau}
   F_{j+1}(t)  =1- \exp\left( - \int_0^t \alpha( Y^{(j)}_u)du \right).\end{equation}
 
 \item  Given $X_{T_j} = z_j$,   $(Y^{(j)}_t)_{t\geq 0}$ and  $\tau_{j+1}$, generate $z_{j+1}$ according to $K \left( Y_{\tau_{j+1}}^{(j)}, \cdotp \right)$.
 
 \item Set $T_{j+1}=T_j + \tau_{j+1}$,  $X_t = Y_{t-T_j}^{(j)}$ for $t \in [T_j,T_{j+1})$ and $X_{T_{j+1}}=z_{j+1}$.
 \end{itemize}
 \end{itemize}

In this construction  $(T_j)_{j\geq 1}$ is the sequence of jump times of the process $(X_t)_{t\geq 0}$ and we set $T_0=0$. The number of jumps  before $t$ is denoted by $(N_t)_{t\geq0}$, that is $N_t=Card\{j\geq 1 : T_j \leq t\}$.

Let  $\F_t^{0} =  \sigma(X_u, u\leq t )$, $t>0$, be the natural filtration of the process $(X_t)_{t \geq 0}$. 
To avoid any measurability issues, we make this filtration complete \citep[Section 20.1]{Bass} and denote it by $\F_t$, i.e. $\F_t= \sigma(\F_t^{0} \cup \mathcal{N})$ where $\mathcal{N}$ is the collection of sets that are $\P_x$-null for every $x \in E$. For the same reason, we also complete the natural filtration $(\F_t^{Y})_{t \geq 0}$ of $(Y_t)_{t\geq 0}$. We recall that $(Y_t)_{t\geq 0}$ remains a time-homogeneous Markov process with respect to this completed filtration \citep[Proposition 20.2]{Bass}. We verify in the supplementary material that $((X_t)_{t \geq 0}, (\P_x)_{x \in E})$ is a Markov process with respect to $(\F_t)_{t>0}$. Note that the specific form \eqref{cdf tau} of the distribution of the waiting times is necessary to get a memoryless property. This is in line with the construction of piecewise-deterministic Markov processes \citep{davis1984}.

\begin{theo}\label{Markov}
 Let $(X_t)_{t \geq 0}$  be a spatial jump-move process.
 Then $\left( (X_t)_{t\geq 0}, (\P_x)_{x \in E}\right)$ is a time-homogeneous Markov process with respect to $(\F_t)_{t \geq 0}.$
\end{theo}

\subsection{Birth-death-move processes}\label{appendix:bdm}

A birth-death-move process is a jump-move process in $E = \bigcup_{n=0}^{+\infty} E_n$, associated with the $\sigma$-field  $\mathcal{E} = \sigma\left( \bigcup_{n=0}^{+\infty} \mathcal{E}_n\right),$ where $(E_n)_{n\geq 0}$  is a sequence of disjoint Polish spaces, each equipped with the Borel $\sigma$-algebra $\mathcal{E}_n.$ We assume that $E_0$ consists of a single element, written $\varnothing$ for short, i.e. $E_0=\{\varnothing\}$. 
 In addition, the move process $(Y_t)_{t\geq 0}$ is assumed to satisfy for all $x\in E$ and all $n\geq 0$, 
  \[\P_x(Y_t \in E_n, \ \forall t \geq 0) = \1_{x \in E_n}.\] 
 In other words,  $(Y_t)_{t\geq 0}$ may be defined separately in each  space $E_n$. 
 Concerning the  intensity of jumps, it writes  $\alpha=\beta+\delta$, where  $\beta : E \rightarrow \R^+$  is the birth intensity function and $\delta : E \rightarrow \R^+$ is the death intensity function, both assumed to be continuous on $E$.  We prevent a death in $E_0$ by assuming that $\delta(\varnothing)=0$. 
 Similarly,  the probability transition kernel depends on a kernel for the births $K_{\beta} : E \times \mathcal{E} \rightarrow [0,1]$ and a kernel for the deaths $K_{\delta} : E \times \mathcal{E} \rightarrow [0,1].$ They satisfy,  
for all $x \in E$,
\[K_{\beta}(x,\Et_{n+1}) = \1_{x \in \Et_n} \quad \textrm{ if } n\geq 0\]
and
\[K_{\delta}(x,\Et_{n-1}) = \1_{x \in \Et_n} \quad \textrm{ if } n\geq 1.\]
The transition kernel for the jumps of $(X_t)_{t\geq 0}$ is then, for any $x\in E$ and any $A\in\mathcal E$, 
\begin{equation}\label{defK}
K(x,A) = \frac{\beta(x)}{\alpha(x)} K_{\beta}(x,A) + \frac{\delta(x)}{\alpha(x)} K_{\delta}(x,A).\end{equation}
Note that the intensity functions $\beta$, $\delta$, $\alpha$ and the transition kernels $K_\beta$, $K_\delta$, $K$ play exactly the same role here as for a pure spatial birth-death process as introduced in \cite{preston}, which is the particular case of a birth-death-move process with no move ($Y_t=Y_0$ for all $t\geq 0$).

 \subsection{Some basic properties of birth-death-move processes}\label{appendix:bdm theory}

We verify in the next proposition that under Assumption~\eqref{existence}, a birth-death-move process $(X_t)_{t \geq 0}$ converges to a stationary distribution, uniformly over the initial configurations.

 \begin{prop}\label{InvMeas}
Under \eqref{existence},  $(X_t)_{t \geq 0}$ admits a unique stationary distribution $\mu_\infty$ and  
 there exist $a>0$ and $c>0$ such that for any measurable bounded function $g$ and any $t>0$,
\begin{equation}\label{H2 sature}
\sup_{y\in\Et} \left | \int_\Et g(z) Q_t(y,dz) - \int_\Et g(z)\mu_{\infty}(dz)\right |\leq a  e^{-ct}\|g\|_\infty.
\end{equation}
\end{prop}

For a pure spatial birth-death process, this proposition is a consequence of \cite{preston} and \cite{moller1989}. 
In these references, geometric ergodicity is also proven under a less restrictive setting than \eqref{existence}, a generalisation that is not straightforward to establish for a general birth-death-move process and that will be addressed in a separate contribution. Moreover, the uniformity in \eqref{H2 sature} only holds under \eqref{existence}, even for pure spatial birth-death processes, and this is a crucial property needed to establish \eqref{variance} in the next corollary. In the following (and all along the article) we omit the dependence on the initial distribution $\mu_0$ and simply write $\E$ and $\var$ instead of $\E_{\mu_0}$ and $\var_{\mu_0}.$

\begin{cor}\label{lemVar}
Under \eqref{existence}, for any measurable bounded non-negative function $g$ and any $t\geq 0$,
   \begin{equation}\label{H2 expectation}
 \left | \E \left(\frac 1t\int_0^t g(X_s)ds\right)  -  \int_E g(z)\mu_{\infty}(dz)\right |\leq \frac{a \|g\|_\infty}{ct} \end{equation}
where $a$ and $c$ are the same positive constants as in \eqref{H2 sature}. Moreover, 
   \begin{equation}\label{variance}
   \var\left( \int_{0}^t g(X_s) ds\right) \leq c_0 \|g\|_\infty\E\left(\int_{0}^t  g(X_s) ds \right)\end{equation}
where $c_0$ is some positive constant independent of $t$ and $g$.
\end{cor}

For $g(X_s)=k_T(x,X_s)$ and $t=T$, we immediately deduce the following result. 
\begin{cor}\label{moments T}
Under  \eqref{H3}, $\hat T(x)$ defined in \eqref{defT} verifies, as  $T\to \infty$, 
\begin{equation*}
\E(\hat T(x))\sim Tv_T(x) \quad \text{and} \quad \var\left(\hat{T}(x)\right)\leq c\,\E(\hat{T}(x)),
\end{equation*}
where $c$ is a positive constant.
\end{cor}

Finally, we clarify the martingale properties of the counting processes $N_t$ that are used in the proofs of consistency and exploited in the likelihood cross-validation of Section~\ref{CV}.  We set $\F_{t^+} = \bigcap_{s >t} \F_s.$
\begin{prop}\label{lemMart}
A left-continuous version of the intensity  of $N_t$ with respect to $\F_{t^+}$ is $\alpha(X_{t^-})$. Moreover, for any measurable bounded function  $g$, the process $(M_t)_{t\geq 0}$ defined by 
   \[M_t=\int_0^t g(X_{s^-}) [ dN_s - \alpha( X_{s})ds]\]
is a martingale with respect to $\F_{t^+}$ and  for all $t\geq 0$
   \begin{equation}\label{moments M}
   \E(M_t^2)=\E\left(\int_0^t g^2(X_s)\alpha(X_s)ds\right).
   \end{equation}
   The same statements hold true by replacing $N_t$ by the number of births (of deaths, respectively) and $\alpha(x)$ by $\beta(x)$ (by $\delta(x)$, respectively).
  \end{prop}

\section{Proofs of Propositions~\ref{Maintheo} and~\ref{th discret} of the main manuscript}\label{proofs prop13}

\subsection{Proof of Proposition~\ref{Maintheo}}
The proof is similar for $\hat\alpha(x)$, $\hat\beta(x)$ and $\hat\delta(x)$. We focus on $\hat\alpha(x)$ and we consider the decomposition
\[\hat{\alpha}(x) - \alpha(x) =  \frac{M_T}{\hat T(x)} + R_T,\]
where 
\[M_T = \int_0^T  k_T(x,X_{s^-}) [ dN_s - \alpha(X_{s})ds] \quad \text{and}\quad R_T = \frac{1}{\hat T(x)} \int_0^T( \alpha(X_s) - \alpha(x) ) k_T(x,X_s ) ds.\]
For all $\epsilon>0$,
\begin{align}\label{decomp alpha}
\P(|\hat{\alpha}(x) - \alpha(x)|>\epsilon)&\leq \P\left(\hat T(x)<\frac {\E(\hat T(x))} 2 \right) +\P\left( |\hat{\alpha}(x) - \alpha(x)|>\epsilon,\hat T(x)\geq \frac {\E(\hat T(x))} 2\right) \nonumber\\
&\leq \P\left(\hat T(x)<\frac {\E(\hat T(x))} 2 \right) + \P\left( \frac{|M_T|}{\hat T(x)} >\frac \epsilon 2, \hat T(x)\geq\frac {\E(\hat T(x))} 2 \right)+   \P\left(|R_T| >\frac \epsilon 2\right)\nonumber\\
&\leq \P\left(|\hat T(x)-\E(\hat T(x))|>\frac {\E(\hat T(x))} 2 \right) + \P\left( |M_T|>\frac{\epsilon \E(\hat T(x))}4  \right)+   \P\left(|R_T| >\frac \epsilon 2\right)\nonumber\\
&\leq \frac{4\var(\hat T(x))}{\E(\hat T(x))^2} + \frac{16 \E(M_T^2)}{\epsilon^2\E(\hat T(x))^2} +  \frac{4 \E(R_T^2)}{\epsilon^2},
\end{align}
using the Markov inequality in the last line.  By \eqref{moments M} applied with $g(X_s)=k_T(x,X_s)$ and $t=T$, and since $k_T$ is uniformly bounded, we get that $\E\left(M_T^2\right)\leq c_0 \E(\hat T(x))$.  
By Corollary~\ref{moments T}, we deduce that
\begin{equation*}
\P(|\hat{\alpha}(x) - \alpha(x)|>\epsilon)\leq  \frac{c_0}{Tv_T(x)} + \frac{4 \E(R_T^2)}{\epsilon^2}.
\end{equation*}
 The proof is then complete once we show that
\begin{equation}\label{R2}
E\left(R_T^2\right) \leq c_0\left( \frac{1}{T v_T(x)} + w_T^2(x)\right).
\end{equation}
To this end, notice that almost surely $|R_T| \leq 2 \alpha^*$ because $\alpha\leq\alpha^*$, so
\begin{align*}
&E\left(R_T^2\right) \leq 4 (\alpha^*)^2 \P\left( \hat{T}(x) \leq \frac{\E(\hat{T}(x))}{2}\right) + \frac{4}{\E(\hat{T}(x))^2 } \E\left[\left(\int_0^T ( \alpha(X_s) - \alpha(x)) k_T( d( x,X_s)) ds\right)^2 \right].
\end{align*}
Denoting $I_T= \int_0^T \alpha(X_s) k_T(x,X_s )ds$ and using the Chebyshev inequality, we obtain
\begin{align*}
 E\left(R_T^2\right)& \leq \frac{c_0}{\E(\hat{T}(x))^2} \left(\var\left(\hat{T}(x)\right) + \E\left[\left(I_T - \alpha(x) \hat{T}(x)\right)^2\right]\right)\\ 
& \leq \frac{c_0}{\E(\hat{T}(x))^2} \left(\var\left(\hat{T}(x)\right) + 3 \var\left(I_T\right) + 3 \var\left(\alpha(x) \hat{T}(x)\right) + 3\left( \E(I_T) - \alpha(x) \E(\hat{T}(x))\right)^2\right).
 \end{align*}
By \eqref{H2 expectation}, we know that 
$$\left | \E(I_T) - T\int_\Et \alpha(z) k_T(x,z)\mu_\infty(dz)\right|\leq c_0 \alpha^* k^*$$
and 
$$\left | \alpha(x)\E(\hat T(x)) - T\alpha(x)\int_\Et  k_T(x,z)\mu_\infty(dz)\right|\leq c_0 \alpha^* k^*$$
so, using the notation from \eqref{H3}-\eqref{H4}, 
 \begin{align*}
 \left | \E(I_T) - \alpha(x) \E(\hat{T}(x))\right | & \leq 2c_0 \alpha^* k^* + T \left| \int_{\Et}  (\alpha(z) - \alpha(x)) k_T( x,z) \mu_{\infty} (dz) \right|= c_1 + Tv_T(x) |w_T(x)|
\end{align*}
where $c_1=2c_0 \alpha^* k^*$. 
Moreover, from \eqref{variance}, $\var\left(I_T\right) \leq c_0 \E(I_T) \leq c_0 \alpha^* \E(\hat{T}(x))$ and by Corollary~\ref{moments T}, we obtain
\begin{align*}
 E\left(R_T^2\right)& \leq  \frac{c_0}{(Tv_T(x))^2}\left( Tv_T(x) + (c_1 + Tv_T(x) |w_T(x)|)^2\right)\leq   \frac{2c_0}{(Tv_T(x))^2}\left( Tv_T(x) + c_1^2+ T^2v_T^2(x)w_T^2(x)\right)
 \end{align*}
 which implies \eqref{R2} under \eqref{H3}.

\subsection{Proof of Proposition~\ref{th discret}}

We detail the proof for $\gamma=\alpha$, the other cases being similar. Recall that $\hat{T}_{(d)}(x) =\sum_{j=0}^{m-1} \Delta t_{j+1} k_T( x, X_{t_j})$. To save space, we write $\hat{T}_{(d)}$ for $\hat{T}_{(d)}(x)$, $\hat T$ for $\hat T(x)$, $\ell_T$ for $\ell_T(x)$ and $v_T$ for $v_T(x)$. We have
  \begin{align*}
 |\hat{\alpha}_{(d)}(x) - \alpha(x)|\leq &\frac{1}{\hat{T}_{(d)}} \sum_{j=0}^{m-1} k_T( x, X_{t_j})|D_{j+1} - \Delta N_{t_{j+1}}|
  +\frac{1}{\hat{T}_{(d)}} \left|\sum_{j=0}^{m-1}  k_T( x, X_{t_j}) \Delta N_{t_{j+1}} - \int_0^T  k_T( x, X_{s^-})dN_s\right|\\
 &+ \frac{\int_0^T k_T(x,X_{s^-})dN_s}{\hat{T}_{(d)} \hat{T}}|\hat{T}-\hat{T}_{(d)}|+ |\hat{\alpha}(x) - \alpha(x)|\\
 \leq& \frac{k^*}{\hat{T}_{(d)}}  \sum_{j=0}^{m-1} \Delta N_{t_{j+1}} \1_{\Delta N_{t_{j+1}} \geq 2} + \frac{1}{\hat{T}_{(d)}}  \sum_{j=0}^{m-1} | k_T( x, X_{t_j}) \Delta N_{t_{j+1}} - \int_{t_{j}}^{t_{j+1}} k_T( x, X_{s^-})dN_s|\\
 &+ \frac{k^* N_T}{\hat{T}_{(d)} \hat{T}}|\hat{T}-\hat{T}_{(d)}| + |\hat{\alpha}(x) - \alpha(x)|.
 \end{align*}
Concerning the second term in this sum we have 
\[\left|  k_T( x,X_{t_j}) \Delta N_{t_{j+1}} - \int_{t_j}^{t_{j+1}} k_T(x,X_{s^-}) dN_s\right|\1_{\Delta N_{t_{j+1}}=0}=0,\]
  \begin{equation*}
  \left|  k_T( x,X_{t_j}) \Delta N_{t_{j+1}} - \int_{t_j}^{t_{j+1}} k_T(x,X_{s^-}) dN_s\right|\1_{\Delta N_{t_{j+1}}=1}= \left|  k_T( x, X_{t_j}) - k_T(x,X_{T_{N_{t_{j+1}}}^-}) \right| \1_{\Delta N_{t_{j+1}}=1}
  \end{equation*}
and 
\[\left| k_T( x,X_{t_j}) \Delta N_{t_{j+1}}- \int_{t_j}^{t_{j+1}} k_T(x,X_{s^-}) dN_s\right |\1_{\Delta N_{t_{j+1}}\geq 2} \leq 2 k^* \Delta N_{t_{j+1}} \1_{\Delta N_{t_{j+1}}\geq 2},\]
so that 
  \begin{align*}
  &\sum_{j=0}^{m-1} | k_T( x, X_{t_j}) \Delta N_{t_{j+1}} - \int_{t_{j}}^{t_{j+1}} k_T( x, X_{s^-})dN_s|\\
 & \leq\sum_{j=0}^{m-1}  \left|  k_T( x, X_{t_j}) - k_T(x,X_{T_{N_{t_{j+1}}}^-}) \right| \1_{\Delta N_{t_{j+1}}=1}  +2 k^* \sum_{j=0}^{m-1} \Delta N_{t_{j+1}} \1_{\Delta N_{t_{j+1}} \geq 2}.
   \end{align*}
 On the other hand, 
    \begin{align}\label{EspDiffT}
  |\hat{T}-\hat{T}_{(d)}|& \leq \sum_{j=0}^{m-1}   \int_{t_j}^{t_{j+1}}  \left|k_T( x,X_{s}) -  k_T( x,X_{t_j})\right| ds \nonumber\\
  &\leq \sum_{j=0}^{m-1}   \int_{t_j}^{t_{j+1}}  \left|k_T( x,X_{s}) -  k_T( x,X_{t_j})\right| ds \1_{\Delta N_{t_{j+1}} = 0} + 2k^*   \Delta_m \sum_{j=0}^{m-1}  \1_{\Delta N_{t_{j+1}} \geq 1} \nonumber\\
  & \leq \sum_{j=0}^{m-1}   \int_{t_j}^{t_{j+1}}  \left|k_T( x,X_{s}) -  k_T( x,X_{t_j})\right| ds \1_{\Delta N_{t_{j+1}} = 0} + 2k^*   \Delta_m N_T.
      \end{align}  
  We thus obtain 
  \begin{align}\label{main-discrete}
 |\hat{\alpha}_{(d)}(x) - &\alpha(x)|\nonumber\\
 & \leq  \frac{3k^*}{\hat{T}_{(d)}}  \sum_{j=0}^{m-1} \Delta N_{t_{j+1}} \1_{\Delta N_{t_{j+1}} \geq 2} +  \frac{1}{\hat{T}_{(d)}} \sum_{j=0}^{m-1} \left|  k_T( x, X_{t_j}) - k_T(x,X_{T_{N_{t_{j+1}}}^-}) \right| \1_{\Delta N_{t_{j+1}}=1} \nonumber \\
 &+ \frac{k^* N_T}{\hat{T}_{(d)} \hat{T}}\sum_{j=0}^{m-1}   \int_{t_j}^{t_{j+1}}  \left|k_T( x,X_{s}) -  k_T( x,X_{t_j})\right| ds \1_{\Delta N_{t_{j+1}} = 0} + \frac{2(k^* N_T)^2  \Delta_m}{\hat{T}_{(d)} \hat{T}}  + |\hat{\alpha}(x) - \alpha(x)|.
   \end{align}
   The last term has been treated in Proposition~\ref{Maintheo}. Let us consider the other terms. First note that $\eqref{EspDiffT}$ entails with \eqref{discrete condition}
   \begin{align*}
  \E[|\hat{T}-\hat{T}_{(d)}|]\leq l_T \Delta_m^a T + 2k^*\Delta_m \E(N_T).   
   \end{align*}
 By Proposition~\ref{lemMart} applied with $g=1$, we get that $\E(N_T)=\E \int_0^T \alpha(X_s) ds\leq \alpha^* T$ and then
     \begin{align}\label{decomposDiscret3}
   \P(\hat T_{(d)}<Tv_T/4)& \leq \P(|\hat{T}-\hat{T}_{(d)}|>Tv_T/4) + \P(\hat T<Tv_T/2) \nonumber\\
   &\leq \frac{4 \E[|\hat{T}-\hat{T}_{(d)}|]}{Tv_T} + \frac{c_0}{Tv_T} \nonumber\\
      &\leq  \frac{c_0 l_T \Delta_m^a}{v_T} + \frac{c_0 \Delta_m}{v_T} + \frac{c_0}{Tv_T}. 
    \end{align}

 Since $v_T$ is bounded, we deduce from \eqref{H3}, \eqref{stepto0} and \eqref{discrete condition} that $ \P(\hat T_{(d)}<Tv_T/4)$ tends to 0 and for the rest of the proof we place ourselves on the event $\{\hat T_{(d)}\geq Tv_T/4\}$. 
We shall use the following lemma.
 \begin{lem}\label{lemcount}
  For any $j=1,\dots,m$ and any $i\in\N$, $\P( \Delta N_{t_j}=i)\leq (\alpha^* \Delta t_j)^i/(i!)$.
 \end{lem}
 
 \begin{proof}
 By the Markov property, $\P( \Delta N_{t_j}=i) = \E[\P(\Delta N_{t_j}=i | X_{t_j})] = \E[\P( N_{\Delta t_j}=i | X_0 = X_{t_j} )]$. 
The lemma is proved once we verify that 
\[\forall t \geq 0, \forall i \in \N, \quad \sup_{x \in E} \P(N_t =i | X_0=x ) \leq \frac{(\alpha^* t)^i}{i!}.\]
This inequality is obvious if $i=0$. Let $i \geq 1$,
 \[\P(N_t =i | X_0=x ) = \E\left[ \P\left( T_1 < T_2 < ...< T_i < t < T_{i+1} \Big | X_0=x, X_{T_j}, (Y^{(j)}_t)_{t\geq 0},\forall j\geq 0\right) \Big |X_0=x\right].\]
 Let $\tau_j=T_j-T_{j-1}$. Conditional on  $X_{T_j}$ and $(Y^{(j)}_t)_{t\geq 0}$, the random variable $\tau_j$ has density \[t\mapsto \alpha(Y^{(j)}_t)e^{-\int_0^t \alpha(Y^{(j)}_u)du},\] whereby
 \begin{align*}
  &\P\left( T_1 < T_2 < ...< T_i < t < T_{i+1} \Big | X_0=x, X_{T_j}, (Y^{(j)}_t)_{t\geq 0},\forall j\geq 0\right) \\
  &= \P \left( \tau_1<t, \tau_2 < t-\tau_1, \dots , \tau_i < t-\sum_{j=1}^{i-1} \tau_j, \tau_{i+1} > t- \sum_{j=1}^{i} \tau_j  \Big | X_0=x, X_{T_j}, (Y^{(j)}_t)_{t\geq 0},\forall j\geq 0\right )\\
  &= \int_0^t  \alpha(Y^{(0)}_{s_1}) e^{-\int_0^{s_1}  \alpha(Y^{(0)}_{u})du} ds_1 \int_0^{t-s_1} \alpha(Y^{(1)}_{s_2}) e^{-\int_0^{s_2}  \alpha(Y^{(1)}_{u})du}  ds_2 \\
 &\qquad \qquad \cdots \int_0^{t-\sum_{j=1}^{i-1} s_j}  \alpha(Y^{(i-1)}_{s_i}) e^{-\int_0^{s_i}  \alpha(Y^{(i-1)}_{u})du} ds_i \, e^{-\int_0^{t-\sum_{j=1}^{i}s_j} \alpha(Y^{(i)}_{u})du  } \\
 &\leq (\alpha^*)^i \int_0^t ds_1 \int_0^{t-s_1} ds_2  ... \int_0^{t-\sum_{j=1}^{i-1} s_j} ds_k
 = (\alpha^*)^i \int_{\R^i} \1_{t_1 < t_2 < ... < t_i < t} dt_1\dots dt_i
 = \frac{(\alpha^* t)^i}{i!}. 
 \end{align*}
\end{proof}

Let us control the first term in \eqref{main-discrete}. For any $\epsilon>0$, using the previous lemma,
   \begin{align*} 
   \P\left( \frac{3k^*}{\hat{T}_{(d)}}  \sum_{j=0}^{m-1} \Delta N_{t_{j+1}} \1_{\Delta N_{t_{j+1}} \geq 2} >\epsilon, \hat T_{(d)}\geq \frac{Tv_T} 4\right )
  &\leq    \P\left( \sum_{j=0}^{m-1} \Delta N_{t_{j+1}} \1_{\Delta N_{t_{j+1}} \geq 2} >\frac{\epsilon T v_T}{12 k^*}\right)\\
  &\leq \frac{12 k^*}{\epsilon Tv_T} \sum_{j=0}^{m-1} \E(\Delta N_{t_{j+1}} \1_{\Delta N_{t_{j+1}} \geq 2})\\
  &\leq \frac{12 k^*}{\epsilon Tv_T} \sum_{j=0}^{m-1} \sum_{i\geq 2} i \P(\Delta N_{t_{j+1}}=i) \\
  &\leq \frac{12 k^*}{\epsilon Tv_T} m \sum_{i\geq 2} i \frac{(\alpha^* \Delta_m)^i}{i!}=  \frac{12 k^* m \alpha^*\Delta_m}{\epsilon Tv_T}(e^{\alpha^*\Delta_m}-1).
  \end{align*}
 This last expression is lower than $m\Delta_m^2/(T v_T)$, up to a positive constant, because $\Delta_m\to 0$ as a consequence of  \eqref{stepto0}. Since by assumption  $m\Delta_m/T$ is uniformly bounded  
  we deduce that 
 \begin{equation}\label{decomposDiscret1}
\P\left( \frac{3k^*}{\hat{T}_{(d)}}  \sum_{j=0}^{m-1} \Delta N_{t_{j+1}} \1_{\Delta N_{t_{j+1}} \geq 2} >\epsilon \right)\leq c_0\frac{\Delta_m}{v_T}.
\end{equation}

Concerning now the second term in \eqref{main-discrete},
 \begin{align*} 
   &\P\left(\frac{1}{\hat{T}_{(d)}} \sum_{j=0}^{m-1}  |  k_T( x, X_{t_j}) - k_T(x,X_{T_{N_{t_{j+1}}}^-}) | \1_{\Delta N_{t_{j+1}}=1} >\epsilon, \hat T_{(d)}\geq \frac{Tv_T} 4\right )\\
  &\leq   \P\left(\frac{4}{Tv_T} \sum_{j=0}^{m-1}  |  k_T( x, X_{t_j}) - k_T(x,X_{T_{N_{t_{j+1}}}^-}) | \1_{\Delta N_{t_{j+1}}=1} >\epsilon \right )\\
  &\leq \frac{ 4 \sum_{j=0}^{m-1} \E[|  k_T( x, X_{t_j}) - k_T(x,X_{T_{N_{t_{j+1}}}^-}) | \1_{\Delta N_{t_{j+1}}=1}]}{T v_T \varepsilon}.
   \end{align*}
 Let us give an upper bound for this expectation. Using the Markov property, we compute
   \begin{align*}
    \E[|  k_T( x, X_{t_j}) - k_T(x,X_{T_{N_{t_{j+1}}}^-}) | \1_{\Delta N_{t_{j+1}}=1}] &\leq \E[ \E_{X_{t_j}} [|  k_T( x, X_0) - k_T(x,X_{T_{1}^-}) | \1_{T_1 < \Delta t_{j+1}} \1_{T_2 >\Delta t_{j+1}} ]].
   \end{align*}
Moreover for every $y \in E,$
\begin{align*}
 \E_{y} [|  k_T( x, X_0) - k_T(x,X_{T_{1}^-}) |  & \1_{T_1 < \Delta t_{j+1}}   \1_{T_2 >\Delta t_{j+1}} ] \\
 &\leq  \E_{y} [|  k_T( x, Y^{(0)}_0) - k_T(x,Y^{(0)}_{T_{1}}) | \1_{T_1 < \Delta t_{j+1}}  ]\\
 &= \E_{y}[\E_{y} [|  k_T( x, Y^{(0)}_0) - k_T(x,Y^{(0)}_{T_{1}}) | \1_{T_1 < \Delta t_{j+1}}   | Y^{(0)}]]\\
 &= \E_{y}[ \int_0^{\Delta t_{j+1}} \alpha( Y^{(0)}_s) e^{-\int_0^s \alpha( Y^{(0)}_u) du} |  k_T( x, Y^{(0)}_0) - k_T(x,Y^{(0)}_{s}) | ds ] \\
 &\leq \int_0^{\Delta t_{j+1}} \alpha^* \E_{y}[ |  k_T( x, Y^{(0)}_0) - k_T(x,Y^{(0)}_{s}) |] ds  \\
 &\leq \alpha^* l_T \Delta_m^{a+1},
\end{align*}
where we have used \eqref{discrete condition} in the last line. We therefore obtain 
\begin{equation}\label{decomposDiscret2}
 \P\left(\frac{1}{\hat{T}_{(d)}} \sum_{j=0}^{m-1}  |  k_T( x, X_{t_j}) - k_T(x,X_{T_{N_{t_{j+1}}}^-}) | \1_{\Delta N_{t_{j+1}}=1} >\epsilon \right ) \leq c_0 \frac{l_T \Delta_m^a}{v_T}.
\end{equation}

Let us now control the third term in \eqref{main-discrete}. From Proposition~\ref{lemMart} with $g=1$, we deduce that $\var(N_T) = \E \int_0^T \alpha(X_s)ds \leq \alpha^* T$, which along with $\E(N_T)\leq \alpha^* T$ gives $\P( N_T > 2 \alpha^*T) \leq c_0/T$. This result combined with $\P( \hat{T} <Tv_T/2) \leq c_0/Tv_T$ and \eqref{decomposDiscret3} entails

\begin{equation}\label{decomposDiscret4}
 \P(\frac{N_T}{\hat{T}_{(d)} \hat{T}} > \frac{16\alpha^*}{Tv_T^2} ) \leq  \frac{c_0 l_T \Delta_m^a}{v_T} + \frac{c_0 \Delta_m}{v_T} + \frac{c_0}{Tv_T} .
\end{equation}
From this inequality and  \eqref{discrete condition}, we deduce that 
  \begin{align}
   &\P(\frac{k^* N_T}{\hat{T}_{(d)} \hat{T}}\sum_{j=0}^{m-1}   \int_{t_j}^{t_{j+1}}  \left|k_T( x,X_{s}) -  k_T( x,X_{t_j})\right| ds \1_{\Delta N_{t_{j+1}} = 0} >\varepsilon ) \nonumber\\
   &\leq  \frac{c_0 l_T \Delta_m^a}{v_T} + \frac{c_0 \Delta_m}{v_T} + \frac{c_0}{Tv_T} + \P(\sum_{j=0}^{m-1}   \int_{t_j}^{t_{j+1}}  \left|k_T( x,X_{s}) -  k_T( x,X_{t_j})\right| ds \1_{\Delta N_{t_{j+1}} = 0} > \frac{\varepsilon T v_T^2}{16 \alpha^*} )\nonumber\\
   &\leq \frac{c_0 l_T \Delta_m^a}{v_T^2} + \frac{c_0 \Delta_m}{v_T} + \frac{c_0}{Tv_T} \label{decomposDiscret5}.
  \end{align}

 Finally, for the remaining term in \eqref{main-discrete}, we have
 \begin{align} 
   \P&\left( \frac{2(k^* N_T)^2\Delta_m}{\hat{T}_{(d)} \hat{T}} >\epsilon, \hat T_{(d)}\geq \frac{Tv_T} 4\right )\nonumber\\
  & \leq   \P\left( \frac{2(k^* N_T)^2\Delta_m}{ \hat{T}} >\frac{\epsilon Tv_T} 4, \hat T\geq \frac{Tv_T} 2\right ) + \P\left(\hat T< \frac{Tv_T} 2\right )\nonumber\\
  &\leq  \P\left( 2(k^* N_T)^2\Delta_m>\frac{\epsilon (Tv_T)^2} 8\right)  +\frac{c_0}{Tv_T}\nonumber\\
  &\leq \frac{16 (k^*)^2 \Delta_m}{\epsilon (Tv_T)^2} \E(N_T^2)  +\frac{c_0}{Tv_T}\nonumber\\
  &\leq  \frac{c_0\Delta_m}{v_T^2}+\frac{c_0}{Tv_T}\label{decomposDiscret6}
  \end{align}
where we have used the fact that $ \E(N_T^2)<c_0 T^2$ (that can be proven as a consequence of Proposition~\ref{lemMart}).  The proof is then complete by gathering \eqref{decomposDiscret1}, \eqref{decomposDiscret2}, \eqref{decomposDiscret5}, \eqref{decomposDiscret6}.

\section{Proofs concerning the supplementary material}\label{sec:proofs}

\subsection{Proof of Proposition~\ref{theo:MSE}}
We detail the proof for $\gamma=\alpha$, the other cases being similar. Using the same decomposition and the same notation as in the proof of Proposition~\ref{Maintheo}, we get for $T$ large enough
\begin{align}\label{firstineq}
\E((\hat{\alpha}_\xi (x) - \alpha(x))^2) &\leq 3\, \E\left(\frac{\1_{\hat T(x)>\xi}}{\hat T(x)^2} M_T^2(k)\right) + 3\,\E(R_T^2) + 3 (\alpha^*)^2  \P\left(\hat T(x) \leq \xi\right) \nonumber \\ 
 &\leq3\, \E\left(\frac{\1_{\hat T(x)>\xi}}{\hat T(x)^2} M_T^2(k)\right) + c_0\left( \frac{1}{T v_T(x)} + w_T^2(x)\right),
 \end{align}
where we  write $M_T(k)$ for $M_T$ to highlight the dependence in $k$. With the latter convention, let us introduce $\tilde{M}_T = M_T^2(k) - \int_0^T k_T^2(x,X_{s^-}) dN_s$ and set $I_T(k^2)= \int_0^T k_T^2(x,X_s) \alpha(X_s)ds$. We have
 \begin{align}\label{A123}
  \, \E\left(\frac{\1_{\hat T(x)>\xi}}{\hat T(x)^2} M_T^2(k)\right) &= \, \E\left(\frac{\1_{\hat T(x)>\xi}}{\hat T(x)^2} \tilde{M}_T\right) + \, \E\left(\frac{\1_{\hat T(x)>\xi}}{\hat T(x)^2} M_T(k^2)\right) + \, \E\left(\frac{\1_{\hat T(x)>\xi}}{\hat T(x)^2}  I_T(k^2)\right) \nonumber\\
 &=: A_1 + A_2 + A_3.
 \end{align}
 For the third term $A_3$, since $\alpha(x)\leq \alpha^*$,
\begin{equation*}
 A_3 \leq k^* \alpha^*\E\left(\frac{\1_{\hat T(x)>\xi}}{\hat T(x)}   \right) \leq \frac{k^* \alpha^*}{\E[\hat T(x)]} + k^* \alpha^* \E\left( \left|\frac{1}{\hat T(x)} - \frac{1}{\E[ \hat T(x)]} \right| \1_{\hat T(x)>\xi}\right) \leq \frac{k^* \alpha^*}{\E[ \hat T(x)]} +\frac{k^*\alpha^*}{\xi} \frac{\sqrt{\var(\hat T(x))}}{\E[ \hat T(x)]}.
\end{equation*}
So for $T$ large enough, by Corollary~\ref{moments T}, 
\begin{equation}\label{ineg3}
 A_3 \leq \frac{c_0}{\sqrt{Tv_T(x)}}.
\end{equation}
For the second term $A_2$,  we have
  \begin{align*}A_2=& \E\left[ \frac{M_T(k^2)}{\hat{T}(x)^2} \1_{\hat{T}(x) >\xi} \1_{I_T(k^2)>Tv_T(x)} + \1_{\hat{T}(x) >\xi} \1_{I_T(k^2) \leq Tv_T(x)} M_T(k^2)\left(\frac{1}{\hat{T}(x)} - \frac{1}{\E[ \hat{T}(x)]} + \frac{1}{\E[ \hat{T}(x)]} \right)^2\right]\\
 \leq &\E\left[ \frac{|M_T(k^2)|}{\hat{T}(x)^2} \1_{\hat{T}(x) >\xi} \1_{I_T(k^2)>Tv_T(x)}\right] + 2\E\left[ \frac{|M_T(k^2)|}{\E[ \hat{T}(x)]^2} \frac{(\hat{T}(x) - \E[\hat{T}(x) ] )^2}{\hat{T}(x)^2}\1_{\hat{T}(x) >\xi} \1_{I_T(k^2) \leq Tv_T(x)}\right]\\
  &+ 2\E\left[ \frac{| M_T(k^2)|}{\E[ \hat{T}(x)]^2} \1_{\hat{T}(x) >\xi} \1_{I_T(k^2) \leq Tv_T(x)}\right]\\
  =:&A_{2,1} + A_{2,2}+A_{2,3}.
 \end{align*}
 Let us control each term. First note that by Jensen inequality and Proposition~\ref{lemMart}, 
 \[\E[ |M_T(k^2)|]^2\leq \E[ |M_T(k^2)|^2] = \E\left( \int_0^T k_T^4( x, X_s)\alpha(X_s)ds\right)\leq c_0 \E[ \hat{T}(x)] \leq c_0 Tv_T(x).\]
Therefore, since $I_T(k^2) \leq k^*\alpha^* \hat{T}(x)$,
 \begin{align}
 A_{2,1} &=\E\left[\frac{|M_T(k^2)|}{I_T(k^2)^2} \frac{I_T(k^2)^2}{\hat{T}(x)^2} \1_{\hat{T}(x) >\xi} \1_{I_T(k^2)>Tv_T(x)}\right] \leq \frac{(k^*\alpha^*)^2}{(Tv_T(x)^2} \E[ |M_T(k^2)|]\leq \frac{c_0}{( Tv_T(x))^{\frac{3}{2}}} \label{ineg5}
\end{align}
and for $T$ large enough,
 \begin{align}
  A_{2,3} &\leq 2\frac{\E[ |M_T(k^2)|]}{\E[ \hat{T}(x)]^2} \leq \frac{c_0}{( Tv_T(x))^{\frac{3}{2}}}. \label{ineg4}
 \end{align}
 For the term $A_{2,2}$, let $\theta \in (\frac{1}{2},1)$ and consider the decomposition
\begin{align}
 A_{2,2} =& \frac{2}{\E[ \hat{T}(x)]^2}\E\left[|M_T(k^2)| \frac{(\hat{T}(x) - \E[\hat{T}(x) ] )^2}{\hat{T}(x)^2}\1_{\hat{T}(x) >\xi} \1_{I_T(k^2) \leq Tv_T(x)}\1_{|M_T(k^2)| \leq (Tv_T(x) )^{\theta} \hat{T}(x)^2}\right] \nonumber \\
 &+ \frac{2}{\E[ \hat{T}(x)]^2} \E\left[|M_T(k^2)| \frac{(\hat{T}(x) - \E[\hat{T}(x) ] )^2}{\hat{T}(x)^2}\1_{\hat{T}(x) >\xi} \1_{I_T(k^2) \leq Tv_T(x)}\1_{|M_T(k^2)| > (Tv_T(x) )^{\theta} \hat{T}(x)^2}\right] \nonumber \\
 \leq & \frac{2(Tv_T(x) )^{\theta} \var(\hat{T}(x))}{\E[ \hat{T}(x)]^2} + c_0\E[|M_T(k^2)|\1_{I_T(k^2) \leq Tv_T(x)}\1_{|M_T(k^2)| > (Tv_T(x) )^{\theta} \xi^2}] \nonumber\\
\leq & c_0\left( ( Tv_T(x))^{\theta-1} + \E[|M_T(k^2)|\1_{I_T(k^2) \leq Tv_T(x)}\1_{|M_T(k^2)| > (Tv_T(x) )^{\theta} \xi^2}] \right). \label{ineg6}
 \end{align}
Letting $Z=M_T(k^2)\1_{I_T(k^2) \leq Tv_T(x)}$, 
\begin{multline}
 \E[|M_T(k^2)|\1_{I_T(k^2) \leq Tv_T(x)}\1_{|M_T(k^2)| > (Tv_T(x) )^{\theta} \xi^2}] 
  = \E[ |Z| \1_{|Z| > (Tv_T(x) )^{\theta} \xi^2}]\\
  = \int_{( Tv_T(x))^{\theta}\xi^2}^{+\infty} \P(|Z|\geq u) du + ( Tv_T(x))^{\theta} \xi^2 \P( |Z| > ( Tv_T(x))^{\theta}\xi^2). \label{ineg7}
\end{multline}
 Since $Z= \int_0^T k_T^2( x,X_{s^-})\1_{I_T(k^2) \leq Tv_T(x)} [ dN_s - \alpha(X_s)ds],$ Theorem 1 of \cite{leguevel} with $H_s =   k_T^2( x,X_{s^-})\1_{I_T(k^2) \leq Tv_T(x)},$ 
 $\|H\|_{\infty} = (k^*)^2$ and $\|H\|_2^2 = (k^*)^2 Tv_T(x)$ leads for $u \geq 0$ to
  \begin{equation}\label{inegexp1}
  \P(| Z| \geq u ) \leq 2\exp\left( - \frac{Tv_T(x)}{(k^*)^2} I\left(\frac{u}{Tv_T(x)} \right)\right)
 \end{equation}
where $I(u) = (1+u)\log(1+u)-u.$ Using the fact that $I(u) \sim_0 u^2/2$, (\ref{inegexp1}) leads to
\begin{align}
( Tv_T(x))^{\theta} \xi^2 \P( |Z| \geq ( Tv_T(x))^{\theta}\xi^2) &\leq 2(Tv_T(x) )^{\theta}\xi^2 \exp \left( - \frac{Tv_T(x)}{(k^*)^2} I\left(\frac{\xi^2}{(Tv_T(x))^{1-\theta}} \right)\right) \leq \frac{c_0}{Tv_T(x)}.\label{ineg9}
\end{align} 
Moreover, thanks again to (\ref{inegexp1}) we have
 \begin{align*}
  \int_{( Tv_T(x))^{\theta}\xi^2}^{+\infty} \P(| Z| \geq u )du &\leq 2\int_{( Tv_T(x))^{\theta}\xi^2}^{+\infty} \exp\left( - \frac{Tv_T(x)}{(k^*)^2} I\left(\frac{u}{Tv_T(x)} \right)\right) du
 \end{align*}
and by the inequality $I(u) \geq (u^2/3)\1_{0\leq u<1} + (u/3)\1_{u \geq 1}$ we get
 \begin{align}
   \int_{( Tv_T(x))^{\theta}\xi^2}^{+\infty}  \P(| Z| \geq u )du  &\leq 2  \int_{( Tv_T(x))^{\theta}\xi^2}^{Tv_T(x)} \exp \left( - \frac{u^2}{3(k^*)^2 Tv_T(x)}\right)du + 2 \int_{Tv_T(x)}^{+\infty} \exp\left(-\frac{u}{3(k^*)^2}\right) du \nonumber \\
   &\leq c_0\left( \sqrt{T v_T(x)} \int_{( Tv_T(x))^{2\theta-1}\xi^4}^{Tv_T(x)} e^{-\frac{\omega}{3(k^*)^2}} \frac{d\omega}{\sqrt{\omega}} + e^{-\frac{Tv_T(x)}{3(k^*)^2}}\right) \nonumber\\
   &\leq c_0\left((T v_T(x))^{1-\theta} e^{-\frac{\xi^4(Tv_T(x))^{2\theta-1}}{3(k^*)^2}} + e^{-\frac{Tv_T(x)}{3(k^*)^2}}\right)
   \leq \frac{c_0}{Tv_T(x)}. \label{ineg10}
 \end{align}
Gathering $(\ref{ineg6})$,  $(\ref{ineg7})$,  $(\ref{ineg9})$ and  $(\ref{ineg10}),$ we obtain $A_{2,2} \leq c_0(Tv_T(x))^{\theta-1}$, which combined with   $(\ref{ineg5})$ and $(\ref{ineg4})$  provides
\begin{equation}\label{ineg11}
|A_2| \leq \frac{c_0}{(Tv_T(x))^{1-\theta}}.
\end{equation}

It remains to control  the term $A_1$ in \eqref{A123}. Using exactly the same decomposition and arguments than for $A_2$, along with  the inequality $\E[|\tilde{M}_T|] \leq c_0 \E[\hat{T}(x)]$, we obtain for $\theta \in (0,1)$ and $\rho \in (\frac{1}{2},1)$ such that $\theta/2 <\rho <\theta$ 
  \begin{equation}\label{ineg12}
  \E[|A_1] \leq \frac{c_0}{(Tv_T(x))^{2 \rho -1}} +  \frac{c_0}{(Tv_T(x))^{1-\theta}} +c_0 \E\left[|\tilde{M}_T \1_{I_T(k^2) \leq (Tv_T(x))^{\rho}}\1_{|\tilde{M}|>(Tv_T(x))^{\theta} \xi^2} \right].
 \end{equation}
To control the last term in \eqref{ineg12}, we can use \eqref{ineg7} where now $Z=\tilde{M}_T \1_{I_T(k^2) \leq (Tv_T(x))^{\rho}}$. 
Theorem 2 of \cite{leguevel} with $H_s =   k_T( x,X_{s^-})\1_{I_T(k^2) \leq (Tv_T(x))^{\rho}},$ 
 $\|H\|_{\infty} = k^*$ and $\|H\|_2^2 = (Tv_T(x))^{\rho}$ leads for $u \geq 0$ to
  \begin{equation}\label{inegexp2}
  \P(|Z| \geq u ) \leq 6\exp\left( - \frac{(Tv_T(x))^{\rho}}{(k^*)^2} I\left(\frac{k^*}{(Tv_T(x))^{\rho}}\sqrt{\frac{u}{2}} \right)\right),
 \end{equation}
which gives for $u=( Tv_T(x))^{\theta}\xi^2$, since $I(u) \sim_0 \frac{u^2}{2}$ and $\theta/2<\rho<\theta$,
\begin{align}
( Tv_T(x))^{\theta} \xi^2 \P\left( |Z| \geq ( Tv_T(x))^{\theta}\xi^2\right)  \leq \frac{c_0}{Tv_T(x)}.\label{ineg15}
\end{align} 
Moreover, using again  \eqref{inegexp2} and the same arguments as to get \eqref{ineg10}, we obtain 
 \begin{align*}
  \int_{( Tv_T(x))^{\theta}\xi^2}^{+\infty} \P(|Z| \geq u )du &\leq 6\int_{( Tv_T(x))^{\theta}\xi^2}^{+\infty} \exp\left( - \frac{(Tv_T(x))^{\rho}}{(k^*)^2} I\left(\frac{k^*}{(Tv_T(x))^{\rho}}\sqrt{\frac{u}{2}} \right)\right) du\leq \frac{c_0}{Tv_T(x)}. 
 \end{align*}
Coming back to \eqref{ineg12}, we deduce that 
\begin{equation}\label{ineg17}
|A_1| \leq \frac{c_0}{(Tv_T(x))^{2 \rho-1}} + \frac{c_0}{(Tv_T(x))^{1-\theta}}.
\end{equation}
 The best rate is clearly achieve for the highest values of $\rho$, or writing $\rho=\theta-\eta/2$ with  $0<\eta<2\theta-1$, for the smallest values of $\eta$. The best trade-off in $\theta$ between $2\rho-1=2 \theta-1-\eta$ and $1-\theta$ is  obtained for $\theta=2/3$. This result combined with \eqref{ineg3} and \eqref{ineg11} in \eqref{firstineq} concludes the proof.

\subsection{Proof of Theorem \ref{Markov}}

Our aim is to prove that $\left( (X_t)_{t\geq 0}, (\P_x)_{x \in E}\right)$ is a time-homogeneous Markov process with respect to $(\F_t)_{t \geq 0}.$ We know from \cite[Proposition 20.2]{Bass} that it is enough to show that it is a time-homogeneous Markov process with respect to $(\F_t^{0})_{t \geq 0}.$ We therefore want to prove that for  $0 \leq s \leq t,$ $A \in \mathcal{E} $  and $x \in E,$ 
\begin{equation*}\label{markov1}
\P_x(X_{t} \in A |\F^{0}_s ) = \P_{X_s}(X_{t-s} \in A), \quad \P_x \textrm{ a.s.} 
\end{equation*}

Let $T^{(s)}$ be the next jump after $s$, i.e. $T^{(s)} = \inf_{T_j >s} \{T_j\}$. We use the decomposition 
\begin{equation}\label{decompmarkov}\P_x(X_t \in A |\F^{0}_s ) = \E_x[\1_{X_t \in A} \1_{T^{(s)} \leq t} |\F^{0}_s ]+ \E_x[\1_{X_t \in A} \1_{T^{(s)} >t} |\F^{0}_s ].\end{equation}

We begin by controlling  the first term. Since the event $\{T_j>s\}=\{N_s\leq j-1\}$ is $\F^{0}_s$-measurable,
\begin{align*}
 \E_x[\1_{X_t \in A} \1_{T^{(s)} \leq t} |\F^{0}_s ] &= \sum_{j\geq 1} \E_x[\1_{X_t \in A} \1_{T_{j} \leq t} |\F^{0}_s ]\1_{T_{j-1}  \leq s < T_j}\\
&=\sum_{j\geq 1} \E_x[ \E_x[\1_{X_t \in A} |\F^{0}_{T_j}] \1_{T_{j} \leq t} |\F^{0}_s ]\1_{T_{j-1}  \leq s < T_j}.
\end{align*}

By construction of the process $(X_t)_{t \geq 0}$, given $\F^{0}_{T_j},$ the process $(X_t)_{t \geq T_j}$ has the same distribution as the process $(X_{t-T_j})_{t \geq T_j}$ given $X_0=X_{T_j},$ so for $t \geq T_j,$  $\E_x[\1_{X_t \in A} |\F^{0}_{T_j}]$ is a function of $X_{T_j}$, $t-T_j$ and $A$. We thus may write $\E_x[\1_{X_t \in A}|\F^{0}_{T_j}] = h(X_{T_j},t-T_j,A)$ where $h$ is a borelian function, so that
\begin{equation}\label{markov2} \E_x[\1_{X_t \in A} \1_{T^{(s)} \leq t} |\F^{0}_{s} ] = \sum_{j\geq 1} \E_x[ h(X_{T_j},t-T_j,A) \1_{T_{j} \leq t} |\F^{0}_s]\1_{T_{j-1}  \leq s <T_j}.\end{equation}

We have
\begin{align}\label{equation h}
\E_x[ h(X_{T_j},t-T_j,A)& \1_{T_{j} \leq t} |\F^{0}_s]\1_{T_{j-1}  \leq s <T_j}\nonumber\\
&= \E_x[ \E_x[ h(X_{T_j},t-T_j,A) \1_{T_{j} \leq t} |\F^{0}_s,(Y_u^{(j-1)} )_{u\geq 0},\tau_j] |\F^{0}_s ]\1_{T_{j-1}  \leq s <T_j}\nonumber\\
&= \E_x[ \int_{\Et} h(z,t-T_j,A) K(Y_{\tau_{j}}^{(j-1)},dz) \1_{ T_{j} \leq t} |\F^{0}_s ]\1_{T_{j-1}  \leq s <T_j}\nonumber\\
&= \E_x[  g( Y_{\tau_j}^{(j-1)}, t-T_j , A) \1_{T_{j} \leq t} |\F^{0}_s]\1_{T_{j-1}  \leq s <T_j},
\end{align}
where $g(y, t , A) = \int_{\Et} h(z,t,A) K(y,dz).$

In the sequel, for a $\sigma$-field $\mathcal{G}$ and an event $B$ satisfying $\P_x(B) >0,$ $\E_x[X | \mathcal{G},B]$ stands for the conditional expectation of $X$ given $\mathcal{G},$ under the conditional probability measure $\P_{x|B}(.) = \P_x(B \cap . )/\P_x(B)$, that is $ \E_x[X | \mathcal{G},B] = \E_{x|B}[X | \mathcal{G}]$. Specifically, $Z= \E_{x|B}[X | \mathcal{G}]$ if and only if $Z$ is $\mathcal{G}$-measurable, integrable and for any $\mathcal{G}$-measurable bounded random variable $U$, $\E_x(XU\1_B)=\E_x(ZU\1_B)$. In particular, for any events $B$ and $C$,
\begin{equation}\label{conditionne2}
 \P_x(B \cap C | \mathcal{G} ) = \P_x(C | \mathcal{G},B) \P_x( B|\mathcal{G}),
\end{equation}
and if $B$ is $\mathcal{G}$-measurable, this gives $\P_x(B \cap C | \mathcal{G} ) = \P_x(C | \mathcal{G},B)\1_B$.
Accordingly, for $v \geq 0$, since  $\{T_j>s\}$ belongs to $\F^{0}_s$, we get

\begin{align}
\P_x(\tau_j >v | \F^{0}_s,\F_{T_{j-1}},Y^{(j-1)}) \1_{T_{j-1}  \leq s <T_j}
 &= \P_x\left(\tau_j >v|\F_{T_{j-1}},Y^{(j-1)}, T_{j-1}  \leq s < T_j \right) \1_{T_{j-1}  \leq s <T_j}.\label{conditionne3}
\end{align}

This probability can be computed using \eqref{conditionne2} 
 and \eqref{cdf tau}, so that \eqref{conditionne3} simplifies into
 \begin{align}\label{resTimeTauj}
 &\P_x(\tau_j >v | \F^{0}_s,\F_{T_{j-1}},Y^{(j-1)}) \1_{T_{j-1}  \leq s <T_j} \nonumber \\& = \left(\1_{v \leq s-T_{j-1}} + \1_{v>s-T_{j-1}}  e^{-\int_{s-T_{j-1}}^{v}\alpha(Y_u^{(j-1)}) du}\right)\1_{T_{j-1}  \leq s <T_j}.
\end{align}
By continuity of $(Y_u^{(j-1)})_{u \geq 0}$ and $\alpha$, this proves that \[B \mapsto \E_x[\1_{\tau_j \in B}|\F^{0}_s,\F_{T_{j-1}},Y^{(j-1)}] \1_{T_{j-1}  \leq s <T_j}\] defines a measure with density $v \mapsto \alpha(Y_v^{(j-1)}) e^{-\int_{s-T_{j-1}}^{v}\alpha(Y_u^{(j-1)}) du}\1_{v>s-T_{j-1}}\1_{T_{j-1}  \leq s <T_j}.$
Therefore,
\begin{align*}
 &\E_x[  g( Y_{\tau_j}^{(j-1)},  t-T_j  , A) \1_{T_{j} \leq t} |\F^{0}_s,\F_{T_{j-1}},Y^{(j-1)}]\1_{T_{j-1}  \leq s <T_j}\\
 &= \int_{s-T_{j-1}}^{t-T_{j-1}} g( Y_{v}^{(j-1)}, t-T_{j-1} - v , A) \alpha(Y_v^{(j-1)}) e^{-\int_{s-T_{j-1}}^{v}\alpha(Y_u^{(j-1)}) du} dv\1_{T_{j-1}  \leq s <T_j}\\
 &=  \int_{s}^{t} g( Y_{v-T_{j-1}}^{(j-1)}, t - v, A) \alpha(Y_{v-T_{j-1}}^{(j-1)}) e^{-\int_{s-T_{j-1}}^{v-T_{j-1}}\alpha(Y_u^{(j-1)}) du} dv\1_{T_{j-1}  \leq s <T_j}.
\end{align*}
Coming back to \eqref{equation h}, this gives
\begin{align}
& \E_x[ h(X_{T_j},t-T_j,A) \1_{T_{j} \leq t} |\F^{0}_s ]\1_{T_{j-1}  \leq s <T_j}\nonumber\\
&= \E_x[\int_{s}^{t} g( Y_{v-T_{j-1}}^{(j-1)}, t-v , A) \alpha(Y_{v-T_{j-1}}^{(j-1)}) e^{-\int_{s-T_{j-1}}^{v-T_{j-1}}\alpha(Y_u^{(j-1)}) du} dv | \F^{0}_s]\1_{T_{j-1}  \leq s <T_j}.\label{intermed1}
\end{align}
Let us now prove  that the latter conditional expectation is equal to $H(X_s) \1_{T_{j-1}\leq s < T_j}$, $\P_x$-almost surely, where
 \[H(z) = \E_{z}[ \int_{0}^{t-s} g( Y_{v}, t-v-s , A) \alpha(Y_v) e^{-\int_{0}^{v}\alpha(Y_{u}) du} dv ].\]
 Denoting $Z = \int_{s}^{t} g( Y_{v-T_{j-1}}^{(j-1)}, t-v , A) \alpha(Y_{v-T_{j-1}}^{(j-1)}) e^{-\int_{s-T_{j-1}}^{v-T_{j-1}}\alpha(Y_u^{(j-1)}) du} dv$, this amounts to prove that
\begin{equation}\label{monotone} \E_x[ Z \1_{T_{j-1}  \leq s <T_j} | \F^{0}_s] = \E_x[H(X_s) \1_{T_{j-1}\leq s < T_j} | \F^{0}_s]\end{equation}
because $H(X_s) \1_{T_{j-1}\leq s < T_j}$ is $\F^{0}_s$-mesurable. Since $\F^{0}_s$ is generated by functions of the form 
$\prod_{i=1}^n f_i(X_{s_i})$ for $n\geq 0$, $0\leq s_1< s_2 < ... < s_n \leq s$ and measurable functions $f_1,\dots,f_n$, it is sufficient to prove that
\begin{align*}
 \E_x[ Z  \1_{T_{j-1}\leq s < T_j} \prod_{i=1}^n f_i(X_{s_i})]& =  \E_x[ H(X_s) \1_{T_{j-1}\leq s < T_j} \prod_{i=1}^{n} f_i(X_{s_i}) ].
\end{align*}
Let $B_{n_0}^j=\{s_{n_0}< T_{j-1} \leq s_{n_0+1}\}$, for $n_0=0,\dots,n$ with the convention $s_0=0$ and $s_{n+1}=s$. We have
\begin{align}
 \E_x[ Z\1_{T_{j-1}\leq s < T_j} \prod_{i=1}^n f_i(X_{s_i})]& = \sum_{n_0=0}^n \E_x[ Z\1_{T_{j-1}\leq s < T_j} \1_{B_{n_0}^j}\prod_{i=1}^{n_0} f_i(X_{s_i}) \prod_{i=n_0+1}^{n} f_i( Y^{(j-1)}_{s_i-T_{j-1}}) ]\nonumber\\
 &= \sum_{n_0=0}^n \E_x[\prod_{i=1}^{n_0} f_i(X_{s_i})  \1_{T_{j-1}\leq s}  \1_{B_{n_0}^j}  \E_x[Z\1_{T_j >s}\prod_{i=n_0+1}^{n} f_i( Y^{(j-1)}_{s_i-T_{j-1}}) |\F_{T_{j-1}} ]]\label{espcondition}.
\end{align}
Conditioning by $Y^{(j-1)},$ we obtain
\[\E_x[Z\1_{T_j >s}\prod_{i=n_0+1}^{n} f_i( Y^{(j-1)}_{s_i-T_{j-1}}) |\F_{T_{j-1}},Y^{(j-1)} ] = Z e^{-\int_{0}^{s-T_{j-1}} \alpha( Y^{(j-1)}_u)du}\prod_{i=n_0+1}^{n} f_i( Y^{(j-1)}_{s_i-T_{j-1}})\]
and 
\[ \E_x[Z\1_{T_j >s}\prod_{i=n_0+1}^{n} f_i( Y^{(j-1)}_{s_i-T_{j-1}}) |\F_{T_{j-1}} ] = \E_x[Z e^{-\int_{0}^{s-T_{j-1}} \alpha( Y^{(j-1)}_u)du}\prod_{i=n_0+1}^{n} f_i( Y^{(j-1)}_{s_i-T_{j-1}}) | \F_{T_{j-1}}  ].\]
Let us write 
\begin{equation}\label{eqU}
U = e^{-\int_{0}^{s-T_{j-1}} \alpha( Y^{(j-1)}_u)du}\prod_{i=n_0+1}^{n} f_i( Y^{(j-1)}_{s_i-T_{j-1}}) = \E_x[ \1_{T_j >s} \prod_{i=n_0+1}^{n} f_i( Y^{(j-1)}_{s_i-T_{j-1}}) | \F_{T_{j-1}}, Y^{(j-1)} ] .
\end{equation}
Conditional on $\F_{T_{j-1}} ,$ the distribution of $Y^{(j-1)}$ is independent of $T_{j-1}$, so using the Markov property of $Y^{(j-1)},$ we then obtain
\begin{align*}
 \E_x[Z\1_{T_j >s} & \prod_{i=n_0+1}^{n} f_i( Y^{(j-1)}_{s_i-T_{j-1}}) |\F_{T_{j-1}} ]\\
&= \E_x[Z U  |\F_{T_{j-1}} ]\\
&= \E_x[ \E_{Y^{(j-1)}_{s-T_{j-1}}}\left( \int_{s}^{t} g( Y_{v-s}^{(j-1)}, t-v , A) \alpha(Y_{v-s}^{(j-1)}) e^{-\int_{s-T_{j-1}}^{v-T_{j-1}}\alpha(Y_{u-s+T_{j-1}}^{(j-1)}) du} dv \right) U |\F_{T_{j-1}}]\\
&=\E_x[ \E_{Y^{(j-1)}_{s-T_{j-1}}}\left(\int_{0}^{t-s} g( Y_{v}^{(j-1)}, t-v-s , A) \alpha(Y_v^{(j-1)}) e^{-\int_{0}^{v}\alpha(Y_{u}^{(j-1)}) du} dv \right) U |\F_{T_{j-1}}]\\
&= \E_x[ H( Y^{(j-1)}_{s-T_{j-1}}) U |\F_{T_{j-1}}]\\
&= \E_x[ H( Y^{(j-1)}_{s-T_{j-1}})  \1_{T_j >s}\prod_{i=n_0+1}^{n} f_i( Y^{(j-1)}_{s_i-T_{j-1}}) |\F_{T_{j-1}}],
\end{align*}
where we have used the fact that the distribution of $Y^{(j)}$ does not depend on $j$ and used \eqref{eqU} back. 
We finally obtain from \eqref{espcondition}
\begin{align*}
 \E_x[ Z \1_{T_{j-1}\leq s < T_j} &\prod_{i=1}^n f_i(X_{s_i})] \\
 &= \sum_{n_0=0}^n \E_x[\prod_{i=1}^{n_0} f_i(X_{s_i})  \1_{T_{j-1}\leq s}  \1_{B_{n_0}^j} \E_x[ H( Y^{(j-1)}_{s-T_{j-1}})  \1_{T_j >s}\prod_{i=n_0+1}^{n} f_i( Y^{(j-1)}_{s_i-T_{j-1}}) |\F_{T_{j-1}}]]\\
 &= \E_x[ H(X_s) \1_{T_{j-1}\leq s < T_j} \prod_{i=1}^{n} f_i(X_{s_i}) ],
 \end{align*}
proving \eqref{monotone}.
We then obtain with \eqref{intermed1}
\[ \E_x[ h(X_{T_j},t-T_j,A) \1_{T_{j} \leq t} |\F^{0}_s ]\1_{T_{j-1}  \leq s <T_j} = H(X_s) \1_{T_{j-1}\leq s < T_j}\quad \P_x\textrm{-a.s.},\]
and  from \eqref{markov2},
\begin{align}\label{firsttermdecompMarkov}
\E[\1_{X_t \in A} \1_{T^{(s)} \leq t} |\F_s^0 ] = H(X_s) =  \E_{X_s}[ \int_{0}^{t-s} g( Y_{v}, t-v-s , A) \alpha(Y_v) e^{-\int_{0}^{v}\alpha(Y_{u}) du} dv ] \quad \P_x\textrm{-a.s.}
\end{align}

For the second term in \eqref{decompmarkov},
\begin{align*}
 \E_x[\1_{X_t \in A} \1_{T^{(s)} >t} |\F_s^0 ] &= \sum_{j\geq 1} \E_x[\1_{Y_{t-T_{j-1}}^{(j-1)} \in A} \1_{T^{(s)} >t}|\F_s^0] \1_{T_{j-1}  \leq s <T_j}\\
  &= \sum_{j\geq 1} \E_x[\P_x(T_j >t| \F^{0}_s,\F_{T_{j-1}},Y^{(j-1)}) \1_{Y_{t-T_{j-1}}^{(j-1)} \in A} |\F_s^0]\1_{T_{j-1}  \leq s <T_j}\\
  &= \sum_{j\geq 1} \E_x[e^{-\int_{s}^{t} \alpha(Y_{v-T_{j-1}}^{(j-1)}) dv} \1_{Y_{t-T_{j-1}}^{(j-1)} \in A} |\F_s^0]\1_{T_{j-1}  \leq s <T_j}
   \end{align*}
with \eqref{resTimeTauj}. Using the same arguments as for the first term in  \eqref{decompmarkov}, we obtain
  \begin{align}
 \E_x[\1_{X_t \in A} \1_{T^{(s)} >t} |\F_s^0 ] 
  &= \sum_{j\geq 1} \E_{X_s}[e^{-\int_{0}^{t-s} \alpha(Y_{v}) dv} \1_{Y_{t-s} \in A}] \1_{T_{j-1}  \leq s <T_j} \quad \P_x\textrm{-a.s.}\nonumber\\
  &=\E_{X_s}[e^{-\int_{0}^{t-s} \alpha(Y_{v}) dv} \1_{Y_{t-s} \in A}] \quad \P_x\textrm{-a.s.}\label{secondtermdecompMarkov}
 \end{align}
 The two expressions \eqref{firsttermdecompMarkov} and \eqref{secondtermdecompMarkov} then imply
 \[ \P_x( X_t \in A | \F_s^0) = \E_{X_s}[ \int_{0}^{t-s} g( Y_{v}, t-v-s , A) \alpha(Y_v) e^{-\int_{0}^{v}\alpha(Y_{u}) du} dv ] + \E_{X_s}[e^{-\int_{0}^{t-s} \alpha(Y_{v}) dv} \1_{Y_{t-s} \in A}]\quad \P_x\textrm{-a.s.}\]
For $s=0,$ this leads to
 \[ \P_x( X_t \in A | \F_0^0) = \E_{Y^{(0)}_0}[ \int_{0}^{t} g( Y_{v}, t-v , A) \alpha(Y_v) e^{-\int_{0}^{v}\alpha(Y_{u}) du} dv ] + \E_{Y^{(0)}_0}[e^{-\int_{0}^{t} \alpha(Y_{v}) dv} \1_{Y_{t} \in A}]\quad \P_x\textrm{-a.s.}\]
 Since $((Y_t^{(0)})_{t \geq 0}, (\P_x)_{x \in E})$ is a Markov process, $\P_x(Y_0^{(0)}=x)=1$, so 
 \[ \P_x( X_t \in A ) =  \E_{x}[ \int_{0}^{t} g( Y_{v}, t-v , A) \alpha(Y_v) e^{-\int_{0}^{v}\alpha(Y_{u}) du} dv ] + \E_{x}[e^{-\int_{0}^{t} \alpha(Y_{v}) dv} \1_{Y_{t} \in A}]\]
 and 
 \[ \P_x(X_t \in A |\F_{s}^0) = \P_{X_s}(X_{t-s}\in A ) \quad \P_x\textrm{-a.s.},\]
which concludes the proof that $((X_t)_{t \geq 0},(\P_x)_{x \in E})$ is a time-homogeneous Markov process with respect to $(\F_t)_{t\geq 0}.$

\subsection{Proof of Proposition~\ref{InvMeas}}

To prove that $(X_t)_{t\geq 0}$ admits an invariant measure, we can view the process as a classical regenerative process with regeneration times $\{t, X_t=\varnothing\}$, see for instance \cite[Chapter 10]{thorisson2000} for a definition. Under \eqref{existence}, it is not difficult to verify that the expected time between two regenerations is finite, which implies the existence of an invariant measure \cite[Chapter 10, Theorem 3.1]{thorisson2000}.

It remains to establish the uniform geometric ergodicity \eqref{H2 sature}. The proof uses a standard coupling argument as carried out for pure spatial birth-death processes in  \cite[Theorem~A]{lotwick1981} and \cite[Theorem 3.1 and Corollary 3.1]{moller1989}. Let $(X_t^{(1)})_{t\geq 0}$ and $(X_t^{(2)})_{t\geq 0}$ be two birth-death-move processes with the same transition kernel $Q_t$ as $(X_t)_{t\geq 0}$ and with respective initial distribution $\phi_1$ and $\phi_2$ on $\Et$. Consider the stopping time 
\[\tau = \inf \{ t >0 : (X_t^{(1)},X_t^{(2)}) = (\varnothing,\varnothing) \}.\]
We verify in Lemma~\ref{Lemmaexpboundtau} in Appendix~\ref{sec:appendixB} that $\tau<\infty$ $\P_{\phi_1\times \phi_2}$-almost surely and we get from Lemma~\ref{Lemmanumero2} that for any  $A \in \mathcal E$ and any $t \geq 0,$
 \[\P_{\phi_1 \times \phi_2}(X_t^{(1)} \in A, \tau \leq t) = \P_{\phi_1 \times \phi_2}(X_t^{(2)} \in A, \tau \leq t).\]
Therefore by the coupling argument, 
 \begin{equation}\label{couplingkey}\left|\int_\Et Q_t(x,A)\phi_1(dx) -\int_\Et Q_t(x,A)\phi_2(dx)\right|\leq 2P_{\phi_1\times\phi_2}(\tau>t).\end{equation}
By Lemma~\ref{Lemmaexpboundtau}, we further get
\[\P_{\phi_1\times\phi_2}(\tau>t)\leq a e^{-ct}\]
for some $a>0$ and $c>0$ that do not depend on $\phi_1$ and $\phi_2$. This inequality for the choice $\phi_1(.)=\1_{y\in .}$ and $\phi_2=\mu_\infty$ implies \eqref{H2 sature} when $g=\1_A$, for some $A\in\mathcal E$.
To extend it to any measurable bounded function $g$, first consider the case where $g$ takes its values in $[0,1]$. Then $g$ can be approximated by the step function $g_n(z)=2^{-n}\lfloor 2^n g(z)\rfloor$, where $\lfloor . \rfloor$ denotes the integer part function, so that $\|g-g_n\|_\infty\leq 2^{-n}$. Let $A_{j,n}=\{j2^{-n}\leq g(z)<(j+1)2^{-n}\}$.  Using \eqref{H2 sature} for indicator functions, we deduce that for any $y\in E$ and any $n$
\begin{align*}
\left | \int_\Et g_n(z) Q_t(y,dz) - \int_\Et g_n(z)\mu_{\infty}(dz)\right | =2^{-n}\sum_{j=0}^{2^n} j \left| Q_t(y,A_{j,n}) -  \mu_{\infty}(A_{j,n})\right|\leq 2^n a  e^{-c t}.
\end{align*}
Therefore for any $y\in E$ and any $n$
\begin{align*}
\left | \int_\Et g(z) Q_t(y,dz) - \int_\Et g(z)\mu_{\infty}(dz)\right |\leq 2 \|g-g_n\|_\infty +  2^n a  e^{-c t} \leq  2^{-n+1} + 2^n a e^{-c t}.
\end{align*}
Choosing $n=\lfloor c t / (2\log 2)\rfloor$, we get \eqref{H2 sature} (for new constants $a>0$ and $c>0$) in the case where $g$ takes its values in $[0,1]$. Applying this result to $g/\|g\|_\infty$ proves  \eqref{H2 sature} for positive bounded functions $g$. The extension to  any measurable bounded function $g$ is obtained by considering the decomposition $g=g_+-g_-$ where $g_+$ and $g_-$ respectively denote the positive and negative part of $g$.

\subsection{Proof of Corollary~\ref{lemVar}}
 
To prove the first statement of Corollary~\ref{lemVar}, note that 
$$\E \left(\int_0^t g(X_s)ds\right)  - t \int_E g(z)\mu_{\infty}(dz) =\int_0^t \int_E \left(\int_E g(z)Q_s(y,dz) - \int_E g(z)\mu_{\infty}(dz)\right)\mu_0(dy)ds$$
where $\mu_0$ denotes the distribution of $X_0$. Using Proposition~\ref{InvMeas}, we get 
\begin{align*}
\left|\E \left(\int_0^t g(X_s)ds\right)  - t \int_E g(z)\mu_{\infty}(dz)\right| &\leq \int_0^t  \sup_{y\in E} \left|\int_E g(z)Q_s(y,dz) - \int_E g(z)\mu_{\infty}(dz)\right|ds\\
&\leq a\|g\|_\infty \int_0^t e^{-c s}ds
\end{align*}
hence the result. Next, in order to prove the second statement, let us write $E_s$ for
   \[\Et_s:=\E[g(X_s)] = \int_{\Et} \int_{\Et} g(y) Q_s(x,dy) \mu_0(dx).\]
For any $v \geq s$ we have $\E[g(X_v) | \F_s] = \int_{\Et} g(z) Q_{v-s}(X_s,dz)$, so 
  \begin{align*}
    &\var \left( \int_0^t g(X_s) ds \right)  =  2 \int_{s=0}^t \int_{v=s}^t \E \left[ (g(X_s) - E_s)(\int_{\Et} g(z) Q_{v-s}(X_s,dz)-E_v)\right] dv ds \\
    & =  2 \int_{s=0}^t \int_{v=0}^{t-s} \iint_{\Et}  (g(y) - E_s)\left(\int_{\Et} g(z) Q_v(y,dz)-E_{v+s}\right) Q_s(x,dy)\mu_0(dx) dv ds \\
    & =  2 \int_{s=0}^t \iint_{\Et}  (g(y) - E_s)\left[\int_{v=0}^{t-s} \left(\int_{\Et} g(z) Q_v(y,dz)- \int_{z \in \Et} g(z) \mu_{\infty}(dz)\right) dv \right] Q_s(x,dy)\mu_0(dx) ds\\
    &    + 2 \int_{s=0}^t \iint_{\Et}  (g(y) - E_s)\left[ \int_{v=s}^{t} \left(\int_{z \in \Et} g(z) \mu_{\infty}(dz)  - \iint_{\Et} g(z) Q_{v}(u,dz) \mu_0(du) \right)dv\right] Q_s(x,dy)\mu_0(dx) ds.
     \end{align*}

  Thanks to Proposition~\ref{InvMeas}, each term in the square brackets above is uniformly bounded in $s$, $t$ and $y$, so there exists a positive constant $c_0$ such that 
   \begin{align*}
     \var \left( \int_0^t g(X_s) ds \right) & \leq c_0 \|g\|_\infty \int_{s=0}^t \iint_{\Et}  (g(y) + E_s ) Q_s(x,dy)\mu_0(dx) ds=2c_0 \|g\|_\infty\int_{s=0}^t E_s ds.
   \end{align*}

\subsection{Proof of Proposition~\ref{lemMart}}

In this proof we denote by $N_t^\beta$ and $N_t^\delta$ the number of births and of deaths before $t$. In order to encompass all cases in the same proof, we consider $N^\gamma_t$ where $\gamma$ is either $\beta$, $\delta$ or $\alpha$ and $N_t^\gamma=N_t$ when $\gamma=\alpha$.
 We first show the result for the filtration $\F_t$ instead of the filtration $\F_{t^+}.$  
The counting process $N^\gamma_t$ is clearly adapted to $\F_t$ and its $\F_t$-intensity is obtained by $\lambda^\gamma(t)=\lim_{h\to 0^+} \lambda_h^\gamma(t)$, almost surely,  where
\[\lambda_h^\gamma(t)= \frac 1 h \E \left( N^\gamma_{t+h}-N^\gamma_t \big | \F_t\right).\]
This makes sense  if for instance $\lambda_h^\gamma(t)$ is uniformly bounded for any $t\geq 0$ and any $0\leq h\leq 1$ (see  formula (3.5) in Chapter~2 of \cite{bremaud}), which is our case as shown below. It is indeed not difficult to deduce that under these assumptions, by applications of Fubini and the dominated convergence theorem, 
\begin{align*}
 \E(N^\gamma_{t+s}-N^\gamma_t \big |\F_t)&=\E\left(\lim_{h\to 0^+}\int_t^{t+s} \frac 1 h (N^\gamma_{u+h}-N^\gamma_u)du\Big |\F_t\right)\\
&=\lim_{h\to 0^+}\E\left(\int_t^{t+s} \lambda_h^\gamma(u)du\Big |\F_t\right)=
 \E\left(\int_t^{t+s} \lambda^\gamma(u)du\Big |\F_t\right),
 \end{align*}
 which shows that $\int_0^{t} \lambda^\gamma(u)du$ is the $\F_t$-compensator of $N^\gamma_t$, i.e. $N^\gamma_t -\int_0^{t} \lambda^\gamma(u)du$ is a  $\F_t$-martingale.

Let us prove that $\lambda_h^\gamma(t)$ is uniformly bounded and $\lim_{h\to 0^+}  \lambda_h^\gamma(t) = \gamma(X_t)$. By the Markov property, for any $y\in \Et$,
\begin{align}\label{lambdah}
\frac 1h\E \left( N^\gamma_{t+h}-N^\gamma_t \big | X_t=y \right) &=  \frac 1h \E \left( N^\gamma_{h}\big | X_0=y\right)\nonumber\\
&=\frac 1h\P\left(N^\gamma_{h}=1\big | X_0=y\right) +\frac 1h \sum_{j\geq 2}j\P\left(N^\gamma_{h}=j\big | X_0=y\right).
\end{align}
On one hand, for any $j\geq 2$,
\begin{align*}
\P\left(N^\gamma_{h}=j\big | X_0=y\right) \leq \P\left(N_{h}=j\big | X_0=y\right)  \leq \P\left(N_{h}\geq j\big | X_0=y\right).
\end{align*}
To control this last term, we use the following lemma.
\begin{lem}\label{lemNumber}
Let  $N^* \sim \mathcal{P}(\alpha^* h)$ and $N_* \sim \mathcal{P}(\alpha_* h)$, where $\mathcal P(a)$ denotes the Poisson distribution with rate $a>0$.  Then for any $n\in\N$, $$\P(N^* \leq n) \leq \P(N_h \leq n) \leq \P(N_* \leq n).$$  
 \end{lem}
 
 \begin{proof}
 For any $t\geq 0$ and any $n$,
 \begin{align*}
  \P(N_t \geq n) = \P(T_n \leq t) & = \P(0 \leq T_n- T_{n-1} \leq t - T_{n-1})\\
  & = \E \left( \1_{t \geq T_{n-1}} \P\left( T_n- T_{n-1} \leq t - T_{n-1} | T_{n-1}, X_{T_{n-1}},Y^{(n-1)}\right)\right)\\
  & = \E \left( \1_{t \geq T_{n-1}} \left(1- e^{-\int_{0}^{t-T_{n-1}}\alpha(Y_{u}^{(n-1)}) du} \right)\right).
 \end{align*}
We deduce that
 \begin{equation*}
   \E\left(  \1_{t \geq T_{n-1}}\left(1-e^{-\alpha_*( t-T_{n-1})} \right)\right) \leq \P(N_t \geq n) \leq  \E\left(  \1_{t \geq T_{n-1}}\left(1-e^{-\alpha^*( t-T_{n-1})} \right)\right).
 \end{equation*}
Let  $e_*$ and $e^*$ be two random variables independent of $T_{n-1}$ distributed according to an exponential distribution with rate $\alpha_*$ and $\alpha^*$, respectively. The bounds above are nothing else than  $\P( e_* \leq t- T_{n-1})$ and $\P( e^* \leq t- T_{n-1})$ so that
\begin{equation}\label{YY}
  \P( e_* \leq t- T_{n-1}) \leq \P(N_t \geq n) \leq \P( e^* \leq t- T_{n-1}). 
 \end{equation}
We start from the lower bound to prove the first inequality of the lemma.  Denote by $\gamma_{j}$ the probability density function of a Gamma distribution with parameters $j$ and $\alpha_*$. We have that 
\[\P( e_* \leq t- T_{n-1}) = \int_{0}^{t} \P( T_{n-1} \leq t-s) \gamma_{1}(s) ds, \]
so for any $n$, 
 \[ \P(T_n \leq t) =   \P(N_t \geq n)  \geq  \int_{0}^{t} \P( T_{n-1} \leq t-s) \gamma_{1}(s) ds.\]
 Iterating this inequality, we obtain
 \begin{align*}
\int_{0}^{t} \P( T_{n-1} \leq t-s) \gamma_{1}(s) ds & \geq  \int_{0}^{t} \int_{0}^{t-s} \P( T_{n-2} \leq t-s-u ) \gamma_{1}(u) du  \gamma_{1}(s) ds\\
   & =  \int_{0}^{t} \P( T_{n-2} \leq t-z) \left(\int_{0}^{z} \gamma_{1}(z-s) \gamma_{1}(s) ds \right) dz\\
   & =  \int_{0}^{t}  \P( T_{n-2} \leq t-z) \gamma_{2}(z) dz
  \end{align*}
and  recursively 
 $$ \P(N_t \geq n)  \geq \int_{0}^{t}\P( T_{n-1} \leq t-s ) \gamma_{1}(s) ds \geq \int_{0}^{t} \gamma_{n}(z) dz.$$ 
This lower bound equals $\P(N_* \geq n)$ where $N_*\sim \mathcal P(\alpha_* t)$, which proves one inequality in Lemma \ref{lemNumber}. 
The other inequality is obtained similarly by starting from the upper-bound in \eqref{YY} and by replacing $\gamma_{j}$ with a Gamma distribution with parameters $j$ and $\alpha^*$. 
\end{proof}

It is easy to verify that this lemma remains true by replacing the probabilities in its statement by conditional probabilities, whereby   $\P\left(N_{h}\geq j\big | X_0=y\right)\leq\P(N^*\geq j)\leq (\alpha^*h)^{j}/j!$. Consequently
\[\frac 1h\sum_{j\geq 2}j\P\left(N^\gamma_{h}=j\big | X_0=y\right)\leq \frac 1h \sum_{j\geq 2} j \frac{(\alpha^*h)^{j}}{j!}\leq h(\alpha^*)^2 e^{\alpha^*h}.\]
On the other hand, 
\[\frac 1h\P\left(N^\gamma_{h}=1\big | X_0=y\right) = \frac 1h\P\left(N^\gamma_{h}=1,N_{h}=1\big | X_0=y\right) + \frac 1h\P\left(N^\gamma_{h}=1,N_{h}\geq 2\big | X_0=y\right).\]

The last term is lower than $\frac 1h\P\left(N_{h}\geq 2\big | X_0=y\right)$ which is less than $\frac 1h (\alpha^*h)^{2}/2$ by the same arguments as above.
 We shall finally prove that $\left(\frac 1h\P(N^\gamma_{h}=1,N_{h}=1\big | X_0=y)\right)_{h\leq 1}$ is uniformly bounded and that
 $\lim_{h \rightarrow 0} \frac 1h\P(N^\gamma_{h}=1,N_{h}=1\big | X_0=y) = \gamma(y).$
 Let us first consider the case where $\gamma=\beta$ and recall that $n(y)$ is the index $n$ such that $y\in\Et_n.$
 By \eqref{cdf tau} along with the continuity of $\alpha$ and $(Y_u^{(n(y))})_{u \geq 0}$,
 \begin{align}
  &\P\left(N^\beta_{h}=1,N_{h}=1\big | X_0=y\right) = \P\left(T_1 \leq h, \text{first jump is a birth}, T_2>h\big | X_0=y\right)\nonumber\\
  &= \int_{z \in \Et_{n(y)+1}}  \int_0^h\E\left[ \alpha( Y_s^{(0)}) K(Y_s^{(0)},dz) e^{-\int_0^s \alpha( Y_u^{(0)})du} e^{-\int_0^{h-s} \alpha( Y_v^{(1)})dv} \big | X_0=y\right] ds \label{control1naissance}\\
  &\leq h \alpha^*. \nonumber
 \end{align}
We obtain similarly for $\gamma=\delta$ and $\gamma=\alpha$
\begin{align*}
  &\P\left(N^\delta_{h}=1,N_{h}=1\big | X_0=y\right) = \P\left(T_1 \leq h, \text{first jump is a death}, T_2>h\big | X_0=y\right)\\
  &= \int_{z \in \Et_{n(y)-1}}  \int_0^h\E\left[ \alpha( Y_s^{(0)}) K(Y_s^{(0)},dz) e^{-\int_0^s \alpha( Y_u^{(0)})du} e^{-\int_0^{h-s} \alpha( Y_v^{(1)})dv} \big | X_0=y\right] ds\\
  &\leq h \alpha^*
  \end{align*}
 and
 \begin{align*}
  \P\left(N_{h}=1\big | X_0=y\right) &= \int_{z \in \Et}  \int_0^h\E\left[ \alpha( Y_s^{(0)}) K(Y_s^{(0)},dz) e^{-\int_0^s \alpha( Y_u^{(0)})du} e^{-\int_0^{h-s} \alpha( Y_v^{(1)})dv} \big | X_0=y \right] ds\\
  &\leq h \alpha^*.
 \end{align*}
So $\left(\frac 1h\P(N^\gamma_{h}=1,N_{h}=1\big | X_0=y)\right)_{h\leq 1}$ is uniformly bounded whatever $\gamma=\beta$ or $\gamma=\delta$ or $\gamma=\alpha$. 

Let us show that $\lim_{h \rightarrow 0} \frac 1h\P(N^\gamma_{h}=1,N_{h}=1\big | X_0=y) = \gamma(y)$ for $\gamma=\beta,$ the other cases $\gamma=\delta$ and $\gamma = \alpha$ being treated similarly.
 Using $(\ref{control1naissance}),$ the inequalities $|1-e^{-x}|\leq x$ for $x\geq0$,  $\int_0^{h-s} \alpha( Y_v^{(1)})dv \leq (h-s) \alpha^*,$  $\int_0^s \alpha( Y_u^{(0)})du \leq \alpha^* s$ and $\int_{z \in \Et_{n(y)+1}} \alpha( Y_s^{(0)}) K(Y_s^{(0)},dz) = \beta(Y_s^{(0)}),$ we obtain
 \begin{align*}
 & \left|\frac 1h\P(N^\beta_{h}=1,N_{h}=1\big | X_0=y) - \beta(y)\right| \\
 &\leq \int_{z \in \Et_{n(y)+1}} \frac 1h \int_0^h\E\left[ \alpha( Y_s^{(0)}) K(Y_s^{(0)},dz) e^{-\int_0^s \alpha( Y_u^{(0)})du} |1-e^{-\int_0^{h-s} \alpha( Y_v^{(1)})dv}|\big | X_0=y\right] ds\\
 &+ \int_{z \in \Et_{n(y)+1}} \frac 1h \int_0^h\E\left[ \alpha( Y_s^{(0)}) K(Y_s^{(0)},dz)|1- e^{-\int_0^s \alpha( Y_u^{(0)})du} | \big | X_0=y\right] ds\\
 &+ \frac 1h \int_0^h\E\left[| \beta( Y_s^{(0)}) - \beta(y) | \big | X_0=y\right] ds\\
 &\leq \frac{(\alpha^*)^2 h}{2} + \frac{(\alpha^*)^2 h}{2} + \frac 1h \int_0^h\E\left[| \beta( Y_s^{(0)}) - \beta(y) | \big | X_0=y\right] ds,
 \end{align*}
and the continuity of $\beta$ and $(Y_s^{(0)})_{s \geq 0}$ entail the result.

This concludes the proof that $\gamma(X_{t})$ is the $\F_t$-intensity of $N_t^\gamma$, meaning that  $N_t^\gamma-\int_0^t \gamma(X_{s})ds$ is a martingale with respect to $\F_{t}$. The extension of this result for the filtration $\F_{t^+}$ is a consequence of Lemma \ref{LemmaMartAugmented} in Appendix~\ref{sec:appendixB},  since the process $N_t^{\gamma} - \int_0^t \gamma( X_s)ds$ is a right-continuous $\F_t$-martingale satisfying $\sup_{0<\varepsilon <\eta}|N_{t+\varepsilon}^{\gamma} - \int_0^{t+\varepsilon} \gamma( X_s)ds |\leq N_{t+\eta}^{\gamma} + \alpha^*( t+\eta)$ and $N_{t+\eta}^{\gamma} \in L^1$ by Lemma \ref{lemNumber}.

A left-continuous (predictable) version of the intensity is $\gamma(X_{t^-})$. Since for any bounded measurable function $g$, the function $s\mapsto g(X_{s^-})$ is predictable with respect to $\F_{s^+}$, we deduce that $M_t=\int_0^t g(X_{s^-})[dN_s^{\gamma}- \gamma(X_{s})ds]$ is also a $\F_{t^+}$-martingale  (see Section~17.2 in \cite{Bass}), proving the second statement of the lemma. 
Finally,  writing $M_t(g)$ to stress the dependence of $M_t$ on $g$, we have that for any bounded measurable function $g$, $\E(M_t(g))=0$ and by (17.8) in \cite{Bass}, 
\begin{align*}
\E(M_t^2(g))=\E\left(\int_0^t g^2(X_{s^-})dN^\gamma_s\right)
=\E\left(M_t(g^2)\right)+\E\left(\int_0^t g^2(X_{s})\gamma(X_s)ds\right)=\E\left(\int_0^t g^2(X_{s})\gamma(X_s)ds\right).
\end{align*}

 \appendix

 \renewcommand{\thesection}{Appendix S-\Alph{section}}

 \section{Technical lemmas for jump-move processes}\label{sec:appendixA}

Recall that for a filtration $\F = (\F_t)_{t \geq 0}$, we define its right-continuous augmented filtration $(\F_{t^+})$ by $\F_{t^+} = \bigcap_{u>t} \F_u.$ For a stopping time $\tau$ we also define
\[\F_\tau = \{ A \in \mathcal{E}: A \cap \{\tau \leq t\} \in \F_t, \forall t\}\]
and 
\[\F_{\tau^+} = \{ A \in \mathcal{E}: A \cap \{\tau < t\} \in \F_t, \forall t\}.\]

We set  $\tau_n = \sum_{k \geq 0} \frac{k+1}{2^n}\1_{\frac{k}{2^n} < \tau \leq \frac{k+1}{2^n}},$ which defines a non-increasing sequence of stopping times that converges to $\tau.$

\begin{lem}\label{Lemmanumero0}
 
  Let $\F = (\F_t)_{t \geq 0}$ be a complete filtration and $\tau$ a $\F$-stopping time which is almost surely finite and satisfies $\P(\tau=t) = 0$ for every $t.$

  \begin{enumerate}
  
  \item[(i)] We have the equality $\F_\tau = \F_{\tau^+}.$ As a consequence
  \[F_\tau = \bigcap_{n \geq 1} \F_{\tau_n}.\]

  \item[(ii)] For every random variable $Z \in L^1(\P),$ we get the following almost sure and $L^1(\P)$ convergence:
    \[ \E[Z |\mathcal{F}_\tau] = \li \E[Z |\mathcal{F}_{\tau_n}].\]

  \end{enumerate}
  
  \end{lem}

\paragraph{Proof of Lemma \ref{Lemmanumero0}}

\begin{enumerate}

 \item[$(i)$] Recall that $\F_\tau \subset \F_{\tau^+},$ $\F_{\tau_n} \subset \F_{\tau_n^+}$ and $\F_{\tau^+} = \bigcap_{n\geq 1} \F_{\tau_n^+}$ is always true. 
 
 Set $A \in \F_{\tau^+}$ and $t\geq 0.$ By definition of $\F_{\tau^+},$ to prove that $A \in \F_\tau$ it is enough to show that $A \cap \{\tau=t\} \in \F_t,$ which follows from $\P(A \cap \{\tau=t\})=0$ and the fact that $(\F_t)_{t \geq 0}$ is complete. We get then $\F_\tau = \bigcap_{n \geq 1} \F_{\tau_n^+}.$ Since $\F_{\tau_n} \subset \F_{\tau_n^+},$  then $\bigcap_{n \geq 1} \F_{\tau_n} \subset \bigcap_{n \geq 1} \F_{\tau_n^+} = \F_\tau.$ For the other inclusion, for every $n\geq 1$, $\F_\tau \subset \F_{\tau_n}$ since $\tau \leq \tau_n.$ We therefore obtain $\F_\tau \subset \bigcap_{n\geq 1} \F_{\tau_n}.$

 \item[$(ii)$] Let $Z \in L^1(\P).$  For $n\leq 0,$ we set $\mathcal{G}_n=\F_{\tau_{-n}}$ and $Y_n = \E[Z|\mathcal{G}_n].$
  We get then for $n <0,$ $\mathcal{G}_n=\F_{\tau_{-n}} \subset \F_{\tau_{-n-1}} = \mathcal{G}_{n+1}$ and 
\begin{align*}
 \E[Y_{n+1} | \mathcal{G}_n ] &= \E[ \E[ Z | \F_{\tau_{-n-1}}] | \F_{\tau_{-n}} ]\\
 &= \E[ Z|\F_{\tau_{-n}}]\\
 &= Y_n,
\end{align*}
so
$(Y_n)_{n \in -\mathbb{N}}$ is a backward martingale indexed by $-\N.$ More, $\E[ |Y_n|] \leq \E[|Z|]$ so $\sup_{n \in -\N} \E[|Y_n|] <+\infty.$ As a consequence, $(Y_n)_{n \in -\mathbb{N}}$ converges almost surely and in $L^1(\P)$  to a random variable $Z_{\infty}$ when $n \rightarrow -\infty$ (see \cite{legall2013}). Hence $\lim_{n \rightarrow +\infty}\limits \E[ Z|\F_{\tau_n}] =Z_{\infty}$ almost surely and in $L^1(\P).$
Since $\F_{\tau_m} \subset \F_{\tau_n}$ for $m \geq n,$ $Y_m=\E[Z|\F_{\tau_m}]$ is $\F_{\tau_n}$-measurable for all $m \geq n$ so $Z_{\infty}$ is  $\F_{\tau_n}$-measurable for all $n$, meaning that it is $\bigcap_{n \geq 1}\limits \F_{\tau_n}$-measurable, which implies that  $Z_{\infty}$ is $\F_\tau$-measurable by $(i)$.

Now, for $A \in \F_\tau = \bigcap_{n \geq 1}\limits \F_{\tau_n},$
\begin{align*}
 \E[Z_{\infty}1_A ]&= \lim_{n \rightarrow +\infty} \E[\E[Z|\F_{\tau_n}]1_A] \quad \textrm{ by the } L^1(\P) \textrm{-convergence}\\
 &= \E[ Z\1_A] \quad \textrm{ because $A$ is } \F_{\tau_n} \textrm{-measurable},
\end{align*}
that is $Z_{\infty} = \E[Z|\F_\tau]$ almost surely. $\Box$

\end{enumerate}

\begin{lem}\label{LemmaMartAugmented}
 Let $M$ be a right-continuous $\F$-martingale. Assume that for each $t \geq 0,$ we may find $\eta >0$ such that $\sup_{0<\varepsilon <\eta} |M_{t+\varepsilon}| \in L^1.$ Then $M$ is a martingale with respect to the filtration $(\F_{t^+})_{t \geq 0}.$
\end{lem}

\paragraph{Proof of Lemma \ref{LemmaMartAugmented}}

For every $t\geq 0,$ $M_t \in L^1$ and $M_t$ is $\F_t$-adapted so it is $\F_{t^+}$-adapted. Let $s>t$ and $A \in \F_{t^+}.$ By definition of $\F_{t^+},$ for all $\varepsilon >0,$ $A \in \F_{t+\varepsilon}$ and for $\varepsilon <s-t,$
\begin{align*}
 \E[M_s\1_A] &= \E[ \E[M_s | \F_{t+\varepsilon}]\1_A]\\
 &= \E[M_{t+\varepsilon}\1_A]
\end{align*}
because $M$ is a $\F$-martingale. Then
\begin{align*}
 \E[M_s\1_A] &= \lim_{\varepsilon \rightarrow 0} \E[M_{t+\varepsilon}\1_A]\\
 &= \E[M_t\1_A]
\end{align*}
by dominated convergence. This yields $\E[M_s|\F_{t^+}] = M_t$ which concludes the proof. $\Box$

\section{Coupling  lemmas for birth-death-move processes}\label{sec:appendixB}

 Let $(X^{(1)}_t)_{t \geq 0}$ and $(X^{(2)}_t)_{t \geq 0}$ be two independent copies of the birth-death-move process $X$, with respective initial distributions $\phi_1$ and $\phi_2$, and jumping times $(T_j^{(1)})_{j \geq 0}$ and $(T_j^{(2)})_{j \geq 0}.$ We shall write $(\F_t^{(k)})_{t \geq 0}$ for the completed natural filtration of $X^{(k)}$, $k=1,2$, and  $\mathcal{G}_t=\sigma(\F_t^{(1)},\F_t^{(2)}) = \sigma(X_u^{(1)}, X_v^{(2)}, u \leq t, v \leq t ) \cup \mathcal{N}.$ Consider 
\[\tau = \inf \{ t >0 : (X_t^{(1)},X_t^{(2)}) = (\varnothing,\varnothing) \}.\]

\begin{lem}\label{Lemmaexpboundtau}
 Assume \eqref{existence}. There exist $a>0$ and $c>0$ (which do not depend on $\phi_1$ and $\phi_2$) such that for all $t>0,$
 \[\P_{\phi_1\times \phi_2}(\tau >t) \leq a e^{-ct}.\]
In particular, $\tau <+\infty$ $\P_{\phi_1\times \phi_2}$-almost surely.
 \end{lem}

 \paragraph{Proof of Lemma \ref{Lemmaexpboundtau}}
 
 Let us denote by $(Y_{k,u}^{(n)})_{u\geq 0}$, for $k=1,2$, the Markov process driving the motion of $X^{(k)}_t$ in $E_n$ and by $(T^{(k)}_j)_{j\geq 0}$ the jump times of $X^{(k)}_t$. 
Let $t_0>0$, we have
\begin{multline}\label{coupling1}
\P_{\phi_1\times\phi_2}(\tau\leq t_0|X_0^{(1)},X_0^{(2)})\\=\1_{\{(X_0^{(1)},X_0^{(2)})=(\varnothing,\varnothing)\}} +\sum_{(n_1,n_2)\neq (0,0)} \P_{\phi_1\times\phi_2}(\tau\leq t_0|X_0^{(1)},X_0^{(2)})\1_{\{n(X_0^{(1)})=n_1,n(X_0^{(2)})=n_2\}}.
\end{multline}
On the event $\{n(X_0^{(1)})=n_1,n(X_0^{(2)})=n_2\}$ with $n_1,n_2\geq 1$,
\begin{align}\label{couplingtime}
\P_{\phi_1\times\phi_2}(\tau\leq t_0|X_0^{(1)},X_0^{(2)})\geq \P_{\phi_1\times\phi_2}\bigg(&\bigcap_{j=1}^{n_1}\{ T^{(1)}_j-T^{(1)}_{j-1}\leq \frac{t_0}{n_1}, X^{(1)}_{T^{(1)}_j}\in E_{n_1-j}\},\   T^{(1)}_{n_1+1}>t_0,\nonumber\\
& \bigcap_{j=1}^{n_2}\{T^{(2)}_j-T^{(2)}_{j-1}\leq \frac{t_0}{n_2}, X^{(2)}_{T^{(2)}_j}\in E_{n_2-j}\},\   T^{(2)}_{n_2+1}>t_0|X_0^{(1)},X_0^{(2)}\bigg)\nonumber\\
\geq \P_{\phi_1}\bigg(\bigcap_{j=1}^{n_1}&\{ T^{(1)}_j-T^{(1)}_{j-1}\leq \frac{t_0}{n_1}, X^{(1)}_{T^{(1)}_j}\in E_{n_1-j}\},\   T^{(1)}_{n_1+1}- T^{(1)}_{n_1}>t_0| X_0^{(1)}\bigg)\nonumber\\
\times \P_{\phi_2}\bigg(\bigcap_{j=1}^{n_2}&\{T^{(2)}_j-T^{(2)}_{j-1}\leq \frac{t_0}{n_2}, X^{(2)}_{T^{(2)}_j}\in E_{n_2-j}\},\   T^{(2)}_{n_2+1}-T^{(2)}_{n_2}>t_0|X_0^{(2)}\bigg),
\end{align}
where we have used the independence between $\phi_1$ and $\phi_2$. 
Each term of this product is treated similarly and we only detail the first one. On the event $n(X_0^{(1)})=n_1$ with $n_1\geq 1$,
\begin{align}\label{alldead}
&\P_{\phi_1}\bigg(\bigcap_{j=1}^{n_1}\{ T^{(1)}_j-T^{(1)}_{j-1}\leq \frac{t_0}{n_1}, X^{(1)}_{T^{(1)}_j}\in E_{n_1-j}\},\   T^{(1)}_{n_1+1}- T^{(1)}_{n_1}>t_0| X_0^{(1)}\bigg)\nonumber\\
=&\P\left(T^{(1)}_{n_1+1}- T^{(1)}_{n_1}>t_0|X^{(1)}_{T^{(1)}_{n_1}}=\varnothing\right) \E_{\phi_1}\left[ \P_{\phi_1}\left(T^{(1)}_1\leq \frac{t_0}{n_1}, X^{(1)}_{T^{(1)}_1}\in E_{n_1-1} |X^{(1)}_{0}, (Y_{1,u}^{(n_1)})_{u\geq 0}\right)| X_0^{(1)}\right]
\nonumber\\
&\prod_{j=2}^{n_1}\E_{\phi_1}\left[ \P_{\phi_1}\left(T^{(1)}_j-T^{(1)}_{j-1}\leq \frac{t_0}{n_1}, X^{(1)}_{T^{(1)}_j}\in E_{n_1-j} |X^{(1)}_{T^{(1)}_{j-1}}, (Y_{1,u}^{(n_1-j+1)})_{u\geq 0}\right)\bigg | A_{j-1}, X_0^{(1)} \right] 
\end{align}
where  for $j\geq1$,  $A_j=\bigcap_{i=1}^{j}\{ T^{(1)}_i-T^{(1)}_{i-1}\leq \frac{t_0}{n_1}, X^{(1)}_{T^{(1)}_i}\in E_{n_1-i}\}$. For any $j\geq 2$ and any $1\leq n\leq n^*$, on the event $n(X^{(1)}_{T^{(1)}_{j-1}})=n$, by definition of the process $(X_t)_{t\geq 0}$,
\begin{align*}
\P_{\phi_1}\bigg(T^{(1)}_j-T^{(1)}_{j-1}\leq \frac{t_0}{n_1}, X^{(1)}_{T^{(1)}_j}\in E_{n-1} & |X^{(1)}_{T^{(1)}_{j-1}}, (Y_{1,u}^{(n)})_{u\geq 0}\bigg) \\
&= \P_{\phi_1}\bigg(T^{(1)}_1\leq \frac{t_0}{n_1}, X^{(1)}_{T^{(1)}_1}\in E_{n-1} |X^{(1)}_{0}, (Y_{1,u}^{(n)})_{u\geq 0}\bigg)\\
&= \frac{\delta(Y_{1,T_1^{(1)}}^{(n)})}{\alpha(Y_{1,T_1^{(1)}}^{(n)})}(1-e^{-\int_0^{t_0/n_1}\alpha(Y_{1,u}^{(n)})du})\\
&\geq \frac{\delta_*}{\alpha^*}(1-e^{-\alpha_* t_0/n_1})\\
&\geq \frac{\delta_*}{\alpha^*}(1-e^{-\alpha_* t_0/n^*}),
\end{align*}
where $\delta_*=\min_{1\leq n\leq n^*} \delta_n.$ This inequality is uniform in $n$. Applying it in \eqref{alldead} along with the inequality $$\P\left(T^{(1)}_{n_1+1}- T^{(1)}_{n_1}>t_0|X^{(1)}_{T^{(1)}_{n_1}}=\varnothing\right)=e^{-\alpha(\varnothing)t_0}\geq e^{-\alpha^*t_0},$$ we get on the event $n(X_0^{(1)})=n_1$ with $n_1\geq 1$ that
\begin{align*}
\P_{\phi_1}\bigg(\bigcap_{j=1}^{n_1}\{ T^{(1)}_j-T^{(1)}_{j-1}\leq \frac{t_0}{n_1}, X^{(1)}_{T^{(1)}_j}\in E_{n_1-j}\},\   T^{(1)}_{n_1+1}- T^{(1)}_{n_1}>t_0| X_0^{(1)}\bigg)& \geq e^{-\alpha^*t_0} (\frac{\delta_*}{\alpha^*}(1-e^{-\alpha_* t_0/n^*}))^{n_1}\\
&\geq e^{-\alpha^*t_0} (\frac{\delta_*}{\alpha^*}(1-e^{-\alpha_* t_0/n^*}))^{n^*}.
\end{align*}
 Coming back to \eqref{couplingtime}, we deduce that on the event $\{n(X_0^{(1)})=n_1,n(X_0^{(2)})=n_2\}$ with $n_1,n_2\geq 1$
\[\P_{\phi_1\times\phi_2}(\tau\leq t_0|X_0^{(1)},X_0^{(2)})\geq \rho_{t_0}\]
for some $\rho_{t_0}>0$ that depends on $t_0$ but not on $n_1$, $n_2$,  $X_0^{(1)}$ and $X_0^{(2)}$. By a similar argument we obtain the same result on the event $\{n(X_0^{(1)})=n_1,n(X_0^{(2)})=n_2\}$ with $n_1\geq 1,n_2=0$ or $n_1=0,n_2\geq 1$. We conclude from \eqref{coupling1} that for any $t_0>0$,
\[\P_{\phi_1\times\phi_2}(\tau>t_0)\leq 1-\rho_{t_0}\]
where $\rho_{t_0}$ does not depend on $\phi_1$ and $\phi_2$. By a standard argument, see for instance \cite{lotwick1981}, we then deduce that 
\[\P_{\phi_1\times\phi_2}(\tau>t)\leq a e^{-ct}\]
for some $a>0$ and $c>0$ that do not depend on $\phi_1$ and $\phi_2$. $\Box$

\bigskip

Let us now consider for $k=1,2,$
\[ \tau_+^{(k)}=\inf \{ t >\tau : X_t^{(k)} \neq \varnothing \}.\]

 \begin{lem}\label{Lemmanumero1}
  For every positive $t$,
  \begin{align*}
   \P_{\phi_1 \times \phi_2} (\tau_+^{(1)} > t+\tau | \mathcal{G}_{\tau} ) &= \P( T_1^{(1)} > t | X_0^{(1)} = \varnothing)\\
   \P_{\phi_1 \times \phi_2} (\tau_+^{(2)} > t+\tau | \mathcal{G}_{\tau} ) &= \P( T_1^{(2)} > t | X_0^{(2)} = \varnothing)
  \end{align*}
thereby we have the following equalities in distribution, denoted by $\overset{\mathcal L}{=}$,
 \[ \left[(\tau_+^{(1)} - \tau) | \mathcal{G}_{\tau}\right] \overset{\mathcal L}{=} \left[(\tau_+^{(2)} - \tau) | \mathcal{G}_{\tau} \right]\overset{\mathcal L}{=} \left[T_1^{(1)} | (X_0^{(1)}=\varnothing)\right] \overset{\mathcal L}{=}  \left[T_1^{(2)} | (X_0^{(2)} = \varnothing)\right].\]
 \end{lem}

\paragraph{Proof of Lemma \ref{Lemmanumero1}}
 We only prove the first equality since the second one is obtained similarly. Given the two first identities, the last statement of the lemma is straightforward. We known by Lemma~\ref{Lemmaexpboundtau} that $\tau <+\infty$, so for any $t>0$
 \begin{equation}\label{ad_eg1}
  \P_{\phi_1 \times \phi_2}( \tau^{(1)}_+ > t+\tau |\mathcal{G}_{\tau}) = \sum_{j\geq 1}  \P_{\phi_1 \times \phi_2}( \tau^{(1)}_+ > t+\tau |\mathcal{G}_{\tau}) \1_{\tau=T_j^{(1)}} +  \P_{\phi_1 \times \phi_2}( \tau^{(1)}_+ > t+\tau |\mathcal{G}_{\tau}) \1_{\tau=T_j^{(2)}}.
  \end{equation}
On one hand, for any $j\geq 1$, since $\{\tau=T_j^{(1)}\} \in \mathcal{G}_{\tau} \cap  \mathcal{G}_{T_{j}^{(1)}}$,
 \begin{align*}
  \P_{\phi_1 \times \phi_2}( \tau^{(1)}_+ > t+\tau |\mathcal{G}_{\tau}) \1_{\tau=T_j^{(1)}}&= 
    \P_{\phi_1  \times \phi_2}( T_{j+1}^{(1)} >t+T_j^{(1)} |\mathcal{G}_{T_j^{(1)}}) \1_{\tau=T_j^{(1)}}\nonumber\\
    &=  \P_{\phi_1} (T_{1}^{(1)} >t | X_0^{(1)} = X^{(1)}_{T_j^{(1)}})\1_{\tau=T_j^{(1)}}  \end{align*}
 by construction of the process $X^{(1)}$, and so 
  \begin{align}\label{ad_eg2}
 \P_{\phi_1 \times \phi_2}( \tau^{(1)}_+ > t+\tau |\mathcal{G}_{\tau}) \1_{\tau=T_j^{(1)}}&= \P(T_{1}^{(1)} >t | X_0^{(1)} = \varnothing )\1_{\tau=T_j^{(1)}}.
 \end{align}
On the other hand
\begin{align}\label{ap_lim3}
 \P_{\phi_1 \times \phi_2}( \tau^{(1)}_+ > t+\tau |\mathcal{G}_{\tau}) \1_{\tau=T_j^{(2)}}&= \sum_{k \geq 1} 
  \P_{\phi_1 \times \phi_2}( \tau^{(1)}_+ > t+\tau |\mathcal{G}_{\tau}) \1_{\tau=T_j^{(2)}}  \1_{T_{k-1}^{(1)} \leq \tau < T_{k}^{(1)}}\nonumber\\
 &= \sum_{k \geq 1}  \P_{\phi_1 \times \phi_2}( T_k^{(1)} >t+T^{(2)}_j |\mathcal{G}_{T^{(2)}_j})\1_{\tau=T_j^{(2)}} \1_{T_{k-1}^{(1)} \leq \tau < T_{k}^{(1)}}.
 \end{align}
 
 Consider now $(T^{(2)}_{n,j})_{n\geq 1}$ a sequence of stopping times, decreasing to $T^{(2)}_j$ and taking values in $\{\frac{k}{2^n}, k\in\N\}$. We get from Lemma \ref{Lemmanumero0} that
 \begin{equation}\label{ap_lim1}
 \P_{\phi_1 \times \phi_2}( T_k^{(1)} >t+T^{(2)}_j |\mathcal{G}_{T^{(2)}_j})= \li \P_{\phi_1 \times \phi_2}( T_k^{(1)} >t+T^{(2)}_j |\mathcal{G}_{T^{(2)}_{n,j}})\end{equation}
 almost surely. 
Letting $\epsilon_n=\P_{\phi_1 \times \phi_2} (t+T^{(2)}_{j}<  T_k^{(1)} \leq t+T^{(2)}_{n,j} | \mathcal{G}_{T^{(2)}_{n,j}})$, we have \begin{align}\label{ap_lim2}
 \P_{\phi_1 \times \phi_2}( T_k^{(1)} >t+T^{(2)}_j |\mathcal{G}_{T^{(2)}_{n,j}})&= \P_{\phi_1 \times \phi_2}(T_k^{(1)} >t+T^{(2)}_{n,j} | \mathcal{G}_{T^{(2)}_{n,j}}) + \epsilon_n \nonumber\\
 &= \sum_{l \geq 0} \P_{\phi_1 \times \phi_2}(T_k^{(1)} >t+\frac{l}{2^n} | \mathcal{G}_{\frac{l}{2^n}}) \1_{T^{(2)}_{n,j}=\frac{l}{2^n}} + \epsilon_n\nonumber \\
  &= \sum_{l \geq 0} \P_{\phi_1}(T_k^{(1)} >t+\frac{l}{2^n} | \mathcal{F}^{(1)}_{\frac{l}{2^n}})\1_{T^{(2)}_{n,j}=\frac{l}{2^n}} + \epsilon_n.
 \end{align}
 Here we have used the fact that the event $B=\{T_k^{(1)} >t+\frac{l}{2^n}\}$ being independent of $\mathcal{F}^{(2)}_{\frac{l}{2^n}}$, and $\mathcal{F}^{(2)}_{\frac{l}{2^n}}$ being independent of $\mathcal{F}^{(1)}_{\frac{l}{2^n}}$, $\P_{\phi_1 \times \phi_2}(B|\mathcal{G}_{\frac{l}{2^n}})=\P_{\phi_1 \times \phi_2}(B|\mathcal{F}^{(1)}_{\frac{l}{2^n}})$.
Now, on the event $\{\tau=T_j^{(2)}\} \cap \{T_{k-1}^{(1)} \leq \tau < T_{k}^{(1)}\}$, since $l/2^n>T_{k-1}^{(1)}$ and $X_{T_{k-1}^{(1)}}^{(1)}=\varnothing$,
 \[ \P_{\phi_1}(T_k^{(1)} >t+\frac{l}{2^n} | \mathcal{F}^{(1)}_{\frac{l}{2^n}}) = \P_{\phi_1}(T^{(1)}_k - T^{(1)}_{k-1} > t + \frac{l}{2^n} - T^{(1)}_{k-1} |\mathcal{F}^{(1)}_{\frac{l}{2^n}}) = \P(T^{(1)}_1 >t | X_0^{(1)} = \varnothing)\1_{T^{(1)}_{k} > t\frac{l}{2^n}},\]
where the last equality comes from the fact that given $X^{(1)}_{T^{(1)}_{k-1}}=\varnothing,$ the distribution of $T_k^{(1)} - T^{(1)}_{k-1}$ is an exponential distribution with rate $\alpha(\varnothing)$ having the memoryless property. 
We get then
\[\li \sum_{l \leq 0} \P_{\phi_1}(T_k^{(1)} >t+\frac{l}{2^n} | \mathcal{F}^{(1)}_{\frac{l}{2^n}}) \1_{T^{(2)}_{n,j}=\frac{l}{2^n}} \1_{\tau=T_j^{(2)}} \1_{T_{k-1}^{(1)} \leq \tau < T_{k}^{(1)}} =\P(T^{(1)}_1 >t | X_0^{(1)} = \varnothing) \1_{\tau=T_j^{(2)}} \1_{T_{k-1}^{(1)} \leq \tau < T_{k}^{(1)}} .\]
On the other hand,  $\E[\varep_n] = \P_{\phi_1 \times \phi_2}( t+T^{(2)}_j < T^{(1)}_k \leq t + T^{(2)}_{n,j})$ tends to 0 by the dominated convergence theorem, so there exists a subsequent $\varphi(n)$ such that $\varep_{\varphi(n)}\to 0$ almost surely. 
From \eqref{ap_lim1}, taking the limit in \eqref{ap_lim2} for this subsequence, we deduce that 
\[  \P_{\phi_1 \times \phi_2}( T_k^{(1)} >t+T^{(2)}_j |\mathcal{G}_{T^{(2)}_j})\1_{\tau=T_j^{(2)}} \1_{T_{k-1}^{(1)} \leq \tau < T_{k}^{(1)}} = \P(T^1_{(1)} >t | X_0^{(1)} = \varnothing) \1_{\tau=T_j^{(2)}} \1_{T_{k-1}^{(1)} \leq \tau < T_{k}^{(1)}}\]
which in view of \eqref{ap_lim3} gives
\[ \P_{\phi_1 \times \phi_2}( \tau^{(1)}_+ > t+\tau|\mathcal{G}_{\tau}) \1_{\tau=T_j^{(2)}} = \P( T_1^{(1)} >t | X_0^{(1)}=\varnothing) \1_{\tau=T_j^{(2)}}.\]
Combining \eqref{ad_eg1}, \eqref{ad_eg2} and this last result concludes the proof. $\Box$

 \begin{lem}\label{Lemmanumero2}
  For every bounded measurable function $g$ and every $t\geq 0$,
     \begin{align*}
    \E_{\phi_1 \times \phi_2} [g(X^{(1)}_{t+\tau})|\mathcal{G}_{\tau}] &= \E[g(X^{(1)}_t) | X_0^{(1)}=\varnothing]\\
    \E_{\phi_1 \times \phi_2} [g(X^{(2)}_{t+\tau})|\mathcal{G}_{\tau}] &= \E[g(X^{(2)}_t) | X_0^{(2)}=\varnothing],
   \end{align*}
thereby 
\[ \E_{\phi_1 \times \phi_2} [g(X^{(1)}_{t+\tau})|\mathcal{G}_{\tau}] = \E_{\phi_1 \times \phi_2} [g(X^{(2)}_{t+\tau})|\mathcal{G}_{\tau}].\]
 In particular, for every $A \in \mathcal{E}$ and every $t \geq 0,$
 \[\P_{\phi_1 \times \phi_2}(X_t^{(1)} \in A, \tau \leq t) = \P_{\phi_1 \times \phi_2}(X_t^{(2)} \in A, \tau \leq t).\]
 \end{lem}

\paragraph{Proof of Lemma \ref{Lemmanumero2}} 
 Let $g$ be a bounded measurable function and $t \geq 0$. Denote by $P_t$ the transition function of  $(X_t)_{t\geq 0}$, i.e. for any $s,t\geq 0$, $P_tg(X_s)=\E(g(X_{s+t})|\F_s)$. Note that both processes $X^{(1)}$ and $X^{(2)}$ have the same transition function $P_t$. We only prove the first equality of the lemma since the rest can be verified similarly or is a straightforward consequence. 
  \begin{align*}
  \E_{\phi_1 \times \phi_2} [g(X_{t+\tau}^{(1)} )|\mathcal{G}_{\tau} ] &= \E_{\phi_1 \times \phi_2} [g(X_{t+\tau}^{(1)} ) \1_{t+\tau <T_+^{(1)}}|\mathcal{G}_{\tau} ] + \E_{\phi_1 \times \phi_2} [g(X_{t+\tau}^{(1)} ) \1_{t+\tau \geq T_+^{(1)}}|\mathcal{G}_{\tau} ]\\
  &= g(\varnothing) \P_{\phi_1 \times \phi_2}(T_+^{(1)} >t+\tau|\mathcal{G}_{\tau}) + \E_{\phi_1 \times \phi_2}[P_{t+\tau-T_+^{(1)}}g(X^{(1)}_{T_+^{(1)}}) \1_{t+\tau \geq T_+^{(1)}}|\mathcal{G}_{\tau}]\\
  &= g(\varnothing) \P_{{\phi_1 \times \phi_2}}(T_+^{(1)} >t+\tau|\mathcal{G}_{\tau}) + \int_{z \in E_1} \E_{\phi_1 \times \phi_2}[ P_{t+\tau-T_+^{(1)}}g(z)\1_{t+\tau \geq T_+^{(1)}}|\mathcal{G}_{\tau}]K( \varnothing, dz).
 \end{align*}
We get from Lemma \ref{Lemmanumero1} that 
\[ \E_{\phi_1 \times \phi_2} [g(X_{t+\tau}^{(1)} )|\mathcal{G}_{\tau} ] = g(\varnothing) \P(T_1^{(1)} >t|X_0^{(1)}=\varnothing) + \int_{z \in E_1} \E[ P_{t-T_1^{(1)}}g(z)\1_{t \geq T_1^{(1)}}|X_0^{(1)}=\varnothing]K( \varnothing, dz).\]
 Notice to conclude that 
  \begin{align*}
  \E[g(X_t^{(1)})| X_0^{(1)} = \varnothing] &= \E[g(X_t^{(1)})\1_{t < T_1^{(1)}}| X_0^{(1)} = \varnothing]+\E[g(X_t^{(1)})\1_{t \geq T_1^{(1)}}| X_0^{(1)} = \varnothing]\\
  &= g(\varnothing) \P(T_1^{(1)} >t|X_0^{(1)}=\varnothing) + \E[P_{t-T_1^{(1)}}g(X_{T_1}^{(1)}) \1_{t \geq T_1^{(1)}}| X_0^{(1)} = \varnothing]\\
  &= g(\varnothing) \P(T_1^{(1)} >t|X_0^{(1)}=\varnothing) + \int_{z \in E_1} \E[ P_{t-T_1^{(1)}}g(z)\1_{t \geq T_1^{(1)}}|X_0^{(1)}=\varnothing]K( \varnothing, dz). \ \Box
 \end{align*}

\bibliographystyle{chicago}
\bibliography{bibref}

 \end{document}